\documentclass[11pt]{article}

\usepackage{amsthm,amsmath,amssymb}
\usepackage{natbib}
\usepackage{multirow}
\usepackage[pdftex]{graphicx}
\usepackage{subfigure}
\usepackage{makecell}
\usepackage{booktabs}
\usepackage{array}
\usepackage{fullpage}
\usepackage{url}
\usepackage{algorithm}
\usepackage{algorithmic}
\usepackage{bm}
\usepackage{smile}
\usepackage{mathtools}
\usepackage{wrapfig}
\usepackage{lipsum}
\usepackage{mathrsfs}
\usepackage{dsfont}
\usepackage{titling}
\usepackage{epstopdf}

\usepackage{natbib}
\usepackage{multirow}

\usepackage[dvipsnames,svgnames,table]{xcolor}
\usepackage[colorlinks=true,
            linkcolor=blue,
            urlcolor=blue,
            citecolor=blue]{hyperref}

\providecommand{\norm}[1]{\|#1\|}

\newcommand*{\supp}{\mathrm{supp}}
\newcommand{\R}{\mathbb{R}}
\newcommand{\E}{\mathbb{E}}

\begin{document}

\title{\huge Active Subsampling for Measurement-Constrained M-Estimation of Individualized Thresholds with High-Dimensional Data}

\author{Jingyi Duan\thanks{Department of Statistics and Data Science, Cornell University, Ithaca, NY 14850, USA; e-mail: \texttt{jd2222@cornell.edu}.}~~~~Lehao Fu\thanks{Department of Statistics and Data Science, Cornell University, Ithaca, NY 14850, USA; e-mail: \texttt{lf449@cornell.edu}.}~~~~Yang Ning\thanks{Department of Statistics and Data Science, Cornell University, Ithaca, NY 14850, USA; e-mail: \texttt{yn265@cornell.edu}.}}

\date{\today}

\maketitle

\vspace{-0.5in}

\begin{abstract}

Measurement-constrained problems frequently arise in modern applications such as electronic health record studies. In such problems, despite the availability of large datasets, collecting labeled data can be highly costly or time-consuming, allowing only a small portion of the data to be labeled within a given budget. This raises a critical question: which data points are most beneficial to label given the budget constraint? We study this question in the context of estimating an optimal individualized threshold under a measurement-constrained M-estimation framework. 
In particular, our goal is to estimate a high-dimensional parameter $\btheta$ in a linear threshold $\btheta^T\bZ$ for a continuous variable $X$ such that the discrepancy between whether $X$ exceeds the threshold $\btheta^T\bZ$ and a binary outcome $Y$ is minimized. In the measurement-constrained setting, we propose a novel $K$-step active subsampling algorithm to estimate $\btheta$, which iteratively samples the most informative observations in the dataset and solves a regularized M-estimator. 
Our theoretical analysis reveals a sharp phase transition phenomenon with respect to $\beta$, the smoothness of the conditional density of $X$ given $Y$ and $\bZ$. In particular, for $\beta>(1+\sqrt{3})/2$, we show that the two-step algorithm (with $K=2$) yields an estimator with the parametric convergence rate $O_p((s \log d /N)^{1/2})$ in $l_2$ norm, where $d$ and $s$ are the dimension and sparsity of $\btheta$ respectively and $N$ is the label budget. The rate  is strictly faster than the minimax optimal rate $O_p((s \log d /N)^{\beta/(2\beta+1)})$ with $N$ i.i.d. samples drawn from the population, demonstrating the benefit of active sampling. For $1<\beta\leq (1+\sqrt{3})/2$, the estimator from the two-step algorithm is sub-optimal; running $K>2$ steps recovers the parametric rate. For $\beta\leq 1$, a super-parametric rate is obtained when $K$ is scaled properly with $N$. Furthermore, we formulate an  $N$-budget minimax framework for the measurement-constrained M-estimation problem and prove that our estimator is minimax rate optimal up to a logarithmic factor. We also develop Lepski's methods to achieve adaptation to the unknown
smoothness and sparsity respectively, and provide practical guidelines for the implementation of our algorithm. Finally, we demonstrate the superior performance of our method in simulation studies and apply the method to analyze a large diabetes dataset from 130 US hospitals.  

\end{abstract}


\section{Introduction}

In many applications, the scientific questions can be formulated as identifying and interpreting the optimal individualized threshold value for a continuous variable such that the discrepancy with a binary outcome is minimized. One prominent example is the so-called minimum clinically important difference (MCID), which has attracted increasing interests in medical research over the past decade. The MCID is defined as the smallest difference in post-treatment changes, such that a patient is considered experiencing a clinically meaningful improvement
if her/his change exceeds this value. Since introduced by \cite{jaeschke1989measurement}, the MCID has been widely used by clinicians and health policy makers to evaluate the clinical effectiveness of the treatment, because it is tailored to reflect the patient's satisfaction or the improvement of her/his health condition. More recently, to account for the population heterogeneity, it was suggested by \cite{hedayat2015minimum,zhou2020estimation} to incorporate individual patient’s clinical profile to construct the individualized MCID (iMCID).


Formally, let $X$ denote a continuous variable representing the measurement of post-treatment change, $Y$ denote a binary outcome in $\{-1,+1\}$, where $Y=+1$ if the patient's health condition is improved after receiving the treatment and $Y=-1$ otherwise, and $\bZ$ denote a vector of $d$-dimensional covariates such as the patient's demographic information. \cite{zhou2020estimation} defined the iMCID as a function $c(\bZ)$ which minimizes 
\begin{equation}\label{objective0}
\PP\left(X <  c(\bZ) \mid Y= 1\right) + \PP\left(X > c(\bZ) \mid Y= -1\right).
\end{equation}
In other words, $c(\bZ)$ is the optimal individualized threshold for $X$ which minimizes the disagreement between the estimated patient's health condition and the binary outcome $Y$. 
The optimization problem (\ref{objective0}) provides a unified framework for a wide range of applications, including the estimation of optimal policy in causal inference \citep{zhao2012estimating}, semiparametric binary choice models in econometrics \citep{manski1975maximum}, and high-dimensional classification in machine learning. We will explore the connections to these and other examples in greater detail in Section \ref{sec_application}.


From a practical standpoint, a linear structure on the threshold $c(\bZ)$ is favored for its transparency and ease of interpretation, particularly when dealing with high-dimensional covariates $\bZ$. In this paper, we assume that $c(\bZ) = \boldsymbol{\theta}^T \bZ$ for some high-dimensional parameter $\btheta$. Under these assumptions, we can reformulate (\ref{objective0}) as the following M-estimation problem  
\begin{equation}\label{objective}
    \btheta^* =  \underset{{\btheta}}{\operatorname{argmin}} ~R(\btheta), \text{\ where\ } R(\btheta)= \EE\left[\gamma(Y) L_{01}\{Y(X-\btheta^T\bZ)\}\right],
\end{equation}
$L_{01}(u)=\frac{1}{2}\{1-\operatorname{sign}(u)\}$ is the 0-1 loss, with $\operatorname{sign}(u)=1$ if $u \geq 0$ and $-1$ otherwise, and $\gamma(\cdot)$ is a user-specified weight function. While this model-free formulation offers greater generality and potential for accommodating model misspecifications, minimization of the empirical version of $R(\btheta)$ is  computationally NP-hard due to the 0-1 loss.  Additionally, it is well known that the M-estimator of (\ref{objective}) is non-regular, resulting in non-standard limiting distributions and rates of convergence \citep{kim1990cube}. Recently, \cite{feng2022nonregular,feng2024test} proposed a regularized M-estimation framework with a smoothed surrogate loss to estimate and make inference on the high-dimensional parameter $\btheta^*$. In particular, they demonstrated that the convergence rate for estimating $\btheta^*$ in $l_2$ norm is given by $(s \log d /N)^{\beta/(2\beta+1)}$, where $d$ and $s$ are the dimension and sparsity of $\btheta^*$ respectively, $N$ is the sample size, and $\beta\geq 1$ is the smoothness of the conditional density of $X$ given the response $Y$ and the covariates $\bZ$. With the slower-than-classic root-$N$ rate, they also established that the resulting estimator is minimax rate optimal up to a logarithmic factor.

Up to this point, all aforementioned methods rely on the assumptions that the observations are i.i.d. and that the estimation process does not influence the data collection process. However, in practice, these two processes are often intertwined, especially under the measurement-constrained setting, where $(X, \bZ)$ are observed for all samples, but the response variable $Y$ can be measured only for a small subset \citep{wang2017computationally,zhang2021optimal}. Such constraints arise whenever labeling is expensive or time-consuming. For instance, in EHR (electronic health records) studies, while the database may contain rich patient-level covariates $(X,\bZ)$, the gold-standard outcome $Y$ often requires manual chart reviews from medical experts. Due to very expensive cost, only a carefully sampled subset of patients can be labeled, and the sampling scheme clearly affects the accuracy of the corresponding M-estimator.




To address this issue, we propose a novel $K$-step active subsampling algorithm for estimating the high-dimensional parameter $\btheta^*\in\RR^d$ in (\ref{objective}) under the following measurement-constrained setting. Formally, assume that we have access to a very large unlabeled dataset (e.g., the EHR database) $D = \{X_i,\bZ_i\}_{i=1}^n$ with $n$ i.i.d. samples, where the outcome $Y_i$ is  unavailable. Let $N$ denote the label budget, that is the expected total number of samples we are allowed to select. Once a specific data point is sampled, we can observe the outcome $Y_i$ (e.g., via chart reviews in the EHR studies). Our goal is to design a computationally and statistically efficient interactive sampling and estimation procedure for $\btheta^*$, subject to the budget constraint $N$, in the scenario $N\ll n$ and $N\ll d$. Our proposed algorithm begins by uniformly sampling a set of independent data from $D$, and solving  a regularized M-estimator with a smoothed surrogate loss to construct an initial estimator for $\btheta^*$. We then iteratively use the estimator from the previous iteration to guide the selection of a new set of independent samples and solve the corresponding regularized M-estimator. Repeating this process $K$ times yields our final estimator $\hat\btheta_K$. 
The underlying principle behind our algorithm is that, as the algorithm iterates, the M-estimators lead to a sequence of intervals with decreasing lengths that contain the true threshold $\btheta^{*T}\bZ$ with high probability, and in return sampling data in the corresponding neighborhoods around the threshold can further improve the estimation accuracy of our M-estimators. 
Importantly, the probability of the $i$th data point being sampled only depends on $(X_i,\bZ_i)$ and the previously sampled data, making it applicable to the measurement-constrained setting. Computationally, by leveraging the smoothness of the surrogate loss, we can design gradient-based
algorithms to solve the regularized M-estimator at each iteration. As a result, our proposed  $K$-step active subsampling algorithm is computationally efficient.


To investigate the theoretical results of our estimators, we assume that the conditional density of $X$ given $Y$ and $\bZ$ satisfies the Hölder smoothness condition with parameter $\beta$. Our analysis reveals a sharp phase transition  with respect to $\beta$. In particular, for $\beta>(1+\sqrt{3})/2$, with a proper choice of tuning parameters, the two-step algorithm, i.e., our algorithm with only $K=2$ iterations, attains the parametric rate $\|\hat\btheta_K-\btheta^*\|_2=O_p((s \log d /N)^{1/2})$, where $d$ and $s$ are the dimension and sparsity of $\btheta^*$ respectively and $N$ is the label budget. The rate is strictly faster than the minimax optimal rate $O_p((s \log d /N)^{\beta/(2\beta+1)})$ with $N$ i.i.d. samples drawn from the population  \citep{feng2022nonregular}. In other words, the two-step algorithm requires less data to attain the same order of convergence rate. However, for $1<\beta\leq (1+\sqrt{3})/2$, the estimator from the two-step algorithm is sub-optimal. To achieve the same parametric rate $O_p((s \log d /N)^{1/2})$, we need to run at least $K =  \lceil \log_{\frac{\beta}{2\beta+1}}(1-\frac{\beta+1}{2\beta^2})\rceil+1$ iterations in our algorithm, where $K$ is strictly greater than 2 but is fixed and finite. For the last case $\beta\leq 1$, 
we can achieve a super-parametric rate $O_p((s \log d /N)^{\frac{1}{2\beta}})$  up to logarithmic factors with $K = \lceil \log_{2\beta+1}(\log N) \rceil$ iterations, where $K$ diverges slowly as $N$ tends to infinity. The distinct regimes reflect how quickly the sequence of estimators enters a fast convergence neighborhood of $\btheta^*$.


Moreover, we formulate a minimax framework to study the optimality of our estimators. Unlike the traditional minimax framework, the distribution of a generic estimator $\hat\btheta$ depends on the joint distribution $P$ of $(X,\bZ,Y)$ as well as the sampling distribution $Q$, where $P$ is determined by the nature but we have the freedom to choose $Q$. After introducing two proper classes $\mathcal{P}(\beta,s)$ for $P$ and $\mathcal{Q}_N(\mathcal{P}(\beta,s))$ for $Q$,  we define the $N$-budget minimax risk in  $l_2$ norm as 
$$
\inf_{Q \in \mathcal{Q}_N(\mathcal{P}(\beta,s))}
\inf_{\hat{\btheta}} \sup_{P \in \mathcal{P}(\beta,s)} \mathbb{E}_{P,Q}\|\widehat{\boldsymbol{\theta}}-\boldsymbol{\theta}^*(P)\|_2,
$$
where the supremum is only for the distribution $P$ in $\mathcal{P}(\beta, s)$, the inner infimum is over all possible estimators $\hat{\btheta}$ based on the observed data and the outer infimum is over all possible sampling distributions $Q$ in $\mathcal{Q}_N(\mathcal{P}(\beta,s))$, which contains necessary constraints on $Q$ such as the conditional independence assumptions for the sampling mechanism and the budget constraint. We prove in Theorem~\ref{lowerbound1} that the $N$-budget minimax risk for estimating $\btheta$ in $l_2$ norm is lower bounded by $(s \log (d/s) /N)^{\frac{1}{2(\beta\wedge 1)}}$, which implies that our proposed estimators are indeed rate optimal up to logarithmic factors in the measurement-constrained setting. Finally, we also develop Lepski's methods to achieve adaptation to the unknown
smoothness $\beta$ and sparsity $s$ respectively.

\subsection{Related Work}

Subsampling is an effective method for handling computational constraints when dealing with massive datasets. 
There's a large literature on subsampling algorithms for regression models, such as linear regression  \citep{drineas2011faster,ma2014statistical,wang2019information,raskutti2016statistical} and generalized linear models \citep{wang2018optimal}. Given the budget constraint, the goal is to construct an estimator based on the sampled data to approximate the  least squares or the maximum likelihood estimator from the entire dataset, and find the optimal subsampling weight by minimizing the asymptotic variance. The similar idea has been extended to deal with the measurement-constrained problems \citep{wang2017computationally,zhang2021optimal}. More recently,  \cite{zrnic2024active} proposed to use a machine learning model to identify which data points are most beneficial to label, and then find the optimal sampling weight by minimizing the variance of the estimator or classification uncertainty.  
For binary data, the case-control subsampling is  considered by \cite{fithian2014local} among many others.

In recent years, substantial research has focused on adaptive experimental design, often aimed at efficiently estimating average treatment effects. 
i\cite{hadad2021confidence} considered how to construct confidence intervals for the average treatment effect with adaptively collected data. A recent overview of adaptive design is given by \cite{perera2020inventory}.

All aforementioned works share the similarity that the data are collected adaptively to improve asymptotic efficiency of the standard root-$N$ statistical inference. However, our work focuses on the threshold estimation problem which is known as a non-regular problem with rate slower than root-$N$. Our goal is to design a subsampling procedure to accelerate the convergence rate. Thus, both the method and its theory are fundamentally different from the prior literature.

Another closely related area is active learning, see \cite{balcan2007margin, koltchinskii2010rademacher, balcan2013active, castro2008minimax, wang2016noise}, among many others. 
While our problem setup is similar to the pool-based active learning algorithms, our work is distinct from the active learning literature in both method and theory. Specifically, our algorithm iteratively solves regularized M-estimators with a smoothed surrogate loss via gradient-based methods, which is computationally efficient. However, most of the margin-based active learning algorithms such as 
\cite{balcan2007margin, wang2016noise} require to minimize the empirical 0-1 loss, which is computationally intractable especially in high-dimensional setting. In addition, when the smoothness parameter $\beta$ is greater than $(1+\sqrt{3})/2\approx 1.37$, our proposed algorithm employs a streamlined two-step process (i.e., $K=2$ iterations), which significantly enhances implementation simplicity and efficiency. In contrast, the pipeline of a typical active learning algorithm often involves iterative model updates (usually $\log N$ iterations) until a stopping criterion is met. In theory, the Tsybakov noise condition plays a pivotal role in deriving theoretical guarantees in the active learning literature. 
Technically, when establishing the theoretical guarantees, the stopping criteria and the number of total iterations in the active learning algorithm are determined to ensure compliance with the assumed Tsybakov noise condition.  In contrast, our analysis relies on the smoothness of the conditional density of $X$ given $\bZ$ and $Y$ rather than such noise conditions, and our estimator exhibits a sharp phase transition in $\beta$ that has no counterpart in the active learning literature. 
Consequently, the two lines of work adopt different algorithms, rely on different structures, and yield distinct theoretical results and proof strategies. 

This work is also related to semi-supervised inference, where we observe a set of labeled data together with a set of unlabeled data \citep{tony2020semisupervised,zhang2022high,azriel2022semi,deng2024optimal}. Under the assumption that the covariates in both labeled and unlabeled samples follow the same distribution, the labeled data can be viewed as being uniformly sampled from the combined dataset in a measurement-constrained setting. Therefore, in semi-supervised inference, the sampling scheme is fixed and known, and the goal is to develop an inference procedure with improved statistical efficiency by leveraging information from the unlabeled data. In contrast, our problem involves designing not only an estimation procedure but also a subsampling algorithm to attain estimation optimality.

\subsection{Organization of the paper}
The rest of this paper is organized as follows. Section~\ref{method} introduces our proposed active subsampling algorithm and the corresponding estimator. Section~\ref{sec_theory} derives upper bounds on the estimation error and establishes a matching minimax lower bound. Section~\ref{sec_practice} discusses the practical implementation and provides a data-driven version of the procedure. Simulation studies and a real data application are presented in Sections~\ref{sec_sim} and \ref{sec_data}, respectively.

\subsection{Notations}
We write $\ind\{\}$ for the indicator function. For any set $S$, we write $|S|$ for its cardinality.
For $\bv=\left(v_1, \ldots, v_d\right)^T \in \mathbb{R}^d$, we use $\bv_S$ to denote the subvector of $\bv$ with entries indexed by
the set $S$.
For $q = [1, \infty)$, $\|\bv\|_q=(\sum_{i=1}^d\left|v_i\right|^q)^{1 / q}$ and $\|\bv\|_0=\sum_{i=1}^d \ind\{v_i \neq 0\}$. For any $a, b \in \RR$, we write $a \vee b = \max\{a, b\}$ and $a \wedge b = \min\{a, b\}$. For any positive sequences $\{a_1, a_2, \ldots\}$ and $\{b_1, b_2, \ldots\}$, we write $a_n \lesssim b_n$ or $a_n = O(b_n)$ if there exists a constant $c$ such that $a_n \leq c b_n$ for any $n$, and $a_n \asymp b_n$ if $a_n \lesssim b_n$ and $b_n \lesssim a_n$. Let $\lfloor a\rfloor$ be the greatest integer strictly less than $a$, and $\lceil a\rceil$ be the smallest integer strictly greater than $a$. Let $\lambda_{\min}(M)$ and $\lambda_{\max}(M)$ be the smallest and largest eigenvalues of $M$. 
A random variable $X$ is called sub-Gaussian if there exists a positive constant $K$ such that $\mathbb{P}(|X|\geq t) \leq 2\exp (-t^2 / K^2)$ for all $t\geq 0$. The sub-Gaussian norm of $X$ is defined as $\|X\|_{\psi_2}=\inf \left\{c>0: \EE[\exp (X^2/c^2)] \leq 2\right\}$.
A vector $\bX \in \RR^d$ is a sub-Gaussian vector if the one-dimensional marginals $\bv^T\bX$ are sub-Gaussian for all $\bv \in \RR^d$, and its sub-Gaussian norm is defined as $\|\bX\|_{\psi_2}=\sup _{\|\bv\|_2=1}\left\|\bv^T \bX\right\|_{\psi_2}$.

\section{Proposed Method}\label{method}

\subsection{Background and Heuristics for Subsampling}\label{background}
In this section, we briefly review the regularized M-estimation approach proposed by \cite{feng2022nonregular} for estimating $\btheta^*$ in (\ref{objective}) and use it to motivate our subsampling strategy. For now, assume that we observe $n$ i.i.d copies of $(X,\bZ, Y)$. Recall that the risk function $R(\btheta)$ is defined in (\ref{objective}). While $R(\btheta)$ is typically a smooth function of the parameter $\btheta$, the empirical version $R_n(\btheta)=\frac{1}{n}\sum_{i=1}^n \gamma(Y_i) L_{01}\{Y_i(X_i-\btheta^T\bZ_i)\}$ is  non-smooth. Minimizing $R_n(\btheta)$ is computationally intractable when $d$ is large and yields the cube root asymptotics  \citep{kim1990cube}.

To address these challenges, \cite{feng2022nonregular} proposed to approximate the 0-1 loss by the following smoothed surrogate loss 
\begin{equation} \label{surrogate}
     L_{\delta}(u) = \int_{u/\delta}^\infty K(t)dt,
\end{equation}
 where $K(t)$ is a proper kernel function defined in Assumption \ref{asp1}, and $\delta>0$ is a bandwidth parameter. As the bandwidth $\delta\rightarrow 0$, we have $L_\delta(u)\rightarrow L_{01}(u)$ for any $u\neq 0$. Thus, it is intuitive to estimate $\btheta^*$ by the minimizer of the regularized smoothed empirical risk, $\hat\btheta_{iid}=\argmin \{R^n_{\delta,iid}(\btheta)+\lambda\|\btheta\|_1\}$, where 
$$
 R^n_{\delta,iid}(\btheta)=\frac{1}{n}\sum_{i=1}^n \gamma(Y_i) L_{\delta}(Y_i(X_i-\btheta^T\bZ_i) ),
$$
and $\lambda$ is a tuning parameter. While $ R^n_{\delta,iid}(\btheta)$ is still non-convex, \cite{feng2022nonregular} showed that the  entire solution
path for the lasso type estimator $\hat\btheta_{iid}$ can be  computed efficiently via the path-following algorithm. In addition, with a proper choice of $\delta$ and $\lambda$, the convergence rate of $\hat\btheta_{iid}$ is faster than the classic cube root rate. 
 
By the M-estimation theory, since $ R^n_{\delta,iid}(\btheta)$ is differentiable in $\btheta$, the gradient of $R^n_{\delta,iid}(\btheta)$ at $\btheta^*$, 
$$
\nabla R^n_{\delta,iid}(\btheta^*)=\frac{1}{n}\sum_{i=1}^n \gamma(Y_i) \frac{Z_iY_i}{\delta}K\Big(\frac{Y_i(X_i-\btheta^{*T}\bZ_i)}{\delta} \Big),
$$
together with some curvature conditions, determine the convergence rate of $\hat\btheta_{iid}$. A key observation that inspires our subsampling method is that, by the property of the kernel function $K(\cdot)$, the gradient assigns larger weight to observations whose linearized residual $X_i-\btheta^{*T}\bZ_i$ is closer to zero. These ``near-threshold" points are therefore the most informative for estimating $\btheta^*$, and should be prioritized for labeling in our subsampling algorithm.

 \begin{algorithm} 
    \caption{$\btheta \leftarrow$ $K$-step Active Subsampling}
    \label{algoK}
 \begin{algorithmic}
     \STATE \textbf{input:} $D = \{X_i,\bZ_i\}_{i=1}^n$, label budget $N$, the number of iterations $K$ 
      \STATE \textbf{parameter:} $ \{\lambda_k\}_{k=1}^K, \{b_k\}_{k=1}^{K-1}, \{\delta_k\}_{k=1}^K$ and $\{N_k\}_{k=1}^K$ with $\sum_{k=1}^K N_k=N$
      \STATE \hspace*{5mm} Randomly split $D$ into $K$ batches: $D_1, \cdots, D_K$, each with batch size $n/K$.
     \STATE \hspace*{5mm} Draw  data $(X_i, \bZ_i)$ from $D_1$ with probability  $c_{n,1}=N_1 K/n$. Acquire the label $Y_i$ for each sampled data and form the dataset $D_1^*=\{X_{i}, \bZ_{i}, Y_{i}\}_{R_i=1}$, where $(X_i,\bZ_i)\in D_1$.  
      \STATE \hspace*{5mm} $\hat{\btheta}_1 \leftarrow \argmin_{\btheta}\{R^{D_1}_{\delta_1}(\btheta) + \lambda_1\|\btheta\|_1\}$. 
       \FOR{$k=2 $  to $K$}
       \STATE \hspace*{5mm} Compute the active set: $S_k \leftarrow \Big\{ (X, \bZ): -b_{k-1} \leq \frac{X-\hat{\btheta}_{k-1}^T\bZ}{\sqrt{1+\|\hat{\btheta}_{k-1}\|_2^2}} \leq b_{k-1}\Big\}$.
       \STATE \hspace*{5mm} Given $(X_i, \bZ_i) \in S_k$, draw the data point $(X_i, \bZ_i)$ from   $D_k$ with probability $c_{n,k} = N_k K / \left(n \PP\left((X,\bZ) \in S_k\right)\right)$. Acquire the label $Y_i$ for each sampled data and form $D_{k}^*=\{X_{i}, \bZ_{i}, Y_{i}\}_{R_i=1}$, where $(X_i,\bZ_i)\in D_{k}$.
         \STATE \hspace*{5mm} $\hat{\btheta}_k \leftarrow \argmin_{\btheta}\{R^{D_k}_{\delta_{k}}(\btheta) + \lambda_k\|\btheta\|_1\}$.
     \ENDFOR
     \RETURN $\hat{\btheta}_K$
 \end{algorithmic}
    \end{algorithm}

\subsection{Active Subsampling Algorithm}



Now, let us consider the measurement-constrained setting. Recall that we observe a very large unlabeled dataset $D = \{X_i,\bZ_i\}_{i=1}^n$ with $n$ i.i.d. samples, where the outcome $Y_i$ is  unavailable. We seek to sample $N\ll n$ data points (on average) from the dataset $D$ and collect their outcomes to construct an optimal estimator of $\btheta^*$. 

We introduce a binary random variable $R_i$ with $R_i=1$ if  $(X_i,\bZ_i)$ is sampled and $R_i=0$ otherwise. Our active subsampling approach is outlined in Algorithm~\ref{algoK}. The algorithm runs for a total of $K$ iterations, where $K$ is to be specified later on. To keep the data distribution during each iteration identical to the original data $D$, we randomly divide $D$ into $K$ batches $D_1, \cdots, D_K$ with equal size $n/K$.
 In the first iteration, since there is no prior information on $\btheta^*$, we uniformly sample data from $D_1$ with probability $0<c_{n,1}<1$.  That is, for each $(X_i, \bZ_i) \in D_1$, $R_i$ is generated independently with probability 
\begin{equation}\label{step1prob}
     \PP(R_i=1) = c_{n,1}=\frac{N_1K}{n},
\end{equation}
where $N_1$ is the expected number of data points sampled in the first iteration.  Given the sampled dataset $D_1^*=\{X_{i}, \bZ_{i}, Y_{i}\}_{R_i=1}$, we then minimize the regularized smoothed empirical risk function to obtain
\begin{equation}\label{penalized}
    \hat{\btheta}_1 := \argmin_{\btheta}\{R^{D_1}_{\delta_1}(\btheta) + \lambda_1\|\btheta\|_1\},
\end{equation}
where $R^{D_1}_{\delta_1}(\btheta)$ is a special case of (\ref{empiricalriskMk}) with a bandwidth parameter $\delta_1$ and $\lambda_1 >0$ is a regularization parameter. In general, with the sampled data from the dataset $D_k$ with $|D_k| = n/K$, we define
\begin{equation}\label{empiricalriskMk}
    R^{D_k}_{\delta_k}(\btheta) = \frac{K}{n}\sum_{(X_i, \bZ_i) \in D_k}\gamma(Y_i) L_{\delta_k}(Y_i(X_i-\btheta^T\bZ_i) )R_i,
\end{equation}
for all $1 \leq k \leq K$. For the iteration $2 \leq k \leq K$, we define an active set
\begin{equation}\label{activeset}
     S_k := \left\{ (X, \bZ) \in \RR \times \RR^d: -b_{k-1} \leq \frac{X-\hat{\btheta}_{k-1}^T\bZ}{\sqrt{1+\|\hat{\btheta}_{k-1}\|_2^2}} \leq b_{k-1}\right\},
\end{equation} 
where $\hat{\btheta}_{k-1}$ is the estimator derived from the $(k-1)$th iteration, and $b_{k-1}>0$ is the tuning parameter controlling the size of $S_k$. Then we uniformly sample the data points in $D_k$ which fall within the active set $S_k$. Specifically, given $(X_i, \bZ_i) \in D_k$ and $\hat\btheta_{k-1}$, we generate $R_i$ from a Bernoulli distribution with 
\begin{equation}\label{sampling}
\PP(R_i=1 \mid  X_i, \bZ_i, \hat\btheta_{k-1})=c_{n,k}\cdot \ind\{(X_i,\bZ_i) \in S_k\},
\end{equation}
where $c_{n,k} = N_k K / \left(n \PP\left((X,\bZ) \in S_k\right)\right)$ enforces the label budget constraint $N_k$ for this iteration. The sampling mechanism implies that only points whose estimated residual $X_i-\hat\btheta_{k-1}\bZ_i$ is sufficiently close to 0 are potentially sampled, aligning with the intuition  in Section \ref{background}.  

Given how $R_i$ is generated, we can verify that the independence assumptions $(X_i,\bZ_i, Y_i) \perp \bar{H}_{i-1}$ and $R_i \perp Y_i \mid X_i, \bZ_i, \bar{H}_{i-1}$ hold, where $\bar{H}_{i-1}$ denotes all observed data right before we decide whether $(X_i,\bZ_i)$ is sampled or not. These two independence assumptions play an important role in the minimax lower bound; see Section \ref{sec_lower} for more detailed discussions.

Our $k$-step estimator is 
\begin{equation}\label{estimatork}
    \hat{\btheta}_k:=\argmin_{\btheta}\{R^{D_k}_{\delta_{k}}(\btheta) + \lambda_k\|\btheta\|_1\},
\end{equation}
where $R^{D_k}_{\delta_{k}}(\btheta)$ is defined in (\ref{empiricalriskMk}). Repeating this procedure $K$ times, we obtain our final estimator $\hat{\btheta}_K$. 

In our algorithm, the sampling probability $c_{n,k}$, which depends on $\PP\left((X,\bZ) \in S_k\right)$, is assumed to be known. In practice, provided the active set $S_k$ is given, we can estimate $\PP\left((X,\bZ) \in S_k\right)$ easily from a large amount of unlabeled data. Specifically, in Algorithm \ref{algoK}, we can instead randomly divide the dataset $D$ into $2K-1$ batches $D_{1}, D_{20}, D_{21}, \cdots, D_{K0}, D_{K1}$, where we  use $D_{k0}$ to compute an empirical estimator $\hat{p}_k$ of $\PP\left((X,\bZ) \in S_k\right)$,  and then construct a plug-in estimator $\hat{c}_{n,k} := N_k K/(n\hat{p}_{k})$ of $c_{n,k}$. 
We draw samples from $D_{k1}$ according to (\ref{sampling}) with $c_{n,k}$ replaced by $\hat{c}_{n,k}$  and compute the estimator in (\ref{estimatork}). The sample splitting strategy guarantees the desired independence  when deriving the theoretical properties of the estimators; see Section \ref{sec_est_active_set} for the theoretical results. 
Further computational details are given in Section \ref{sec_practice}.



\section{Theoretical Results}\label{sec_theory}

We first list the technical assumptions in Section \ref{sec_ass}, then derive the convergence rate of our estimator in Section \ref{sec_rate} and establish the minimax lower bound in Section \ref{sec_lower}. Without loss of generality, we set the weight function in (\ref{objective}) to  $\gamma(y)=1/\PP(Y=y)$. 

\subsection{Assumptions}\label{sec_ass}

\begin{assumption}\label{asp3} 
$\btheta^*$ is $s$-sparse with $\left\|\boldsymbol{\theta}^*\right\|_0 \leq s$
and $\|\btheta^*\|_2\leq C$ for some constant $C$.
\end{assumption}

Assumption~\ref{asp3} quantifies the properties of $\btheta^*$. In addition to the sparsity of $\btheta^*$, we impose the bound on $\|\btheta^*\|_2$, so that 
the threshold $\btheta^{*T}\bZ$ has the same scale as $X$. Technically, this condition is used to verify the restricted strong convexity (RSC) condition \citep{feng2022nonregular}. In particular, we give an example in Section \ref{sec_example} showing that the RSC condition (such as Assumption \ref{asp6}  below) fails if $\|\btheta^*\|_2\rightarrow\infty$. 



\begin{assumption}\label{asp4}
(i) There exists a constant $0<c<1/2$ such that $c \leq \mathbb{P}(Y=1) \leq 1-c$.\\
(ii) Assume that $\left|Z_{ij}\right| \leq M_n$ for any $1\leq i\leq n$ and $1\leq j\leq d$, where $M_n$ is allowed to increase with $n$ such that  
    \begin{equation}\label{eq_asp4}
    M_n \leq C \sqrt{\frac{ n \min_{1 \leq k \leq K}\delta_k c_{n,k} }{K\log d}}
    \end{equation} 
    for some constant $C$, where $K$ is the number of iterations,  $\delta_k$ is the bandwidth parameter in the $k$th iteration and 
    $c_{n,k}$ is defined in (\ref{sampling}).
    In addition, it holds that
\begin{equation}\label{sparseeigen2}
    \sup _{\|\boldsymbol{v}\|_0 \leq s^{\prime}} \frac{\boldsymbol{v}^T \mathbb{E}\left(\boldsymbol{Z} \boldsymbol{Z}^T \mid Y=y\right) \boldsymbol{v}}{\|\boldsymbol{v}\|_2^2} \leq  M_1, 
\end{equation}
for some constant $M_1>0$, where $s^{\prime}=C s$ for some sufficiently large constant $C$. \\
(iii) $\bZ$ given $Y=y$ is a sub-Gaussian vector with a bounded sub-Gaussian norm. 
\end{assumption}

Part (i) ensures that the weight function $\gamma(y)=1/\PP(Y=y)$ is bounded away from infinity. For part (ii), if each component of $\bZ$ is sub-Gaussian with bounded sub-Gaussian norm, $\max_{1\leq i\leq n, 1 \leq j \leq d}\left|Z_{ij}\right| \leq M_n$ holds with high probability with $M_n \asymp (\log (d\vee n))^{1/2}$. Given the choices of $\delta_k$ and $c_{n,k}$ in {Theorem~\ref{beta>beta_*}}, (\ref{eq_asp4}) reduces to 
$M_n \leq C  \sqrt{\frac{N}{  K \log d}}
$, 
which is a mild condition provided $N$ is large enough. Furthermore, (\ref{sparseeigen2}) controls the maximal sparse eigenvalues of  $\mathbb{E}(\boldsymbol{Z} \boldsymbol{Z}^T \mid Y=y)$, we refer to \cite{buhlmann2011statistics} for the detailed discussion. Finally, part (iii) is used to provide a sharp bound for the plug-in error of $\hat\btheta_{k-1}$ in the active set $S_k$. This condition can be removed with the price of obtaining a sub-optimal rate for our estimator; see Section \ref{sec_theory_weaker} for the detailed result. 


The following definition and assumption quantify the smoothness of the conditional density of $X$ given  $Y$ and $\bZ$.

\begin{definition}\label{densityclass}
Let $l=\lfloor\beta\rfloor$ be the greatest integer strictly less than $\beta$. We say $P \in \mathcal{P}\left(\beta,L\right)$ if the conditional density $f(x \mid y, \boldsymbol{z}) $ of $ X $ given $(Y, \boldsymbol{Z})$ is $l$ times differentiable w.r.t $x\in [\btheta^{*T} \bz-c, \btheta^{*T} \bz+c]$ for an arbitrarily small constant $c>0$, and satisfies
  \begin{equation}\label{tsybakov}
      \left|f^{(l)}\left(\btheta^{*T}
      \bz +\Delta \mid y, \boldsymbol{z}\right)-f^{(l)}\left( \btheta^{*T}
      \bz\mid y, \boldsymbol{z}\right)\right| \leq L\left|\Delta\right|^{\beta-l},
  \end{equation}
for any $|\Delta|\leq c, y \in\{-1,1\}, \boldsymbol{z} \in \mathbb{R}^d$ and some constant $L>0$. 
\end{definition}


\begin{assumption}\label{asp5}
We assume $P \in \mathcal{P}\left(\beta,L\right)$, where $\beta>0$ and $L>0$ are constants. In addition,
\begin{equation}\label{eq_upper_density}
    \sup_{x \in \cX, y \in\{-1,1\}, \boldsymbol{z} \in \cZ} f(x \mid y, \boldsymbol{z})<p_{\max }<\infty,
\end{equation}
 and there exists a set $\mathcal{G} \in \RR^d$ such that $\PP(\bZ \in \mathcal{G})\geq C$ for some constant $0<C\leq 1$ and
\begin{equation}\label{eq_lower_density}
    \inf_{x \in B(\btheta^{*T}\bz, \epsilon_n), \boldsymbol{z} \in \mathcal{G}} f(x \mid \boldsymbol{z}) \geq p_{\min}>0,
\end{equation}
    where $\cX$ and $\cZ$ are the support sets of $X$ and $\bZ$, {$ B(\btheta^{*T}\bz, \epsilon_n) := \{x\in\cX : |x - \btheta^{*T}\bz| \leq \epsilon_n \}$, $\epsilon_n = C \max_{2\leq k\leq K}b_{k-1}$} for some constant $C$ large enough,
    and 
     $p_{\max}, p_{\min}$ are positive constants.
\end{assumption}

Here, $f(x \mid y, \bz)$ is assumed to belong to a $\beta$-smooth Hölder class at the point $x = \btheta^{*T} \bz$. Differentiability of $f(x \mid y, \bz)$ is required only in a small neighborhood of $x = \btheta^{*T} \bz$.
In addition, $f(x \mid y, \boldsymbol{z})$ is assumed to be upper bounded by a constant. To guarantee that the active set $S_k$ receives enough probability mass, we further require that on some region $\mathcal{G}$ of $\bz$, the density
$f(x \mid \bz)$ is lower bounded by some constant for any $\bz \in \mathcal{G}$ and $x\in B(\btheta^{*T}\bz, \epsilon_n)$.


\begin{assumption} \label{asp1}
Assume that $K(t)$ is a proper kernel of order $l=\lfloor\beta\rfloor$ with bounded support, where $\beta$ is the smoothness parameter in Assumption \ref{asp5}. That is $K(t)$ satisfies $K(t)=K(-t)$, $|K(t)| \leq K_{\max }<\infty$, $\int K(t) d t=1$, $\int K^2(t) d t<\infty$,  $ \int t^j K(t) d t=0, \forall j=1, \ldots, l$, and $\int |K(t)||t|^q dt$ is bounded by a constant for any $q \in [l,l+1]$.
\end{assumption}

As in classical non-parametric estimation, we adopt a kernel of order $l$ to control the higher order bias of the gradient of the smoothed empirical risk $\EE(\nabla  R^{D_k}_{\delta_k}(\btheta^*) \mid \hat{\btheta}_{k-1})$; see  Proposition \ref{prop2}. For clarity we present our theoretical results for the kernel with bounded support. With minor changes to the proofs, our results can be extended to kernels with mild tail conditions such as Gaussian kernels. The detailed results are deferred to Section \ref{sec_theory_weaker}.

\begin{assumption}\label{asp6}
There exists a sequence of sets $\Omega_{k-1}$, such that $\btheta^*\in \Omega_{k-1}$ and the following restricted strong convexity (RSC) and restricted smoothness (RSM) conditions hold for $R_{\delta_k}^{D_k}\left(\boldsymbol{\theta}\right)$ over sparse vectors in $\Omega_{k-1}$, for $1\leq k\leq K$. That is, uniformly over $1\leq k\leq K$ and  $\boldsymbol{\theta}, \boldsymbol{\theta}^{\prime} \in \Omega_{k-1}$ with  $(\left\|\boldsymbol{\theta}^{\prime}\right\|_0\vee \|\boldsymbol{\theta}\|_0) \lesssim s$, we have
    \begin{equation}\label{lowerrho}
R_{\delta_k}^{D_k}\left(\boldsymbol{\theta}^{\prime}\right) \geq R_{\delta_k}^{D_k}(\boldsymbol{\theta})+\nabla R_{\delta_k}^{D_k}(\boldsymbol{\theta})^T\left(\boldsymbol{\theta}^{\prime}-\boldsymbol{\theta}\right)+\frac{1}{2} \rho^{-}_{n,k}\left\|\boldsymbol{\theta}^{\prime}-\boldsymbol{\theta}\right\|_2^2,
\end{equation}
and 
\begin{equation}\label{upperrho}
R_{\delta_k}^{D_k}\left(\boldsymbol{\theta}^{\prime}\right) \leq R_{\delta_k}^{D_k}(\boldsymbol{\theta})+\nabla R_{\delta_k}^{D_k}(\boldsymbol{\theta})^T\left(\boldsymbol{\theta}^{\prime}-\boldsymbol{\theta}\right)+\frac{1}{2} \rho^{+}_{n,k}\left\|\boldsymbol{\theta}^{\prime}-\boldsymbol{\theta}\right\|_2^2,
\end{equation}
where, for some positive constants $C_1\leq C_2$,
\[
(\rho^{-}_{n,k}, \rho^{+}_{n,k}) =
\begin{cases}
(C_1 c_{n,k}, \; C_2 c_{n,k}), & \text{if } \beta \geq 1, \\[6pt]
(C_1 \delta_k^{\beta-1} c_{n,k}, \; C_2 \delta_k^{\beta-1} c_{n,k}), & \text{if } 0 < \beta < 1.
\end{cases}
\]
\end{assumption}


The RSC and RSM conditions are commonly used to analyze the statistical rate and  computational guarantees for non-convex optimization problems in high-dimensional regression. Similar conditions have been discussed extensively in the literature; see \cite{buhlmann2011statistics}. In our context, we require that the smoothed empirical risk $R_{\delta_k}^{D_k}\left(\boldsymbol{\theta}^{\prime}\right)$ is $\rho^{-}_{n,k}$-strongly convex  in (\ref{lowerrho}) and $\rho^{+}_{n,k}$-smooth in (\ref{upperrho}) when restricted to sparse vectors in $\Omega_{k-1}$, where both $\rho^{-}_{n,k}$ and $\rho^{+}_{n,k}$ scale with $c_{n,k}$ when $\beta \geq 1$ and $\delta_k^{\beta-1} c_{n,k}$ when $0 < \beta < 1$. The factor $c_{n,k}$ appears in the rate of $(\rho^{-}_{n,k}, \rho^{+}_{n,k})$ because only the selected data points with $R_i=1$ enter $R_{\delta_k}^{D_k}\left(\boldsymbol{\theta}\right)$. However, the assumed rate of $(\rho^{-}_{n,k}, \rho^{+}_{n,k})$ differs in terms of whether the smoothness parameter $\beta$ exceeds 1 or not. When $\beta> 1$, under mild  conditions,  the population risk function $R(\btheta)$ defined in (\ref{objective}) satisfies 
\begin{equation}\label{eq_population_curvature_beta}
\rho_- \le \lambda_{\min}\!\big(\nabla^2 R(\boldsymbol{\theta}^*)\big)
\le \lambda_{\max}\!\big(\nabla^2 R(\boldsymbol{\theta}^*)\big)
\le \rho_+,
\end{equation}
for some constants $\rho_+>\rho_->0$, so the smoothed empirical curvature is naturally of order $c_{n,k}$. When $0 < \beta \leq 1$, the function $R(\btheta)$ is not second-order differentiable and instead grows much faster than a quadratic function locally around $\btheta^*$. More precisely, we can show that, when $0 < \beta < 1$, there exists a small neighborhood $\bar\Omega$ of $\btheta^*$, such that for any $\btheta\in\bar\Omega$, 
\begin{equation}\label{eq_population_curvature_beta_2}
\rho_- \,\|\boldsymbol{\theta}-\boldsymbol{\theta}^*\|_2^{1+\beta}
\le R(\boldsymbol{\theta})-R(\boldsymbol{\theta}^*)
\le \rho_+ \,\|\boldsymbol{\theta}-\boldsymbol{\theta}^*\|_2^{1+\beta},
\end{equation}
for some constants $\rho_+>\rho_->0$. After smoothing, the smoothed empirical risk $R_{\delta_k}^{D_k}(\btheta)$ is always second-order differentiable, provided that $K'(\cdot)$ exists. However, the corresponding eigenvalues of $\nabla^2 R_{\delta_k}^{D_k}(\btheta)$ blow up with the rate $\delta_k^{\beta-1}\rightarrow \infty$ as $\delta_k\rightarrow 0$ (ignoring the effect of $c_{n,k}$), which is indeed driven by the geometric structure of $R(\btheta)$ in (\ref{eq_population_curvature_beta_2}).


Assumption~\ref{asp6} as well as (\ref{eq_population_curvature_beta}) and (\ref{eq_population_curvature_beta_2}) can be verified under specific distributional assumptions. The proof is deferred to Section \ref{rscexample}. 



\subsection{Convergence Rate of the Proposed Estimator}\label{sec_rate}

We first present a master theorem that characterizes the effect of subsampling  on the convergence rate of our estimators at each iteration.

\begin{theorem}\label{rate_k}
Under Assumptions~\ref{asp3}-\ref{asp6}, for any $1 \leq k \leq K $,  choose $\lambda_k\asymp  \sqrt{\frac{c_{n,k} K \log d}{n\delta_k}}$, and $\delta_k\asymp \left(\frac{K s  \log d}{nc_{n,k}}\right)^{1/(2\beta+1)}$. With probability greater than $1-2d^{-1}$, we have
\begin{equation}\label{eq_rate_k_1}
     \|\hat{\btheta}_1-\boldsymbol{\theta}^*\|_2 \lesssim 
     \left(\frac{s \log d }{N_1}\right)^{\frac{\beta \vee 1}{2\beta+1}},\ 
      \|\hat{\btheta}_1-\boldsymbol{\theta}^*\|_1 \lesssim 
         \sqrt{s}\left(\frac{ s \log d }{N_1}\right)^{\frac{\beta \vee 1}{2\beta+1}}.    
\end{equation}
For $2 \leq k \leq K$, if we further assume 
\begin{align}\label{eq_rate_k_3}
b_{k-1}\geq C\delta_k,~~\textrm{the event}~~\cW_{k-1}=\Big\{b_{k-1}\geq C\|\hat\btheta_{k-1}-\btheta^*\|_2 \sqrt{\log \frac{N_k}{s \log d}}\Big\}~~\textrm{holds}
\end{align}
for some large constant $C>0$ and there exists a large constant $\zeta$ such that $(\frac{s\log d}{N_k})^\zeta\lesssim \delta_k$ with $\delta_k=o(1)$, then with probability greater than $1-2d^{-1}$,
\begin{align}\label{eq_rate_k_2}
    \|\hat{\btheta}_k-\boldsymbol{\theta}^*\|_2 &\lesssim 
     \left(\frac{ \PP\left((X,\bZ)\in S_k \right)  s \log d }{N_k}\right)^{\frac{\beta \vee 1}{2\beta+1}},\nonumber\\
      \|\hat{\btheta}_k-\boldsymbol{\theta}^*\|_1 &\lesssim 
         \sqrt{s}\left(\frac{\PP\left((X,\bZ)\in S_k \right) s \log d }{N_k}\right)^{\frac{\beta \vee 1}{2\beta+1}},
\end{align}
where $N_k = \sum_{(X_i,\bZ_i)\in D_k}\EE\left( R_i\right)$ is the expected sample size at the $k$th iteration.    
\end{theorem}

In this theorem, the tuning parameters $\lambda_k$ and $\delta_k$ are determined by a bias-variance trade-off. Specifically, Proposition \ref{prop2} implies that the bias of the smoothed empirical risk $R^{D_k}_{\delta_k}(\btheta)$ due to kernel smoothing satisfies $$
\Big\|\EE(\nabla R^{D_k}_{\delta_k}(\btheta^*)|\hat\btheta_{k-1}) - \nabla R(\btheta^*)\Big\|_\infty \lesssim c_{n,k}\delta_k^\beta.
$$ 
Moreover, Proposition \ref{E_1} characterizes the stochastic error of $\nabla  R^{D_k}_{\delta_k}(\btheta^*)$, 
$$
\Big\|\nabla  R^{D_k}_{\delta_k}(\btheta^*)  - \EE(\nabla R^{D_k}_{\delta_k}(\btheta^*)|\hat\btheta_{k-1}) \Big\|_\infty \lesssim \sqrt{\frac{c_{n,k}K\log d }{n\delta_k}}.
$$
The bandwidth $\delta_k$ needs to be chosen to balance the bias $c_{n,k}\delta_k^\beta$ with the stochastic error $\sqrt{\frac{c_{n,k}K\log d }{n\delta_k}}$ multiplied by $\sqrt{s}$, which yields $\delta_k\asymp \Big(\frac{K s  \log d}{nc_{n,k}}\Big)^{1/(2\beta+1)}$. The shrinkage parameter $\lambda_k$ needs to dominate the stochastic error of $\nabla  R^{D_k}_{\delta_k}(\btheta^*)$ to exploit the sparsity of $\btheta$. In practice, the tuning parameters can be determined by a cross-validation approach  shown in Section \ref{sec_practice}. 

In the following, we comment on the convergence rate of $\hat\btheta_k$. For $k=1$, the non-standard convergence rate of $\hat{\btheta}_1$ in (\ref{eq_rate_k_1}) generalizes the results in \cite{feng2022nonregular}. In particular, when $\beta\geq 1$, the rate of $\hat{\btheta}_1$ in the $l_2$ norm is $O_p((\frac{s \log d }{N_1})^{\frac{\beta}{2\beta+1}})$, which recovers the result obtained by \cite{feng2022nonregular}  when data are i.i.d drawn from the population. However, their results do not cover the case $0<\beta<1$, as the population risk exhibits the non-quadratic behavior shown in (\ref{eq_population_curvature_beta_2}). We close this gap by showing that the rate of $\hat{\btheta}_1$ in the $l_2$ norm is $O_p((\frac{s \log d }{N_1})^{\frac{1}{2\beta+1}})$ for $0<\beta<1$. 

For $k\geq 2$, the convergence rate of $\hat{\btheta}_k$ in (\ref{eq_rate_k_2}) contains an additional factor $\PP((X,\bZ)\in S_k)$. By the definition of the active set $S_k$ in (\ref{activeset}), $\PP((X,\bZ)\in S_k)$ depends on the regularity of the joint distribution of $(X,\bZ)$, the accuracy of the estimator $\hat{\btheta}_{k-1}$ from the previous iteration and the choice of $b_{k-1}$. Under Assumption \ref{asp5} and when $\hat{\btheta}_{k-1}$ is close enough to $\btheta^*$, we can show that $\PP((X,\bZ)\in S_k)$ is proportional to $b_{k-1}$. If we allow $b_{k-1}\rightarrow 0$ and $N_k$ no smaller than $N_1$, the convergence rate of $\hat{\btheta}_k$ obtained via active subsampling is faster than $\hat{\btheta}_1$ obtained under uniform subsampling. Moreover, since (\ref{eq_rate_k_2}) holds for any choice of  $b_{k-1}$ provided (\ref{eq_rate_k_3}) is satisfied, there is a trade-off for determining the optimal choice of $b_{k-1}$ and the corresponding optimal rate of $\hat{\btheta}_k$. Our previous argument suggests that a smaller value of  $b_{k-1}$ is desirable to attain a faster rate of $\hat{\btheta}_k$ in (\ref{eq_rate_k_2}). However, condition (\ref{eq_rate_k_3}) prevents us from choosing the value of $b_{k-1}$ too small, see Remark \ref{remark_bk} below for interpretation of condition (\ref{eq_rate_k_3}). These together yield the optimal choice of $b_{k-1}$ and the corresponding optimal rate of $\hat{\btheta}_k$.

\begin{remark}\label{remark_bk}
Condition (\ref{eq_rate_k_3}) is inherited from Proposition \ref{prop2},  which  controls the approximation error of $\EE(\nabla R^{D_k}_{\delta_k}(\btheta^*)|\hat\btheta_{k-1})$ to  $\nabla R(\btheta^*)$ (which is 0 by definition). Since we show that $\PP((X,\bZ)\in S_k)\asymp b_{k-1}$, we can interpret $b_{k-1}$ as the size of the active set $S_k$. The first condition $b_{k-1} \geq C \delta_k$ in (\ref{eq_rate_k_3}) requires that the bandwidth $\delta_k$ should be chosen in a smaller order than the size of the active set. Otherwise, the surrogate risk $R_\delta(\cdot)$ is not shrunk to $R(\cdot)$ sufficiently fast. The event $\cW_{k-1}$ in (\ref{eq_rate_k_3}) is concerned with the stability of the active set $S_k$ with respect to the plug-in estimator $\hat\btheta_{k-1}$. Recall that the samples with $|X_i-\btheta^{*T}\bZ_i|\approx 0$ are the most informative for estimating $\btheta$. To keep them inside $S_k$ despite the plug-in error of $\hat\btheta_{k-1}$, we need $b_{k-1}$ to dominate $\|\hat\btheta_{k-1}-\btheta^*\|_2$ up to some logarithmic factors. 
\end{remark}

To obtain the optimal rate for our final estimator $\hat\btheta_K$ via Theorem \ref{rate_k}, we need to optimize the parameters $ \{\lambda_k\}_{k=1}^K, \{b_k\}_{k=1}^{K-1}, \{\delta_k\}_{k=1}^K$ and $\{N_k\}_{k=1}^K$ as well as the number of iterations $K$ in Algorithm \ref{algoK}. In the following, we show that our algorithm exhibits a phase transition at two critical values 1 and $(1+\sqrt{3})/2$ of the smoothness parameter $\beta$. We start from the theoretical result for  $\beta\in ((1+\sqrt{3})/2, +\infty)$.

\begin{theorem}[Optimal rate for $\beta > \frac{1+\sqrt{3}}{2}$]\label{beta>beta_*}
Assume that Assumptions~\ref{asp3}-\ref{asp6} hold, $K \geq 2$ and $\beta > \frac{1+\sqrt{3}}{2}$ are both fixed. We set $N_k= N/K$ for  $1 \leq k \leq K$ and 
\[
\delta_1 =c_{1}\left(\frac{s \log d}{N}\right)^{1/(2\beta+1)}, \ \lambda_1 =c_{2} \sqrt{\frac{N\log d }{n^2\delta_1}} ,
\]
\[
 \delta_k = c_{1} \left(\frac{s\log d}{N}\right)^{1/(2\beta)},
 \ \lambda_k = c_{2} \sqrt{\frac{N \log d}{n^2 b_{k-1}\delta_k}},
 \ b_{k-1} = c_{3} \left(\frac{s\log d}{N}\right)^{1/(2\beta)},     2\leq k \leq K,
\]
for some constants $c_{1}, c_{2}, c_{3}>0$. If
\begin{equation}\label{Nncondition.}
    N \lesssim (s \log d)^{\frac{1}{2\beta+1}} n^{\frac{2\beta}{2\beta+1}},
\end{equation}
and $s\log d=o(N)$, then with probability greater than $1-2K/d$, we have
\begin{equation}\label{eq_Kiteration_rate.}
   \|\hat{\btheta}_k-\boldsymbol{\theta}^*\|_2 \lesssim \left(\frac{s \log d}{N}\right)^{1/2},\  ~~\|\hat{\btheta}_k-\boldsymbol{\theta}^*\|_1 \lesssim \sqrt{s} \left(\frac{s \log d}{N}\right)^{1/2},
\end{equation}
uniformly over $2\leq k\leq K$, where $N$ is the pre-specified label budget. 
\end{theorem}

The tuning parameters $\delta_k$ and $\lambda_k$ are chosen in the same way as in Theorem \ref{rate_k}, where we plug in $c_{n,k} = \frac{KN_k}{n\PP\left( (X,\bZ) \in S_k \right)}$ and use the intermediate result $\PP\left( (X,\bZ) \in S_k \right)\asymp b_{k-1}$. In addition, when $\beta > \frac{1+\sqrt{3}}{2}$, we choose the smallest value of $b_{k-1}$ such that the condition $b_{k-1} \geq C \delta_k$ in (\ref{eq_rate_k_3}) holds, and we verify that the event $\cW_{k-1}$ in (\ref{eq_rate_k_3}) also holds with high probability for such choice of $b_{k-1}$.  The condition (\ref{Nncondition.}) guarantees that the active set $S_k$ contains enough unlabeled data to draw from. In this theorem, the budget $N$ is evenly divided across $K$ iterations so that the expected sample size at the final iteration is $N_K=N/K$.  Intuitively, it may be desirable to allocate more budget to compute $\hat\btheta_k$ as $k$ increases. However, this does not improve the convergence rate of $\hat\btheta_K$; see Section \ref{iterationexample} for a variant of Theorem \ref{beta>beta_*} in this case.

The convergence rate of $\hat\btheta_k$ at each iterations is shown in (\ref{eq_Kiteration_rate.}). For any $k\geq 2$, the rate can be viewed as the parametric rate for sparse models and is faster than the minimax optimal rate with $N$ i.i.d. data \citep{feng2022nonregular}, demonstrating the benefit of active sampling. Moreover, as the rate (\ref{eq_Kiteration_rate.}) stays the same for any $k\geq 2$, it suffices to only run $K=2$ iterations in Algorithm \ref{algoK}. 

\begin{theorem}[Optimal rate for $1 < \beta \leq  \frac{1+\sqrt{3}}{2}$]\label{beta^{**}< beta < beta^*}
    Assume that Assumptions~\ref{asp3}-\ref{asp6} hold, $K =  \lceil \log_{\frac{\beta}{2\beta+1}}(1-\frac{\beta+1}{2\beta^2})\rceil+1$ and $1 < \beta \leq  \frac{1+\sqrt{3}}{2}$ are fixed. We set  $N_k= N/K$ for  $1 \leq k \leq K$, 
\[
\delta_1 = c_{1}\left(\frac{s \log d}{N}\right)^{1/(2\beta+1)}, \ \lambda_1 =c_{2}  \sqrt{\frac{N \log d }{n^2\delta_1}},
\]
for $2\leq k \leq K-1$,
 \[
   \ b_{k-1} = c_3 
        \left(\log(\frac{N}{s \log d}) \right)^{\frac{(2\beta+1)(1-(\frac{\beta}{2\beta+1})^{k-1})}{2(\beta+1)}} \left(\frac{s \log d}{N}\right)^{\frac{\beta}{\beta+1}(1-(\frac{\beta}{2\beta+1})^{k-1})},
   \]
\[
     \delta_k  = c_1 \left(\frac{b_{k-1} s \log d}{N}\right)^{1/(2\beta+1)},\ 
 \lambda_k = c_{2} \sqrt{\frac{N \log d}{n^2 b_{k-1}\delta_k}},   
\]
and
\[
b_{K-1} = c_3 \left(\frac{s\log d}{N}\right)^{1/(2\beta)}, \delta_{K} = c_1 \left(\frac{s\log d}{N}\right)^{1/(2\beta)}, 
\lambda_K = c_{2} \sqrt{\frac{N \log d}{n^2 b_{K-1}\delta_K}},
\]
for some constants $c_{1}, c_{2}, c_{3}>0$. If
\begin{equation}\label{Nncondition1}
N\lesssim  \Big(\log(\frac{N}{s \log d})\Big)^{\frac{\beta+1}{2(2\beta+1)}} (s \log d)^{\frac{\beta}{2\beta+1}}n^{\frac{\beta+1}{2\beta+1}},
\end{equation}
(\ref{Nncondition.}) and $s\log d=o(N)$ hold, then with probability greater than $1-2K/d$, we have
\begin{equation}\label{eq_beta^{**}< beta < beta^*_rate}
   \|\hat{\btheta}_K-\boldsymbol{\theta}^*\|_2 \lesssim \left(\frac{ s \log d }{N}\right)^{1/2},\  ~~~\|\hat{\btheta}_K-\boldsymbol{\theta}^*\|_1 \lesssim \sqrt{s} \left(\frac{ s \log d }{N}\right)^{1/2}.
\end{equation}
\end{theorem}

Compared with Theorem \ref{beta>beta_*}, our estimator  behaves differently  when $1 < \beta \leq  (1+\sqrt{3})/2$.  While we still obtain the same parametric rate for $\hat{\btheta}_K$ in (\ref{eq_beta^{**}< beta < beta^*_rate}), we need to run at least $K =  \lceil \log_{\frac{\beta}{2\beta+1}}(1-\frac{\beta+1}{2\beta^2})\rceil+1$ iterations in Algorithm \ref{algoK} to attain (\ref{eq_beta^{**}< beta < beta^*_rate}) as opposed to only $K=2$ iterations in Theorem \ref{beta>beta_*}. Note that $\log_{\frac{\beta}{2\beta+1}}(1-\frac{\beta+1}{2\beta^2})$ is well defined and is strictly greater than 1 for $1 < \beta \leq  (1+\sqrt{3})/2$. Indeed, following a mathematical induction argument, we show that, for any $2\leq k\leq K-1$, 
\begin{equation}\label{eq_beta^{**}< beta < beta^*_rate_sub}
\|\hat{\btheta}_k - \btheta^*\|_2 \lesssim \left(\log(\frac{N}{s\log d})\right)^{\frac{\beta}{2(1+\beta)}(1-(\frac{\beta}{2\beta+1})^{k-1})} \Big(\frac{s\log d}{N}\Big)^{(1-(\frac{\beta}{2\beta+1})^k)\frac{\beta}{1+\beta}},
\end{equation}
with high probability. It can be verified from (\ref{eq_beta^{**}< beta < beta^*_rate_sub}) that stopping early produces estimators with sub-optimal rates. This phenomenon occurs because, for $1 < \beta \leq  (1+\sqrt{3})/2$, the rate of $\hat\btheta_1$ is not fast enough  so that we have to choose a larger value of $b_1$ to satisfy the event $\cW_1$ in  (\ref{eq_rate_k_3}). This remains the case, until $\hat{\btheta}_k$ enters a fast convergence region 
\begin{equation}\label{eq_fast}
\Theta_{fast,\beta}=\Big\{\btheta: \|\btheta-\btheta^*\|_2\lesssim \Big(\frac{Ks\log d}{N}\Big)^{1/(2\beta)}\sqrt{\log (\frac{Ks\log d}{N})}\Big\},
\end{equation} 
which is derived by matching the rate of $\hat{\btheta}_{k-1}$ with the order of $\delta_k$ up to a logarithmic factor (see the two conditions in (\ref{eq_rate_k_3})). Once $\hat{\btheta}_k\in\Theta_{fast,\beta}$, we only need one more iteration to achieve the parametric rate, mirroring the analysis of $K = 2$  in Theorem \ref{beta>beta_*}.


\begin{theorem}[Optimal rate for $0 < \beta \leq 1$]\label{beta < beta^**}
Assume that Assumptions~\ref{asp3}-\ref{asp6} hold, $0 < \beta \leq 1$ and $K = \lceil \log_{2\beta+1}(\log N) \rceil$. We set  $ N_k= N/K$ for  $1 \leq k \leq K$,  
  \[
 \delta_1 = c_{1}\left(\frac{K s \log d}{N}\right)^{1/(2\beta+1)}, \ \lambda_1 =c_{2} \sqrt{\frac{N K\log d }{n^2\delta_1}},
  \]
and for $2\leq k \leq K$,
\begin{align} \label{parameterchoice}
b_{k-1} =& c_3 \left(\log (\frac{N}{K s \log d})\right)^{\frac{2\beta+1}{4\beta}[1 -  (\frac{1}{2\beta+1})^{k-1}]} \left(\frac{K s \log d}{N}\right)^{\frac{1}{2\beta}[1 - (\frac{1}{2\beta+1})^{k-1}]}, \\
   \delta_k  =& c_1 \left(\frac{b_{k-1} K s \log d}{N}\right)^{1/(2\beta+1)},\ 
 \lambda_k = c_{2} \sqrt{\frac{N K \log d}{n^2 b_{k-1}\delta_k}},   
\end{align}
    for some constants $c_1, c_2, c_3> 0$. If
\begin{equation}
    \label{Nncondition1.}
 N \lesssim  \Big(\log(\frac{N}{K s \log d})\Big)^{\frac{1}{2}} (K s \log d)^{\frac{1}{2\beta+1}} n^{\frac{2\beta}{2\beta+1}},
\end{equation}
and $Ks\log d=o(N)$ hold, then with probability greater than $1-2 Kd^{-1}$,
\begin{align} 
 \|\hat{\btheta}_K - \btheta^*\|_2 &\lesssim    \left(\log(\frac{N}{K s \log d})\right)^{\frac{1}{4\beta}}
   \left(\frac{K s \log d}{N}\right)^{\frac{1}{2\beta}},\label{eq_beta < beta^**_rate1}\\
 \|\hat{\btheta}_K - \btheta^*\|_1 &\lesssim 
   \sqrt{s}  \left(\log(\frac{N}{K s \log d})\right)^{\frac{1}{4\beta}}
   \left(\frac{K s \log d}{N}\right)^{\frac{1}{2\beta}}.\label{eq_beta < beta^**_rate2}
\end{align}
\end{theorem}

For any $\beta \leq 1$, ignoring all the logarithmic factors, the convergence rate of $\hat{\btheta}_K$ in (\ref{eq_beta < beta^**_rate1}) reduces to $O_p((s/N)^{\frac{1}{2\beta}})$, which differs from the parametric rate $O_p(\sqrt{s/N})$ in Theorems \ref{beta>beta_*} and \ref{beta^{**}< beta < beta^*}. In addition, unlike the previous two cases (i) $\beta >  (1+\sqrt{3})/2$ and (ii) $1 < \beta \leq  (1+\sqrt{3})/2$, we have to run at least $K = \lceil \log_{2\beta+1}(\log N) \rceil$ iterations in Algorithm \ref{algoK} to attain the optimal rate, where $K$ has to grow with $N$, although, very slowly. To see the reason, we first apply a similar induction argument to show that, for $0<\beta\leq 1$ and $k\geq 2$, 
  \begin{align}\label{eq_beta < beta^**_rate_sub}
              \|\hat{\btheta}_k - \btheta^*\|_2
        \lesssim & \left(\log(\frac{N}{K s \log d})\right)^{\frac{1}{4\beta}(1-1/(2\beta+1)^{k-1})}
        \Big(\frac{K s \log d}{N}\Big)^{\frac{1}{2\beta}(1-1/(2\beta+1)^k)},
   \end{align}
with high probability. Therefore, for any fixed $\beta\leq 1$ and any fixed $k\geq 2$, the rate of $\hat\btheta_k$ in (\ref{eq_beta < beta^**_rate_sub}) is always slower than the bound in (\ref{eq_fast}), and therefore the sequence of estimators $\hat\btheta_1, \hat\btheta_2,...,\hat\btheta_k,...$ can only approach the fast convergence region $\Theta_{fast,\beta}$ but never enters it in finite $k$. Finally, we stop the algorithm at $K = \lceil \log_{2\beta+1}(\log N) \rceil$ steps when $\hat\btheta_K$ is close enough to the fast convergence region. More precisely, we choose  $K$ such that the convergence rate of $\hat{\btheta}_{K+1}$ matches with that of $\hat{\btheta}_K$. 

\begin{remark}
\cite{mallik2020m} studied a general M-estimation problem with multistage sampling procedures. In particular, they considered the following classification problem $d^*=\argmin \EE[L_{01}(Y(X-d))]$, where $L_{01}(\cdot)$ is the 0-1 loss. Under the assumption that $\eta(x)=\PP(Y=1|X=x)$ is continuously differentiable in a neighborhood of $d^*$, they proposed an isotonic regression approach to estimate $\eta(x)$ and then invert this function at $1/2$ to estimate $d^*$. Using this approach, their first-stage estimator $\hat d_1$ with $N$ i.i.d data sampled uniformly from the population has the rate $O_p(N^{-1/3})$. In the second stage, they sampled another $N$ i.i.d data in a zoomed-in neighborhood of $\hat d_1$ and the resulting  second-stage estimator $\hat d_2$ has the rate $O_p(N^{-(1+\gamma)/3})$ for any $\gamma<1/3$, whose  limiting distribution was also established. Their setting is similar to our  $\beta=1$ case. With $d,s$ being fixed, the proof of Theorem \ref{beta < beta^**} shows that our estimator $\hat\btheta_1$ has the rate $O_p((\log(\log N)/N)^{1/3})$ and $\hat\btheta_2$ has the rate $O_p((\log N)^{\frac{1}{6}}(\log(\log N)/N)^{\frac{4}{9}})$, which are comparable to the rates of $\hat d_1$ and $\hat d_2$, respectively. In their paper, they did not pursue the theoretical results of the estimators beyond two stages, whereas our results confirm that the convergence rate can be further accelerated to the near-parametric rate $O_p((\log N)^{\frac{1}{4}}(\log(\log N)/N)^{\frac{1}{2}})$ with $K = \lceil \log_3(\log N) \rceil$ stages. 

 \end{remark}

\begin{remark}
It is seen that the optimality of the proposed estimator $\hat\btheta_K$ critically depends on the choice of tuning parameters $ \{\lambda_k\}_{k=1}^K, \{b_k\}_{k=1}^{K-1}, \{\delta_k\}_{k=1}^K$ and the number of iterations $K$, which requires the knowledge of the smoothness parameter $\beta$ and the sparsity parameter $s$. To overcome this limitation, we propose Lepski’s methods \citep{lepskii1991problem} to achieve adaptation to the unknown smoothness $\beta$ and sparsity $s$ respectively. The detailed theoretical guarantees are presented in Section~\ref{sec_adaptation}.
\end{remark}

Finally, we briefly summarize our main conclusions: 
\begin{itemize}
\item [(i)] $\beta>(1+\sqrt{3})/2$. The first step estimator $\hat\btheta_1$ already lies in the fast convergence region $\Theta_{fast,\beta}$; one more iteration yields the parametric rate.
\item [(ii)] $1<\beta\leq (1+\sqrt{3})/2$.  After $K-1=\lceil \log_{\frac{\beta}{2\beta+1}}(1-\frac{\beta+1}{2\beta^2})\rceil$ iterations, we have $\hat\btheta_{K-1}\in\Theta_{fast,\beta}$ with high probability; the next step attains the parametric rate.
\item [(iii)] $0 < \beta \leq 1$. There does not exist a finite $k$ that places $\hat\btheta_k$ inside  $\Theta_{fast,\beta}$; $K=\lceil \log_{2\beta+1}(\log N) \rceil$ iterations are required to approach the fast convergence region.
\end{itemize}

\subsection{Minimax Lower Bound}\label{sec_lower}

For clarity, we write $R_P(\boldsymbol{\theta})$ for $R(\boldsymbol{\theta})$ in (\ref{objective}), emphasizing that the expectation is taken under the joint distribution $P$ of $(X, Y, \bZ)$. Similarly, we use $\btheta^*(P)$ to denote the unique minimizer of $R_P(\boldsymbol{\theta})$. Let $\mathcal{P}(\beta, L, p_{\min},p_{\max})$ denote the class of distributions which belong to the Hölder class $\mathcal{P}(\beta, L)$ in Definition~\ref{densityclass} and satisfy $\sup_{x \in \cX, y \in\{-1,1\}, \boldsymbol{z} \in \cZ} f(x \mid y, \boldsymbol{z})<p_{\max }$ and $\inf_{x \in \cX, \boldsymbol{z} \in \cZ} f(x \mid \boldsymbol{z}) \geq p_{\min}$, where $\cX$ and $\cZ$ are the support sets of $X$ and $\bZ$. We consider the following class of distributions
\begin{align}\label{def: prob measure class}
\mathcal{P}(\beta,s)=
\begin{cases}
\begin{aligned}
\big\{\, P\in &\mathcal{P}(\beta,L,p_{\min},p_{\max}) :
\|\boldsymbol{\theta}^*(P)\|_0 \le s,\;
\|\boldsymbol{\theta}^*(P)\|_2 \le C,\;
\eqref{sparseeigen2}\ \text{holds},\\
&\rho_- \le \lambda_{\min}\!\big(\nabla^2 R_P(\boldsymbol{\theta}^*(P))\big)
\le \lambda_{\max}\!\big(\nabla^2 R_P(\boldsymbol{\theta}^*(P))\big)
\le \rho_+ \,\big\}
\end{aligned}
& \text{if }\beta>1, \\[15pt]
\begin{aligned}
\big\{\, P\in &\mathcal{P}(\beta,L,p_{\min},p_{\max}) :
\|\boldsymbol{\theta}^*(P)\|_0 \le s,\;
\|\boldsymbol{\theta}^*(P)\|_2 \le C,\;
\eqref{sparseeigen2}\ \text{holds},\\
&\rho_- \,\|\boldsymbol{\theta}-\boldsymbol{\theta}^*(P)\|_2^{1+\beta}
\le R_P(\boldsymbol{\theta})-R_P(\boldsymbol{\theta}^*(P))\le \rho_+ \,\|\boldsymbol{\theta}-\boldsymbol{\theta}^*(P)\|_2^{1+\beta},
\\
&\qquad~~\text{for any }\|\boldsymbol{\theta}-\boldsymbol{\theta}^*(P)\|_2 \le r_n\big\}
\end{aligned}
& \text{if }0<\beta\le 1,
\end{cases}
\end{align}
where we treat $L,p_{\min},p_{\max}, C, M_1, \rho_-$ and $\rho_+$ as positive constants, and $r_n$ is a sequence converging to $0$. Thus $\cP(\beta,s)$ encodes both the sparsity of $\boldsymbol{\theta}^*(P)$ and the curvature of the population risk $R(\btheta)$ at $\boldsymbol{\theta}^*(P)$; see  (\ref{eq_population_curvature_beta}) for $\beta>1$ and (\ref{eq_population_curvature_beta_2}) for $0<\beta\leq 1$.  

For $1\leq i\leq n$, assume that $(X_i, \bZ_i, Y_i)$ are i.i.d. from the distribution $P$. We observe $O_i=(X_i, \bZ_i, Y_i)$ if $R_i = 1$ and $O_i=(X_i, \bZ_i)$ if $R_i = 0$. Here, the assumed observed data mechanism is more flexible than that considered in Section \ref{method}, as we can still observe the data $O_i=(X_i, \bZ_i)$ even if the $i$th data point is not sampled. Let $\bar{H}_{i-1}=\{O_1,...,O_{i-1}\}$ denote the collection of the first $i-1$ observed data, which can be viewed as the historical data prior to the $i$th sample. For simplicity, denote $\bar{H}_{i-1}=\emptyset$ for $i=1$. In terms of the sampling procedure, we allow $R_i$ to depend on $(X_i, \bZ_i)$ as well as the first $i-1$ observed data $\bar{H}_{i-1}$. We denote $Q_i:= \PP(R_i =1 \mid X_i, \bZ_i, \bar{H}_{i-1})$ and $Q=(Q_1,...,Q_n)$. Based on the  observed data $\{O_i\}_{i=1}^n$, our goal is to estimate the unknown parameter $\btheta^*(P)$. In this setting, the estimator $\hat{\btheta}$ is a measurable function of the observed data $(O_1,...,O_n)$, whose accuracy can be assessed by the following $l_q$ risk
$$
\mathbb{E}_{P,Q}\|\widehat{\boldsymbol{\theta}}(O_1,...,O_n)-\boldsymbol{\theta}^*(P)\|_q,
$$
where the expectation is taken under $(P,Q)$ which determines the distribution of $(O_1,...,O_n)$. For any given sampling distribution $Q$, the minimax risk for estimating  $\btheta^*(P)$ is defined as 
\[
\mathcal{M}_n(\mathcal{P}(\beta,s), Q) := \inf_{\hat{\btheta}} \sup_{P \in \mathcal{P}(\beta,s)} \mathbb{E}_{P,Q}\|\widehat{\boldsymbol{\theta}}(O_1,...,O_n)-\boldsymbol{\theta}^*(P)\|_q,
\]
where the supremum is only for the distribution $P$ in $\mathcal{P}(\beta, s)$ and the infimum is over all possible estimators $\hat{\btheta}$ based on the observed data  $\{O_i\}_{i=1}^n$.

Since $\mathcal{M}_n(\mathcal{P}(\beta,s), Q)$ depends on the sampling distribution $Q$, to  characterize the assumptions on $Q$, we define the following class. For any given distribution $P\in \mathcal{P}(\beta,s)$ and budget $N$, let
\begin{align}\label{sampling1}
    \mathcal{Q}_N(P) := \Big\{ (Q_1, \cdots, Q_n) &:
            \forall 1 \leq i \leq n, (X_i,\bZ_i, Y_i) \perp_{(P,Q)} \bar{H}_{i-1}, ~~R_i \perp_{(P,Q)} Y_i \mid X_i, \bZ_i, \bar{H}_{i-1},\nonumber\\
&~~\EE_P\Big(\sum_{i=1}^n Q_i\Big)\leq N,~~\sup _{\|\boldsymbol{v}\|_0 \leq s} \frac{\boldsymbol{v}^T \mathbb{E}_{P}\left(\sum_{i=1}^n Q_i\boldsymbol{Z}_i \boldsymbol{Z}_i^T\right) \boldsymbol{v}}{\|\boldsymbol{v}\|_2^2} \leq CN\Big\},
\end{align}
where $C$ is a positive constant. In (\ref{sampling1}), $(X_i,\bZ_i, Y_i) \perp_{(P,Q)} \bar{H}_{i-1}$ and $R_i \perp_{(P,Q)} Y_i \mid X_i, \bZ_i, \bar{H}_{i-1}$ formalize the assumptions on the data generating process for $R_i$. More precisely, $(X_i,\bZ_i, Y_i) \perp_{(P,Q)} \bar{H}_{i-1}$ is valid, since we generate $(R_1,...,R_{i-1})$ without using the future data $(X_i,\bZ_i, Y_i)$. In addition, the assumption $R_i \perp_{(P,Q)} Y_i \mid X_i, \bZ_i, \bar{H}_{i-1}$ is satisfied by the measurement-constrained sampling, that is $Y_i$ is not used to decide whether the $i$th data point is sampled or not. The assumption $\EE_P(\sum_{i=1}^n Q_i)\leq N$ is equivalent to $\EE_{P,Q}(\sum_{i=1}^n R_i)\leq N$, corresponding to our budget constraint. Similarly, the last assumption in (\ref{sampling1}) can be rewritten as 
\begin{align}\label{sampling2}
\sup _{\|\boldsymbol{v}\|_0 \leq s} \frac{\boldsymbol{v}^T \mathbb{E}_{P,Q}\left(\sum_{i=1}^n R_i\boldsymbol{Z}_i \boldsymbol{Z}_i^T\right) \boldsymbol{v}}{\|\boldsymbol{v}\|_2^2} \lesssim N,
\end{align}
which controls the maximal sparse eigenvalues of $\mathbb{E}_{P,Q}\left(\sum_{i=1}^n R_i\boldsymbol{Z}_i \boldsymbol{Z}_i^T\right)$. When $Z$ is univariate with bounded support, the condition (\ref{sampling2}) automatically holds, as $\sum_{i=1}^n\EE_{P,Q}\left(R_i Z_i^2\right)\lesssim \sum_{i=1}^n\EE_{P,Q}(R_i)\leq N$, where the last inequality follows from the budget constraint. When the support of $Z$ is unbounded, this assumption prevents sampling schemes that concentrate on extreme covariate values. They often correspond to high leverage points, that may indeed deteriorate the estimation accuracy.

Since the class $\mathcal{Q}_N(P)$ depends on the given distribution $P\in \mathcal{P}(\beta,s)$, we define $\mathcal{Q}_N(\mathcal{P}(\beta,s)):=\cap_{P \in \mathcal{P}(\beta,s)} \mathcal{Q}_N(P)$, as the set of sampling distributions that satisfy the conditions in (\ref{sampling1}) for all  $P\in \mathcal{P}(\beta,s)$.

\begin{remark}
    We give some examples of sampling distributions in $\mathcal{Q}_N(\mathcal{P}(\beta,s))$. 
    \begin{itemize}
    \item Bounded-probability sampling. Consider an arbitrary sampling distribution $Q$, which satisfies the independence assumptions, the budget constraint $\EE_P(\sum_{i=1}^n Q_i)\leq N$ in (\ref{sampling1}) and more importantly a bounded probability assumption $\max_{1\leq i\leq n} Q_i\leq \frac{CN}{n}$ for some constant $C>0$.  To show $Q\in \mathcal{Q}_N(P)$ for any $P\in \mathcal{P}(\beta,s)$, we note that for any $\|\bv\|_0\leq s$,
$$
\boldsymbol{v}^T \mathbb{E}_{P}\Big(\sum_{i=1}^n Q_i\boldsymbol{Z}_i \boldsymbol{Z}_i^T\Big)\boldsymbol{v}\leq \frac{CN}{n}\sum_{i=1}^n\EE_P(\bv^T\bZ_i)^2\leq CM_1N \|\bv\|_2^2,
$$
where we use the assumption $Q_i\leq \frac{CN}{n}$ in the first inequality, and the second inequality follows from $P\in \mathcal{P}(\beta,s)$ and the corresponding sparse eigenvalue assumption (\ref{sparseeigen2}). Thus, $Q\in \mathcal{Q}_N(\mathcal{P}(\beta,s))$ holds. The canonical example is the uniform sampling, where the data are sampled completely at random with $Q_i=\frac{N}{n}$. We note that, unlike the positivity assumption in the missing data literature, we do not impose any lower bound for $Q_i$, which allows $Q_i=0$ in our scenario. Intuitively, the bounded probability assumption rules out  extreme over-sampling of individual points relative to the uniform sampling.

\item Region-based sampling. Consider the following sampling distribution
\begin{equation}\label{eq_sampling_reg}
    Q_i = g_{3i}(\bar{H}_{i-1}) \ind\{f_i(\bZ_i, \bar{H}_{i-1}) - g_{1i}(\bar{H}_{i-1})\leq X_i \leq f_i(\bZ_i, \bar{H}_{i-1}) + g_{2i}(\bar{H}_{i-1})\} ,    
\end{equation}
where $ f_i: (\bZ_i, \bar{H}_{i-1}) \rightarrow \RR, g_{1i}, g_{2i}: \bar{H}_{i-1} \rightarrow \RR^+$, $g_{3i}: \bar{H}_{i-1} \rightarrow [0,1]$ are user specified functions and $\EE_P(\sum_{i=1}^n Q_i)\leq N$ holds. Apparently, the independence assumptions in (\ref{sampling1}) are  guaranteed by the data generating mechanism in (\ref{eq_sampling_reg}). Moreover, in Section \ref{sec_sampling} we show that (\ref{sampling2}) holds for any $P\in \mathcal{P}(\beta,s)$, which implies $Q\in \mathcal{Q}_N(\mathcal{P}(\beta,s))$. Recall that the proposed sampling mechanism in Section \ref{method} can be written as follows, for any $(X_i,\bZ_i)\in D_k$ with $k>1$,
    \begin{align*}
   Q_i= c_{n,k} \ind\{ 
    -b_{k-1} \sqrt{1+ \|\hat\btheta(\bar{H}_{i-1})\|_2^2} \leq 
    X_i - \hat\btheta(\bar{H}_{i-1})^T\bZ_i 
    \leq b_{k-1} \sqrt{1+ \|\hat\btheta(\bar{H}_{i-1})\|_2^2}
    \},
\end{align*}
where we write $\hat\btheta(\bar{H}_{i-1})$ for $\hat\btheta_{k-1}$ to emphasize that it is a function of the historical data. Hence, our sampling method in Section \ref{method} belongs to the class (\ref{eq_sampling_reg}). Indeed, the class (\ref{eq_sampling_reg}) covers the zoom-in approaches that oversample the data within a (small) region of $X$, complementing the bounded-probability sampling scheme. We also note that (\ref{eq_sampling_reg}) can be generalized to the sampling distributions within multiple regions of $X$, such as $Q_i= g^{(k)}_{3i}(\bar{H}_{i-1})$ if $X_i\in S_i^{(k)}$ and $Q_i=0$ otherwise, for some $k>1$, where $S_i^{(k)}=[f^{(k)}_i(\bZ_i, \bar{H}_{i-1}) - g^{(k)}_{1i}(\bar{H}_{i-1}),f^{(k)}_i(\bZ_i, \bar{H}_{i-1}) + g^{(k)}_{2i}(\bar{H}_{i-1})]$ are non-overlapping intervals.
 \end{itemize}

\end{remark}

Finally, we define the $N$-budget minimax risk for estimating $\btheta^*(P)$ as
\[
\mathcal{M}_n(\mathcal{P}(\beta,s), N) := \inf_{Q \in \mathcal{Q}_N(\mathcal{P}(\beta,s))}
\inf_{\hat{\btheta}} \sup_{P \in \mathcal{P}(\beta,s)} \mathbb{E}_{P,Q}\|\widehat{\boldsymbol{\theta}}(O_1,...,O_n)-\boldsymbol{\theta}^*(P)\|_q.
\]

\begin{theorem} \label{lowerbound1}
Assume that $s(\frac{\log(d/s) }{N})^{1/2}=o(1)$.  We have 
\[
\inf_{Q \in \mathcal{Q}_N(\mathcal{P}(\beta,s))}
\inf_{\hat{\btheta}} \sup_{P \in \mathcal{P}(\beta,s)}  \PP_{P,Q} \left[\Big\|\widehat{\boldsymbol{\theta}}(O_1,...,O_n)-\boldsymbol{\theta}^*(P)\Big\|_q \geq c s^{\frac{1}{q}-\frac{1}{2}}\left(\frac{s \log (d / s)}{N}\right)^{\frac{1}{2(\beta \wedge 1)}} \right]
    \geq c^\prime
\] 
    for $q = 1,2$, where  $c, c^\prime$ are positive constants.
\end{theorem}

Compared with the upper bounds in Section \ref{sec_rate}, this theorem shows that our proposed estimator via the active subsampling algorithm is minimax rate optimal up to some logarithmic factors.  By Markov inequality,  we can also obtain the lower bound for the $N$-budget minimax risk 
$$\mathcal{M}_n(\mathcal{P}(\beta,s), N)\geq c^\prime cs^{\frac{1}{q}-\frac{1}{2}}\left(\frac{s \log (d / s)}{N}\right)^{\frac{1}{2(\beta \wedge 1)}}.
$$ 

As a final remark, the same lower bound in Theorem \ref{lowerbound1} holds if we restrict the infimum to estimators $\hat\btheta(\{O_i\}_{i:R_i=1})$ that use only the labeled data $\{O_i\}_{i:R_i=1}$. Indeed, this lower bound over $\hat\btheta(\{O_i\}_{i:R_i=1})$ is tight, as it is attained by our estimators in Section \ref{sec_rate} that only use the labeled data as well. Thus, the unlabeled samples $\{X_i,\bZ_i\}_{i: R_i=0}$ offer no rate improvement for estimating $\boldsymbol{\theta}^*(P)$ in our model.

\section{Practical Considerations}\label{sec_practice}

In this section, we discuss several implementation issues in our Algorithm \ref{algoK}. 

First, we discuss the computational challenge for solving the optimization problem (\ref{estimatork}). Despite addressing the discontinuity of the 0-1 loss, the smoothed empirical risk function $R^{D_k}_{\delta_{k}}(\btheta)$ remains non-convex. Consequently, obtaining the global solution to (\ref{estimatork}) presents computational challenges. To overcome this issue, we focus on the path-following algorithm \citep{feng2022nonregular}. For any $1\leq k\leq K$, this algorithm computes approximate local solutions to (\ref{estimatork}) corresponding to a sequence of decreasing regularization parameters $\lambda$ until the desired regularization parameter is reached. We set $\hat{\btheta}_k= \tilde{\btheta}_{k,tgt}$, where $\tilde{\btheta}_{k,tgt}$ represents the final approximate local solution with the desired regularization parameter $\lambda_{k,tgt}$.  By employing this path-following approach, we can efficiently compute the entire solution path for $\hat\btheta_k$ while preserving sparsity across the sequence. For $k\geq 2$, to speed up the computation of $\hat\btheta_k$,  we can use the estimator $\hat\btheta_{k-1}$ obtained from the previous iteration as a warm start in the path-following approach. The detailed algorithm for solving (\ref{estimatork}) is provided in Appendix~\ref{path-following}.


Second, we consider how to choose $K$. Our results in Section \ref{sec_rate} reveal that the choice of $K$ may depend on the smoothness parameter $\beta$.  In particular, when $\beta$ is greater than $(1+\sqrt{3})/2\approx 1.37$, the two-step algorithm with $K=2$ yields the estimator with the optimal rate. In practice, we recommend choosing $K=2$ for the following reasons.  First, in many applications, it is often reasonable to assume that the conditional density of $X$ given $Y$ and $\bZ$ is  sufficiently smooth (e.g., twice differentiable), implying that $\beta>1.37$. Even if $1\leq \beta\leq (1+\sqrt{3})/2$, ignoring all logarithmic factors, the results (\ref{eq_beta^{**}< beta < beta^*_rate_sub}) and (\ref{eq_beta < beta^**_rate_sub}) show that the slowest rate of $\hat\btheta_2$ in the $l_2$ norm is  $O_p((s/N)^{4/9})$, attained when $\beta=1$. In this case, the improvement of $\hat\btheta_2$ over $\hat\btheta_1$ (which converges with rate $O_p((s/N)^{1/3})$) is much more substantial than the improvement of $\hat\btheta_3$ (with rate $O_p((s/N)^{13/27})$) over $\hat\btheta_2$. Second, compared with $\hat\btheta_k$ with $k>2$, the estimator $\hat\btheta_2$ depends on fewer tuning parameters, making it easier to implement and numerically more stable. 

Finally, we discuss how to choose the tuning parameters $\lambda_1,\lambda_2$ and $b_1$ in our active subsampling algorithm with $K=2$. Empirically, the algorithm is insensitive to the choice of bandwidth parameters $\delta_1$ and $\delta_2$. Thus, we set $\delta_1=\delta_2=1$ in both simulation and real data analysis. 
Indeed, following \cite{feng2022nonregular}, we can easily modify our cross-validation algorithm to choose $\lambda_1,\lambda_2$, $b_1$ and $\delta_1, \delta_2$ simultaneously. However, such an approach tends to be computationally far more expensive and may require a very large $N$ to obtain stable results. We do not pursue it here. 

Instead, we apply the one standard error rule to choose $\lambda$ in Algorithm \ref{alg_cv}. Built on this algorithm, the data-driven two-step active subsampling with cross-validation is shown in Algorithm \ref{algodatadriven}. In this algorithm, $\tilde N_1$, $\tilde N_{cv}$ and $\tilde N_{2}$ stand for the budget for obtaining $\hat\btheta_1$ in step 1, the budget for selecting $b_1$ in step 2, and the budget for obtaining the final estimator $\hat\btheta_2$ in step 3, respectively. In general, we recommend choosing the budget $\tilde N_{2}$ in the final step relatively large, as $\tilde N_{2}$ plays the role of effective sample size for the final estimator $\hat\btheta_2$. 

\begin{algorithm}
    \caption{$M$-fold Cross-Validation for $\lambda$ with one standard error rule}
    \label{alg_cv}
 \begin{algorithmic}
     \STATE \textbf{Input:} Data $D = \{X_i,\bZ_i,Y_i\}, \text{a grid for } \lambda$. 
     \STATE \textbf{Parameters:} $\delta$
     \STATE \hspace*{5mm} Randomly split $D$ into $M$ folds, $D_1,...,D_M$ with equal size.         
     \STATE \hspace*{5mm} For each $\lambda$ in the grid, compute
$$
\hat{CV}_{\lambda}^k=\frac{1}{|D_k|}\sum_{a\in D_k} \gamma(Y_a) L_\delta(Y_a(X_a-\bZ^T_a\hat\btheta^{(-k)}_\lambda)),
$$
where $\hat\btheta^{(-k)}_\lambda$ is our estimator with the tuning parameters $\delta$ and $\lambda$ using the data excluding $D_k$. 
\STATE \hspace*{5mm} Compute the cross-validation error $\hat{CV}_{\lambda}=\frac{1}{M}\sum_{k=1}^M\hat{CV}_{\lambda}^k$. 
\STATE \hspace*{5mm} Define $\hat{CV}_{\min}=\min_{\lambda}\hat{CV}_{\lambda}$ as the minimum cross-validation error over this grid, and $\bar\lambda=\argmin_{\lambda}\hat{CV}_{\lambda}$. Define $\hat{SE}_{\min}$ as the standard error of $\{\hat{CV}_{\bar \lambda}^k\}_{k=1}^M$ over these M folds.
\STATE \hspace*{5mm} Find $\hat\lambda_{CV}$ via the following ``one standard error rule'', 
$$\hat\lambda_{CV} = \max~\lambda\in\textrm{grid}~ s.t.~\hat{CV}_{\lambda} \leq \hat{CV}_{\min} + \hat{SE}_{\min}.$$ 
     \RETURN $\hat\lambda_{CV}$ and $\hat{CV}_{\hat\lambda_{CV}}$.
 \end{algorithmic}
\end{algorithm}

\begin{algorithm}
    \caption{Data-Driven Two-Step Active Subsampling with Cross-Validation}
    \label{algodatadriven}
 \begin{algorithmic}
     \STATE \textbf{Input:} $D = \{X_i,\bZ_i\}_{i=1}^n, \text{a grid $\Delta$ for } b$ and $\tilde N_1, \tilde N_{cv}, \tilde N_2$ with $\tilde N_1+\tilde N_{cv}+\tilde N_2=N$.
     \STATE \textbf{Parameters:} $\delta_1, \delta_2$
     \STATE \hspace*{5mm} Randomly split $D$ into 3 batches with $|D_1|=\frac{\tilde N_1}{N}n$, $|D_{cv}|=\frac{\tilde N_{cv}}{N}n$ and $|D_2|=\frac{\tilde N_2}{N}n$.
     \STATE \textbf{Step 1: Obtain Initial Estimator $\hat{\btheta}_1$}
     \STATE \hspace*{5mm} Draw data from $D_1$ with probability $c_{n,1} = \frac{N}{n}$. Acquire the label $Y_i$ for each sampled data and form the dataset $D_1^*=\{X_{i}, \bZ_{i}, Y_{i}\}_{R_i=1}$, where $(X_i,\bZ_i)\in D_1$. 
     \STATE \hspace*{5mm} Apply the 5-fold cross-validation Algorithm \ref{alg_cv} to $D_1^*$. Return the optimal parameter $\lambda_{1,opt}$.
     \STATE \hspace*{5mm} Compute the initial estimator: $\hat{\btheta}_1 \leftarrow \argmin_{\btheta} \{R^{D_1}_{\delta_1}(\btheta) + \lambda_{1,opt} \|\btheta\|_1\}$.

     \STATE \textbf{Step 2: Select Optimal $b_1$}
     \STATE \hspace*{5mm} For each candidate $b\in\Delta$:
     \STATE \hspace*{10mm} Define the active set: $S_2 \leftarrow \left\{(X, \bZ): -b \leq \frac{X - \hat{\btheta}_1^T \bZ}{\sqrt{1 + \|\hat{\btheta}_1\|_2^2}} \leq b\right\}$.
     \STATE \hspace*{10mm} Given $(X_i, \bZ_i) \in S_2$, draw data $(X_i, \bZ_i)$ from $D_{cv}$ with probability $\frac{N}{n|\Delta|\PP\left((X, \bZ) \in S_2\right)}$. Acquire the label $Y_i$ for each sampled data and form $D_{cv,b}^*=\{X_{i}, \bZ_{i}, Y_{i}\}_{R_i=1}$, where $(X_i,\bZ_i)\in D_{cv}$. 
     \STATE \hspace*{10mm} Apply the 5-fold cross-validation Algorithm \ref{alg_cv} to $D_{cv,b}^*$. Return the optimal parameter $\lambda_{b,opt}$ and the minimum CV error $\hat{CV}_b$.
     \STATE \hspace*{5mm} Compute $\hat b_1=\argmin_{b\in\Delta} \hat{CV}_b$ that minimizes the cross-validation error.
     
     \STATE \textbf{Step 3: Obtain Final Estimator $\hat{\btheta}_2$}
     \STATE \hspace*{5mm} Define the active set: $S_2 \leftarrow \left\{(X, \bZ): -
     \hat b_1 \leq \frac{X - \hat{\btheta}_1^T \bZ}{\sqrt{1 + \|\hat{\btheta}_1\|_2^2}} \leq \hat b_1\right\}$.
     \STATE \hspace*{5mm} Given $(X_i, \bZ_i) \in S_2$, draw data $(X_i, \bZ_i)$ from   $D_2$ with probability $c_{n,2} = \frac{N}{n\PP\left((X, \bZ) \in S_2\right)}$. Acquire the label $Y_i$ for each sampled data and form $D_{2}^*=\{X_{i}, \bZ_{i}, Y_{i}\}_{R_i=1}$, where $(X_i,\bZ_i)\in D_{2}$.
     \STATE \hspace*{5mm} Apply the 5-fold cross-validation Algorithm \ref{alg_cv} to $D_2^*$. Return the optimal parameter $\lambda_{2,opt}$.
     \STATE \hspace*{5mm} Compute the final estimator: $\hat{\btheta}_2 \leftarrow \argmin_{\btheta} \{R^{D_2}_{\delta_2}(\btheta) + \lambda_{2,opt} \|\btheta\|_1\}$.
     
     \RETURN $\hat{\btheta}_2$
 \end{algorithmic}
\end{algorithm}

\section{Simulation Studies}\label{sec_sim}
We conduct simulations to evaluate the performance of our proposed method. We consider the following two classes of models, and for both models, it can be shown that the parameter $\btheta$ coincides with the estimand in (\ref{objective}).
\begin{itemize}
    \item \textbf{Binary response model:} We consider $Y = \operatorname{sign}(\tilde{Y})$, where 
    \[
    \tilde{Y} = X - \btheta^T\bZ + \epsilon,
    \]
    $X \in \RR, \bZ \in \RR^d$, and $\epsilon$ is a random noise such that $\operatorname{Median}(\tilde{Y} \mid X, \bZ)=X-\btheta^T \bZ$.  Logistic regression belongs to the class of binary response models by setting $\epsilon$ to follow the logistic distribution independent with $(X,\bZ)$. We conduct simulations for both logistic regression and a more general case where we allow $\epsilon$ to depend on $(X, \bZ)$.
    \item \textbf{Conditional mean model:} We consider $Y \in \{-1,1\}, \bZ \in \RR^d$, and 
    \[
    X = \mu Y + \btheta^T\bZ + \epsilon,
    \]
    where $\epsilon \perp Y,\bZ$ is a random noise and $\mu>0$ is a constant.
\end{itemize}

Under each model, we simulate i.i.d. samples with $n = 20000$ and dimension  $d = 200$. We refer to the collection of these samples as the full dataset. We set sparsity $s = 10$, and generate the nonzero elements of  $\btheta^*$ from $\text{Uniform}(1,2)$. We then normalize $\btheta^*$ such that $\|\btheta^*\|_2 = 1$. For the logistic regression, we generate $X \sim N(0,1)$, $\bZ \sim N_d(0,1)$. For the binary response model where $\epsilon$ can depend on $(X,\bZ)$, we generate $X \sim N(0,1)$, $\bZ \sim N_d(0,1)$, and $\epsilon \sim N(0, \sigma^2(1+2(X-\btheta^T\bZ)^2))$ with $\sigma = 0.5$. We refer to this case as binary response model in the following discussion. For the conditional mean model, we generate $Y \sim \text{Uniform}\{-1,1\}$, $\bZ \sim N_d(0,1)$, $\epsilon \sim N(0, (0.1)^2)$ and set $\mu = 2$. 

We fix the label budget at $N = 2000$ in all scenarios. For our proposed method, we implement the two-step active subsampling algorithm outlined in Algorithm~\ref{algoK} with $N_1=N/8$ and $N_2=7N/8$. Empirically, we find that assigning a larger proportion of the label budget to the second step leads to a more stable final estimator $\hat\btheta_2$. The numerical results for our algorithm with $N_1=N/5$ and $N_2=4N/5$ are quite similar and are deferred to the Appendix Section \ref{sec_additional_sim}. We compare the performance of our proposed method with \cite{feng2022nonregular}, where
the path-following algorithm 
is applied to a dataset with size $N = 2000$ uniformly sampled from the full dataset. We refer to this method as ``passive PF'', while our proposed method is denoted as ``two-step sampling with PF". In the application of the path-following algorithm, we set the bandwidth parameter $\delta = 1$ and use the standard Gaussian density as the kernel function. The number of regularization stages is fixed at $T=20$, and we choose $\nu = 1/4$, $\phi = (\lambda_{tgt}/\lambda_0)^{1/T}$ and $\eta = 1$ in the path-following Algorithm~\ref{pathfollow}. The tuning parameter $\lambda_{tgt}$ is chosen by the 5-fold cross-validation via Algorithm \ref{alg_cv}. 

In our comparison, we also include two simple benchmarks ``passive LR'' and ``two-step sampling with LR''. The former refers to the $\ell_1$-penalized logistic regression applied to the dataset with size $N = 2000$ uniformly sampled from the full dataset. For the ``two-step sampling with LR'', we employ a similar two-step approach as in Algorithm~\ref{algoK}. However,  instead of solving (\ref{penalized}) using the path-following algorithm, we estimate $\btheta$ by the $\ell_1$-penalized logistic regression. Similar to our proposed method, the active set in second step has the same form as (\ref{activeset}), which depends on the choice of $b_1$.

Figures~\ref{log_8}- \ref{manski_8} plot the statistical errors $\|\hat{\btheta} - \btheta^*\|$ in $\ell_1$ and $\ell_2$ norms as the size of the active set $b_1$ increases under three models. The smallest $b_1$ value is chosen such that approximately 10\% of the full dataset fall into the active set, while the largest $b_1$ value corresponds to the case that the active set covers nearly 100\% of the full dataset. The simulation is repeated 50 times.

Since the passive LR and passive PF methods are independent of $b_1$, their estimation errors correspond to two horizontal lines. In Figure \ref{log_8}, the passive LR can be viewed as the benchmark estimator, since the data are generated under the logistic regression. We can see that applying the two-step sampling idea to logistic regression generally does not improve the performance of the passive LR. This may be due to the fact that the logistic regression is a regular model and the parameter estimation may not be improvable. However, our proposed two-step sampling with PF method can significantly improve upon the passive PF method, especially for small or moderate values of $b_1$. This is consistent with our theoretical results on the convergence rate of the estimators. Similar patterns hold under the conditional mean model in Figure \ref{condi_8}. In particular, our proposed method yields the smallest $\ell_1$ and $\ell_2$ estimation errors among all the competing methods, for small or moderate values of $b_1$. For the binary response model in Figure~\ref{manski_8}, when $b_1$ is small, our proposed method has larger errors than the passive PF method. However, as $b_1$ keeps increasing, our proposed method outperforms the passive PF.

We also compare the four methods with respect to  $\|\hat{\btheta} - \btheta^*\|_{\infty}$ and prediction errors. The results exhibit similar patterns as $\|\hat{\btheta} - \btheta^*\|_1$ and $\|\hat{\btheta} - \btheta^*\|_2$; see Section \ref{sec_additional_sim}. 

Since the optimal choice of $b_1$ is data-dependent, we evaluate the fully data-driven procedure in Algorithm~\ref{algodatadriven}. Table~\ref{comparisoncv} compares our proposed method with the passive path-following method across the three models. The simulation is repeated 50 times. From the result in Table~\ref{comparisoncv} we can see that our proposed method uniformly dominates the passive path-following method across all three models. For example, the $\ell_1$ estimation error of our proposed method is around 45\% smaller than the passive path-following method under the conditional mean model.

\begin{table}
\caption{Comparison of the two-step path-following method  with the passive path-following algorithm. The number in the parentheses are standard deviations.}
\label{comparisoncv}
\begin{tabular}{ *{7}{c}} 
 \toprule
 Error & \multicolumn{2}{c}{Conditional mean model}
            & \multicolumn{2}{c}{Logistic model} 
                        & \multicolumn{2}{c} {Binary response  model}\\ 
\cmidrule(lr){2-3} \cmidrule(lr){4-5} \cmidrule(lr){6-7}
& Two-step PF & Passive PF & Two-step PF &  Passive PF & Two-step PF &  Passive PF \\
\midrule
  $\ell_1$   & 0.918(0.101)  & 1.648(0.165) 
  & 1.514(0.319)  & 1.625(0.416) 
  & 0.835(0.201) & 0.937(0.200) \\ 
  \hline
  $\ell_2$ & 0.313(0.036) & 0.525(0.052) 
  & 0.525(0.093) & 0.559(0.129)  
  & 0.319(0.068) & 0.341(0.066) \\
  \hline
  $\ell_{\infty}$ 
  & 0.150(0.026) & 0.196(0.022) 
  & 0.270(0.047) & 0.275(0.050) 
  & 0.192(0.044) & 0.187(0.038)  \\
 \bottomrule
\end{tabular}
\end{table}

\section{Real Data Analysis}\label{sec_data}
In this section, we apply our proposed method to a dataset of hospitalized patient diagnosed with diabetes, obtained from the UCI Machine Learning Repository \citep{misc_diabetes_130-us_hospitals_for_years_1999-2008_296}. This dataset contains 101,766 hospital records collected over a decade (1999-2008) from 130 US hospitals. The data includes various attributes such as patient ID, race, gender, age, admission type, readmission status, length of hospital stay, medical specialty of the admitting physician, number of lab tests performed, Hemoglobin A1c (HbA1c) test results, diagnosis, and more. Among these attributes, the HbA1c test result is a key indicator of glucose control and is widely used to evaluate the quality of diabetes care \citep{baldwin2005eliminating}.
In an existing work \citep{https://doi.org/10.1155/2014/781670}, researchers  were interested in identifying important risk factors that lead to early readmission. They defined the readmission attribute as a binary outcome: ``readmitted" if the patient was readmitted within 30 days of discharge, and ``otherwise," which includes both readmission after 30 days and no readmission.

In our analysis, we adopt the same definition of response variable $Y$ as in \cite{https://doi.org/10.1155/2014/781670}. Specifically, $Y_i = 1$ if a patient was readmitted within 30 days of discharge and $Y_i = -1$ otherwise. We chose the patient's HbA1c test result as the primary measurement $X_i$, with $\bZ_i$ representing additional patient demographic statistics and clinical biomarkers.   The primary goal of our analysis is to determine the optimal individualized threshold, $\btheta^T\bZ_i$, such that a patient's early readmission can be predicted based on whether the HbA1c test result exceeds this threshold ($X_i \geq \btheta^T\bZ_i$) or falls below it ($X_i < \btheta^T\bZ_i$).  This problem can be formulated as an estimation task of the form (\ref{objective0}) or (\ref{objective}), with weights $\gamma(y) = 1/\PP(Y = y)$.

The original dataset includes multiple inpatient visits from the same patients, making the observations statistically dependent. To address this, as suggested by \cite{https://doi.org/10.1155/2014/781670}, we used only the first encounter per patient as the primary admission and determined whether they were readmitted within 30 days. Next, we removed redundant features and those with a high percentage of missing values. Following the discussion in \cite{https://doi.org/10.1155/2014/781670}, we also added pairwise interactions among features as new variables. After these preprocessing steps, the dataset contains 69,984 observations and $d = 60$ variables, excluding $X$. Furthermore, we observed significant imbalance in the dataset, with only 6,293 positive instances ($Y = 1$) compared to 63,691 negative instances ($Y = -1$). To address this issue, we randomly selected 6,293 samples from the negative class to match the number of positive cases and then combined these with the positive instances. This resulted in a final dataset containing $n= 12,586$ observations.

Since the true value $\btheta^*$ is unknown, we first applied the path-following algorithm to the full dataset ($n= 12,586$) to derive an estimator, which is used as the benchmark or equivalently treated as $\btheta^*$ when evaluating $\|\hat\btheta - \btheta^*\|$ for different methods. Suppose that there is a budget constraint due to study design or administration cost, which prevents us from using the benchmark estimator. Our goal is to illustrate the efficacy of our active subsampling methods under the budget constraint. For our proposed method, we implement the data-driven two-step active subsampling algorithm outlined in Algorithm~\ref{algodatadriven} with label budgets of $N = 3000$, 4000, and 5000, respectively. We compare the performance of our method with “passive path-following,” where the budgeted data is uniformly sampled from the full dataset to derive the estimator of $\btheta^*$. The result is shown in Table~\ref{realdata}. It is seen that, as the label budget increases, the estimation errors for both methods decrease, with the two-step path-following method consistently outperforming the passive path-following method across all three settings. From a scientific point of view, our two-step path-following method identified a few important covariates in our model, including `time in hospital' and `change of medications'. These findings are consistent with the previous work \citep{https://doi.org/10.1155/2014/781670}, and seem to be clinically meaningful.

\begin{table}
\caption{Comparison of the active path-following method  with the passive path-following algorithm.}
\label{realdata}
\begin{tabular}{ *{7}{c}} 
 \toprule
 $N$ & \multicolumn{2}{c}{3000}
            & \multicolumn{2}{c}{4000} 
                        & \multicolumn{2}{c} {5000}\\ 
\cmidrule(lr){2-3} \cmidrule(lr){4-5} \cmidrule(lr){6-7}
& Two-step PF & Passive PF & Two-step PF &  Passive PF & Two-step PF &  Passive PF \\
\midrule
  $\ell_1$   & 2.422  & 3.192
  & 0.968  & 1.276
  & 0.379 & 0.651 \\ 
  \hline
  $\ell_2$ & 1.403 & 2.478 
  & 1.649 & 2.048 
  & 0.269 & 0.543 \\
  \hline
  $\ell_{\infty}$ 
  & 0.869 & 2.369 
  & 0.824 & 2.048 
  & 0.203 & 0.529 \\
 \bottomrule
\end{tabular}
\end{table}

\begin{figure}
\centering
\subfigcapskip -5pt
\subfigure[]{
\includegraphics[width=0.48\textwidth]{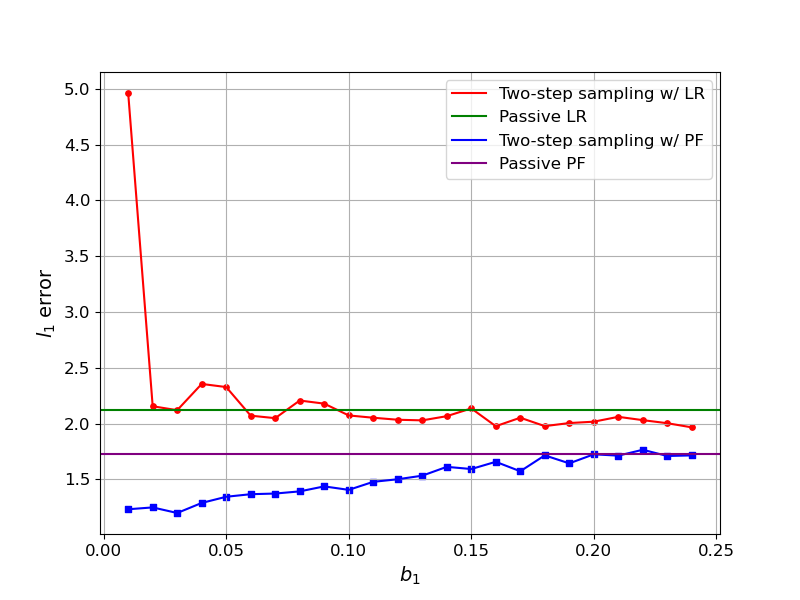}}
\subfigure[]{
\includegraphics[width=0.48\textwidth]{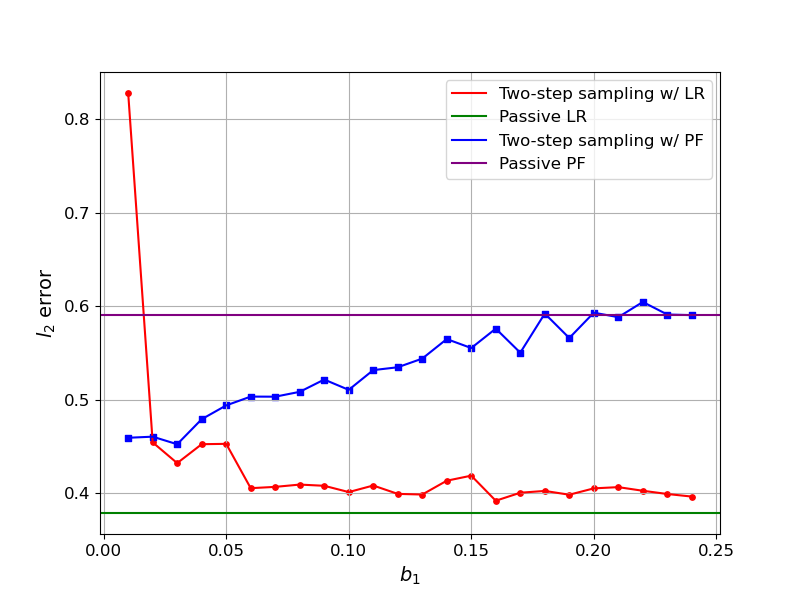}} \vskip -5pt 
\caption{$\|\hat{\btheta} - \btheta^*\|$ in $\ell_1$ and $\ell_2$ norms under the logistic regression. LR: $\ell_1$ penalized logistic regression; PF: path-following algorithm. 1/8 of the label budget is used in the first step for both two-step sampling methods.}
\label{log_8}
\end{figure}

\begin{figure}
\centering
\subfigcapskip -5pt
\subfigure[]{
\includegraphics[width=0.48\textwidth]{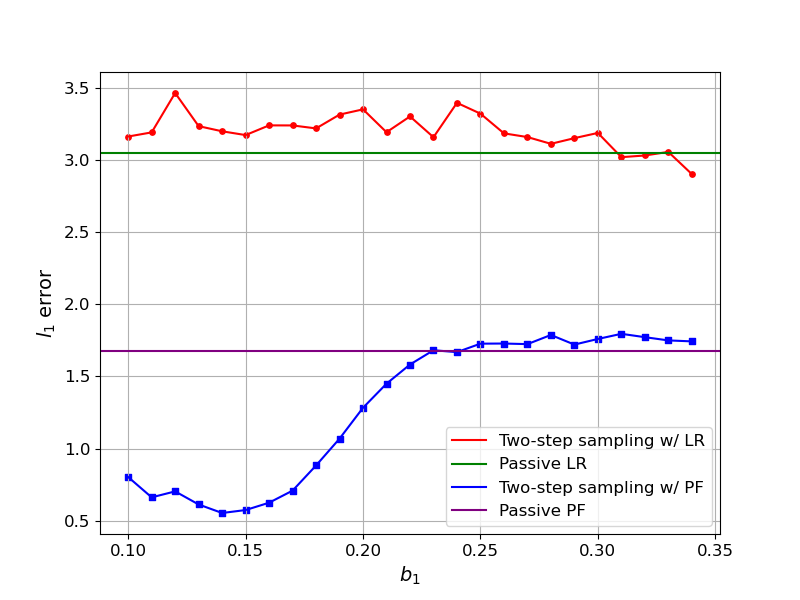}}
\subfigure[]{
\includegraphics[width=0.48\textwidth]{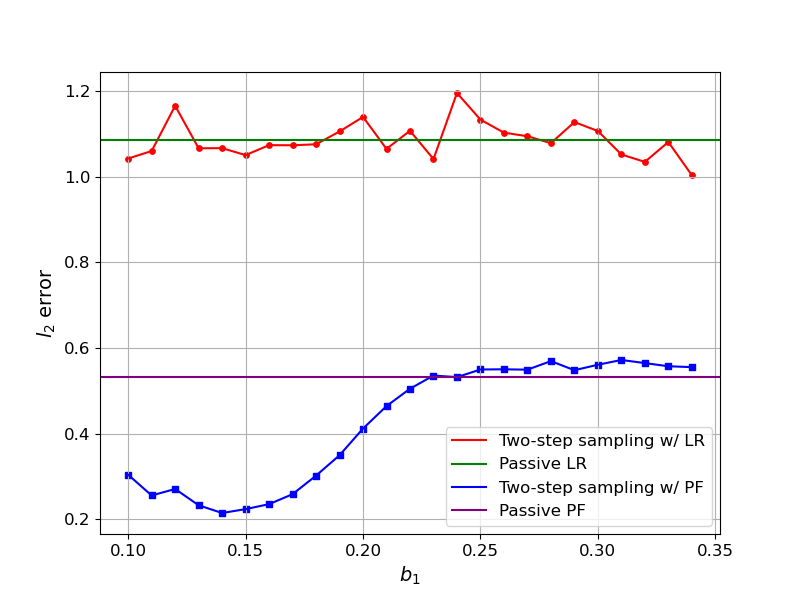}} \vskip -5pt 
\caption{$\|\hat{\btheta} - \btheta^*\|$ in $\ell_1$ and $\ell_2$ norms under the conditional mean model. LR: $\ell_1$ penalized logistic regression; PF: path-following algorithm. 1/8 of the label budget is used in the first step for both two-step sampling methods.}
\label{condi_8}
\end{figure}

\begin{figure}
\centering
\subfigcapskip -5pt
\subfigure[]{
\includegraphics[width=0.48\textwidth]{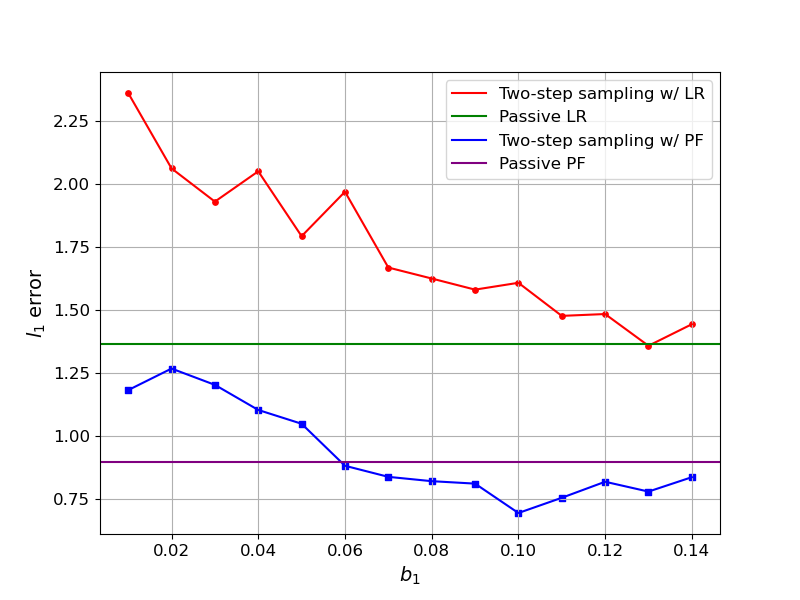}}
\subfigure[]{
\includegraphics[width=0.48\textwidth]{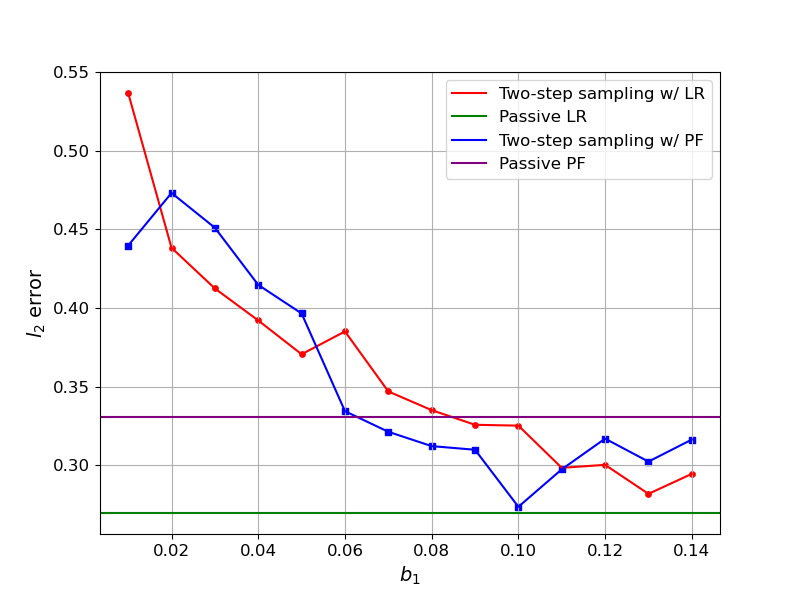}} \vskip -5pt 
\caption{$\|\hat{\btheta} - \btheta^*\|$ in $\ell_1$ and $\ell_2$ norms under the binary response model. LR: $\ell_1$ penalized logistic regression; PF: path-following algorithm. 1/8 of the label budget is used in the first step for both two-step sampling methods.}
\label{manski_8}
\end{figure}

\section*{Acknowledgment}

Yang is supported by the NSF grant DMS-1941945 and DMS-2311291 and NIH 1RF1AG077820-01A1.

\bibliographystyle{ims}
\bibliography{refer}

\begin{thebibliography}{42}
\expandafter\ifx\csname natexlab\endcsname\relax\def\natexlab#1{#1}\fi
\expandafter\ifx\csname url\endcsname\relax
  \def\url#1{\texttt{#1}}\fi
\expandafter\ifx\csname urlprefix\endcsname\relax\def\urlprefix{URL }\fi

\bibitem[{Azriel et~al.(2022)Azriel, Brown, Sklar, Berk, Buja and Zhao}]{azriel2022semi}
\textsc{Azriel, D.}, \textsc{Brown, L.~D.}, \textsc{Sklar, M.}, \textsc{Berk, R.}, \textsc{Buja, A.} and \textsc{Zhao, L.} (2022).
\newblock Semi-supervised linear regression.
\newblock \textit{Journal of the American Statistical Association} \textbf{117} 2238--2251.

\bibitem[{Balcan et~al.(2007)Balcan, Broder and Zhang}]{balcan2007margin}
\textsc{Balcan, M.-F.}, \textsc{Broder, A.} and \textsc{Zhang, T.} (2007).
\newblock Margin based active learning.
\newblock In \textit{International Conference on Computational Learning Theory}. Springer.

\bibitem[{Balcan and Long(2013)}]{balcan2013active}
\textsc{Balcan, M.-F.} and \textsc{Long, P.} (2013).
\newblock Active and passive learning of linear separators under log-concave distributions.
\newblock In \textit{Conference on Learning Theory}. PMLR.

\bibitem[{Baldwin et~al.(2005)Baldwin, Villanueva, McNutt and Bhatnagar}]{baldwin2005eliminating}
\textsc{Baldwin, D.}, \textsc{Villanueva, G.}, \textsc{McNutt, R.} and \textsc{Bhatnagar, S.} (2005).
\newblock Eliminating inpatient sliding-scale insulin: a reeducation project with medical house staff.
\newblock \textit{Diabetes care} \textbf{28} 1008--1011.

\bibitem[{Boufounos and Baraniuk(2008)}]{boufounos20081}
\textsc{Boufounos, P.~T.} and \textsc{Baraniuk, R.~G.} (2008).
\newblock 1-bit compressive sensing.
\newblock In \textit{2008 42nd Annual Conference on Information Sciences and Systems}. IEEE.

\bibitem[{B{\"u}hlmann and Van De~Geer(2011)}]{buhlmann2011statistics}
\textsc{B{\"u}hlmann, P.} and \textsc{Van De~Geer, S.} (2011).
\newblock \textit{Statistics for high-dimensional data: methods, theory and applications}.
\newblock Springer Science \& Business Media.

\bibitem[{Cai and Guo(2020)}]{tony2020semisupervised}
\textsc{Cai, T.} and \textsc{Guo, Z.} (2020).
\newblock Semisupervised inference for explained variance in high dimensional linear regression and its applications.
\newblock \textit{Journal of the Royal Statistical Society Series B: Statistical Methodology} \textbf{82} 391--419.

\bibitem[{Castro and Nowak(2008)}]{castro2008minimax}
\textsc{Castro, R.~M.} and \textsc{Nowak, R.~D.} (2008).
\newblock Minimax bounds for active learning.
\newblock \textit{IEEE Transactions on Information Theory} \textbf{54} 2339--2353.

\bibitem[{Chernozhukov et~al.(2014)Chernozhukov, Chetverikov and Kato}]{Kengo2014MaxIneq}
\textsc{Chernozhukov, V.}, \textsc{Chetverikov, D.} and \textsc{Kato, K.} (2014).
\newblock Gaussian approximation of suprema of empirical processes.
\newblock \textit{The Annals of Statistics} \textbf{42} 1564--1597.
\newline\urlprefix\url{http://www.jstor.org/stable/43556334}

\bibitem[{Clore et~al.(2014)Clore, Cios, DeShazo and Strack}]{misc_diabetes_130-us_hospitals_for_years_1999-2008_296}
\textsc{Clore, J.}, \textsc{Cios, K.}, \textsc{DeShazo, J.} and \textsc{Strack, B.} (2014).
\newblock {Diabetes 130-US Hospitals for Years 1999-2008}.
\newblock UCI Machine Learning Repository.
\newblock {DOI}: https://doi.org/10.24432/C5230J.

\bibitem[{Deng et~al.(2024)Deng, Ning, Zhao and Zhang}]{deng2024optimal}
\textsc{Deng, S.}, \textsc{Ning, Y.}, \textsc{Zhao, J.} and \textsc{Zhang, H.} (2024).
\newblock Optimal and safe estimation for high-dimensional semi-supervised learning.
\newblock \textit{Journal of the American Statistical Association} \textbf{119} 2748--2759.

\bibitem[{Drineas et~al.(2011)Drineas, Mahoney, Muthukrishnan and Sarl{\'o}s}]{drineas2011faster}
\textsc{Drineas, P.}, \textsc{Mahoney, M.~W.}, \textsc{Muthukrishnan, S.} and \textsc{Sarl{\'o}s, T.} (2011).
\newblock Faster least squares approximation.
\newblock \textit{Numerische mathematik} \textbf{117} 219--249.

\bibitem[{Feng et~al.(2024)Feng, Duan, Ning and Zhao}]{feng2024test}
\textsc{Feng, H.}, \textsc{Duan, J.}, \textsc{Ning, Y.} and \textsc{Zhao, J.} (2024).
\newblock Test of significance for high-dimensional thresholds with application to individualized minimal clinically important difference.
\newblock \textit{Journal of the American Statistical Association} \textbf{119} 1396--1408.

\bibitem[{Feng et~al.(2022)Feng, Ning and Zhao}]{feng2022nonregular}
\textsc{Feng, H.}, \textsc{Ning, Y.} and \textsc{Zhao, J.} (2022).
\newblock Nonregular and minimax estimation of individualized thresholds in high dimension with binary responses.
\newblock \textit{The Annals of Statistics} \textbf{50} 2284--2305.

\bibitem[{Fithian and Hastie(2014)}]{fithian2014local}
\textsc{Fithian, W.} and \textsc{Hastie, T.} (2014).
\newblock Local case-control sampling: Efficient subsampling in imbalanced data sets.
\newblock \textit{Annals of statistics} \textbf{42} 1693.

\bibitem[{Hadad et~al.(2021)Hadad, Hirshberg, Zhan, Wager and Athey}]{hadad2021confidence}
\textsc{Hadad, V.}, \textsc{Hirshberg, D.~A.}, \textsc{Zhan, R.}, \textsc{Wager, S.} and \textsc{Athey, S.} (2021).
\newblock Confidence intervals for policy evaluation in adaptive experiments.
\newblock \textit{Proceedings of the national academy of sciences} \textbf{118} e2014602118.

\bibitem[{Hedayat et~al.(2015)Hedayat, Wang and Xu}]{hedayat2015minimum}
\textsc{Hedayat, A.}, \textsc{Wang, J.} and \textsc{Xu, T.} (2015).
\newblock Minimum clinically important difference in medical studies.
\newblock \textit{Biometrics} \textbf{71} 33--41.

\bibitem[{Jaeschke et~al.(1989)Jaeschke, Singer and Guyatt}]{jaeschke1989measurement}
\textsc{Jaeschke, R.}, \textsc{Singer, J.} and \textsc{Guyatt, G.~H.} (1989).
\newblock Measurement of health status: ascertaining the minimal clinically important difference.
\newblock \textit{Controlled clinical trials} \textbf{10} 407--415.

\bibitem[{Kim and Pollard(1990)}]{kim1990cube}
\textsc{Kim, J.} and \textsc{Pollard, D.} (1990).
\newblock Cube root asymptotics.
\newblock \textit{The Annals of Statistics}  191--219.

\bibitem[{Koltchinskii(2010)}]{koltchinskii2010rademacher}
\textsc{Koltchinskii, V.} (2010).
\newblock Rademacher complexities and bounding the excess risk in active learning.
\newblock \textit{The Journal of Machine Learning Research} \textbf{11} 2457--2485.

\bibitem[{LeCun et~al.(1998)LeCun, Bottou, Bengio and Haffner}]{lecun1998mnist}
\textsc{LeCun, Y.}, \textsc{Bottou, L.}, \textsc{Bengio, Y.} and \textsc{Haffner, P.} (1998).
\newblock Gradient-based learning applied to document recognition.
\newblock \textit{Proceedings of the IEEE} \textbf{86} 2278--2324.

\bibitem[{Lepskii(1991)}]{lepskii1991problem}
\textsc{Lepskii, O.} (1991).
\newblock On a problem of adaptive estimation in gaussian white noise.
\newblock \textit{Theory of Probability \& Its Applications} \textbf{35} 454--466.

\bibitem[{Lounici et~al.(2011)Lounici, Pontil, van~de Geer and Tsybakov}]{lounici2011oracle}
\textsc{Lounici, K.}, \textsc{Pontil, M.}, \textsc{van~de Geer, S.} and \textsc{Tsybakov, A.~B.} (2011).
\newblock Oracle inequalities and optimal inference under group sparsity.
\newblock \textit{The Annals of Statistics}  2164--2204.

\bibitem[{Ma et~al.(2014)Ma, Mahoney and Yu}]{ma2014statistical}
\textsc{Ma, P.}, \textsc{Mahoney, M.} and \textsc{Yu, B.} (2014).
\newblock A statistical perspective on algorithmic leveraging.
\newblock In \textit{International conference on machine learning}. PMLR.

\bibitem[{Mallik et~al.(2020)Mallik, Banerjee and Michailidis}]{mallik2020m}
\textsc{Mallik, A.}, \textsc{Banerjee, M.} and \textsc{Michailidis, G.} (2020).
\newblock M-estimation in multistage sampling procedures.
\newblock \textit{Sankhya A} \textbf{82} 261--309.

\bibitem[{Manski(1975)}]{manski1975maximum}
\textsc{Manski, C.~F.} (1975).
\newblock Maximum score estimation of the stochastic utility model of choice.
\newblock \textit{Journal of econometrics} \textbf{3} 205--228.

\bibitem[{Nesterov(2013)}]{nesterov2013gradient}
\textsc{Nesterov, Y.} (2013).
\newblock Gradient methods for minimizing composite functions.
\newblock \textit{Mathematical programming} \textbf{140} 125--161.

\bibitem[{Perera et~al.(2020)Perera, Fahimnia and Tokar}]{perera2020inventory}
\textsc{Perera, H.~N.}, \textsc{Fahimnia, B.} and \textsc{Tokar, T.} (2020).
\newblock Inventory and ordering decisions: a systematic review on research driven through behavioral experiments.
\newblock \textit{International Journal of Operations \& Production Management} \textbf{40} 997--1039.

\bibitem[{Raskutti and Mahoney(2016)}]{raskutti2016statistical}
\textsc{Raskutti, G.} and \textsc{Mahoney, M.~W.} (2016).
\newblock A statistical perspective on randomized sketching for ordinary least-squares.
\newblock \textit{Journal of Machine Learning Research} \textbf{17} 1--31.

\bibitem[{Strack et~al.(2014)Strack, DeShazo, Gennings, Olmo, Ventura, Cios and Clore}]{https://doi.org/10.1155/2014/781670}
\textsc{Strack, B.}, \textsc{DeShazo, J.~P.}, \textsc{Gennings, C.}, \textsc{Olmo, J.~L.}, \textsc{Ventura, S.}, \textsc{Cios, K.~J.} and \textsc{Clore, J.~N.} (2014).
\newblock Impact of hba1c measurement on hospital readmission rates: Analysis of 70,000 clinical database patient records.
\newblock \textit{BioMed Research International} \textbf{2014} 781670.
\newline\urlprefix\url{https://onlinelibrary.wiley.com/doi/abs/10.1155/2014/781670}

\bibitem[{Tsybakov(2008)}]{10.5555/1522486}
\textsc{Tsybakov, A.~B.} (2008).
\newblock \textit{Introduction to Nonparametric Estimation}.
\newblock 1st ed. Springer Publishing Company, Incorporated.

\bibitem[{Vaart and Wellner(1996)}]{vaart1996weak}
\textsc{Vaart, A.} and \textsc{Wellner, J.~A.} (1996).
\newblock Weak convergence and empirical processes.

\bibitem[{Wang et~al.(2019)Wang, Yang and Stufken}]{wang2019information}
\textsc{Wang, H.}, \textsc{Yang, M.} and \textsc{Stufken, J.} (2019).
\newblock Information-based optimal subdata selection for big data linear regression.
\newblock \textit{Journal of the American Statistical Association} \textbf{114} 393--405.

\bibitem[{Wang et~al.(2018)Wang, Zhu and Ma}]{wang2018optimal}
\textsc{Wang, H.}, \textsc{Zhu, R.} and \textsc{Ma, P.} (2018).
\newblock Optimal subsampling for large sample logistic regression.
\newblock \textit{Journal of the American Statistical Association} \textbf{113} 829--844.

\bibitem[{Wang and Singh(2016)}]{wang2016noise}
\textsc{Wang, Y.} and \textsc{Singh, A.} (2016).
\newblock Noise-adaptive margin-based active learning and lower bounds under tsybakov noise condition.
\newblock In \textit{Thirtieth AAAI Conference on Artificial Intelligence}.

\bibitem[{Wang et~al.(2017)Wang, Yu and Singh}]{wang2017computationally}
\textsc{Wang, Y.}, \textsc{Yu, A.~W.} and \textsc{Singh, A.} (2017).
\newblock On computationally tractable selection of experiments in measurement-constrained regression models.
\newblock \textit{Journal of Machine Learning Research} \textbf{18} 1--41.

\bibitem[{Xu et~al.(2014)Xu, Wang and Fang}]{https://doi.org/10.1002/sim.6290}
\textsc{Xu, T.}, \textsc{Wang, J.} and \textsc{Fang, Y.} (2014).
\newblock A model-free estimation for the covariate-adjusted youden index and its associated cut-point.
\newblock \textit{Statistics in Medicine} \textbf{33} 4963--4974.
\newline\urlprefix\url{https://onlinelibrary.wiley.com/doi/abs/10.1002/sim.6290}

\bibitem[{Zhang et~al.(2021)Zhang, Ning and Ruppert}]{zhang2021optimal}
\textsc{Zhang, T.}, \textsc{Ning, Y.} and \textsc{Ruppert, D.} (2021).
\newblock Optimal sampling for generalized linear models under measurement constraints.
\newblock \textit{Journal of Computational and Graphical Statistics} \textbf{30} 106--114.

\bibitem[{Zhang and Bradic(2022)}]{zhang2022high}
\textsc{Zhang, Y.} and \textsc{Bradic, J.} (2022).
\newblock High-dimensional semi-supervised learning: in search of optimal inference of the mean.
\newblock \textit{Biometrika} \textbf{109} 387--403.

\bibitem[{Zhao et~al.(2012)Zhao, Zeng, Rush and Kosorok}]{zhao2012estimating}
\textsc{Zhao, Y.}, \textsc{Zeng, D.}, \textsc{Rush, A.~J.} and \textsc{Kosorok, M.~R.} (2012).
\newblock Estimating individualized treatment rules using outcome weighted learning.
\newblock \textit{Journal of the American Statistical Association} \textbf{107} 1106--1118.

\bibitem[{Zhou et~al.(2020)Zhou, Zhao and Bisson}]{zhou2020estimation}
\textsc{Zhou, Z.}, \textsc{Zhao, J.} and \textsc{Bisson, L.~J.} (2020).
\newblock Estimation of data adaptive minimal clinically important difference with a nonconvex optimization procedure.
\newblock \textit{Statistical Methods in Medical Research} \textbf{29} 879--893.

\bibitem[{Zrnic and Cand{\`e}s(2024)}]{zrnic2024active}
\textsc{Zrnic, T.} and \textsc{Cand{\`e}s, E.~J.} (2024).
\newblock Active statistical inference.
\newblock \textit{arXiv preprint arXiv:2403.03208} .

\end{thebibliography}

\clearpage

\appendix

\section{Appendix}

\subsection{Further Examples}\label{sec_application}

\textbf{Example 1: Covariate-adjusted Youden index.}\label{exp_2}
The Youden's J statistic is a fundamental metric for evaluating the performance of a binary diagnostic test based on its Receiver Operating Characteristic (ROC) curve. Consider a continuous biomarker $X$ and a binary disease status $Y$ ($Y=1$ if a patient is diseased and $Y=-1$ if a patient is healthy). A common diagnostic rule classifies a patient as diseased if $X$ exceeds a threshold $c$. Recent work, including \cite{https://doi.org/10.1002/sim.6290}, has shown that diagnostic accuracy can be improved by tailoring this threshold to patient-specific covariates $\bZ$. This leads to a covariate-adjusted threshold of the form $c_{\btheta}(\bZ)$, where for simplicity we assume the threshold $c_{\btheta}(\bZ)= \btheta^T\bZ$ is linear in $\bZ$. The covariate-adjusted Youden index is
	\[
	J = \max_{\btheta} \{ \text{sen}(c_{\btheta}(\bZ)) + \text{spe}(c_{\btheta}(\bZ)) -1 \},
	\]
	where $\text{sen}(c_{\btheta}(\bZ))=\PP(X\geq c_{\btheta}(\bZ)|Y =1)$ and $\text{spe}(c_{\btheta}(\bZ)) = \PP(X < c_{\btheta}(\bZ)|Y =-1)$ are the sensitivity and specificity of the test, respectively. Thus, estimating the covariate-adjusted Youden index $J$ is equivalent to solving the problem (\ref{objective0}).

\textbf{Example 2: One-bit compressed sensing.}\label{exp_1} Classical compressed sensing aims to recover a high-dimensional signal $(\alpha^*,\eta^*) \in \RR^{d+1}$ from a small number of linear measurements $\tilde{y}_i = x_i\alpha^*+\bz_i^T\eta^*$,  where $(x_i, \bz_i)$ is a measurement vector. In practice, these measurements $\tilde{y}_i$ are often quantized. One-bit compressed sensing represents an extreme case, where each measurement is quantized to a single bit based on the sign of $x_i \alpha^*+ \bz_i^T\eta^*$ \citep{boufounos20081}.  Assume that $\alpha^*>0$. The goal of one-bit compressed sensing is to recover the signal $\btheta^*$ from the measurements
\[y_i = \sign(x_i - \bz_i^T\btheta^*),\]
where we can only identify $\btheta^*=-\eta^*/\alpha^*$. We can easily verify that 
$$
\btheta^*=\argmin_{\btheta}\Big\{\pi\PP\left(X <  \bZ^T\btheta \mid Y= 1\right) + (1-\pi)\PP\left(X > \bZ^T\btheta \mid Y= -1\right)\Big\},
$$
where $\pi=\PP(Y=1)$. Therefore, the one-bit compressed sensing can be solved by minimizing a weighted version of (\ref{objective0}).

\textbf{Example 3: Semiparametric binary response model and maximum score estimator.}\label{exp_3}
The binary response model plays a fundamental role in economic analysis, as many economic decisions and outcomes are inherently discrete, such as labor force participation, market entry, and credit default. Conventional parametric approaches, such as logistic regression, rely on stringent distributional assumptions for the response variable conditional on covariates. Once the distributional assumptions are violated, parameter estimation and inference may become unreliable. To mitigate potential model misspecification, the following well-known semiparametric binary response model is often preferred 
 \begin{equation}\label{eqbinary}
 Y = \sign(X\alpha^*+\bZ^T\eta^*+ \epsilon),
 \end{equation}
where $(X,\bZ)\in \RR^{d+1}$ are covariates and the random noise $\epsilon$ is not required to be independent of $(X,\bZ)$.  Similar to Example 2, we assume that $\alpha^*>0$. The model (\ref{eqbinary}) can be equivalently written as $Y = \sign(X-\bZ^T\btheta^*+ \epsilon)$, where  only $\btheta^*=-\eta^*/\alpha^*$ is identifiable. Under the assumption that the conditional median of $\epsilon$ given $X$ and $\bZ$ is 0, we can similarly verify that 
$$
\btheta^*=\argmin_{\btheta}\Big\{\pi\PP\left(X <  \bZ^T\btheta \mid Y= 1\right) + (1-\pi)\PP\left(X > \bZ^T\btheta \mid Y= -1\right)\Big\},
$$
where $\pi=\PP(Y=1)$. As a seminal work, \cite{manski1975maximum} proposed the maximum score estimator for $\btheta^*$ by minimizing a weighted empirical version of (\ref{objective0}). Thus, our framework includes the semiparametric binary response model as a special case.

\textbf{Example 4:  Optimal individualized treatment rules and policy learning.} Many clinical studies have demonstrated that patient responses to the same treatment can be highly heterogeneous. Motivated by this heterogeneity, determining the optimal treatment strategy for each individual has become a central problem in personalized medicine. Similarly, decision making problems are ubiquitous  in economics and marketing. For instance, a firm must decide which customers will receive targeted offers. The government may need to decide which groups of participants to prioritize in a welfare-to-work program. In the statistics and economics literature, extensive research has focused on developing methods to estimate optimal individualized treatment rules (ITRs) or optimal policy using high-dimensional covariates.

Let $R$ denote a continuous outcome variable coded such that higher values indicate better conditions, $(X,\bZ) \in \mathbb{R}^{d+1}$ represent baseline  features, and $A \in \{-1, 1\}$ denote the assigned treatment. An ITR, $D(\cdot)$, is a mapping from $\mathbb{R}^{d+1}$ to $\{-1, 1\}$, prescribing treatment $D(X,\bZ)$ for a patient with covariates $(X,\bZ)$. \citet{zhao2012estimating} showed that the optimal ITR $D^*$ minimizes the  weighted misclassification risk
$$
\EE\Big[\frac{R}{\pi(A|X,\bZ)}\mathds{1}(A\neq D(X,\bZ))\Big],
$$
where $\pi(a|X,\bZ)=\PP(A=a|X,\bZ)$ is the propensity score. 

In practice, simple ITRs, such as linear decision rules, are often preferred for interpretability and clinical applicability. When some prior domain knowledge can identify the direction of association for certain baseline variables (e.g., $X$ has a positive effect), the linear ITR can be expressed as $D_{\btheta}(X,\bZ)=\sign(X - \btheta^T \bZ)$. Consequently, estimating the optimal ITR reduces to solving 
$$
\btheta^*=\argmin_{\btheta}\Big\{\gamma_1(X,\bZ)\PP\left(X <  \bZ^T\btheta \mid A= 1\right) + \gamma_{-1}(X,\bZ)\PP\left(X > \bZ^T\btheta \mid A= -1\right)\Big\},
$$
where $\gamma_a(X,\bZ)=\mathbb{E}(R \mid A=a, X, \bZ)\PP(A=a) / \pi(a \mid X,\bZ)$ for $a \in \{-1, 1\}$. When the weight function $\gamma_a(X,\bZ)$ is known or can be consistently estimated, our framework can be used to  estimate the optimal ITR.

\textbf{Example 5: High-dimensional linear classifiers.} The framework introduced in Example 4 can be further extended to a general binary classification setting. In particular, let $(X,\bZ)\in\RR^{d+1}$ denote the feature vector and $Y\in\{-1,1\}$ be the binary label. In the high-dimensional setting, where $d$ is much larger than the sample size, to avoid the curse of dimensionality, a linear classifier $D_{\btheta}(X,\bZ)=\sign(X - \btheta^T \bZ)$ is often desirable, where we also assume that $X$ has a positive effect. In the model-agnostic setting, the optimal linear classifier $D_{\btheta^*}(X,\bZ)$ can be defined as the minimizer of a weighted misclassification risk
$$
\btheta^*=\argmin_{\btheta}\Big\{\gamma_1(X,\bZ)\PP\left(X <  \bZ^T\btheta \mid Y= 1\right) + \gamma_{-1}(X,\bZ)\PP\left(X > \bZ^T\btheta \mid Y= -1\right)\Big\},
$$
where $\gamma_1(X,\bZ)$ and $\gamma_{-1}(X,\bZ)$ are nonnegative weight functions, which are known or can be estimated from external information. Hence, our framework naturally extends to the estimation of high-dimensional linear classifiers, providing a flexible approach for classification problems under minimal distributional assumptions.

\subsection{Comments on Estimating $\mathbb{P}\left((X, Z) \in S_k\right)$}\label{sec_est_active_set}

When $p_k:=\mathbb{P}\left((X, Z) \in S_k\right)$ is unknown, we can use the following sample splitting approach to modify our method. We randomly divide the dataset $D$ into $2K-1$ batches $D_{1}, D_{20}, D_{21}, \cdots, D_{K0}, D_{K1}$.  In the first iteration, we uniformly sample data from $D_1$ with probability $0<c_{n,1}<1$,  
$$
     \PP(R_i=1) = c_{n,1}=\frac{N_1K}{n},
$$
where $N_1$ is the expected number of data points sampled in the first iteration. We compute the estimator 
$$
    \hat{\btheta}_1 := \argmin_{\btheta}\{R^{D_1}_{\delta_1}(\btheta) + \lambda_1\|\btheta\|_1\}.
$$
Starting from the second iteration (i.e., $k\geq 2$), we use the dataset $D_{k0}$ to estimate $p_k=\mathbb{P}\left((X, Z) \in S_k\right)$, 
$$
\hat{p}_{k} = \frac{1}{|D_{k0}|} \sum_{(X_i, \bZ_i) \in D_{k0}} \ind \left\{\left(X_i, \boldsymbol{Z}_i\right) \in S_k\right\}.
$$
Then, given $(X_i, \boldsymbol{Z}_i)\in D_{k1}$, $\hat{p}_{k}$ and $\hat\btheta_{k-1}$, we generate $R_i$ from a Bernoulli distribution with
\[
\mathbb{P}\left(R_i=1 \mid X_i, \boldsymbol{Z}_i, \hat{p}_{k}, \widehat{\boldsymbol{\theta}}_{k-1}\right)=\hat{c}_{n, k} \cdot \ind\left\{\left(X_i, \boldsymbol{Z}_i\right) \in S_k\right\},
\]
where $\hat{c}_{n, k}=N_k K /\left(n \widehat{p}_k\right)$. Finally, we compute the estimator from the dataset $D_{k1}$,
$$
    \hat{\btheta}_k:=\argmin_{\btheta}\{R^{D_{k1}}_{\delta_{k}}(\btheta) + \lambda_k\|\btheta\|_1\}.
$$
By splitting the data into $2K-1$ batches $D_{1}, D_{20}, D_{21}, \cdots, D_{K0}, D_{K1}$, we can now easily establish the concentration inequality for $\hat{p}_{k}$. In this case, conditioning on $\hat\btheta_{k-1}$, the terms $\ind\left\{\left(X_i, \boldsymbol{Z}_i\right) \in S_k\right\}$ are independent, and thus by Bernstein inequality, 
$$
\PP(|\hat{p}_{k}-{p}_{k}|\geq t| \hat\btheta_{k-1})\leq 2 \exp \left(-\frac{|D_{k0}| t^2}{2 p_k(1-p_k)+\frac{2}{3}t} \right).
$$
Since the right side doesn't contain $\hat{\btheta}_{k-1}$, we obtain that
$$
\PP(|\hat{p}_{k}-{p}_{k}|\geq t)\leq 2 \exp \left(-\frac{|D_{k0}| t^2}{2 p_k(1-p_k)+\frac{2}{3}t} \right).
$$
Define the event $E = \{|\hat p_k - p_k| \leq u\}$, and thus $P(E^c) \leq 2 \exp \left(-\frac{|D_{k0}| u^2}{2 p_k(1-p_k)+\frac{2}{3}u}\right)$. On the event $E$, we can bound $\hat{c}_{n,k} = \frac{N_k K}{n \hat p_k}$ by $\frac{c_{n,k}}{1+u/p_k} \leq \hat c_{n,k} \leq \frac{c_{n,k}}{1-u/p_k}$.
Finally, we can bound the  unconditional probability
\begin{align*}
 \PP(\|T^k\|_{\infty}>t ) &\leq \PP(\|T^k\|_{\infty} > t \mid E) \PP(E) + \PP(\|T^k\|_{\infty} > t \mid E^c) \PP(E^c)\\
 &\leq \PP(\|T^k\|_{\infty} > t , E) + \PP(E^c)\\
 &\leq 2 d \exp \left(-\frac{\frac{1}{2} t^2 n/K}{  C_2 \frac{c_{n,k}}{(1-u/p_k)\delta_k}+\frac{t}{3} C_1 M_n K_{\max} / \delta_k}\right) + 2 \exp \left(-\frac{|D_{k0}| u^2}{2 p_k(1-p_k)+\frac{2}{3}u}\right)\\
 &\leq 2 d \exp \left(-\frac{\frac{1}{2} t^2 n/K}{  2 C_2 \frac{c_{n,k}}{\delta_k}+\frac{t}{3} C_1 M_n K_{\max} / \delta_k}\right) + 2 \exp \left(-\frac{|D_{k0}| p_k^2/4}{2 p_k(1-p_k)+\frac{1}{3}p_k}\right)
\end{align*}
by taking $u = p_k/2$. By the proof of Theorem \ref{beta>beta_*}, we have shown that $p_k\asymp b_k$. Then, with $t = C \sqrt{\frac{c_{n,k} K \log d}{n \delta_k}}$ for some sufficiently large constant $C$, 
$$
 \PP(\|T^k\|_{\infty}>C \sqrt{\frac{c_{n,k} K \log d}{n \delta_k}}) \leq 2d^{-1}+2d^{-1},
$$
where we assume $b_k\geq C'\frac{\log d}{|D_{k0}|}$ in order to bound the second term. Since we have a very large unlabeled dataset, we can take $|D_{k0}|=c'n/K$ for some constant $0<c'<1$. Then, the same conclusion as in Proposition \ref{E_1} holds, under a very mild assumption $b_k\gtrsim \frac{K\log d}{n}$. By carefully examining the rest of the proof, we can conclude that the convergence rate of the proposed estimator remains the same.

\subsection{Path-following Algorithm}\label{path-following}
We employ the path-following algorithm introduced in \cite{feng2022nonregular} to solve the optimization problem $  \hat{\btheta}_k:=\argmin_{\btheta}\{R^{D_k}_{\delta_{k}}(\btheta) + \lambda_k\|\btheta\|\}$ for each iteration $k$. 
The idea is to compute approximate local solutions  corresponding to a sequence of decreasing regularization parameters $\lambda$, until the target regularization parameter is reached. To be specific, we firstly choose a sequence of $\lambda_{k,0}>\lambda_{k,1}>\ldots>\lambda_{k,T}=\lambda_{k,tgt}$, where $\lambda_{k,t}=\phi^t \lambda_{k,0}, t=0,1, \ldots$ for some constant $\phi \in(0,1)$ and $\lambda_{k,tgt}$ is the target regularization parameter to be specified later. Let $T$ denote the total number of the path-following
stages and we set $T=\frac{\log \left(\lambda_{k,t g t} / \lambda_{k,0}\right)}{\log \phi}$. 
At each stage $t=1, \cdots, T$, the goal is to  approximately compute the exact local solution $\hat{\btheta}_{k,t}$ corresponding to $\lambda_{k,t}$,
\begin{equation}\label{thetahat}
\hat{\boldsymbol{\theta}}_{k,t}=\underset{\boldsymbol{\theta}}{\operatorname{argmin}} R_{\delta_k}^{D_k}(\boldsymbol{\theta})+\lambda_{k,t}\|\boldsymbol{\theta}\|_1.
\end{equation}
To this end, we apply the proximal-gradient method to iteratively approximate $\hat{\boldsymbol{\theta}}_{k,t}$ by minimizing a sequence of quadratic approximations of $R_{\delta_k}^{D_k}(\boldsymbol{\theta})$ over a convex constraint set $\Omega$:
\begin{equation}
\begin{aligned}
\boldsymbol{\theta}_{k,t}^{j+1} & =\underset{\boldsymbol{\theta} \in \Omega}{\operatorname{argmin}}\left\{R_{\delta_k}^{D_k}\left(\boldsymbol{\theta}_{k,t}^j\right)+\left\langle\nabla R_{\delta_k}^{D_k}\left(\boldsymbol{\theta}_{k,t}^j\right),\left(\boldsymbol{\theta}-\boldsymbol{\theta}_{k,t}^j\right)\right\rangle+\frac{1}{2 \eta}\left\|\boldsymbol{\theta}-\boldsymbol{\theta}_{k,t}^j\right\|_2^2+\lambda_{k,t}\|\boldsymbol{\theta}\|_1\right\} \\
& =\underset{\boldsymbol{\theta} \in \Omega}{\operatorname{argmin}}\left\{\frac{1}{2 \eta}\left\|\boldsymbol{\theta}-\boldsymbol{\theta}_{k,t}^j+\eta \nabla R_{\delta_k}^{D_k}\left(\boldsymbol{\theta}_{k,t}^j\right)\right\|_2^2+\lambda_{k,t}\|\boldsymbol{\theta}\|_1\right\} \\
& :=\mathcal{S}_{\lambda_{k,t} \eta}\left(\boldsymbol{\theta}^j_k, \Omega\right),
\end{aligned}
\end{equation}
where $\eta$ is the step size to be specified later.  
The proximal-gradient algorithm is described in Algorithm~\ref{proximal}.

In the algorithm, we use the stopping criteria $w_\lambda(\boldsymbol{\theta})$ defined as
\[
\omega_\lambda(\boldsymbol{\theta})=\min _{\boldsymbol{\xi} \in \partial\|\boldsymbol{\theta}\|_1} \max _{\boldsymbol{\theta}^{\prime} \in \Omega}\left\{\frac{\left(\boldsymbol{\theta}-\boldsymbol{\theta}^{\prime}\right)^T}{\left\|\boldsymbol{\theta}-\boldsymbol{\theta}^{\prime}\right\|_1}\left(\nabla R_{\delta_k}^{D_k}(\boldsymbol{\theta})+\lambda \boldsymbol{\xi}\right)\right\}.
\]
At stage $t$, the proximal-gradient algorithm returns an approximate solution
$\tilde{\boldsymbol{\theta}}_{k,t}$ with precision $\epsilon_{k,t}=\nu \lambda_{k,t}, \nu \in(0,1)$ corresponding to $\lambda_{k,t}$. Then we use $\tilde{\btheta}_{k,t}$ as a warm start for stage $t+1$ and repeat this process. At the final stage, we would compute the
approximate solution $\tilde{\btheta}_{k,T} = \tilde{\btheta}_{k,tgt}$
corresponding to $\lambda_{k,tgt}$ using a high precision $\epsilon_{k,tgt}$. The detail of the path-following algorithm is described in Algorithm~\ref{pathfollow}.

We note that at iteration $k$, we can use  the estimator $\hat\btheta_{k-1}$ obtained from iteration $k-1$ as the initial value for $\btheta$ in Algorithm~\ref{pathfollow}.

\begin{algorithm} 
    \caption{$\btheta \leftarrow$ Path-Following $\left(\lambda_{k,0}, \lambda_{k,t g t}, \nu, T, \epsilon_{k,t g t}, \Omega\right)$}
    \label{pathfollow}
 \begin{algorithmic}
     \STATE \textbf{input:} $\lambda_{k,t g t}>0, \nu>0, \phi>0, \epsilon_{k,t g t}>0, \Omega$
      \STATE \textbf{parameter:} $\eta > 0$
       \STATE \textbf{initialize:} $\tilde{\boldsymbol{\theta}}_{k,0} \leftarrow \mathbf{0}, \lambda_{k,0} \leftarrow\left\|\nabla R_{\delta_k}^{D_k}(\mathbf{0})\right\|_{\infty}, T \leftarrow \frac{\log \left(\lambda_{k, t g t} / \lambda_{k,0}\right)}{\log \phi}$
       \FOR{$t=1, \ldots, T-1$}
        \STATE \hspace*{5mm} $\lambda_{k,t} \leftarrow \phi^t \lambda_{k,0} \quad \epsilon_{k,t} \leftarrow \nu \lambda_{k,t } \quad \tilde{\boldsymbol{\theta}}_{k,t} \leftarrow \text { Proximal-Gradient }\left(\lambda_t, \epsilon_{k,t}, \tilde{\boldsymbol{\theta}}_{k,t-1}, \Omega\right)$
     \ENDFOR
  \STATE     $\tilde{\boldsymbol{\theta}}_{k,T} \leftarrow$ Proximal-Gradient $\left(\lambda_{k,t g t}, \epsilon_{k,t g t}, \tilde{\boldsymbol{\theta}}_{k,T-1}, \Omega\right)$
     \RETURN $\tilde{\boldsymbol{\theta}}_{k,t g t}=\tilde{\boldsymbol{\theta}}_{k,T}$
 \end{algorithmic}
    \end{algorithm}

\begin{algorithm} 
    \caption{$\btheta \leftarrow$ Proximal-Gradient $\left(\lambda, \epsilon, \boldsymbol{\theta}^0, \Omega\right)$}
    \label{proximal}
 \begin{algorithmic}
     \STATE \textbf{input:} $\lambda>0, \epsilon>0, \boldsymbol{\theta}^0 \in \mathbb{R}^d, \Omega$
      \STATE \textbf{parameter:} $\eta > 0$
       \STATE \textbf{initialize:} $j \leftarrow 0$
\WHILE{$w_\lambda\left(\boldsymbol{\theta}^j\right)>\epsilon$}
        \STATE \hspace*{5mm} $j \leftarrow j+1 \quad \boldsymbol{\theta}^{j+1} \leftarrow \mathcal{S}_{\lambda \eta}\left(\boldsymbol{\theta}^j, \Omega\right)$
     \ENDWHILE
     \RETURN $\btheta^{j+1}$
 \end{algorithmic}
    \end{algorithm}

\subsection{Preliminary Results}\label{sec_preliminary}

Recall that we consider the case $\gamma(y)=1/\PP(Y=y)$ and for simplicity denote $\|\hat{\bomega}_k\|_2 = \sqrt{1+\|\hat{\btheta}_{k}\|_2^2}$ in the rest of the proof.

We also allow the kernel with unbounded support which satisfies the following tail condition. Assume that there exits a sequence $C_N>0$ that may depend on $N$ such that
\begin{equation}\label{kerneltailbound}
     \int_{C_{N}/2}^\infty |K(t)|dt \leq C\delta_k^\beta,
\end{equation}
holds for any $2 \leq k \leq K$, where $C$ is a constant that does not depend on $k$ and $\delta_k$ is the bandwidth parameter in the $k$th iteration. 

For example, the Gaussian kernel satisfies  
 $\int_{C_{N}/2}^\infty |K(t)|dt \leq  \frac{C}{C_N} e^{-C_N^2/8}$.
 Note that in Theorem~\ref{Kiteration}, we choose $\delta_k \asymp \left(\frac{C_N K s \log d}{N}\right)^{1/(2\beta)}$, for $2\leq k \leq K$. In this case, (\ref{kerneltailbound}) requires $ \frac{ e^{-C_N^2/8}}{C_N} \leq C \sqrt{\frac{C_N K s \log d}{N}}$, which holds with $C_N \asymp \sqrt{\log N}$. For the kernel with bounded support, we can simply set $C_N=O(1)$ in the rest of the proof. 

\begin{lemma}\label{var1}
   Under Assumptions \ref{asp3}, \ref{asp4} (i) and (ii), \ref{asp5} and \ref{asp1}, we have for any $(X_i, \bZ_i, Y_i) \in D_k$, $2 \leq k \leq K$, and for all $j=1 \cdots, d$, 
   \[
   \mathbb{E}\left[\left(\gamma(Y_i)\frac{Y_i Z_{ij}}{\delta_k} K\left(\frac{Y_i\left(X_i-\boldsymbol{\theta}^{*T} \boldsymbol{Z}_i\right)}{\delta_k}\right) R_i\right)^2 \mid \hat{\btheta}_{k-1}\right] \leq C \frac{c_{n,k}}{\delta_k},
   \]
   for some constant $C>0$.
\end{lemma}
\begin{proof}
For simplicity we omit the subscript $i$ in the proof.
Note that since $\hat{\btheta}_{k-1}$ is independent of $(X,\bZ,Y)$, we have
\begin{align*}
     &\EE\left[\left(\gamma(Y)\frac{Y Z_j}{\delta_k} K\left(\frac{Y\left(X-\boldsymbol{\theta}^{*T} \boldsymbol{Z}\right)}{\delta_k}\right) R\right)^2 \mid \hat{\btheta}_{k-1} \right]\\
     =& \EE\left[  \EE\left[\left(\gamma(Y)\frac{Y Z_j}{\delta_k} K\left(\frac{Y\left(X-\boldsymbol{\theta}^{*T} \boldsymbol{Z}\right)}{\delta_k}\right) R\right)^2 \mid \hat{\btheta}_{k-1}, Y \right] \mid \hat{\btheta}_{k-1}\right]\\
     =&  \EE \left[ \left(\frac{Z_j}{\delta_k} K\left(\frac{X-\boldsymbol{\theta}^{*T} \boldsymbol{Z}}{\delta_k}\right) R\right)^2   \mid \hat{\btheta}_{k-1}, Y=1\right] 
     + \EE \left[ \left(\frac{Z_j}{\delta_k} K\left(\frac{X-\boldsymbol{\theta}^{*T} \boldsymbol{Z}}{\delta_k}\right) R\right)^2   \mid \hat{\btheta}_{k-1}, Y=-1\right].
\end{align*}
 We bound the first term here and the second term follows similarly. Note that
\begin{align*}
   &\EE \left[ \left(\frac{Z_j}{\delta_k} K\left(\frac{X-\boldsymbol{\theta}^{*T} \boldsymbol{Z}}{\delta_k}\right) R\right)^2   \mid \hat{\btheta}_{k-1}, Y=1\right] \\
   =&  \EE\left[ \EE \left[ \left(\frac{Z_j}{\delta_k} K\left(\frac{X-\boldsymbol{\theta}^{*T} \boldsymbol{Z}}{\delta_k}\right) R\right)^2   \mid \hat{\btheta}_{k-1}, Y=1, X, \bZ\right] \mid \hat{\btheta}_{k-1}, Y=1\right]. 
\end{align*}
Recall that 
\[
S_k := \left\{ (X,\bZ) :-b_{k-1} \leq \frac{X-\hat{\btheta}_{k-1}^T\bZ}{\|\hat{\bomega}_{k-1}\|_2} \leq b_{k-1}\right\},
\]
and 
 \[
\PP(R=1 \mid  X, \bZ, \hat{\btheta}_{k-1})=c_{n,k}\cdot \ind\{(X,\bZ) \in S_k\},
\]
we have
\begin{align*}
    &\EE \left[ \left(\frac{Z_j}{\delta_k} K\left(\frac{X-\boldsymbol{\theta}^{*T} \boldsymbol{Z}}{\delta_k}\right) R\right)^2   \mid \hat{\btheta}_{k-1}, Y=1, X, \bZ\right] \\
    =& c_{n,k} \left(\frac{Z_j}{\delta_k} K\left(\frac{X-\boldsymbol{\theta}^{*T} \boldsymbol{Z}}{\delta_k}\right) \right)^2  \cdot \ind\{(X,\bZ) \in S_k\},
\end{align*}
where we use the fact that $R \perp Y \mid (X,\bZ, \hat{\btheta}_{k-1})$.
Hence
\begin{align*}
     &\EE \left[ \left(\frac{Z_j}{\delta_k} K\left(\frac{X-\boldsymbol{\theta}^{*T} \boldsymbol{Z}}{\delta_k}\right) R\right)^2   \mid \hat{\btheta}_{k-1}, Y=1\right] \\
   =& c_{n,k} \EE\left[ \left(\frac{Z_j}{\delta_k} K\left(\frac{X-\boldsymbol{\theta}^{*T} \boldsymbol{Z}}{\delta_k}\right) \right)^2  \cdot \ind\{(X,\bZ) \in S_k\} \mid \hat{\btheta}_{k-1}, Y=1\right] 
    \\
  =& c_{n,k} \int_{\bz} \frac{z_j^2}{\delta_k^2} \int_{-b_{k-1}\|\hat{\bw}_{k-1}\|_2 +\hat{\btheta}_{k-1}^T\bz}^{b_{k-1}\|\hat{\bw}_{k-1}\|_2+ \hat{\btheta}_{k-1}^T\bz } K^2\left(\frac{x-\boldsymbol{\theta}^{*T} \bz}{\delta_k}\right)f(x \mid \bz, Y=1)dx f(\bz \mid Y=1) d\bz\\
  =& \frac{c_{n,k}}{\delta_k} \int_{\bz} \int_{(-b_{k-1}\|\hat{\bw}_{k-1}\|_2 +(\hat{\btheta}_{k-1}-\btheta^*)^T\bz)/\delta_k}^{(b_{k-1}\|\hat{\bw}_{k-1}\|_2 +(\hat{\btheta}_{k-1}-\btheta^*)^T\bz)/\delta_k} z_j^2 K^2(u) f(u\delta_k+\btheta^{*T}\bz \mid \bz, Y=1)du f(\bz \mid Y=1)d\bz.
\end{align*}
 Since  $\sup _{x \in \mathbb{R}, y \in \{-1,1\} ,\boldsymbol{z} \in \mathbb{R}^d} f(x \mid y,  \boldsymbol{z})<p_{\max }<\infty$, we have
\begin{align*}
     \int_{(-b_{k-1}\|\hat{\bw}_{k-1}\|_2 +(\hat{\btheta}_{k-1}-\btheta^*)^T\bz)/\delta_k}^{(b_{k-1}\|\hat{\bw}_{k-1}\|_2 +(\hat{\btheta}_{k-1}-\btheta^*)^T\bz)/\delta_k} K^2(u) f(u\delta_k+\btheta^{*T}\bz \mid \bz, Y=1)du \leq p_{\max} \int K^2(u) du.
\end{align*}
Note that $\EE(Z_j^2|Y=1) \leq M_1 < \infty$, we obtain
\[
\EE \left[ \left(\frac{Z_j}{\delta_k} K\left(\frac{X-\boldsymbol{\theta}^{*T} \boldsymbol{Z}}{\delta_k}\right) R\right)^2   \mid \hat{\btheta}_{k-1}, Y=1\right]
 \leq \frac{c_{n,k}}{\delta_k} M_1 p_{\max} \int K^2(u) du,
\]
hence we finish the proof.
\end{proof}

\begin{proposition}\label{E_1}
Denote $\nabla  R_{\delta_k, \hat{\btheta}_{k-1}}(\btheta^*) = \EE\left(\nabla  R^{D_k}_{\delta_k}(\btheta^*) \mid \hat{\btheta}_{k-1}\right)$ for $2 \leq k \leq K$.
Under Assumptions \ref{asp3}, \ref{asp4} (i) and (ii), \ref{asp5} and \ref{asp1}, with probability greater than $1-2d^{-1}$, we have
\[
\|\nabla  R^{D_k}_{\delta_k}(\btheta^*)  - \nabla  R_{\delta_k, \hat{\btheta}_{k-1}}(\btheta^*) \|_\infty \leq 
C_1\sqrt{\frac{c_{n,k}K\log d }{n\delta_k}},
\]
where $C_1$ is a constant independent of $n$, $d$ and $k$.
\end{proposition}

 \begin{proof} 
 Denote $T^k=\nabla  R^{D_k}_{\delta_k}(\btheta^*)  - \nabla  R_{\delta_k, \hat{\btheta}_{k-1}}(\btheta^*) $. By definition
 \[
 \|T^k\|_\infty = \left\|  \frac{K}{n}\sum_{(x_i, \bz_i) \in D_k}
\left(  \gamma(y_i)\frac{y_i\bz_i}{\delta_k}K(\frac{y_i(x_i-\btheta^{*T}\bz_i)}{\delta_k})R_i\right)- \EE\left[\gamma(Y)\frac{Y\bZ}{\delta_k}K(\frac{Y(X-\btheta^{*T}\bZ)}{\delta_k})R \mid \hat{\btheta}_{k-1}\right] \right\|_\infty.
 \]
 Note that for some constant $C_1$, we have
\begin{align*}
    |T^k_{ij}| = &\left|\gamma(y_i)
  \frac{y_i\bz_{ij}}{\delta_k}K(\frac{y_i(x_i-\btheta^{*T}\bz_i)}{\delta_k})R_i
  - \mathbb{E}\left[\gamma(Y)\frac{Y Z_j}{\delta_k} K\left(\frac{Y\left(X-\boldsymbol{\theta}^{*T} \boldsymbol{Z}\right)}{\delta_k}\right)R \mid \hat{\btheta}_{k-1} \right]
  \right|\\
  \leq & C_1\frac{M_n K_{\max}}{\delta_k},
\end{align*}
and by Lemma~\ref{var1} we know that  for some constant $C_2$, we have
\[
\mathbb{E}((T_{ij}^k)^2 \mid \hat{\btheta}_{k-1}) \leq C_2 \frac{c_{n,k}}{\delta_k} .
\]
Then by Bernstein inequality we have
\begin{align*}
 \mathbb{P}\left(\|T^k\|_{\infty}>t \mid \hat{\btheta}_{k-1} \right) 
\leq& \sum_{j=1}^d \mathbb{P}\left(|T_{j}^k|>t \mid \hat{\btheta}_{k-1}\right) \\
 \leq& 2 d \exp \left(-\frac{\frac{1}{2} t^2 n/K}{  C_2 \frac{c_{n,k}}{\delta_k}+\frac{t}{3} C_1 M_n K_{\max} / \delta_k}\right).
\end{align*}
Since the right side doesn't contain $\hat{\btheta}_{k-1}$, we obtain that
\[
\mathbb{P}\left(
\|\nabla  R^{D_k}_{\delta_k}(\btheta^*)  - \nabla  R_{\delta_k, \hat{\btheta}_{k-1}}(\btheta^*) \|_\infty >t  \right) \leq 2 d \exp \left(-\frac{\frac{1}{2} t^2 n/K}{  C_2 \frac{c_{n,k}}{\delta_k}+\frac{t}{3} C_1 M_n K_{\max} / \delta_k}\right).
\]
Then note that $M_n \lesssim \sqrt{\frac{c_{n,k} n\delta_k}{K\log d}}$ and take $t=C_3\sqrt{\frac{c_{n,k}K\log d }{n\delta_k}}$ for some constant $C_3$ we finish the proof.`
 \end{proof}

\begin{proposition}\label{prop2}
Recall that $\btheta^* = \argmin_{\btheta} R(\btheta)$, where $R(\btheta) =\EE\left( \gamma(Y)L_{01}(Y(X-\btheta^T\bZ))\right)$, and $\nabla  R_{\delta_k, \hat{\btheta}_{k-1}}(\btheta^*) = \EE\left(\nabla  R^{D_k}_{\delta_k}(\btheta^*) \mid \hat{\btheta}_{k-1}\right)$ for $2 \leq k \leq K$. Consider the following two cases:
\begin{itemize}
    \item[(i)] Assumptions \ref{asp3}, \ref{asp4} (i) and (ii), \ref{asp5} and \ref{asp1} hold, and 
\begin{equation}\label{prop2condi1}
    b_{k-1} \geq C_{N} \delta_k  \text{\ and\ }
b_{k-1} \geq 2 \|\hat{\btheta}_{k-1}-\btheta^*\|_1 M_n.
\end{equation}
\item[(ii)] Assumptions \ref{asp3}, \ref{asp4} (i), (ii) and (iii), \ref{asp5} and \ref{asp1} hold, and 
\begin{equation} \label{prop2condi2}
    b_{k-1}\geq C_N\delta_k,~~\textrm{and}~~b_{k-1}\geq C\|\hat\btheta_{k-1}-\btheta^*\|_2 \sqrt{\log \frac{N}{K s \log d}},
\end{equation} 
where $C_{N}>0$ is defined in (\ref{kerneltailbound}) and there exists a large constant $\zeta$ such that $(\frac{Ks\log d}{N})^\zeta\lesssim \delta_k$ with $\delta_k=o(1)$. 
\end{itemize}
If either (i) or (ii) holds, we have for any $\boldsymbol{v} \in \mathbb{R}^d$ with $\|\boldsymbol{v}\|_0 \leq s^{\prime}$,
    \[
    \left|\bv^T \left( \nabla  R_{\delta_k, \hat{\btheta}_{k-1}}(\btheta^*) - \nabla R(\btheta^*)\right)\right| \leq   C_2c_{n,k}\delta_k^\beta \|\bv\|_2,
    \]
    where $s^\prime$ is defined in Assumption~\ref{asp4} and $C_2$ is a constant independent of $n$, $d$ and $k$.
\end{proposition}
\begin{proof}
By definition we have
\begin{align*}
R(\btheta) = \EE\left( \ind(X<\btheta^T\bZ) \mid  Y=1\right)   +  \EE\left( \ind(X>\btheta^T\bZ) \mid  Y=-1\right),
\end{align*}
and 
\begin{align*}
   c_{n,k} \nabla R(\btheta^*) =&  c_{n,k} \int_{\boldsymbol{z}} \boldsymbol{z}  f\left(\boldsymbol{\theta}^{*T} \boldsymbol{z} \mid \boldsymbol{z}, Y=1\right) f(\boldsymbol{z}\mid Y=1) d \boldsymbol{z}\\
    -& c_{n,k} \int_{\boldsymbol{z}} \boldsymbol{z} f\left(\boldsymbol{\theta}^{*T} \boldsymbol{z} \mid \boldsymbol{z} , Y=-1\right) f(\boldsymbol{z} \mid Y=-1) d \boldsymbol{z}.
\end{align*}
 Note that
\begin{align*}
    \nabla  R_{\delta_k, \hat{\btheta}_{k-1}}(\btheta^*)=& \EE\left(\gamma(Y)\frac{Y\bZ}{\delta_k}K(\frac{Y(X-\btheta^{*T}\bZ)}{\delta_k})R    \mid \hat{\btheta}_{k-1} \right)\\
    =&  \EE\left(\frac{\bZ}{\delta_k}K(\frac{X-\btheta^{*T}\bZ}{\delta_k})R \mid Y=1,  \hat{\btheta}_{k-1}\right)
    -  \EE\left(\frac{\bZ}{\delta_k}K(\frac{X-\btheta^{*T}\bZ}{\delta_k})R \mid Y=-1,   \hat{\btheta}_{k-1}\right),
\end{align*}
hence 
\begin{align*}
    &\nabla  R_{\delta_k, \hat{\btheta}_{k-1}}(\btheta^*) - \nabla R(\btheta^*)\\
    =&
    \EE\left(\frac{\bZ}{\delta_k}K(\frac{X-\btheta^{*T}\bZ}{\delta_k})R \mid Y=1,   \hat{\btheta}_{k-1}\right)-
    c_{n,k}\int_{\boldsymbol{z}} \boldsymbol{z} f\left(\boldsymbol{\theta}^{*T} \boldsymbol{z} \mid \boldsymbol{z}, Y=1\right) f(\boldsymbol{z} \mid Y=1) d \boldsymbol{z}\\
    -& \EE\left(\frac{\bZ}{\delta_k}K(\frac{X-\btheta^{*T}\bZ}{\delta_k})R \mid Y=-1,   \hat{\btheta}_{k-1}\right) +c_{n,k}\int_{\boldsymbol{z}} \boldsymbol{z}  f\left(\boldsymbol{\theta}^{*T} \boldsymbol{z} \mid \boldsymbol{z}, Y=-1\right) f(\boldsymbol{z} \mid Y=-1) d \boldsymbol{z}. 
\end{align*}
We will bound the first term and the second term follows similarly. Denote 
\[
S_u = \{u:\frac{-b_{k-1}\|\hat{\bw}_{k-1}\|_2 +(\hat{\btheta}_{k-1}-\btheta^*)^T\bz}{\delta_k} \leq u \leq \frac{b_{k-1}\|\hat{\bw}_{k-1}\|_2 +(\hat{\btheta}_{k-1}-\btheta^*)^T\bz}{\delta_k}\}.
\]
We have
\begin{align*}
    &\bv^T \EE\left(\frac{\bZ}{\delta_k}K(\frac{X-\btheta^{*T}\bZ}{\delta_k})R \mid Y=1,  \hat{\btheta}_{k-1}\right)\\
    =& c_{n,k} \int_{\bz} \int_{-b_{k-1}\|\hat{\bw}_{k-1}\|_2 +\hat{\btheta}_{k-1}^T\bz}^{b_{k-1}\|\hat{\bw}_{k-1}\|_2+ \hat{\btheta}_{k-1}^T\bz } \frac{\bv^T\bz}{\delta_k}K\left(\frac{x-\btheta^{*T} \bz}{\delta_k}\right)f(x \mid \bz, Y=1)dx f(\bz \mid Y=1) d\bz \\
    =&     c_{n,k}\int_{\bz} \bv^T\bz
    \int_{S_u}K(u) f(u\delta_k+\btheta^{*T}\bz\mid\bz, Y=1)duf(\bz \mid Y=1)d\bz\\
    =&  c_{n,k}\int_{\bz} \bv^T\bz
    \int K(u) f(u\delta_k+\btheta^{*T}\bz\mid\bz, Y=1)duf(\bz \mid Y=1)d\bz\\
    & ~~~-c_{n,k}\int_{\bz} \bv^T\bz
    \int_{S_u^c} K(u) f(u\delta_k+\btheta^{*T}\bz\mid\bz, Y=1)duf(\bz \mid Y=1)d\bz.
\end{align*}
Therefore,
\begin{align}
    &\bv^T \left(
     \EE\left(\frac{\bZ}{\delta_k}K(\frac{X-\btheta^{*T}\bZ}{\delta_k})R \mid Y=1, \hat{\btheta}_{k-1}\right)-
    c_{n,k}\int_{\boldsymbol{z}} \boldsymbol{z} f\left(\boldsymbol{\theta}^{*T} \boldsymbol{z} \mid \boldsymbol{z}, Y=1\right) f(\boldsymbol{z}, Y=1) d \boldsymbol{z} \right)\nonumber\\
    =& c_{n,k}\int_{\bz} \bv^T\bz
    \underbrace{\int K(u) \left(f(u\delta_k+\btheta^{*T}\bz\mid\bz , Y=1) - f(\btheta^{*T}\bz\mid\bz , Y=1)\right) du}_{(A)}
    f(\bz \mid Y=1)d\bz\nonumber\\
    -& c_{n,k}\int_{\bz} \bv^T\bz
    \int_{S_u^c} K(u) f(u\delta_k+\btheta^{*T}\bz\mid\bz, Y=1)duf(\bz \mid Y=1)d\bz.\label{eq_prop2_1}
\end{align}
Now we look at the first term when $\beta > 1$. Since $f(x \mid  \boldsymbol{z},y)$ is $l$ times differentiable, by Taylor expansion we have
\begin{align*} 
 &f\left(u \delta_k+\boldsymbol{\theta}^{* T} \boldsymbol{z} \mid \boldsymbol{z}, Y=1\right)-f\left(\boldsymbol{\theta}^{* T} \boldsymbol{z} \mid \boldsymbol{z}, Y=1 \right)\\
 =&\sum_{i=1}^{l-1} \frac{f^{(i)}\left(\boldsymbol{\theta}^{* T} \boldsymbol{z} \mid \boldsymbol{z}, Y=1\right)}{i !}(u \delta_k)^i +\frac{(u \delta_k)^l}{l !} f^{(l)}\left(\boldsymbol{\theta}^{* T} \boldsymbol{z}+\tau u \delta_k \mid \boldsymbol{z}, Y=1\right)  \end{align*}
 for some $\tau\in[0,1]$ where $l=\lfloor\beta\rfloor$. 
By the definition of kernel of order $l$, we obtain
\begin{align*}
 (A) = \int K(u) \frac{(u \delta_k)^l}{l !}\left(f^{(l)}\left(\boldsymbol{\theta}^{* T} \boldsymbol{z}+\tau u \delta_k \mid \boldsymbol{z}, Y=1\right)-f^{(l)}\left(\boldsymbol{\theta}^{* T} \boldsymbol{z} \mid \boldsymbol{z}, Y=1\right)\right) d u,
\end{align*}
hence for any $\bv$ with $\|\bv\|_0 \leq s^\prime$,
\begin{align*}
   &\left| c_{n,k}\int_{\bz} \bv^T\bz
    \int K(u) \left(f(u\delta_k+\btheta^{*T}\bz\mid\bz , Y=1) - f(\btheta^{*T}\bz\mid\bz , Y=1)\right) du
    f(\bz \mid Y=1)d\bz\right|\\
    =& \left| c_{n,k}  \int K(u)  \frac{(u \delta_k)^l}{l !}\int_{\bz} \bv^T\bz
   \left(f^{(l)}\left(\boldsymbol{\theta}^{* T} \boldsymbol{z}+\tau u \delta_k \mid \boldsymbol{z}, Y=1\right)-f^{(l)}\left(\boldsymbol{\theta}^{* T} \boldsymbol{z} \mid \boldsymbol{z}, Y=1\right)\right) 
    f(\bz \mid Y=1)d\bz du\right|\\
    \leq & c_{n,k} \int  |K(u)| \frac{|u \delta_k|^l}{l !}L\|\bv\|_2 |u\delta_k|^{\beta-l}du\\
    \leq &  c_{n,k} L\|\bv\|_2 \int |K(u)|  \frac{|u \delta_k|^\beta}{l !}du.
\end{align*}
When $0 < \beta \leq 1$, by definition (\ref{densityclass}), we have
\[
|f\left(u \delta_k+\boldsymbol{\theta}^{* T} \boldsymbol{z} \mid \boldsymbol{z}, Y=1\right)-f\left(\boldsymbol{\theta}^{* T} \boldsymbol{z} \mid \boldsymbol{z}, Y=1 \right)| \leq L \delta_k^\beta |u|^\beta,
\]
hence for any $\bv$ with $\|\bv\|_0 \leq s^\prime$,
\begin{align*}
   &\left| c_{n,k}\int_{\bz} \bv^T\bz
    \int K(u) \left(f(u\delta_k+\btheta^{*T}\bz\mid\bz , Y=1) - f(\btheta^{*T}\bz\mid\bz , Y=1)\right) du
    f(\bz \mid Y=1)d\bz\right|\\
    \leq& c_{n,k} \int |K(u)|  \int_{\bz} |\bv^T\bz| L \delta_k^\beta |u|^\beta f(\bz \mid Y=1)d\bz du\\
    \leq& c_{n,k}L \|\bv\|_2 \delta_k^\beta \int |K(u)| |u|^\beta du.
\end{align*}
For the second term on the right hand side of (\ref{eq_prop2_1}), we first show the result under (\ref{prop2condi1}), i.e., 
$b_{k-1} \geq C_{N} \delta_k  \text{\ and\ }
b_{k-1} \geq 2 \|\hat{\btheta}_{k-1}-\btheta^*\|_1 M_n$. Then  
we have
\[
\frac{b_{k-1}\|\hat{\bw}_{k-1}\|_2 +(\hat{\btheta}_{k-1}-\btheta^*)^T\bz}{\delta_k} \geq \frac{b_{k-1}\|\hat{\bw}_{k-1}\|_2 -\|\hat{\btheta}_{k-1}-\btheta^*\|_1 M_n}{\delta_k},
\]
and
\[
\frac{-b_{k-1}\|\hat{\bw}_{k-1}\|_2 +(\hat{\btheta}_{k-1}-\btheta^*)^T\bz}{\delta_k} \leq \frac{-b_{k-1}\|\hat{\bw}_{k-1}\|_2 +\|\hat{\btheta}_{k-1}-\btheta^*\|_1 M_n}{\delta_k}.
\]
Therefore, by (\ref{kerneltailbound}) we have
\begin{align*}
\int_{S_u^c} |K(u)| du     =& \int_{\frac{b_{k-1}\|\hat{\bw}_{k-1}\|_2 +(\hat{\btheta}_{k-1}-\btheta^*)^T\bz}{\delta_k}}^\infty |K(u)| du + 
     \int_{-\infty}^{\frac{-b_{k-1}\|\hat{\bw}_{k-1}\|_2 +(\hat{\btheta}_{k-1}-\btheta^*)^T\bz}{\delta_k}} |K(u)| du\\
     \leq &  \int_{\frac{b_{k-1}\|\hat{\bw}_{k-1}\|_2 -\|\hat{\btheta}_{k-1}-\btheta^*\|_1 M_n}{\delta_k}}^\infty |K(u)|du + 
      \int_{-\infty}^{\frac{-b_{k-1}\|\hat{\bw}_{k-1}\|_2 +\|\hat{\btheta}_{k-1}-\btheta^*\|_1 M_n}{\delta_k}} |K(u)| du \\
      = & 2 \int_{\frac{b_{k-1}\|\hat{\bw}_{k-1}\|_2 -\|\hat{\btheta}_{k-1}-\btheta^*\|_1 M_n}{\delta_k}}^\infty |K(u)| du \\
      \leq & 2 \int_{C_{N}/2}^\infty |K(u)|du = O(\delta_k^\beta),
\end{align*}
where the last inequality follows that $\|\hat{\bomega}_{k-1}\|_2 = \sqrt{1+\|\hat{\btheta}_{k-1}\|_2^2} \geq 1$, hence
\begin{align*}
    \frac{b_{k-1}\|\hat{\bw}_{k-1}\|_2 -\|\hat{\btheta}_{k-1}-\btheta^*\|_1 M_n}{\delta_k} \geq& 
    \frac{b_{k-1}\|\hat{\bw}_{k-1}\|_2 - b_{k-1}/2}{\delta_k}
    \geq \frac{b_{k-1}}{2\delta_k} \geq C_{N}/2.
\end{align*}
Finally, since  $\sup _{x \in \mathbb{R}, y \in \{-1,1\} ,\boldsymbol{z} \in \mathbb{R}^d} f(x \mid y,  \boldsymbol{z})<p_{\max }$, and $\sup _{\|\boldsymbol{v}\|_0 \leq s^{\prime}} \frac{\boldsymbol{v}^T \mathbb{E}\left(\boldsymbol{Z} \boldsymbol{Z}^T \mid Y=y\right) \boldsymbol{v}}{\|\boldsymbol{v}\|_2^2} \leq L^2$, we have
\begin{align*}
    &|c_{n,k}\int_{\bz} \bv^T\bz
    \int_{S_u^c} K(u) f(u\delta_k+\btheta^{*T}\bz\mid\bz, Y=1)duf(\bz \mid Y=1)d\bz |\\
    \lesssim & c_{n,k} p_{\max} |\EE(|\bv^T \bz| \mid Y=1)| \delta_k^\beta\\
    \lesssim &  c_{n,k} p_{\max} \sqrt{\EE( (\bv^T \bz)^2 \mid Y=1 )} \delta_k^\beta\\
    \lesssim &  c_{n,k} p_{\max} \|\bv\|_2 L \delta_k^\beta.
\end{align*}
Combining the bound for the two terms in (\ref{eq_prop2_1}) we finish the proof.

Now we show  the result also holds under (\ref{prop2condi2}), i.e., 
$   b_{k-1}\geq C_N\delta_k,~~\textrm{and}~~b_{k-1}\geq C\|\hat\btheta_{k-1}-\btheta^*\|_2 \sqrt{\log \frac{N}{K s \log d}}$. For the second term on the right hand side of (\ref{eq_prop2_1}),  we only need to check that
\[
\int_{\bz} \bv^T\bz
    \int_{S_u^c} K(u) f(u\delta_k+\btheta^{*T}\bz\mid\bz, Y=1)duf(\bz \mid Y=1)d\bz = O(\delta_k^\beta).
\]
Since  $\sup _{x \in \mathbb{R}, y \in \{-1,1\} ,\boldsymbol{z} \in \mathbb{R}^d} f(x \mid y,  \boldsymbol{z})<p_{\max }$,
    we have
    \begin{align*}
        &\int_{\bz} \bv^T\bz
    \int_{S_u^c} K(u) f(u\delta_k+\btheta^{*T}\bz\mid\bz, Y=1)duf(\bz \mid Y=1)d\bz \\
    \leq & p_{\max } \int_{\bz} \bv^T\bz
    \int_{S_u^c} K(u) duf(\bz \mid Y=1)d\bz \\
    \lesssim& \underbrace{\int_{|(\hat{\btheta}_{k-1} -\btheta^*)^T \bz| \leq C \|\hat{\btheta}_{k-1}-\btheta^*\|_2 \sqrt{\log (\frac{N}{K s \log d})}} \bv^T\bz \int_{S_u^c} K(u) duf(\bz \mid Y=1)d\bz}_{\text{A}}\\
    +&
     \underbrace{\int_{|(\hat{\btheta}_{k-1} -\btheta^*)^T \bz| > C \|\hat{\btheta}_{k-1}-\btheta^*\|_2 \sqrt{\log (\frac{N}{K s \log d})}} \bv^T\bz \int_{S_u^c} K(u) duf(\bz \mid Y=1)d\bz}_{\text{B}}.
    \end{align*}
    For the term $\text{A}$,  we have
\begin{align*}
\int_{S_u^c} |K(u)| du     =& \int_{\frac{b_{k-1}\|\hat{\bw}_{k-1}\|_2 +(\hat{\btheta}_{k-1}-\btheta^*)^T\bz}{\delta_k}}^\infty |K(u)| du + 
     \int_{-\infty}^{\frac{-b_{k-1}\|\hat{\bw}_{k-1}\|_2 +(\hat{\btheta}_{k-1}-\btheta^*)^T\bz}{\delta_k}} |K(u)| du\\
     \leq &  \int_{\frac{b_{k-1}\|\hat{\bw}_{k-1}\|_2 -C\|\hat{\btheta}_{k-1}-\btheta^*\|_2 \sqrt{\log (\frac{N}{K s \log d})}/2}{\delta_k}}^\infty |K(u)|du + 
      \int_{-\infty}^{\frac{-b_{k-1}\|\hat{\bw}_{k-1}\|_2 +C\|\hat{\btheta}_{k-1}-\btheta^*\|_2 \sqrt{\log (\frac{N}{K s \log d})}/2}{\delta_k}} |K(u)| du \\
      = & 2 \int_{\frac{b_{k-1}\|\hat{\bw}_{k-1}\|_2 -C\|\hat{\btheta}_{k-1}-\btheta^*\|_2 \sqrt{\log (\frac{N}{K s \log d})}/2}{\delta_k}}^\infty |K(u)| du \\
      \leq & 2 \int_{C_{N}/2}^\infty |K(u)|du = O(\delta_k^\beta),
\end{align*}
where the last inequality follows  that $\|\hat{\bomega}_{k-1}\|_2 = \sqrt{1+\|\hat{\btheta}_{k-1}\|_2^2} \geq 1$, (\ref{kerneltailbound}) and
\begin{align*}
    \frac{b_{k-1}\|\hat{\bw}_{k-1}\|_2 -C\|\hat{\btheta}_{k-1}-\btheta^*\|_2 \sqrt{\log (\frac{N}{K s \log d})}/2}{\delta_k} \geq& 
    \frac{b_{k-1}\|\hat{\bw}_{k-1}\|_2 - b_{k-1}/2}{\delta_k}
    \geq \frac{b_{k-1}}{2\delta_k} \geq C_{N}/2.
\end{align*}
For the term $\text{B}$, since $(\hat{\btheta}_{k-1} -\btheta^*)^T \bZ \mid Y=1$ is sub-Gaussian with sub-Gaussian norm that scales with $\|\hat{\btheta}_{k-1} -\btheta^*\|_2$, we have
\[
\PP\left(|(\hat{\btheta}_{k-1} -\btheta^*)^T \bZ| > C \|\hat{\btheta}_{k-1}-\btheta^*\|_2 \sqrt{\log (\frac{N}{ K s \log d})}
\mid Y=1
\right)
\leq 2 \left(\frac{K s \log d}{N}\right)^{C'},
\]
hence we can choose $C$ sufficiently large such that $C'/2\geq \zeta\beta$ and thus $\left(\frac{K s \log d}{N}\right)^{C^\prime/2}\lesssim \delta_k^\beta$. Therefore, 
\begin{align*}
    B \lesssim& \EE\left[\bv^T\bZ \ind\{|(\hat{\btheta}_{k-1} -\btheta^*)^T \bZ| > C \|\hat{\btheta}_{k-1}-\btheta^*\|_2 \sqrt{\log (\frac{N}{K s \log d})}\}  \mid Y=1\right] \\
    \leq& \sqrt{\EE( (\bv^T \bZ)^2 \mid Y=1 )} \sqrt{\PP\left(|(\hat{\btheta}_{k-1} -\btheta^*)^T \bZ| > C \|\hat{\btheta}_{k-1}-\btheta^*\|_2 \sqrt{\log (\frac{N}{K s \log d})} \mid Y=1\right)} \\
    \lesssim& \|\bv\|_2 \left(\frac{K s \log d}{N}\right)^{C^\prime/2} 
    \lesssim \|\bv\|_2 \delta_k^\beta.
\end{align*}
Combining the bound for the two terms in (\ref{eq_prop2_1}) we finish the proof.

\end{proof}

\subsection{Proof of the Main Results}

\begin{proof}[Proof of Theorem \ref{rate_k}]
Note that our estimator is defined as $\hat{\btheta}_k:= \tilde{\btheta}_{k,tgt}$, where $\tilde{\btheta}_{k,tgt}$ represents the approximate local solution from the path-following algorithm. Therefore, by Theorem~\ref{paththm1},  we have that with probability greater than $1-2d^{-1}$,
     \begin{equation}\label{2norm}
           \|\hat{\btheta}_k-\boldsymbol{\theta}^*\|_2 \lesssim 
           \left(\frac{K s \log d}{nc_{n,k}}\right)^{\frac{\beta \vee 1}{2\beta+1}},
     \end{equation}
    and
   \begin{equation}\label{1norm}
        \|\hat{\btheta}_k-\boldsymbol{\theta}^*\|_1 \lesssim 
         \sqrt{s}\left(\frac{K s \log d}{nc_{n,k}}\right)^{\frac{\beta \vee 1}{2\beta+1}}.
   \end{equation}
     When $k=1$, we have
   \[
    \PP(R_i=1) = c_{n,1}, \ \PP(R_i=0) = 1-c_{n,1},
   \] 
   hence $N_1 = n c_{n,1}/K$. Plugging this back to (\ref{2norm}) and (\ref{1norm}) we get the result for $\hat{\btheta}_1$. For $2 \leq k \leq K$, recall that   
 \[
 \PP(R_i=1 \mid  X_i, \bZ_i, \hat{\btheta}_{k-1})=c_{n,k}\cdot \ind\{(X_i,\bZ_i) \in S_k\}.
 \]
We have
\begin{align*}
    \EE(R_i) =& \EE\left[\EE[R_i \mid X_i, \bZ_i, \hat{\btheta}_{k-1} ] \right]\\
    =& \EE\left[c_{n,k}\cdot \ind\{(X_i, \bZ_i)\in S_k\}\right]\\
    =& c_{n,k} \PP\left((X,\bZ)\in S_k \right),
\end{align*}
hence $N_k =  n\EE(R_i)/K = n c_{n,k} \PP\left((X,\bZ)\in S_k \right) /K$. Plugging this back to (\ref{2norm}) and (\ref{1norm}) we finish the proof.
\end{proof}

\begin{proof}[Proof of Theorem \ref{beta>beta_*}]
First, let's consider the case when $k=1$. For each $(X_i, \bZ_i) \in D_1$, we have
 \[
 \PP(R_i=1) = c_{n,1}, \ \PP(R_i=0) = 1-c_{n,1},
 \]
and $ N/K=N_1= nc_{n,1}/K$. 
According to Theorem~\ref{rate_k},  by selecting 
    \[
    \delta_1 = c_{1}\left(\frac{Ks \log d}{nc_{n,1}}\right)^{1/(2\beta+1)} = c_{1}\left(\frac{K s \log d}{N}\right)^{1/(2\beta+1)}
    \]
    and $\lambda_1 = c_{2} \sqrt{\frac{c_{n,1} K\log d }{n\delta_1}} = c_{2} \sqrt{\frac{N K\log d }{n^2\delta_1}} $,  with probability greater than $1-2d^{-1}$, we obtain
\begin{equation}\label{equalsizestep1.}
    \|\hat{\btheta}_1 -\btheta^*\|_2 \lesssim  \left(\frac{K s \log d}{N}\right)^{\beta/(2\beta+1)},~~~
    \|\hat{\btheta}_1 -\btheta^*\|_1 \lesssim \sqrt{s} \left(\frac{K s \log d}{N}\right)^{\beta/(2\beta+1)}.
\end{equation}

Next, we will establish the bound for $\|\hat{\btheta}_k -\btheta^*\|_2$. The result for  $\|\hat{\btheta}_k -\btheta^*\|_1$ follows similarly. 
According to Assumption~\ref{asp5}, we have $ \sup _{x \in \mathbb{R}, \boldsymbol{z} \in \mathbb{R}^d} f(x \mid \boldsymbol{z})<p_{\max }<\infty$. Recall that
\[
 S_k := \left\{ (X,\bZ) :-b_{k-1} \leq \frac{X-\hat{\btheta}_{k-1}^T\bZ}{\sqrt{1+\|\hat{\btheta}_{k-1}\|_2^2}} \leq b_{k-1}\right\}.
\]
Let us start from $k=2$. 
Since  $\bZ|Y=y$ is sub-Gaussian with a bounded sub-Gaussian norm and independent of $\hat{\btheta}_{k-1}$, we have $(\hat{\btheta}_{k-1} -\btheta^*)^T \bZ |Y=y$ is also sub-Gaussian, with a sub-Gaussian norm that scales with $\|\hat{\btheta}_{k-1} -\btheta^*\|_2$, hence
\[
\PP\left(|(\hat{\btheta}_{k-1} -\btheta^*)^T \bZ| > c \|\hat{\btheta}_{k-1}-\btheta^*\|_2 \sqrt{\log \left(\frac{N}{K s \log d}\right)}
\right)
\leq 2 \left(\frac{K s \log d}{N}\right)^{c^\prime},
\]
where $c^\prime$ is a sufficiently large constant. Consider the following set 
\[
\cE=\Big\{(\btheta,\bZ): \|{\btheta}\|_2 \leq 2C, |({\btheta} -\btheta^*)^T \bZ| \leq c \|{\btheta}-\btheta^*\|_2 \sqrt{\log \left(\frac{N}{K s \log d}\right)}\Big\},
\]
where $C$ is the constant defined in Assumption \ref{asp3}. Note that $\|\btheta^* -\hat{\btheta}_{k-1}\|_2\leq C$ with probability greater than $1-2d^{-1}$. Since  $\|\hat{\btheta}_{k-1}\|_2 \leq \|\btheta^* -\hat{\btheta}_{k-1}\|_2 + \|\btheta^*\|_2$, the event $(\hat\btheta_{k-1},\bZ)\in \cE$ holds with probability greater than $1-2 \left(\frac{K s \log d}{N}\right)^{c^\prime}-2d^{-1}$. 
So we have
\begin{align}
     \PP\left((X,\bZ)\in S_{k} \mid \hat{\btheta}_{k-1},\bZ=\bz \right) &= \int_{-b_{k-1}\sqrt{1+\|\hat{\btheta}_{k-1}\|_2^2}+ \hat{\btheta}_{k-1}^T\bz}^{b_{k-1}\sqrt{1+\|\hat{\btheta}_{k-1}\|_2^2}+ \hat{\btheta}_{k-1}^T\bz} f(x \mid \bz) dx  \nonumber\\
    &\leq  2 b_{k-1} p_{\max } \sqrt{1+\|\hat{\btheta}_{k-1}\|_2^2}.\nonumber
\end{align}
As a result, we have
\begin{align}\label{pmaxbound.}
\PP\left((X,\bZ)\in S_k \right)&=\EE\left(\PP\left((X,\bZ)\in S_{k} \mid \hat{\btheta}_{k-1},\bZ\right) \ind\{(\hat\btheta_{k-1},\bZ)\in \cE\}\right)\nonumber\\
&~~~~+\EE\left(\PP\left((X,\bZ)\in S_{k} \mid \hat{\btheta}_{k-1},\bZ\right) \ind\{(\hat\btheta_{k-1},\bZ)\notin \cE\}\right)\nonumber\\
&\leq  2 b_{k-1} p_{\max } \sqrt{1+4C^2}+\PP((\hat\btheta_{k-1},\bZ)\notin \cE)\nonumber\\
&\leq 2 b_{k-1} p_{\max } \sqrt{1+4C^2}+2 \left(\frac{K s \log d}{N}\right)^{c^\prime}+2d^{-1}\nonumber\\
&\lesssim b_{k-1},
\end{align}
where $b_{k-1} = c_3 (\frac{ Ks\log d}{N})^{1/(2\beta)}$ and the last step follows from the fact that $d\gg N$ and $c^\prime$ is sufficiently large. In addition, it can be shown that 
\begin{align*}
    -b_{k-1}\sqrt{1+\|\hat{\btheta}_{k-1}\|_2^2}+ \hat{\btheta}_{k-1}^T\bz =& \btheta^{*T}\bz  -b_{k-1}\sqrt{1+\|\hat{\btheta}_{k-1}\|_2^2} + \left(\hat{\btheta}_{k-1} - \btheta^*\right)^T\bz \\
     \geq&  \btheta^{*T}\bz - C_1 b_{k-1},
\end{align*}
if the events $(\hat\btheta_{k-1},\bZ)\in \cE$ and
\begin{align*}
   c \|\hat{\btheta}_{k-1}-\btheta^*\|_2 \sqrt{\log \Big(\frac{N}{K s \log d}\Big)} \leq b_{k-1}/2,
\end{align*}
hold. By (\ref{equalsizestep1.}) and the choice of $b_{k-1}$, the two events hold with probability greater than $1-2 \left(\frac{K s \log d}{N}\right)^{c^\prime}-2d^{-1}$. By a similar proof,  we can show that
\[
b_{k-1}\sqrt{1+\|\hat{\btheta}_{k-1}\|_2^2}+ \hat{\btheta}_{k-1}^T\bz \leq \btheta^{*T}\bz + C_2 b_{k-1}.
\]
 Therefore, with probability greater than $1-2 \left(\frac{K s \log d}{N}\right)^{c^\prime}-2d^{-1}$, the event 
 \[
\cA_k=\Big\{ \Big[-b_{k-1}\sqrt{1+\|\hat{\btheta}_{k-1}\|_2^2}+ \hat{\btheta}_{k-1}^T\bz,b_{k-1}\sqrt{1+\|\hat{\btheta}_{k-1}\|_2^2}+ \hat{\btheta}_{k-1}^T\bz\Big] \subset  B(\btheta^{*T}\bz, \epsilon)\Big\},
 \]
holds, where $\epsilon = C_3 b_{k-1}$ for some constant $C_3$ large enough. Following the similar derivations in (\ref{pmaxbound.}), by Assumption~\ref{asp5}, we can show that
\begin{align}
        \PP\left((X,\bZ)\in S_{k} \right)&\geq \EE\Big(\PP\Big((X,\bZ)\in S_{k} \mid \hat{\btheta}_{k-1},\bZ\Big) \ind\{(\hat\btheta_{k-1},\bZ)\in \cE, \cA_k\}\Big) \nonumber\\
        &=\EE \Big\{ \ind\{(\hat\btheta_{k-1},\bZ)\in \cE, \cA_k\} \int_{-b_{k-1}\sqrt{1+\|\hat{\btheta}_{k-1}\|_2^2}+ \hat{\btheta}_{k-1}^T\bz}^{b_{k-1}\sqrt{1+\|\hat{\btheta}_{k-1}\|_2^2}+ \hat{\btheta}_{k-1}^T\bz} f(x \mid \bz) dx \Big\}\nonumber \\
        &\geq \EE \Big\{ \ind\{(\hat\btheta_{k-1},\bZ)\in \cE, \cA_k, \bZ\in\cG\} \int_{-b_{k-1}\sqrt{1+\|\hat{\btheta}_{k-1}\|_2^2}+ \hat{\btheta}_{k-1}^T\bz}^{b_{k-1}\sqrt{1+\|\hat{\btheta}_{k-1}\|_2^2}+ \hat{\btheta}_{k-1}^T\bz} f(x \mid \bz) dx \Big\}\nonumber\\
        & \geq 2b_{k-1} p_{\min }\PP\{(\hat\btheta_{k-1},\bZ)\in \cE, \cA_k, \bZ\in\cG\} \geq Cb_{k-1},\label{pminbound.}
\end{align}
for some constant $C$. Combining (\ref{pmaxbound.}) with (\ref{pminbound.}), we have $\PP\left((X,\bZ)\in S_k \right) \asymp b_{k-1} $.

To apply Theorem~\ref{rate_k}, we need to verify that $b_{k-1}$ satisfies
\begin{equation} \label{propb_k-1.}
    b_{k-1} \geq C \delta_k  ~~\text{\ and \ }~~
b_{k-1} \geq C\|\hat\btheta_{k-1}-\btheta^*\|_2 \sqrt{\log \left(\frac{N}{K s \log d}\right)}.
\end{equation}
Given $b_{k-1} = c_3 (\frac{ Ks\log d}{N})^{1/(2\beta)}$ and from Theorem~\ref{rate_k}  
\begin{equation}\label{deltakequalsize.}
    \delta_k \asymp \left(\frac{K s \log d}{nc_{n,k}}\right)^{1/(2\beta+1)} \asymp \left(\frac{s \log d ~~\PP\left((X,\bZ)\in S_{k} \right)}{N_k}\right)^{1/(2\beta+1)} \asymp \left(\frac{ b_{k-1} K s \log d }{N}\right)^{1/(2\beta+1)},
\end{equation}
where it follows from $N_k =  n\EE(R_i)/K = n c_{n,k} \PP\left((X,\bZ)\in S_k \right) /K$, we can verify that $b_{k-1} \geq C \delta_k$ holds.  In addition, by (\ref{equalsizestep1.}), $Ks\log d=o(N)$ and for any fixed $\beta > \frac{1+\sqrt{3}}{2}$ (which implies $\frac{\beta}{2\beta+1}> \frac{1}{2\beta}$), we conclude that $b_{k-1} \geq C\|\hat\btheta_{k-1}-\btheta^*\|_2 \sqrt{\log (\frac{N}{K s \log d})}$ holds with probability greater than $1-2d^{-1}$. Thus, applying Theorem~\ref{rate_k} with $\delta_k$ in (\ref{deltakequalsize.}) and  $\lambda_k = c_{2} \sqrt{\frac{N K \log d}{n^2 b_{k-1}\delta_k}}$, we obtain that with probability greater than $1-4d^{-1}$  
 \begin{equation}\label{equalsizerate.}
     \|\hat{\btheta}_k-\boldsymbol{\theta}^*\|_2 \lesssim 
         \left(\frac{\PP\left((X,\bZ)\in S_k \right)s \log d }{N_k}\right)^{\beta/(2\beta+1)}
         \lesssim \left(\frac{ b_{k-1} K s \log d  }{N}\right)^{\beta/(2\beta+1)}\lesssim \left(\frac{ K s \log d }{N}\right)^{1/2},
 \end{equation} 
where we plug in $b_{k-1} = c_3 (\frac{ Ks\log d}{N})^{1/(2\beta)}$ in the last step. This completes the proof for $k=2$. 

By mathematical induction, assuming (\ref{equalsizerate.}) holds for $\hat{\btheta}_k$ with probability greater than $1-2kd^{-1}$,  we would like to prove (\ref{equalsizerate.}) holds for $\hat{\btheta}_{k+1}$ with probability greater than $1-2(k+1)d^{-1}$. Following the similar arguments, we can prove that $\PP\left((X,\bZ)\in S_{k+1} \right) \asymp b_k $. Note that $\delta_{k+1}\asymp (\frac{ b_{k} K s \log d }{N})^{1/(2\beta+1)}$ and $b_{k} = c_3 (\frac{ Ks\log d}{N})^{1/(2\beta)}$. As a result, $b_{k} \geq C \delta_{k+1}$, and by (\ref{equalsizerate.}) it holds that $b_{k} \geq C\|\hat\btheta_{k}-\btheta^*\|_2 \sqrt{\log (\frac{N}{K s \log d})}$ with  probability greater than $1-2kd^{-1}$. Finally, as shown in (\ref{equalsizerate.}), we obtain that with probability greater than $1-2(k+1)d^{-1}$,  
$$
     \|\hat{\btheta}_{k+1}-\boldsymbol{\theta}^*\|_2 
         \lesssim \left(\frac{ b_{k} K s \log d  }{N}\right)^{\beta/(2\beta+1)}\lesssim \left(\frac{ K s \log d }{N}\right)^{1/2}.
$$
This completes the proof for $k+1$. It is easily seen that, by the union bound argument, the event $\cap_{2\leq k\leq K}\{\|\hat{\btheta}_{k}-\boldsymbol{\theta}^*\|_2 \lesssim (\frac{ K s \log d }{N})^{1/2}\}$ holds with probability greater than $1-2Kd^{-1}$.

Finally let's consider the assumption (\ref{Nncondition.}).
 By definition, we have $N_k = \sum_{(X_i, \bZ_i) \in D_k}\EE\left( R_i\right)= n\EE(R_i)/K = n c_{n,k} \PP\left((X,\bZ)\in S_k \right) /K$. To ensure that $0 < c_{n,k} \leq 1$, we require
\[
N_k K \leq n  \PP\left((X,\bZ)\in S_k \right), \ 2 \leq k \leq K.
\]
Note that $N_k K = N$,
with (\ref{pminbound.}), it suffices to ensure that $N \leq  Cb_{k-1}n$ for some constant $C$ and for all $2\leq k\leq K$. Some calculation yields (\ref{Nncondition.}).
\end{proof}

\begin{proof}[Proof of Theorem \ref{beta^{**}< beta < beta^*}]
By Theorem~\ref{rate_k}, choosing $\delta_1 = c_1 \left(\frac{K s \log d}{N}\right)^{1/(2\beta+1)}$ and 
    $\lambda_1 = c_{2} \sqrt{\frac{N K\log d }{n^2\delta_1}}$, yields that with probability greater than $1-2d^{-1}$,
\begin{align} \label{step1rate.}
   \|\hat{\btheta}_1 -\btheta^*\|_2 \lesssim  \left(\frac{K s \log d}{N}\right)^{\beta/(2\beta+1)}, \    \|\hat{\btheta}_1 -\btheta^*\|_1 \lesssim \sqrt{s} \left(\frac{K s \log d}{N}\right)^{\beta/(2\beta+1)}.
\end{align}
Similar as in Theorem~\ref{beta>beta_*}, we find $\PP\left((X,\bZ)\in S_2 \right) \asymp b_1$, leading to
\begin{equation}\label{deltachoice.}
    \delta_2 \asymp  \left(\frac{K s \log d}{nc_{n,2}}\right)^{1/(2\beta+1)} \asymp \left(\frac{s \log d ~\PP\left((X,\bZ)\in S_2 \right)}{N_2}\right)^{1/(2\beta+1)} \asymp \left(\frac{b_{1} K s \log d }{N}\right)^{1/(2\beta+1)}.
\end{equation}
To invoke  Theorem~\ref{rate_k} for $\hat\btheta_2$, we need to verify
    \[
    b_1 \geq C \delta_2  ~~\text{\ and\ }~~
b_1 \geq C\|\hat\btheta_{1}-\btheta^*\|_2 \sqrt{\log \left(\frac{N}{K s \log d}\right)}.
    \]
For simplicity, we denote $\Delta = \frac{K s \log d}{N}$ and $\alpha = \frac{\beta}{2\beta+1}$. By (\ref{deltachoice.}) and (\ref{step1rate.}), it suffices to verify 
  \begin{equation} \label{b_1choice.}
   b_1 \geq C^\prime   \Delta^{1/(2\beta)} \text{\ and\ }
b_1 \geq C\Delta^\alpha \sqrt{\log \Big(\frac{1}{\Delta}\Big)}
  \end{equation}
for some constant $C, C^\prime$. For any fixed $1<\beta \leq \frac{1+\sqrt{3}}{2}$, we have $\alpha\leq \frac{1}{2\beta}$. To satisfy (\ref{b_1choice.}), we choose $b_1= C\Delta^\alpha \sqrt{\log (\frac{1}{\Delta})} $ for some constant $C$. With $\lambda_2 = c_{2} \sqrt{\frac{N K \log d}{n^2 b_{1}\delta_2}}$, Theorem~\ref{rate_k} implies, with probability greater than $1-4d^{-1}$,
 \begin{align}\label{beta1step2*}
     \|\hat{\btheta}_2-\boldsymbol{\theta}^*\|_2 &\lesssim 
        \left(\frac{\PP\left((X,\bZ)\in S_2 \right)Ks \log d }{N}\right)^{\beta/(2\beta+1)}\nonumber\\        
         &\lesssim \left(\frac{ b_{1} Ks \log d }{N}\right)^{\beta/(2\beta+1)} \lesssim \left(\log(\frac{1}{\Delta})\right)^{\frac{\alpha}{2}} \Delta^{\alpha^2+\alpha}.
 \end{align}
In the following, we will show that for any $2\leq k\leq \lceil  \log_{\frac{\beta}{2\beta+1}}\left(1-\frac{\beta+1}{2\beta^2}\right) \rceil$ 
  \begin{align}\label{beta1finalrate*}
        \|\hat{\btheta}_k - \btheta^*\|_2 &\lesssim \left(\log(\frac{1}{\Delta})\right)^{\frac{\alpha-\alpha^k}{2(1-\alpha)}} \Delta^{(1-\alpha^k)\frac{\alpha}{1-\alpha}}:=r_{k},
   \end{align}
holds  with probability greater than $1-2kd^{-1}$. Note that $\log_{\frac{\beta}{2\beta+1}}(1-\frac{\beta+1}{2\beta^2}) $ is well defined for $1<\beta \leq \frac{1+\sqrt{3}}{2}$. Clearly, (\ref{beta1finalrate*}) holds for $k=2$. Assuming (\ref{beta1finalrate*}) holds for $k-1$, it suffices to show (\ref{beta1finalrate*}) holds for $k$. Following the same argument above for $k=2$, $b_{k-1}$ needs to satisfy
$$
   b_{k-1} \geq C^\prime   \Delta^{1/(2\beta)} \text{\ and\ }
b_{k-1}\geq C r_{k-1} \sqrt{\log \Big(\frac{1}{\Delta}\Big)},
$$
where $r_{k-1}$ is given by (\ref{beta1finalrate*}). We note that for any $ k\leq \lceil  \log_{\frac{\beta}{2\beta+1}}\left(1-\frac{\beta+1}{2\beta^2}\right) \rceil$ and for any fixed $1<\beta \leq \frac{1+\sqrt{3}}{2}$, 
$$
(1-\alpha^{k-1})\frac{\alpha}{1-\alpha} = \left(1- (\frac{\beta}{2\beta+1})^{k-1}\right) \frac{\beta}{\beta+1}< \frac{1}{2\beta},
$$
which implies 
  \begin{equation} \label{eq_delta_vs_error}
 \Delta^{1/(2\beta)}  = O\left(r_{k-1} \sqrt{ \log(\frac{1}{\Delta}) } \right),
  \end{equation}
and therefore we can choose
   \begin{align*}
        b_{k-1}  &\asymp r_{k-1}  \sqrt{ \log(\frac{1}{\Delta}) } \\
        &\asymp\left(\log(\frac{1}{\Delta})\right)^{\frac{\alpha-\alpha^{k-1}}{2(1-\alpha)}+\frac{1}{2}} \Delta^{(1-\alpha^{k-1})\frac{\alpha}{1-\alpha}} \\
        &\asymp \left(\log(\frac{N}{K s \log d}) \right)^{\frac{(2\beta+1)(1-(\frac{\beta}{2\beta+1})^{k-1})}{2(\beta+1)}} \left(\frac{K s \log d}{N}\right)^{\frac{\beta}{\beta+1}(1-(\frac{\beta}{2\beta+1})^{k-1})}.
   \end{align*}
Similar to (\ref{beta1step2*}),  we have  with probability greater than $1-2kd^{-1}$, 
  \begin{align*}
        \|\hat{\btheta}_k - \btheta^*\|_2  \lesssim \left( b_{k-1} \Delta\right)^{\alpha} \lesssim r_{k-1}^\alpha \Big\{\log(\frac{1}{\Delta})\Big\}^{\alpha/2}\Delta^\alpha=r_k, 
   \end{align*}
which completes the proof of (\ref{beta1finalrate*}).

For $k = K=\lceil  \log_{\frac{\beta}{2\beta+1}}\left(1-\frac{\beta+1}{2\beta^2}\right) \rceil+1$, to satisfy
$$
   b_{K-1} \geq C^\prime   \Delta^{1/(2\beta)} \text{\ and\ }
b_{K-1}\geq C r_{K-1} \sqrt{\log \Big(\frac{1}{\Delta}\Big)},
$$
we set $b_{K-1}=c\Delta^{1/(2\beta)}$. To see this, note that $\Delta =o(1)$, the bound (\ref{beta1finalrate*}) holds for $r_{K-1}$, and thus for any fixed $1<\beta \leq \frac{1+\sqrt{3}}{2}$, some calculation shows that 
\begin{equation}\label{betacondition}
(1-\alpha^{K-1})\frac{\alpha}{1-\alpha} = \left(1- (\frac{\beta}{2\beta+1})^{K-1}\right) \frac{\beta}{\beta+1} > \frac{1}{2\beta},
\end{equation}
which implies
$$
r_{K-1} \sqrt{ \log(\frac{1}{\Delta}) } =O(\Delta^{1/(2\beta)}).
$$
Applying Theorem~\ref{rate_k}, we select
\begin{equation*}
    \delta_K  \asymp \left(\frac{K b_{K-1} s \log d }{N}\right)^{1/(2\beta+1)}  \asymp \left(\frac{Ks\log d}{N}\right)^{1/(2\beta)},
\end{equation*}
and $\lambda_K = c_{2} \sqrt{\frac{N K\log d}{n^2 b_{K-1}\delta_K}}$,
to ensure, with probability greater than $1-2Kd^{-1}$  
\begin{align*}
     \|\hat{\btheta}_K-\boldsymbol{\theta}^*\|_2 \lesssim& 
         \left(\frac{\PP\left((X,\bZ)\in S_K \right)s \log d }{N_K}\right)^{\beta/(2\beta+1)}
         \lesssim \sqrt{s}\left(\frac{K b_{K-1} s \log d }{N}\right)^{\beta/(2\beta+1)}\\
         \lesssim& \left(\frac{ K s \log d }{N}\right)^{1/2}.
\end{align*}
The result for $ \|\hat{\btheta}_k-\boldsymbol{\theta}^*\|_1$ follows similarly.

Finally,  to ensure that $0 < c_{n,k} \leq 1$ for $2\leq k\leq K$, we require
\[
N_k K \leq n  \PP\left((X,\bZ)\in S_k \right), \ 2 \leq k \leq K.
\]
Note that $N_k K = N$. It suffices to ensure that $N \leq  Cb_{k-1}n$ for some constant $C$, which is implied by (\ref{Nncondition1}) for $2 \leq k \leq K-1$ and (\ref{Nncondition.}) for $k=K$.
    
\end{proof}

\begin{proof}[Proof of Theorem \ref{beta < beta^**}]
Following the proof of Theorem~\ref{beta>beta_*} we have  by choosing $\delta_1 = c_1\left(\frac{K s \log d}{N}\right)^{1/(2\beta+1)}$ and $\lambda_1 = c_{2} \sqrt{\frac{c_{n,1} K\log d }{n\delta_1}}= c_{2} \sqrt{\frac{N K\log d }{n^2\delta_1}}$,  with probability greater than $1-2d^{-1}$,
\begin{align} \label{step1rate..}
   \|\hat{\btheta}_1 -\btheta^*\|_2 \lesssim  \left(\frac{K s \log d}{N}\right)^{1/(2\beta+1)}, ~~\    \|\hat{\btheta}_1 -\btheta^*\|_1 \lesssim \sqrt{s} \left(\frac{K s \log d}{N}\right)^{1/(2\beta+1)}.
\end{align}
Similar as in Theorem~\ref{beta>beta_*}, we find $\PP\left((X,\bZ)\in S_2 \right) \asymp b_1$. Applying  Theorem~\ref{rate_k} with 
\begin{equation}\label{deltachoice..}
    \delta_2 \asymp  \left(\frac{K s \log d}{nc_{n,2}}\right)^{1/(2\beta+1)} \asymp \left(\frac{s \log d \PP\left((X,\bZ)\in S_2 \right)}{N_2}\right)^{1/(2\beta+1)} \asymp \left(\frac{b_{1} K s \log d }{N}\right)^{1/(2\beta+1)},
\end{equation}
and $\lambda_2 = c_{2} \sqrt{\frac{N K \log d}{n^2 b_{1}\delta_2}}$,
 we have with probability greater than $1-4d^{-1}$,
 \begin{equation}\label{beta1step2**}
     \|\hat{\btheta}_2-\boldsymbol{\theta}^*\|_2 \lesssim 
         \left(\frac{K \PP\left((X,\bZ)\in S_2 \right)s \log d }{N}\right)^{1/(2\beta+1)}
         \lesssim \left(\frac{ Kb_{1} s \log d }{N}\right)^{1/(2\beta+1)}.
 \end{equation} 
  To  satisfy 
    \[
    b_1 \geq C \delta_2  \text{\ and\ }
b_1 \geq C\|\hat\btheta_{1}-\btheta^*\|_2 \sqrt{\log \left(\frac{N}{K s \log d}\right)},
    \]
  by (\ref{deltachoice..})  it suffices to verify 
  \begin{equation} \label{b_1choice..}
   b_1 \geq C^\prime   \left(\frac{ Ks\log d}{N}\right)^{1/(2\beta)} \text{\ and\ }
b_1 \geq C \left(\frac{K s \log d}{N}\right)^{1/(2\beta+1)} \sqrt{\log \left(\frac{N}{K s \log d}\right)}
  \end{equation}
for some constants $C, C^\prime$. Since ${Ks\log d}=o(N)$, clearly we have 
$$
\left(\frac{Ks\log d}{N}\right)^{1/(2\beta)}  = O\left(\left(\frac{K s \log d}{N}\right)^{1/(2\beta+1)} \sqrt{\log \left(\frac{N}{K s \log d}\right)}\right).
$$ 
Therefore, to satisfy (\ref{b_1choice..}), we choose $b_1 =  C_1  \left(\frac{K s \log d}{N}\right)^{1/(2\beta+1)} \sqrt{\log \left(\frac{N}{K s \log d}\right)} $ for some constant $C_1$. Then by 
(\ref{beta1step2**}) and (\ref{step1rate..}) we have
\[
\|\hat{\btheta}_2-\boldsymbol{\theta}^*\|_2 \lesssim \left(\log(\frac{1}{\Delta})\right)^{\frac{1}{2(2\beta+1)}} \Delta^{\frac{2(\beta+1)}{(2\beta+1)^2}},
\]
where $\Delta = \frac{K s \log d}{N}$. Using a similar mathematical induction argument, we can show that  with probability greater than $1-2kd^{-1}$, 
  \begin{align}\label{kstepiterationrate}
              \|\hat{\btheta}_k - \btheta^*\|_2 \lesssim& \left( 
        \Delta \sqrt{\log(\frac{1}{\Delta})}
        \right)^{\sum_{i=1}^{k-1} 1/(2\beta+1)^i}
        \|\hat{\btheta}_1 - \btheta^*\|_2^{1/(2\beta+1)^{k-1}} \nonumber \\
        \lesssim& 
    \left( 
        \Delta \sqrt{\log(\frac{1}{\Delta})}
        \right)^{\frac{1}{2}(1-1/(2\beta+1)^{k-1})}
        \Delta^{1/(2\beta+1)^k} \nonumber\\
        \lesssim & \left(\log(\frac{1}{\Delta})\right)^{\frac{1}{4\beta}(1-1/(2\beta+1)^{k-1})}
        \Delta^{\frac{1}{2\beta}(1-1/(2\beta+1)^k)},
   \end{align}
for any $2\leq k\leq K$.  Given $K = \lceil \log_{2\beta+1}(\log N) \rceil$, we have
 $\left(\frac{N}{K s \log d}\right)^{\frac{1}{2\beta\cdot(2\beta+1)^K}}\leq C$ for some constant $C>0$, hence
(\ref{kstepiterationrate}) with $k=K$ can reduce to 
\[
   \|\hat{\btheta}_K - \btheta^*\|_2 \lesssim 
   \left(\log(\frac{N}{K s \log d})\right)^{\frac{1}{4\beta}}
   \left(\frac{K s \log d}{N}\right)^{\frac{1}{2\beta}},
 \]
   and the result for $ \|\hat{\btheta}_k-\boldsymbol{\theta}^*\|_1$ follows similarly.
Since $ \frac{1}{2}(1-1/(2\beta+1)^{k-1}) < \frac{1}{2},$
we can verify that 
 $\Delta^{1/(2\beta)}  \lesssim \left(\log(\frac{1}{\Delta})\right)^{\frac{1}{2}+\frac{1}{4\beta}(1-1/(2\beta+1)^{k-2})}\Delta^{\frac{1}{2\beta}(1-1/(2\beta+1)^{k-1})}$ for all $2\leq k\leq K$.
Therefore, we choose
\[
b_{k-1} = C_1 \left(\log (\frac{N}{K s \log d})\right)^{\frac{1}{2}+\frac{1-1/(2\beta+1)^{k-2}}{4\beta}} \left(\frac{K s \log d}{N}\right)^{\frac{1-1/(2\beta+1)^{k-1}}{2\beta}},
\]
and   $ \delta_k = C_2 \left(\frac{b_{k-1}K s \log d}{N}\right)^{1/(2\beta+1)}$
   where $C_1, C_2>0$ are some constants.  
Finally, to ensure that $0 < c_{n,k} \leq 1$, we require
\[
N_k K \leq n  \PP\left((X,\bZ)\in S_k \right), \ 2 \leq k \leq K-1.
\]
Note that $N_k K = N$. It suffices to ensure that $N \leq  Cb_{k-1}n$ for some constant $C$, which is provided in (\ref{Nncondition1.}).
\end{proof}

\begin{proof}[Proof of Theorem \ref{lowerbound1}]

We first consider the case when $\beta > 1$. The proof consists of the following two steps.
\begin{itemize}
    \item For any sampling method $Q \in \mathcal{Q}_N(\mathcal{P}(\beta,s))$,
    construct a set of hypotheses $\mathcal{H}=\left\{P_j(X, Y, \boldsymbol{Z})\right\} \subset \mathcal{P}(\beta,s)$.
    \item Apply Theorem 2.7 in \cite{10.5555/1522486} by checking the following two conditions:
    \begin{enumerate}
        \item $ \mathrm{KL}\left(f^{j} \| f^{0}\right) \leq \gamma \log |\mathcal{H}|$ for some $\gamma \in(0,1 / 8)$, where $ \mathrm{KL}\left(f^{j} \| f^{0}\right)$ is the K-L divergence between probability measures $f^j$ and $f^0$, and $f^j$ is the probability measure of the random variables $\{O_i\}_{i=1}^n$ under hypothesis $j$.
        \item For all $j \neq k$ and $q=1,2,\left\|\boldsymbol{\theta}_j-\boldsymbol{\theta}_k\right\|_q \geq 2 t$, where $t \asymp s^{\frac{1}{q}-\frac{1}{2}}\left(\frac{s \log (d / s)}{N}\right)^{1/2} .$
    \end{enumerate}
\end{itemize}

Given the set
\[
\mathcal{M}=\left\{x \in\{0,1\}^d:\|x\|_0=s\right\}, 
\]
there exists a subset $\mathcal{H}^{\prime}$ of $\mathcal{M}$ such that $\rho_H\left(x, x^{\prime}\right)>s / 16$ for $x, x^{\prime} \in \mathcal{H}^{\prime}, x \neq x^{\prime}$ and $\log \left|\mathcal{H}^{\prime}\right| \geq c^{\prime} s \log \left(\frac{d}{s}\right)$, where $\rho_H$ denotes the Hamming distance and $c^\prime$ is some absolute constant. We let $\boldsymbol{\omega}_0=\mathbf{0} \in \mathbb{R}^d$ and use $\bomega_j$ to denote the elements in $\mathcal{H}^\prime$ for $j=1, \ldots,\left|\mathcal{H}^{\prime}\right|$.

We construct $P_j(X,Y,\bZ)$ as follows. 
We choose weight functions such that  $\gamma(1) = \gamma(-1)$.
For all $j=0, \ldots,\left|\mathcal{H}^{\prime}\right|$, we assume $X$ and each of $Z_1, \cdots, Z_d$ follows a Uniform distribution on $[-1,1]$ independently. For each $j=0, \ldots,\left|\mathcal{H}^{\prime}\right|$,  let
\begin{align}\label{conditional_density}
    f_j( y=1 \mid x, \boldsymbol{z})= \frac{1}{2} + \frac{1}{2\sigma }\left(x - c\left(\frac{s \log(d/s) }{N}\right)^{1/2} \frac{\bomega_j^T\bz}{\sqrt{s}} \right), \\
     f_j( y=-1 \mid x, \boldsymbol{z})= \frac{1}{2} - \frac{1}{2\sigma }\left(x - c\left(\frac{s \log(d/s) }{N}\right)^{1/2} \frac{\bomega_j^T\bz}{\sqrt{s}} \right),
\end{align}
where $c$ is some sufficiently small constant and  $\sigma$ is some sufficiently large constant. Under the assumption that $s(\frac{\log(d/s) }{N})^{1/2}=o(1)$, we can guarantee that $f_j( y=1 \mid x, \boldsymbol{z})$ and $f_j( y=-1 \mid x, \boldsymbol{z})$ are within $[0,1]$ for any $(x,\bz)\in [-1,1]^{d+1}$, and thus are well defined.
By this construction we can conclude that $P_0, \ldots, P_{\mathcal{H}^{\prime}}$ are well defined probability measures. In the following we present
two lemmas which characterize two key properties of $P_j$.

\begin{lemma}
    Under the conditions of Theorem~\ref{lowerbound1} and the construction of $P_j = P_j\left(X,Y,\bZ\right)$ above, we have $P_j \in \mathcal{P}(\beta, L,p_{\min},p_{\max}), \forall j=0, \ldots,\left|\mathcal{H}^{\prime}\right|$ and (\ref{sparseeigen2}) holds.
\end{lemma}
\begin{proof}
    By the construction of $P(X,Y,\bZ)$, denoting $\tilde{c} = \frac{c}{\sqrt{s}
    }\left(\frac{s \log(d/s) }{N}\right)^{1/2}$,
    we have 
    \begin{align}\label{yzjoint}
        f_j(y=1, \bz) =& \int_{-1}^1 f_j(x, y=1, \bz) dx\nonumber\\
        = & \int_{-1}^1 f_j(y=1 \mid x, \bz ) f(x) f(\bz) dx \nonumber\\
        =& \frac{1}{2^{d+1}}(1-\frac{\tilde{c}}{\sigma } \bomega_j^T \bz ),
    \end{align}
    and $f_j(y=-1, \bz) = \frac{1}{2^{d+1}}(1+\frac{\tilde{c}}{\sigma } \bomega_j^T \bz )$. Hence
    \begin{equation}\label{bayesrule}
         f_j(x \mid y, \bz) = \frac{f_j(y \mid x, \bz) (1/2)^{d+1} }{f(y, \bz)} = \frac{1+\frac{y}{\sigma }(x - \tilde{c} \bomega_j^T \bz )}{2(1-y \frac{\tilde{c}}{\sigma } \bomega_j^T \bz )} = \frac{1}{2} + \frac{y x}{2\sigma (1-y \frac{\tilde{c}}{\sigma  } \bomega_j^T \bz )}
    \end{equation}
    is $l = \lfloor\beta\rfloor$ times differentiable w.r.t. $x$ for any $y, \bz$.
Now we check the condition in Definition~\ref{densityclass}, i.e., $f_j(x \mid y, \boldsymbol{z}), j=0, \ldots,\left|\mathcal{H}^{\prime}\right|$ satisfies that
  \begin{equation}\label{lipshitz}
      \left|f_j^{(l)}\left(x_1 \mid y, \boldsymbol{z}\right)-f_j^{(l)}\left(x_2 \mid y, \boldsymbol{z}\right)\right| \leq L\left|x_1-x_2\right|^{\beta-l}
  \end{equation}
   for any $y \in\{-1,1\}, \boldsymbol{z} \in \mathbb{R}^d, x_1, x_2$,  and $L>0$ is some constant. When $\beta=1$, $l = 0$, by (\ref{bayesrule}) we have 
\[
  \left|f_j^{(0)}\left(x_1 \mid y, \boldsymbol{z}\right)-f_j^{(0)}\left(x_2 \mid y, \boldsymbol{z}\right)\right|  = \frac{ |x_1 - x_2| }{2\sigma (1-y \frac{\tilde{c}}{\sigma } \bomega_j^T \bz )} < \frac{|x_1 - x_2|}{\sigma },
   \]
   given $\sigma$ sufficiently large. Then note that $|x_1 - x_2| \leq 2$, therefore choosing $L = \frac{1}{\sigma }$ we ensure that (\ref{lipshitz}) is satisfied. For $\beta > 1$, $l \geq 1$, we have $\left|f_j^{(l)}\left(x_1 \mid y, \boldsymbol{z}\right)-f_j^{(l)}\left(x_2 \mid y, \boldsymbol{z}\right)\right| = 0$, hence (\ref{lipshitz}) holds trivially. This means $P_j\in\mathcal{P}(\beta, L)$. Clearly, $f_j(x|\bz)=f_j(x)=1/2=p_{\min}$, and $f_j(x \mid y, \boldsymbol{z})\leq 1/2+1/\sigma=p_{\max}$. Thus, $P_j \in \mathcal{P}(\beta, L,p_{\min},p_{\max})$ holds.

Now we check the condition (\ref{sparseeigen2}), i.e.,
\begin{equation} \label{holder3}
    \sup_{\|\boldsymbol{v}\|_0 \leq s^{\prime}} \frac{\boldsymbol{v}^T \mathbb{E}\left(\boldsymbol{Z} \boldsymbol{Z}^T \mid Y=y\right) \boldsymbol{v}}{\|\boldsymbol{v}\|_2^2} \leq M_1.
\end{equation}
By (\ref{yzjoint}) we have
\[
P_j(Y=1) = \int_{\bZ} \frac{1}{2^{d+1}}(1-\frac{\tilde{c}}{\sigma } \bomega_j^T \bz )d\bz = \frac{1}{2},
\]
hence $f_j(z \mid Y=y) = \frac{1}{2^{d}}(1-y\frac{\tilde{c}}{\sigma } \bomega_j^T \bz )$.
We have
 \begin{align*}
\bv^T \mathbb{E}\left(\bZ\bZ^T \mid Y=y\right) \bv =&  \int (\bv^T \bz)^2  \frac{1}{2^{d}}(1-y\frac{\tilde{c}}{\sigma } \bomega_j^T \bz ) d\bz  \\ 
=& \bv^T \EE(\bZ \bZ^T) \bv - \frac{y \tilde{c}}{\sigma} \int \frac{1}{2^d}(\bv^T \bz)^2 \bomega_j^T \bz d\bz.
 \end{align*}
 Note that $\mathbb{E}\left(\boldsymbol{Z} \boldsymbol{Z}^T \right) = 1/3 \mathbb{I}_d$ and $\|\bomega_j\|_0 = s$, hence
\begin{align*}
    \Big|\int \frac{1}{2^d}(\bv^T \bz)^2 \bomega_j^T \bz d\bz\Big| \leq s \EE((\bv^T \bZ)^2)
    \leq \frac{s}{3} \|\bv\|_2^2,
\end{align*}
and
\begin{align*}
    |\bv^T \mathbb{E}\left(\bZ\bZ^T \mid Y=y\right) \bv | \leq \frac{1}{3}\|\bv\|_2^2
    + \frac{c}{3\sigma}s\sqrt{\frac{\log(d/s)}{N}}\|\bv\|_2^2.
\end{align*}
Since $s \sqrt{\frac{\log(d/s)}{N}} =o(1)$, (\ref{holder3}) holds.
\end{proof}

\begin{lemma}\label{minimizer}
    Under the conditions of Theorem~\ref{lowerbound1} and the construction of $P_j = P_j\left(X,Y,\bZ\right)$ above, the unique minimizer $\boldsymbol{\theta}_j \in \mathbb{R}^d$ of the risk $R_{P_j}(\boldsymbol{\theta})$ is 
    \begin{equation*}
\boldsymbol{\theta}_j= \begin{cases}0 & \text { if } j=0, \\ 
\frac{c}{\sqrt{s}}\left(\frac{s \log(d/s) }{N}\right)^{1/2} \bomega_j & \text { otherwise,}\end{cases}
\end{equation*}
where $c$ is defined in (\ref{conditional_density}).
In addition, $\|\btheta_j\|_2 \leq C$ for some constant $C>0$, and $\rho_-\leq \lambda_{\min }\left(\nabla^2 R_j\left(\boldsymbol{\theta}_j\right)\right) \leq \lambda_{\max }\left(\nabla^2 R_j\left(\boldsymbol{\theta}_j\right)\right) \leq \rho_+$ for some constants $\rho_+ \geq \rho_- >0$.
\end{lemma}
\begin{proof}
Recall that $R_j(\boldsymbol{\theta})=\mathbb{E}_j\left[\gamma(Y)\left(1-\operatorname{sign}\left(Y\left(X-\boldsymbol{\theta}^T \boldsymbol{Z}\right)\right)\right)\right]$, and $\gamma(1) = \gamma(-1)=C$ for some constant $C>0$, 
    the difference between $R_j(\boldsymbol{\theta})$ and $R_j(\boldsymbol{\theta}_j)$ can be written as
    \begin{align*}
        R_j(\boldsymbol{\theta})-R_j\left(\boldsymbol{\theta}_j\right)=&\mathbb{E}_j\left[\gamma(Y) Y\left(\operatorname{sign}\left(X-\boldsymbol{\theta}_j^T \boldsymbol{Z}\right)-\operatorname{sign}\left(X-\boldsymbol{\theta}^T \boldsymbol{Z}\right)\right)\right] \\
        =& 2 \int_{\mathcal{G}} \operatorname{sign}\left(x-\boldsymbol{\theta}_j^T \boldsymbol{z}\right) \mathbb{E}_j[\gamma(Y)Y \mid x, \boldsymbol{z}] d P_{j ; X, \boldsymbol{Z}}\\
        =& 2C \int_{\mathcal{G}} \operatorname{sign}\left(x-\boldsymbol{\theta}_j^T \boldsymbol{z}\right) \mathbb{E}_j[Y \mid x, \boldsymbol{z}] d P_{j ; X, \boldsymbol{Z}},
    \end{align*}
where 
\[
\mathcal{G}=\left\{(x, \boldsymbol{z}) \mid \operatorname{sign}\left(x-\boldsymbol{\theta}^T \boldsymbol{z}\right) \neq \operatorname{sign}\left(x-\boldsymbol{\theta}_j^T \boldsymbol{z}\right)\right\},
\]
$P_{j ; X, \boldsymbol{Z}}$ is the joint distribution of $(X,\bZ)$ under $P_j$, and
\begin{align}\label{conditionex}
   \mathbb{E}_j[Y \mid x, \boldsymbol{z}]=& f_j(Y=1 \mid x, \boldsymbol{z})-f_j(Y=-1 \mid x, \boldsymbol{z}) \nonumber\\
    =& \frac{1}{\sigma }\left( x - c\left(\frac{s \log(d/s) }{N}\right)^{1/2} \frac{\bomega_j^T\bz}{\sqrt{s}}\right) \nonumber\\
    =& \frac{1}{\sigma }\left(x - \btheta_j^T\bz\right) ,
\end{align}
\[
\operatorname{sign}\left(\mathbb{E}_j[Y \mid x, \boldsymbol{z}]\right) = \operatorname{sign} \left(x - \btheta_j^T\bz \right).
\]
Therefore, 
\begin{equation}\label{signj0}
    R_j(\boldsymbol{\theta})-R_j\left(\boldsymbol{\theta}_j\right)=2C \int_{\mathcal{G}}\left|\mathbb{E}_j[Y \mid x, \boldsymbol{z}]\right| d P_{j ; X, \boldsymbol{Z}} \geq 0,
\end{equation}
hence $\btheta_j$ is a minimizer of $R_j(\btheta)$. In addition, $\|\btheta_j\|_2 = c\sqrt{\frac{\log(d/s)}{N}} \|\bomega_j\|_2 = c\sqrt{\frac{s\log(d/s)}{N}} = O(1)$.

Now we check the uniqueness. By (\ref{conditionex}) we have 
\begin{equation}\label{unique}
     R_j(\boldsymbol{\theta})-R_j\left(\boldsymbol{\theta}_j\right)=\frac{2C}{\sigma } \int_{\mathcal{G}}\left|x - \btheta_j^T\bz\right| d P_{j ; X, \boldsymbol{Z}}.
\end{equation}
For any $\boldsymbol{\theta} \neq \boldsymbol{\theta}_j$, consider the set 
\[
\mathcal{G}_{\boldsymbol{z}}=\left\{\boldsymbol{z}:\left(\boldsymbol{\theta}-\boldsymbol{\theta}_j\right)^T \boldsymbol{z} \neq 0,\left|\boldsymbol{\theta}^T \boldsymbol{z}\right| \leq 1,\left|\boldsymbol{\theta}_j^T \boldsymbol{z}\right| \leq 1\right\}.
\]
Note that there exists an open neighborhood in $\mathcal{G}_{\boldsymbol{z}}$, hence $\mathcal{G}_{\boldsymbol{z}}$ has nonzero measure. Then
define 
$\bar{\mathcal{G}}=\left\{(x, \boldsymbol{z}): \boldsymbol{\theta}^T \boldsymbol{z}<x<\boldsymbol{\theta}_j^T \boldsymbol{z}\right.$ or $\left.\boldsymbol{\theta}_j^T \boldsymbol{z}<x<\boldsymbol{\theta}^T \boldsymbol{z}, \boldsymbol{z} \in \mathcal{G}_{\boldsymbol{z}}\right\}$, we have $\bar{\mathcal{G}} \subset \mathcal{G}_{\boldsymbol{z}}$ and $\bar{\mathcal{G}}$ has nonzero measure as well.  Therefore, (\ref{unique}) implies that $ R_j(\boldsymbol{\theta})-R_j\left(\boldsymbol{\theta}_j\right) >0$, which completes the proof of the uniqueness.
Note that 
\begin{align*}
    \nabla R_j(\boldsymbol{\theta})=& \sum_{y= \pm 1} \gamma(y) \int_{\boldsymbol{Z}} \boldsymbol{z} y f\left(\boldsymbol{\theta}^T \boldsymbol{z} \mid \boldsymbol{z}, y\right) f_j(\boldsymbol{z}, y) d \boldsymbol{z} \\
    =& \sum_{y= \pm 1} \gamma(y) \int_{\boldsymbol{Z}} \boldsymbol{z} y f_j\left(y \mid \btheta^T\bz, \bz \right) f_{X,\bZ}(\btheta^T\bz, \bz) d \boldsymbol{z} \\
    =& \frac{1}{2} \sum_{y= \pm 1} \gamma(y) \int_{\bZ} \bz y \left( \frac{1}{2} + \frac{y}{2\sigma }\left(\btheta^T\bz - \btheta_j^T \bz\right) \right) f_{\bZ}(\bz)d\bz,
\end{align*}
hence
\begin{align*}
    \nabla^2 R_j(\boldsymbol{\theta}) = 
    \frac{1}{4\sigma } \sum_{y= \pm 1} \gamma(y)
    \int_{\bZ}  \bz \bz^T f_{\bZ}(\bz) d\bz =  \frac{C}{2\sigma }\EE(\bZ\bZ^T) = \frac{C}{6\sigma}\mathbb{I}_d,
\end{align*}
and $\lambda_{\min }\left(\nabla^2 R_j\left(\boldsymbol{\theta}_j\right)\right)= \lambda_{\max }\left(\nabla^2 R_j\left(\boldsymbol{\theta}_j\right)\right)  = \frac{C}{6\sigma }>0$. This completes the proof.
\end{proof}

For the second step of the proof, we check both the two conditions in the following.
Recall that  $f^j$ is the probability measure of the random variables $\{(X_i,\bZ_i, Y_i, R_i)^{\ind\{R_i = 1\}}, (X_i,\bZ_i, R_i)^{\ind\{R_i = 0\}} \}_{i=1}^n$ under hypothesis $j$, and $\bar{H}_{i-1} = 
\{(X_j,\bZ_j, Y_j, R_j)^{\ind\{R_j = 1\}}, (X_j,\bZ_j, R_j)^{\ind\{R_j = 0\}} \}_{j=1}^{i-1}$. We have
\begin{align*}
    f^j =& \prod_{i=1}^n \left\{f(R_i=1 \mid X_i, \bZ_i, Y_i, \bar{H}_{i-1}) f_j( X_i, \bZ_i, Y_i \mid \bar{H}_{i-1}) \right\}^{\ind\{R_i = 1\}}\\
    &\cdot 
    \left\{f(R_i=0 \mid X_i, \bZ_i, \bar{H}_{i-1}) f_j( X_i, \bZ_i \mid \bar{H}_{i-1}) \right\}^{\ind\{R_i = 0\}}\\
    =& \prod_{i=1}^n 
    \left\{f(R_i=1 \mid X_i, \bZ_i, \bar{H}_{i-1}) f_j( X_i, \bZ_i, Y_i ) \right\}^{\ind\{R_i = 1\}} \cdot
     \left\{f(R_i=0 \mid X_i, \bZ_i, \bar{H}_{i-1}) f_j( X_i, \bZ_i ) \right\}^{\ind\{R_i = 0\}}\\
     =& \prod_{i=1}^n 
    \left\{f(R_i=1 \mid X_i, \bZ_i, \bar{H}_{i-1})  \right\}^{\ind\{R_i = 1\}}
     \left\{f(R_i=0 \mid X_i, \bZ_i, \bar{H}_{i-1})  \right\}^{\ind\{R_i = 0\}}  \\
     & \ \ \ \ \cdot f_j( Y_i \mid X_i, \bZ_i )^{\ind\{R_i = 1\}}  f_j( X_i, \bZ_i ),
\end{align*}
where in the second last equation we used the fact that 
 $R_i \perp Y_i \mid X_i, \bZ_i, \bar{H}_{i-1}$ and $(X_i,\bZ_i, Y_i) \perp \bar{H}_{i-1}$ as specified in the class of sampling method in (\ref{sampling1}).
Since the sampling method $f(R_i \mid X_i, \bZ_i, \bar{H}_{i-1}) $ and the joint distribution of $(X,\bZ)$ keep invariant under different hypotheses, we have 
\begin{align}\label{klnew}
    &\mathrm{KL}\left(f^{j} \| f^{0}\right) \nonumber\\
    =& \EE \left[\log
    \frac{\prod_{i=1}^n 
    f_j( Y_i \mid X_i, \bZ_i )^{\ind\{R_i = 1\}}  f_j( X_i, \bZ_i )
    }{\prod_{i =1}^n
     f_0( Y_i \mid X_i, \bZ_i )^{\ind\{R_i = 1\}}  f_0( X_i, \bZ_i )
    }
    \right] \nonumber \\
    =& \EE \left[\log
    \frac{\prod_{i=1}^n 
    f_j( Y_i \mid X_i, \bZ_i )^{\ind\{R_i = 1\}}
    }{
    \prod_{i=1}^n 
    f_0( Y_i \mid X_i, \bZ_i )^{\ind\{R_i = 1\}}
    }
    \right] \nonumber\\
      =& \EE \left[\log \prod_{i=1}^n \left( \frac{f_{j}(Y_i \mid X_i,\bZ_i)}{f_{0}(Y_i \mid X_i,\bZ_i)} \right)^{\ind\{R_i=1\}}
    \right] \nonumber\\
    =&\sum_{i=1}^n \EE \left[\ind\{R_i=1\} \cdot \log \frac{f_{j}(Y_i \mid X_i, \bZ_i)}{f_{0}(Y_i \mid X_i, \bZ_i)}
    \right] \nonumber\\
    =& \sum_{i=1}^n \EE \left[
    \EE \left[\ind\{R_i=1\} \cdot \log \frac{f_{j}(Y_i \mid X_i, \bZ_i)}{f_{0}(Y_i \mid X_i, \bZ_i)} \mid X_i, \bZ_i, \bar{H}_{i-1}
    \right]
    \right] \nonumber\\
    =& \sum_{i=1}^n \EE \left[
    \EE\left[\log \frac{f_{j}(Y_i \mid X_i, \bZ_i)}{f_{0}(Y_i \mid X_i, \bZ_i)}  \mid X_i, \bZ_i\right] \cdot \PP\left(R_i = 1 \mid X_i, \bZ_i, \bar{H}_{i-1}\right)
    \right],
\end{align}
where in the last equation we used the fact that $R_i \perp Y_i \mid X_i, \bZ_i, \bar{H}_{i-1}$ and $(X_i,\bZ_i, Y_i) \perp \bar{H}_{i-1}$, and the expectation is with respect to $f^j$.
Note that $\EE\left[\log \frac{f_{j}(Y_i \mid X_i, \bZ_i)}{f_{0}(Y_i \mid X_i, \bZ_i)}  \mid X_i, \bZ_i\right]$ is the K-L divergence between two Bernoulli distributions $f_{j}(Y_i \mid X_i, \bZ_i)$ and $f_{0}(Y_i \mid X_i, \bZ_i)$. Since $\delta_j := c\left(\frac{s \log(d/s) }{N}\right)^{1/2} \frac{\bomega_j^T\bz}{\sqrt{s}}=o(1)$ uniformly over $\bz$, we have
\begin{align*}
    &\EE\left[\log \frac{f_{j}(Y_i \mid X_i, \bZ_i)}{f_{0}(Y_i \mid X_i, \bZ_i)}  \mid X_i = x, \bZ_i = \bz\right] \\
    =& 
    f_j( y=1 \mid x, \boldsymbol{z}) \log \frac{f_j( y=1 \mid x, \boldsymbol{z})}{f_0( y=1 \mid x, \boldsymbol{z})} + 
    f_j( y=-1 \mid x, \boldsymbol{z}) \log \frac{f_j( y=-1 \mid x, \boldsymbol{z})}{f_0( y=-1 \mid x, \boldsymbol{z})}\\
    =& \left(\frac{1}{2} + \frac{1}{2\sigma }\left(x - \delta_j \right)\right) \log\left(1 - \frac{\delta_j}{\sigma + x}\right) + \left(\frac{1}{2} - \frac{1}{2\sigma }\left(x - \delta_j \right)\right) \log\left(1 + \frac{\delta_j}{\sigma - x}\right)\\
    =& \frac{\delta_j^2}{2(\sigma+x)(\sigma-x)}+ o(\delta_j^2) \leq \delta_j^2.
\end{align*}
Plugging this back to (\ref{klnew}), we have
\begin{align}\label{klnew1}
    \mathrm{KL}\left(f^{j} \| f^{0}\right) \leq& 
     \sum_{i=1}^n \EE \left[
     \delta_j^2 \PP\left(R_i = 1 \mid X_i, \bZ_i, \bar{H}_{i-1}\right)
    \right] \nonumber\\
    =& c^2 \frac{s \log(d/s) }{N} \frac{\bomega_j^T}{s} \sum_{i=1}^n\EE\left[Q_i\bZ_i \bZ_i^T  \right]
    \bomega_j.
\end{align}
Since $Q \in \mathcal{Q}_N(\mathcal{P}(\beta,s))$, we have
 \[
 \bomega_j^T \sum_{i=1}^n\EE\left[\bZ_i \bZ_i^T \cdot \PP\left(R_i = 1 \mid X_i, \bZ_i, \bar{H}_{i-1}\right) \right] \bomega_j
 \leq C \| \bomega_j\|_2^2 N.
 \]
Giving this back to (\ref{klnew1}), and choosing $c$ in (\ref{conditional_density}) sufficiently small,  we can ensure 
$C_2$ small enough such that
\begin{align*}
    \mathrm{KL}\left(f^{j} \| f^{0}\right) \leq C_2 \frac{s \log(d/s) }{N}  
   \frac{ \| \bomega_j\|_2^2 N}{s}
    = C_2 s \log(d/s)
    \leq \gamma c^\prime s \log(d/s)
    \leq \gamma \log |\mathcal{H}^\prime|
\end{align*}
for some $\gamma \in(0,1 / 8)$, hence condition 1 is satisfied.

For condition 2, by Lemma~\ref{minimizer}, we have when $j \neq 0$, 
\[
\left\|\boldsymbol{\theta}_0-\boldsymbol{\theta}_j\right\|_2=c \left(\frac{s \log(d/s)}{N}\right)^{1/2}\left\|\boldsymbol{\omega}_j\right\|_2 / \sqrt{s}=c \left(\frac{s \log(d/s)}{N}\right)^{1/2},
\]
\[
\left\|\boldsymbol{\theta}_0-\boldsymbol{\theta}_j\right\|_1=c \left(\frac{s \log(d/s)}{N}\right)^{1/2}\left\|\boldsymbol{\omega}_j\right\|_1 / \sqrt{s}=c \sqrt{s} \left(\frac{s \log(d/s)}{N}\right)^{1/2}.
\]
For all $j,k \neq 0$, we have
\[
\left\|\boldsymbol{\theta}_j-\boldsymbol{\theta}_k\right\|_2=c \left(\frac{s \log(d/s)}{N}\right)^{1/2}\left\|\boldsymbol{\omega}_j-\boldsymbol{\omega}_k\right\|_2 / \sqrt{s} \geq \frac{c}{4} \left(\frac{s \log(d/s)}{N}\right)^{1/2},
\]
\[
\left\|\boldsymbol{\theta}_j-\boldsymbol{\theta}_k\right\|_1=c \left(\frac{s \log(d/s)}{N}\right)^{1/2}\left\|\boldsymbol{\omega}_j-\boldsymbol{\omega}_k\right\|_1 / \sqrt{s} \geq \frac{c \sqrt{s} }{16}\left(\frac{s \log(d/s)}{N}\right)^{1/2} .
\]
Therefore condition 2 holds. Then we apply Theorem 2.7 in \cite{10.5555/1522486} and finish the proof for the case $\beta > 1$.

Next we consider $0 < \beta \leq 1$. Define $\phi(x) = \operatorname{sign}(x) |x|^\beta$ for $x \in \R$, which is an odd function. Let $\tilde{\bomega}_j = s^{\frac{1}{2}(\frac{1}{\beta}-1)} \bw_j$, where $\bw_j$ is the element in $\cH'$. Then, $\norm{\tilde{\bomega}_j}_0 = s, \norm{\tilde{\bomega}_j}_1 = s^{\frac{1}{2}(\frac{1}{\beta}+1)}$ and $\norm{\tilde{\bomega}_j}_2 = s^{\frac{1}{2\beta}}$. We construct $P_j(X,Y,\bZ)$ similarly to the case when $\beta>1$ using $\tilde{\bomega}_j$ and function $\phi$. That is, we choose weight functions $\gamma(1)=\gamma(-1)$. For $j=0,\dots,|\mathcal{H}'|$, assume $X, \bZ_1, \dots, \bZ_d \sim i.i.d. \ U[-1,1]$. For each $j=0, \dots, |\mathcal{H}'|$, let 
\begin{align}\label{conditional_density: beta<1}
    f_j( y=1 \mid x, \boldsymbol{z})= \frac{1}{2} + \frac{1}{2\sigma }\phi(x - \btheta_j^T\bz), \\
     f_j( y=-1 \mid x, \boldsymbol{z})= \frac{1}{2} - \frac{1}{2\sigma }\phi(x - \btheta_j^T\bz),
\end{align}
where $\btheta_j = c \left(\frac{\log(d/s) }{N}\right)^{1/(2\beta)} \tilde \bomega_j$ for $j \neq 0$, $\btheta_0 = \boldsymbol{0}$, $c$ is some sufficiently small constant and $\sigma$ is some sufficiently large constant. Under the assumption that $s(\frac{\log(d/s) }{N})^{1/2}=o(1)$, we can guarantee that $P_0, \ldots, P_{\mathcal{H}^{\prime}}$ are well defined probability measures.
Similar to the argument for the case $\beta>1$, we prove the minimax lower bound by verifying the following three conditions:
\begin{itemize}
    \item For any sampling method $Q \in \mathcal{Q}_N(\mathcal{P}(\beta,s))$ and $j=1, \dots, |\mathcal{H}'|$, $P_j(X, Y, \boldsymbol{Z}) \subset \mathcal{P}(\beta,s)$.
    \item $ \mathrm{KL}\left(f^{j} \| f^{0}\right) \leq \gamma \log |\mathcal{H}|$ for some $\gamma \in(0,1 / 8)$, where $ \mathrm{KL}\left(f^{j} \| f^{0}\right)$ is the K-L divergence between probability measures $f^j$ and $f^0$, and $f^j$ is the probability measure of the random variables $\{O_i\}_{i=1}^n$ under hypothesis $j$.
    \item For all $j \neq k$ and $q=1,2,\left\|\boldsymbol{\theta}^*_j-\boldsymbol{\theta}^*_k\right\|_q \geq 2 t$, where $t \asymp s^{\frac{1}{q}-\frac{1}{2}}\left(\frac{s \log (d / s)}{N}\right)^{1/(2\beta)} .$
\end{itemize}
By the construction of $P(X,Y,\bZ)$, we have 
    \begin{align}\label{yzjoint: beta<1}
        f_j(y, \bz) =& \int_{-1}^1 f_j(x, y, \bz) dx\nonumber\\
        = & \int_{-1}^1 f_j(y \mid x, \bz ) f(x) f(\bz) dx \nonumber\\
        = & \int_{-1}^1 \left( \frac{1}{2} + \frac{y}{2\sigma }\phi(x - \btheta_j^T\bz) \right) \frac{1}{2^{d+1}}\, dx \\
        =& \frac{1}{2^{d+1}} \left(1 + \frac{y}{2\sigma} G(\btheta_j^T\bz)\right),
    \end{align}
    where $G(\delta) = \int_{-1}^1 \phi(x - \delta) \, dx$. Since $G(-\delta) = \int_{-1}^1 \phi(x+\delta) \,dx = \int_1^{-1} \phi(-u+\delta) \, d(-u) = -\int_{-1}^1 \phi(u - \delta) \, du = -G(\delta)$, $G(\cdot)$ is also an odd function. For marginal distribution, we get
    \begin{equation}\label{bayesrule: beta<1}
         f_j(x \mid y, \bz) = \frac{f_j(y \mid x, \bz) f(x)f(\bz)}{f(y, \bz)} = \frac{\sigma+y\phi(x-\btheta_j^T\bz)}{2\sigma+ yG(\btheta_j^T \bz)}
    \end{equation}
    belongs to the H\"older class $\mathcal{P}(\beta, L)$. This is because \[
    \left|f_j^{(0)}\left(x_1 \mid y, \boldsymbol{z}\right)-f_j^{(0)}\left(x_2 \mid y, \boldsymbol{z}\right)\right|  = \frac{ |\phi(x_1 - \btheta_j^T\bz) - \phi(x_2-\btheta_j^T\bz)| }{2\sigma+ yG(\btheta_j^T \bz)} \leq C\frac{|x_1 - x_2|^\beta}{\sigma },
    \]
    given $\sigma$ sufficiently large and $G(\btheta_j^T\bz) = o(1)$ which will be shown later, and the inequality is by Lemma \ref{lem: phi's holder}. Clearly, $p_{\min} = 1/2 - 1/\sigma \leq f_j(x\mid y, \bz) \leq p_{\max} = 1/2+1/\sigma$. Thus, $P_j \in \mathcal{P}(\beta, L,p_{\min},p_{\max})$ holds.
    Now we check the condition (\ref{sparseeigen2}), i.e.,
    \begin{equation} \label{holder3: beta<1}
    \sup_{\|\boldsymbol{v}\|_0 \leq s^{\prime}} \frac{\boldsymbol{v}^T \mathbb{E}\left(\boldsymbol{Z} \boldsymbol{Z}^T \mid Y=y\right) \boldsymbol{v}}{\|\boldsymbol{v}\|_2^2} \leq M_1.
    \end{equation}
    By (\ref{yzjoint: beta<1}) we have
    \[
    P_j(Y=1) = \int_{\bZ} \frac{1}{2^{d+1}} \left(1 + \frac{1}{2\sigma} G(\btheta_j^T\bz)\right)d\bz = \frac{1}{2} = P_j(Y=-1),
    \]
    hence $f_j(\bz \mid Y=y) = \frac{1}{2^{d}}(1 + \frac{y}{2\sigma} G(\btheta_j^T\bz))$.
    We have
    \begin{align*}
    \bv^T \mathbb{E}\left(\bZ\bZ^T \mid Y=y\right) \bv
    =&  \int (\bv^T \bz)^2  \frac{1}{2^{d}}(1 + \frac{y}{2\sigma} G(\btheta_j^T\bz)) d\bz  \\ 
    =& \bv^T \EE(\bZ \bZ^T) \bv + \frac{y}{2\sigma} \int \frac{1}{2^d}(\bv^T \bz)^2 G(\btheta_j^T\bz)) d\bz.
    \end{align*}
    Note that $\mathbb{E}\left(\boldsymbol{Z} \boldsymbol{Z}^T \right) = 1/3 \mathbb{I}_d$, $|\btheta_j^T\bz| \leq \norm{\btheta_j}_1 = c \left(\frac{\log(d/s)}{N} \right)^{\frac{1}{2\beta}} s^{\frac{1}{2}(\frac{1}{\beta} + 1)} \lesssim \left(\frac{\log(d/s)}{N} \right)^{\frac{1}{2\beta}} s^{\frac{1}{\beta}} = o(1)$ and
    \begin{align*}
        |G(\delta)| &\leq \left|\int_{-1}^1 \phi(x-\delta) \,dx \right|\\
        &= \left|\int_{-1-\delta}^{1-\delta} \phi(u) \, du \right|\\
        &= \left|\int_{-1-\delta}^{-1} \phi(u)\,du + \int_1^{1-\delta} \phi(u) \, du \right|\\
        &\leq \int_{\min\{-1-\delta,-1\}}^{\max\{-1-\delta,-1\}} |\phi(u)|\,du + \int_{\min\{1,1-\delta\}}^{\max\{1,1-\delta\}} |\phi(u)|\,du \leq C |\delta|
    \end{align*}
    for some constant $C$ if $\delta = o(1)$. Hence, we get
    \begin{align*}
    \Big|\int \frac{1}{2^d}(\bv^T \bz)^2 G(\btheta_j^T\bz)) d\bz d\bz\Big| \leq C \|\btheta_j\|_1 \EE((\bv^T \bZ)^2)
    \leq C \left(\frac{\log(d/s)}{N} \right)^{\frac{1}{2\beta}} s^{\frac{1}{\beta}} \frac{\|\bv\|_2^2}{3},
    \end{align*}
    and
    \begin{align*}
        |\bv^T \mathbb{E}\left(\bZ\bZ^T \mid Y=y\right) \bv | \leq \frac{1}{3}\|\bv\|_2^2
    + \frac{1}{6\sigma} s^{\frac{1}{\beta}} \left(\frac{\log(d/s)}{N} \right)^{\frac{1}{2\beta}} \|\bv\|_2^2.
    \end{align*}
    Since $s^{\frac{1}{\beta}} \left(\frac{\log(d/s)}{N} \right)^{\frac{1}{2\beta}} =o(1)$, (\ref{holder3: beta<1}) holds.
    Similar to Lemma \ref{minimizer}, we then show the following results: (1)$\boldsymbol{\theta}_j$ is the unique minimizer of $R_{P_j}(\btheta)$,
    (2) $\|\btheta_j\|_2 \leq C$ for some constant $C>0$, and (3) $\rho_- \norm{\btheta - \btheta_j}_2^{1+\beta} \leq R(\btheta) - R(\btheta_j) - \nabla R(\btheta)^T (\btheta - \btheta_j) \leq \rho_+ \norm{\btheta - \btheta_j}_2^{1+\beta}$ for some positive constants $\rho_-$ and $\rho_+$. 
    
    In fact, recall that $R_j(\boldsymbol{\theta})=\mathbb{E}_j\left[\gamma(Y)\left(1-\operatorname{sign}\left(Y\left(X-\boldsymbol{\theta}^T \boldsymbol{Z}\right)\right)\right)\right]$, and $\gamma(1) = \gamma(-1)=C$ for some constant $C>0$, 
    the difference between $R_j(\boldsymbol{\theta})$ and $R_j(\boldsymbol{\theta}_j)$ can be written as
    \begin{align*}
        R_j(\boldsymbol{\theta})-R_j\left(\boldsymbol{\theta}_j\right)
        =& 2C \int_{\mathcal{G}} \operatorname{sign}\left(x-\boldsymbol{\theta}_j^T \boldsymbol{z}\right) \mathbb{E}_j[Y \mid x, \boldsymbol{z}] d P_{j ; X, \boldsymbol{Z}},
    \end{align*}
where 
\[
\mathcal{G}=\left\{(x, \boldsymbol{z}) \mid \operatorname{sign}\left(x-\boldsymbol{\theta}^T \boldsymbol{z}\right) \neq \operatorname{sign}\left(x-\boldsymbol{\theta}_j^T \boldsymbol{z}\right)\right\},
\]
$P_{j ; X, \boldsymbol{Z}}$ is the joint distribution of $(X,\bZ)$ under $P_j$, and
\begin{align}\label{conditionex: beta<1}
   \mathbb{E}_j[Y \mid x, \boldsymbol{z}]=& f_j(Y=1 \mid x, \boldsymbol{z})-f_j(Y=-1 \mid x, \boldsymbol{z}) \nonumber\\
    =& \frac{1}{\sigma }\phi(x-\btheta_j^T\bz),
\end{align}
\[
\operatorname{sign}\left(\mathbb{E}_j[Y \mid x, \boldsymbol{z}]\right) = \operatorname{sign} \left(\phi(x - \btheta_j^T\bz) \right) = \operatorname{sign} \left(x - \btheta_j^T\bz \right).
\]
Therefore, 
\begin{equation}\label{signj0: beta<1}
    R_j(\boldsymbol{\theta})-R_j\left(\boldsymbol{\theta}_j\right)=2C \int_{\mathcal{G}}\left|\mathbb{E}_j[Y \mid x, \boldsymbol{z}]\right| d P_{j ; X, \boldsymbol{Z}} \geq 0,
\end{equation}
hence $\btheta_j$ is a minimizer of $R_j(\btheta)$. In addition, $\|\btheta_j\|_2 = c(\frac{\log(d/s)}{N})^{\frac{1}{2\beta}} \|\tilde\bomega_j\|_2 = c(\frac{s\log(d/s)}{N})^{\frac{1}{2\beta}} = O(1)$.

Now we check the uniqueness. By (\ref{conditionex: beta<1}) we have 
\begin{equation}\label{unique: beta<1}
     R_j(\boldsymbol{\theta})-R_j\left(\boldsymbol{\theta}_j\right)=\frac{2C}{\sigma } \int_{\mathcal{G}}\left|\phi(x - \btheta_j^T\bz)\right| d P_{j ; X, \boldsymbol{Z}}.
\end{equation}
For any $\boldsymbol{\theta} \neq \boldsymbol{\theta}_j$, consider the set 
\[
\mathcal{G}_{\boldsymbol{z}}=\left\{\boldsymbol{z}:\left(\boldsymbol{\theta}-\boldsymbol{\theta}_j\right)^T \boldsymbol{z} \neq 0,\left|\boldsymbol{\theta}^T \boldsymbol{z}\right| \leq 1,\left|\boldsymbol{\theta}_j^T \boldsymbol{z}\right| \leq 1\right\}.
\]
Note that there exists an open neighborhood in $\mathcal{G}_{\boldsymbol{z}}$, hence $\mathcal{G}_{\boldsymbol{z}}$ has nonzero measure. Then
define 
$\bar{\mathcal{G}}=\left\{(x, \boldsymbol{z}): \boldsymbol{\theta}^T \boldsymbol{z}<x<\boldsymbol{\theta}_j^T \boldsymbol{z}\right.$ or $\left.\boldsymbol{\theta}_j^T \boldsymbol{z}<x<\boldsymbol{\theta}^T \boldsymbol{z}, \boldsymbol{z} \in \mathcal{G}_{\boldsymbol{z}}\right\}$, we have $\bar{\mathcal{G}} \subset \mathcal{G}_{\boldsymbol{z}}$ and $\bar{\mathcal{G}}$ has nonzero measure as well.  Therefore, (\ref{unique: beta<1}) implies that $ R_j(\boldsymbol{\theta})-R_j\left(\boldsymbol{\theta}_j\right) >0$ for $\btheta \neq \btheta_j$, which completes the proof of the uniqueness.

Finally, Lemma \ref{lem: beta-smooth class} and Lemma \ref{lem: Taylor remainder lower con for uni} in below implies that $\rho_- \norm{\btheta - \btheta^*}_2^{1+\beta} \leq R(\btheta) - R(\btheta^*) - \nabla R(\btheta^*)^T (\btheta - \btheta^*) \leq \rho_+ \norm{\btheta - \btheta^*}_2^{1+\beta}$ for some positive constants $\rho_-$ and $\rho_+$.

    Next, we verify the KL condition using the similar argument in the proof for the case $\beta>1$. For $f^j$, the probability measure of the random variables $\{(X_i,\bZ_i, Y_i, R_i)^{\ind\{R_i = 1\}}, (X_i,\bZ_i, R_i)^{\ind\{R_i = 0\}} \}_{i=1}^n$ under hypothesis $j$, and $\bar{H}_{i-1} = \{(X_j,\bZ_j, Y_j, R_j)^{\ind\{R_j = 1\}}, (X_j,\bZ_j, R_j)^{\ind\{R_j = 0\}} \}_{j=1}^{i-1}$, we can still get
\begin{align*}
    f^j=& \prod_{i=1}^n 
    \left\{f(R_i=1 \mid X_i, \bZ_i, \bar{H}_{i-1})  \right\}^{\ind\{R_i = 1\}}
     \left\{f(R_i=0 \mid X_i, \bZ_i, \bar{H}_{i-1})  \right\}^{\ind\{R_i = 0\}}  \\
     & \ \ \ \ \cdot f_j( Y_i \mid X_i, \bZ_i )^{\ind\{R_i = 1\}}  f_j( X_i, \bZ_i ),
\end{align*}
and 
\begin{equation}\label{klnew: beta<1}
    \mathrm{KL}\left(f^{j} \| f^{0}\right)= \sum_{i=1}^n \EE \left[\EE\left[\log \frac{f_{j}(Y_i \mid X_i, \bZ_i)}{f_{0}(Y_i \mid X_i, \bZ_i)}  \mid X_i, \bZ_i\right] \cdot \PP\left(R_i = 1 \mid X_i, \bZ_i, \bar{H}_{i-1}\right)
    \right].
\end{equation}
Note that $\EE\left[\log \frac{f_{j}(Y_i \mid X_i, \bZ_i)}{f_{0}(Y_i \mid X_i, \bZ_i)}  \mid X_i, \bZ_i\right]$ is the K-L divergence between two Bernoulli distributions $f_{j}(Y_i \mid X_i, \bZ_i)$ and $f_{0}(Y_i \mid X_i, \bZ_i)$. Denote $p_j = f_j(y=1 \mid x, \bz), \Delta_j = p_j - p_0 = \frac{1}{2\sigma}\left(\phi(x-\btheta_j^T\bz) - \phi(x)\right)$, then we have
\begin{equation}\label{eq: KL beta<1}
\begin{aligned}
    &\EE\left[\log \frac{f_{j}(Y_i \mid X_i, \bZ_i)}{f_{0}(Y_i \mid X_i, \bZ_i)}  \mid X_i = x, \bZ_i = \bz\right] \\
    =& 
    p_j \log \frac{p_j}{p_0} + (1-p_j) \log \frac{1-p_j}{1-p_0}\\
    =& (p_0 + \Delta_j)\left(\frac{\Delta_j}{p_0} - \frac{1}{2} \left(\frac{\Delta_j}{p_0}\right)^2 + o(\Delta_j^2) \right) + (1-p_0-\Delta_j)\left(-\frac{\Delta_j}{1-p_0} - \frac{1}{2} \left(\frac{\Delta_j}{1-p_0}\right)^2 + o(\Delta_j^2) \right)\\
    =& \frac{\Delta_j^2}{2p_0(1-p_0)}+ o(\Delta_j^2) = \frac{\left(\phi(x-\btheta_j^T\bz) - \phi(x) \right)^2}{2(\sigma^2 - \phi(x)^2)} + o((\phi(x-\btheta_j^T\bz) - \phi(x) )^2).
\end{aligned}
\end{equation}
Since $|\phi(x-\btheta_j^T\bz) - \phi(x)| \leq C |\btheta_j^T\bz|^\beta$ by Lemma \ref{lem: phi's holder}, (\ref{eq: KL beta<1}) $\leq C' |\btheta_j^T\bz|^{2\beta}$ for some constant $C' >0$. Plugging this back to (\ref{klnew: beta<1}), we have
\begin{equation}\label{klnew1: beta<1}
\begin{aligned}
    \mathrm{KL}\left(f^{j} \| f^{0}\right) \leq& 
    C' \sum_{i=1}^n \EE \left[|\btheta_j^T\bz|^{2\beta}\PP\left(R_i = 1 \mid X_i, \bZ_i, \bar{H}_{i-1}\right)
    \right] \nonumber\\
    =&C' \frac{\log(d/s)}{N} \sum_{i=1}^n\EE\left[Q_i |\tilde\bomega_j^T\bZ_i|^{2\beta}  \right]\\
    =&C' \frac{\log(d/s)}{N} \sum_{i=1}^n\EE\left[Q_i^{1-\beta} Q_i^\beta |\tilde\bomega_j^T\bZ_i|^{2\beta}\right]\\
    \leq&C' \frac{\log(d/s)}{N} \left(\sum_{i=1}^n\EE Q_i\right)^{1-\beta} \left(\tilde\bomega_j^T \sum_{i=1}^n \EE(Q_i \bZ_i \bZ_i^T) \tilde\bomega_j \right)^\beta  &\text{(By H\"older's Inequality)}\\
    \leq&C' \frac{\log(d/s)}{N} N^{1-\beta} \left(C_2 \norm{\tilde\bomega_j}_2^2 N \right)^\beta &(Q \in \mathcal{Q}_N(\mathcal{P}(\beta,s)))\\
    =& C' \frac{\log(d/s)}{N} N s = C' s \log(d/s) \leq \gamma \log|\mathcal{H}'|
\end{aligned}
\end{equation}
for some $\gamma \in (0, 1/8)$, hence the KL condition is satisfied.
Finally, we verify the last condition that for all $j \neq k$ and $q=1,2,\left\|\boldsymbol{\theta}_j-\boldsymbol{\theta}_k\right\|_q \gtrsim s^{\frac{1}{q}-\frac{1}{2}}\left(\frac{s \log (d / s)}{N}\right)^{1/(2\beta)}$. By definition of $\btheta_j$, we have when $j \neq 0$,
\[
\left\|\boldsymbol{\theta}_0-\boldsymbol{\theta}_j\right\|_2 \gtrsim \left(\frac{\log(d/s)}{N}\right)^{\frac{1}{2\beta}} \norm{\tilde\bomega_j}_2 = \left(\frac{\log(d/s)}{N}\right)^{\frac{1}{2\beta}} s^{\frac{1}{2\beta}} = \left(\frac{s\log(d/s)}{N}\right)^{\frac{1}{2\beta}},
\]
\[
\left\|\boldsymbol{\theta}_0-\boldsymbol{\theta}_j\right\|_1 \gtrsim \left(\frac{\log(d/s)}{N}\right)^{\frac{1}{2\beta}} \norm{\tilde\bomega_j}_1 = \left(\frac{\log(d/s)}{N}\right)^{\frac{1}{2\beta}} s^{\frac{1}{2}(\frac{1}{\beta}+1)} = \sqrt{s} \left(\frac{s\log(d/s)}{N}\right)^{\frac{1}{2\beta}}.
\]
For all $j,k \neq 0$, we have
\begin{equation} 
\begin{aligned}
\left\|\boldsymbol{\theta}_j-\boldsymbol{\theta}_k\right\|_2 &\gtrsim \left(\frac{\log(d/s)}{N}\right)^{\frac{1}{2\beta}} \norm{\tilde\bomega_j - \tilde\bomega_k}_2\\
&= \left(\frac{\log(d/s)}{N}\right)^{\frac{1}{2\beta}} s^{\frac{1}{2}(\frac{1}{\beta} - 1)} \norm{\bomega_j - \bomega_k}_2\\
&\geq \left(\frac{\log(d/s)}{N}\right)^{\frac{1}{2\beta}} s^{\frac{1}{2}(\frac{1}{\beta} - 1)} \frac{\sqrt{s}}{4}\gtrsim \left(\frac{s\log(d/s)}{N}\right)^{\frac{1}{2\beta}},
\end{aligned}
\end{equation}
\begin{equation} 
\begin{aligned}
\left\|\boldsymbol{\theta}_j-\boldsymbol{\theta}_k\right\|_1 &\gtrsim \left(\frac{\log(d/s)}{N}\right)^{\frac{1}{2\beta}} \norm{\tilde\bomega_j - \tilde\bomega_k}_1\\
&= \left(\frac{\log(d/s)}{N}\right)^{\frac{1}{2\beta}} s^{\frac{1}{2}(\frac{1}{\beta} - 1)} \norm{\bomega_j - \bomega_k}_1\\
&\geq \left(\frac{\log(d/s)}{N}\right)^{\frac{1}{2\beta}} s^{\frac{1}{2}(\frac{1}{\beta} - 1)} \frac{s}{16} \gtrsim \sqrt{s}\left(\frac{s\log(d/s)}{N}\right)^{\frac{1}{2\beta}},
\end{aligned}
\end{equation}
which completes the proof for the case $0 < \beta \leq 1$. 
\end{proof}

\begin{lemma}\label{lem: beta-smooth class}
Under Assumption \ref{asp4} (iii), if we further assume $\inf \limits_{\norm{\bu}_2 = 1} \EE[|\bu^T \bZ|^{1+\beta} \mid Y=y] \geq c$ for some constant $c > 0$, 
\begin{equation}\label{eq_smooth_lower_upper_1}
L_1|\Delta|^\beta \leq \operatorname{sign}(\Delta) \left(f(\btheta^{*T}\bz + \Delta \mid \bz, y=1) - f(\btheta^{*T}\bz \mid \bz, y=1)\right) \leq L_2 |\Delta|^\beta
\end{equation}
and
\begin{equation}\label{eq_smooth_lower_upper_2}
L_1|\Delta|^\beta \leq -\operatorname{sign}(\Delta) \left(f(\btheta^{*T}\bz + \Delta \mid \bz, y=-1) - f(\btheta^{*T}\bz \mid \bz, y=-1)\right) \leq L_2 |\Delta|^\beta
\end{equation}
for any $|\Delta| \leq \sqrt{s}M_nr_n$, then we have \[
    \rho_- \norm{\btheta - \btheta^*}_2^{1+\beta} \leq R(\btheta) - R(\btheta^*) - \nabla R(\btheta^*)^T (\btheta - \btheta^*) \leq \rho_+ \norm{\btheta - \btheta^*}_2^{1+\beta}
    \]
    for $\norm{\theta - \theta^*}_2 \leq r_n$, where $\rho_-$ and $\rho_+$ are positive constants and $r_n$ goes to 0 fast enough as (\ref{def: prob measure class}).
\end{lemma}
\begin{proof}
    We take $\norm{\btheta - \btheta^*}_2 \leq r_n$, then $|(\btheta - \btheta^*)^T \bz| \leq \norm{\btheta-\btheta^*}_1 M_n \leq \sqrt{s}\norm{\btheta-\btheta^*}_2 M_n \leq \sqrt{s} M_n r_n$.
    By definition
    \begin{equation}
    \begin{aligned}
        R(\btheta) &= \EE\left[\gamma(Y) L_{01}\{Y(X-\btheta^T\bZ)\}\right]\\
        &= \EE\left[L_{01}\{X-\btheta^T\bZ\} \mid Y=1 \right] + \EE\left[L_{01}\{-X+\btheta^T\bZ\} \mid Y=-1 \right]\\
        &= \PP(X < \btheta^T\bZ \mid Y=1) + \PP(X > \btheta^T\bZ \mid Y=-1)\\
        &= \int \int_{-\infty}^{\btheta^T\bz} f(x\mid \bz, y=1)\ dx f(\bz|y=1)\ d\bz + \int \int_{\btheta^T\bz}^{\infty} f(x\mid \bz, y=-1)\ dx f(\bz|y=-1)\ d\bz.
    \end{aligned}
    \end{equation}
    Hence, \[
    \nabla R(\btheta) = \int \bz f(\btheta^T\bz \mid \bz, y=1) f(\bz \mid y=1) \ d\bz - \int \bz f(\btheta^T\bz \mid \bz, y=-1) f(\bz \mid y=-1) \ d\bz,
    \]
    \begin{equation}\label{eq: Taylor remainder}
    \begin{aligned}
        &R(\btheta) - R(\btheta^*) - \nabla R(\btheta^*)^T(\btheta - \btheta^*) \\
        &=\int\int_{0}^\Delta f(\btheta^{*T}\bz+t \mid \bz, y=1) \ dt f(\bz\mid y=1)\ d\bz - \int\int_0^{\Delta} f(\btheta^{*T}\bz + t \mid\bz, y=-1) \ dt f(\bz\mid y=-1)\ d\bz\\
        &~~- \int (\btheta - \btheta^*)^T \bz f(\btheta^{*T}\bz \mid \bz, y=1) f(\bz \mid y=1) \ d\bz + \int (\btheta - \btheta^*)^T \bz f(\btheta^{*T}\bz \mid \bz, y=-1) f(\bz \mid y=-1) \ d\bz\\
        &=\int\int_{0}^\Delta \left[f(\btheta^{*T}\bz+t \mid \bz, y=1) - f(\btheta^{*T}\bz \mid \bz, y=1)\right] \ dt f(\bz\mid y=1)\ d\bz\\
        &~~- \int\int_0^{\Delta} \left[f(\btheta^{*T}\bz + t \mid\bz, y=-1) - f(\btheta^{*T}\bz \mid \bz, y=-1)\right] \ dt f(\bz\mid y=-1)\ d\bz\\
        &=\int\int_{0}^\Delta \left[f(\btheta^{*T}\bz+t \mid \bz, y=1) - f(\btheta^{*T}\bz \mid \bz, y=1)\right]  (I(\Delta>0)+I(\Delta\leq 0)) \ dt f(\bz\mid y=1)\ d\bz\\
        &~~- \int\int_0^{\Delta} \left[f(\btheta^{*T}\bz + t \mid\bz, y=-1) - f(\btheta^{*T}\bz \mid \bz, y=-1)\right]  (I(\Delta>0)+I(\Delta\leq 0)) \ dt f(\bz\mid y=-1)\ d\bz,
    \end{aligned}
    \end{equation}
    where $\Delta = (\btheta - \btheta^*)^T \bz$. By (\ref{eq_smooth_lower_upper_1}) and (\ref{eq_smooth_lower_upper_2}), 
    \begin{equation*}
    \begin{aligned}
        \text{(\ref{eq: Taylor remainder})} &\leq  \frac{L_2}{1+\beta} \EE\left[\Delta^{1+\beta}I(\Delta>0)+(-\Delta)^{1+\beta}I(\Delta\leq 0)\mid Y=1\right] \\
        &~~+ \frac{L_2}{1+\beta} \EE\left[\Delta^{1+\beta}I(\Delta>0)+(-\Delta)^{1+\beta}I(\Delta\leq 0)\mid Y=-1\right]\\
        &=\frac{L_2}{1+\beta} \EE\left[|\Delta|^{1+\beta}\mid Y=1\right]+ \frac{L_2}{1+\beta} \EE\left[|\Delta|^{1+\beta}\mid Y=-1\right]. 
    \end{aligned}
    \end{equation*}
    Since $\bZ\mid Y=y$ is a sub-Gaussian vector with a bounded sub-Gaussian norm, denoted as $K_y$, $(\btheta - \btheta^*)^T \bZ \mid Y = y$ is also a sub-Gaussian random variable with the sub-Gaussian norm $\norm{\btheta - \btheta^*}_2 K_y$. Then $\EE\left[|(\btheta - \btheta^*)^T \bZ|^{1+\beta} \mid Y=y\right] \leq C_\beta K_y^{1+\beta} \norm{\btheta - \btheta^*}_2^{1+\beta}$ for some constant $C_\beta > 0$. Hence $(\text{\ref{eq: Taylor remainder}}) \leq \rho_+ \norm{\btheta - \btheta^*}_2^{1+\beta}$ for some constant $\rho_+$.
    Similarly, we can get the lower bound
    \[
        \text{(\ref{eq: Taylor remainder})} \geq
        \frac{L_1}{1+\beta} \EE\left[|\Delta|^{1+\beta}\mid Y=1\right]+ \frac{L_1}{1+\beta} \EE\left[|\Delta|^{1+\beta}\mid Y=-1\right]. 
    \]
    Notice that
    \begin{equation}
    \begin{aligned}
        \EE\left[|\Delta|^{1+\beta}\mid Y=y\right] &= \EE[|(\btheta - \btheta^*)^T \bz|^{1+\beta} \mid Y=y]\\
        & = \norm{\btheta - \btheta^*}_2^{1+\beta} \EE[|\bu^T \bz|^{1+\beta} \mid Y=y] \\
        & \geq c \norm{\btheta - \btheta^*}_2^{1+\beta},
    \end{aligned}
    \end{equation}
    where $\bu = \frac{\btheta - \btheta^*}{\norm{\btheta - \btheta^*}_2}$. We also have $(\text{\ref{eq: Taylor remainder}}) \geq \rho_- \norm{\btheta - \btheta^*}_2^{1+\beta}$ for some constant $\rho_-$.
\end{proof}

\begin{lemma} \label{lem: Taylor remainder lower con for uni}
    Suppose $0 < \beta \leq 1$. Let $\bZ = (Z_1, \dots, Z_d)$ where $\{Z_i\}_{i=1}^d$ are i.i.d. and uniformly distributed on $[-1,1]$. $f_j(\bz \mid Y=y) = \frac{1}{2^{d}}(1 + \frac{y}{2\sigma} G(\btheta_j^T\bz))$, where the $G(\delta) = \int_{-1}^1 \sign(x-\delta) |x-\delta|^\beta \ dx$ and $\sigma$, $\btheta_j$ are defined in (\ref{conditional_density: beta<1}). Then $\inf \limits_{\norm{\bu}_2 = 1} \EE[|\bu^T \bZ|^{1+\beta} \mid Y=y] \geq c$ for some constant $c>0$.
\end{lemma}
\begin{proof}
    We first prove that $\inf \limits_{\norm{\bu}_2=1} \EE[|\bu \bZ|^{1+\beta}] \geq \frac{1}{3}$. Denote $S = \bu^T\bZ = \sum_{i=1}^d u_i Z_i$ for $\norm{\bu}_2 = 1$. If $\beta = 1$, then
    \[
        \EE[|\bu^T \bZ|^2] = \EE(\sum_{i=1}^d u_i^2Z_i^2 + \sum_{i\neq j}u_iu_j Z_iZ_j) = \sum_{i=1}^d u_i^2 \EE Z_i^2 = \frac{1}{3} \sum_{i=1}^d u_i^2 = \frac{1}{3}.
    \]
    For $0 < \beta < 1$,
    \begin{align*}
        \EE S^4 &= \EE(\sum_{i=1}^d u_i Z_i)^4 = \sum_{i=1}^d u_i^4 \EE Z_i^4 + 3 \sum_{i \neq j} u_i^2 u_j^2\EE Z_i^2 \EE Z_j^2\\
        &= \frac{1}{5} \sum_{i=1}^d u_i^4 + \frac{1}{3} \sum_{i \neq j} u_i^2 u_j^2\\
        &= \frac{1}{5} \sum_{i=1}^d u_i^4 + \frac{1}{3} \left((\sum_{i=1}^d u_i^2)^2 - \sum_{i=1}^d u_i^4 \right)\\
        &= \frac{1}{3} (\sum_{i=1}^d u_i^2)^2 - \frac{2}{15} \sum_{i=1}^d u_i^4 \leq \frac{1}{3} (\sum_{i=1}^d u_i^2)^2 = \frac{1}{3}.
    \end{align*}
    Since $1+\beta < 2 < 4$, we can write $2 = p(1+\beta)+ (1-p)4$ for some $0 < p < 1$. Then by H\"older's inequality,
    \[
    \EE S^2 = \EE\left[|S|^{p(1+\beta)} |S|^{4(1-p)} \right] \leq (\EE |S|^{1+\beta})^p (\EE S^4)^{1-p}.
    \]
    Hence,
    \[
    (\EE |S|^{1+\beta})^p \geq \frac{\EE S^2}{(\EE S^4)^{1-p}} \geq \frac{\frac{1}{3}}{(\frac{1}{3})^{1-p}} = (\frac{1}{3})^p,
    \]
    which implies $\inf \limits_{\norm{\bu}_2 = 1} \EE[|\bu^T \bZ|^{1+\beta}] \geq \frac{1}{3}$ for $0 < \beta <1$. Now we prove the lower bound of conditional expectation.
    \begin{align*}
        \EE[|\bu^T\bZ|^{1+\beta} \mid Y=y] &= \int |\bu^T \bz|^{1+\beta} f_j(\bz\mid Y=y) \ d\bz\\
        &=\int |\bu^T \bz|^{1+\beta} f(\bz) \left[1+\frac{y}{2\sigma}G(\btheta_j^T\bz) \right] \ d\bz\\
        &= \EE[|\bu^T \bZ|^{1+\beta}] + \frac{y}{2\sigma}\EE\left[|\bu^T \bZ|^{1+\beta}G(\btheta_j^T\bZ) \right].
    \end{align*}
    We have shown the first term $\EE[|\bu^T \bZ|^{1+\beta}] \geq \frac{1}{3}$. Then it is enough to show the second term is $o(1)$. Note that we proved $|\btheta_j^T\bz| \lesssim \left(\frac{\log(d/s)}{N} \right)^{\frac{1}{2\beta}} s^{\frac{1}{\beta}} = o(1)$ and $|G(\delta)| \leq C |\delta|$ for some constant $C$ if $\delta = o(1)$ in the proof of Theorem \ref{lowerbound1}. Hence, $|G(\btheta_j^T\bZ)| \leq C |\btheta_j^T\bZ| \leq C' \left(\frac{\log(d/s)}{N} \right)^{\frac{1}{2\beta}} s^{\frac{1}{\beta}}$, and
    \begin{align*}
        \big|\frac{y}{2\sigma}\EE\left[|\bu^T \bZ|^{1+\beta} G(\btheta_j^T\bZ) \right] \big| &\leq \frac{1}{2\sigma} \EE \left[|\bu^T \bZ|^{1+\beta} |G(\btheta_j^T\bZ)| \right]\\
        &\leq \frac{C'}{2\sigma} \left(\frac{\log(d/s)}{N} \right)^{\frac{1}{2\beta}} s^{\frac{1}{\beta}} \EE\left[|\bu^T \bZ|^{1+\beta} \right] \\
        &\leq \frac{C'}{2\sigma} \left(\frac{\log(d/s)}{N} \right)^{\frac{1}{2\beta}} s^{\frac{1}{\beta}} \left(\EE |\bu^T\bZ|^2\right)^{\frac{1+\beta}{2}}\\
        &= \frac{C'}{2\sigma} \left(\frac{\log(d/s)}{N} \right)^{\frac{1}{2\beta}} s^{\frac{1}{\beta}} (\frac{1}{3})^{\frac{1+\beta}{2}} = o(1),
    \end{align*}
    where the third inequality is due to H\"older's inequality. This completes the proof.
\end{proof}

\begin{lemma}\label{lem: phi's holder}
    Suppose $0 < \beta < 1$. For $\phi(x) = \sign(x) |x|^\beta$ for $-1 \leq x \leq 1$, we have the H\"older condition:
    \begin{equation}\label{eq: holder ineq}
        |\phi(x) - \phi(y)| \leq c |x - y|^\beta
    \end{equation}
    for any $x, y \in [-1, 1]$, and $c$ is a positive constant. Furthermore, let $g(x) = \phi(x-a)$ for some $a \in \RR$. Then $|g(a+\Delta) - g(a)| \geq |\Delta|^\beta$. 
\end{lemma}
\begin{proof}
    For $0<\beta<1$, we first recall that the map $t\mapsto t^\beta$ is concave on $[0,\infty)$ and satisfies $t^\beta \ge t$ for all $t\in[0,1]$. And by Jensen's inequality, for any $a,b\in[0,1]$, $\frac{a^\beta + b^\beta}{2} \le \Big(\frac{a+b}{2}\Big)^\beta$, which leads to
    \begin{equation}\label{ineq: Jensen's ineq}
        a^\beta + b^\beta \leq 2^{1-\beta}(a+b)^\beta.
    \end{equation}
    Now fix $x,y\in(-1,1)$ and set $\phi(x)=\operatorname{sign}(x)|x|^\beta$.
    We distinguish three cases.
    
    First we consider $x,y \geq 0$.
    Without loss of generality, we assume $x \geq y \geq 0$. Let $y = t x$ for some $t \in [0,1]$. Then (\ref{eq: holder ineq}) is equivalent to $1 - t^\beta \leq c(1-t)^\beta$. Since $t^\beta \geq t$ and $(1-t)^\beta \geq 1-t$, this holds for $c = 1$.

    Then if $x,y \leq 0$, we have
    \begin{equation}
    \begin{aligned}
        |\phi(x) - \phi(y)| &= \big|\sign(x) |x|^\beta - \sign(y) |y|^\beta \big|\\
        &= \big| |x|^\beta - |y|^\beta \big|\\ 
        &\leq \big||x| - |y| \big|^\beta ~~(\text{By conclusion in first case})\\
        &\leq |x-y|^\beta.
    \end{aligned}
    \end{equation}
    Finally, if $\sign(x) \neq \sign(y)$, we assume $x \geq 0 \geq y$ without loss of generality. Then
    \[
        |\phi(x) - \phi(y)| = x^\beta + (-y)^\beta \leq 2^{1-\beta} (x-y)^\beta,
    \]
    where the last inequality is from (\ref{ineq: Jensen's ineq}). Combining the three cases, we obtain
    \[
        |\phi(x)-\phi(y)| \le 2^{1-\beta}|x-y|^\beta, \qquad \forall x,y\in(-1,1),
    \]
    which completes the proof. And the conclusion for function $g(\cdot)$ is trivial.
\end{proof}

\subsection{Supplementary Theoretical Results}\label{sec_all_sup}

\subsubsection{Extension to unbounded kernel and relaxed Assumption \ref{asp4} }\label{sec_theory_weaker}

In this section, we assume that there exits a sequence $C_N>0$ that may depend on $N$ such that
\begin{equation}
     \int_{C_{N}/2}^\infty |K(t)|dt \leq C\delta_k^\beta,
\end{equation}
holds for any $2 \leq k \leq K$, where $C$ is a constant that does not depend on $k$ and $\delta_k$ is the bandwidth parameter in the $k$th iteration. In addition, we also remove the sub-Gaussian vector assumption for $\bZ$ in Assumption \ref{asp4} in this section. However, in this case we need the following stronger version of the smoothness condition, 
  \begin{equation}
      \left|f^{(l)}\left(\btheta^{*T}
      \bz +\Delta \mid y, \boldsymbol{z}\right)-f^{(l)}\left( \btheta^{*T}
      \bz\mid y, \boldsymbol{z}\right)\right| \leq L\left|\Delta\right|^{\beta-l},
  \end{equation}
for any $\Delta \in \RR, y \in\{-1,1\}, \boldsymbol{z} \in \mathbb{R}^d$ and some constant $L>0$. Compared to (\ref{tsybakov}), the above inequality holds for any $\Delta \in \RR$.

The proof of the following theorems is similar to the proof of the main results in the previous section, and is omitted to avoid repetition.  We note that without the  sub-Gaussian vector assumption, it becomes complicated to pinpoint the critical points for $\beta$ at which the transition of the property of the algorithm occurs. Indeed, the conditions (\ref{KNcondition}), (\ref{secondstage}) and (\ref{probeta2.}) in the following three theorems correspond to the cases (i) $\beta\in ((1+\sqrt{3})/2, +\infty)$; (ii) $\beta\in (1, (1+\sqrt{3})/2]$; and (iii) $\beta=1$ considered in the main paper.

\begin{theorem}\label{Kiteration}
Assume that Assumptions \ref{asp3}, \ref{asp4} (i) and (ii), \ref{asp5}-\ref{asp6} hold, and $K \geq 2$. We set  $ N_k= N/K$ for  $1 \leq k \leq K$, and
\[
\delta_1 = c_{1}\left(\frac{K s \log d}{N}\right)^{1/(2\beta+1)}, \ \lambda_1 =c_{2} \sqrt{\frac{N K\log d }{n^2\delta_1}},
\]
\[
 \delta_k = c_{1}\left(\frac{C_N K s \log d}{N}\right)^{1/(2\beta)},
 \ \lambda_k = c_{2} \sqrt{\frac{N K\log d}{n^2 b_{k-1}\delta_k}},
 \ b_{k-1} = c_{3}\left(\frac{C_{N}^{2\beta+1}K s \log d}{N}\right)^{1/(2\beta)},     2\leq k \leq K,
\]
for some constants $c_{1}, c_{2}, c_{3}>0$ and $c_{3} \geq c_{1}$.  If
\begin{align} \label{KNcondition}
 \sqrt{s} M_n \left(\left(\frac{K s \log d}{N}\right)^{\beta/(2\beta+1)} \vee 
 \left(\frac{C_{N} K s \log d }{N}\right)^{1/2}\right)
 =O\left(\left(\frac{C_{N}^{2\beta+1}K s \log d}{N}\right)^{1/(2\beta)} \right), 
\end{align}
and
\begin{equation}\label{Nncondition}
    N\leq C n^{2\beta/(2\beta+1)}(Ks\log d)^{1/(2\beta+1)}C_N
\end{equation}
hold for some constant $C$, then with probability greater than $1-2K/d$, we have
\begin{equation}\label{eq_Kiteration_rate}
   \|\hat{\btheta}_K-\boldsymbol{\theta}^*\|_2 \lesssim \left(\frac{C_N K s \log d }{N}\right)^{1/2},\  \|\hat{\btheta}_K-\boldsymbol{\theta}^*\|_1 \lesssim \sqrt{s}\left(\frac{C_N K s \log d }{N}\right)^{1/2},
\end{equation}
where $N$ is the given label budget. 
\end{theorem}

\begin{theorem} 
\label{cor_betabetween}
Assume that Assumptions \ref{asp3}, \ref{asp4} (i) and (ii), \ref{asp5}-\ref{asp6} hold, and 
suppose there exists an integer $K^*$ such that 
\begin{equation}
    \label{secondstage}
\frac{
    M_n^{\frac{2\beta+1}{\beta+1} - \frac{\beta^{K^*-1}}{(\beta+1)(2\beta+1)^{K^*-2}}}    s^{\frac{2\beta+1}{2(\beta+1)} -\frac{\beta^{K^*-1}}{2(\beta+1)(2\beta+1)^{K^*-2}}}
   }{C_N^{(2\beta+1)/(2\beta)}}
= O\left(
 \left(\frac{K  \log d}{N}\right)^{\frac{1}{2\beta}-\frac{\beta}{\beta+1}+ \frac{\beta^{K^*}}{(\beta+1)(2\beta+1)^{K^*-1}}}
   \right). 
\end{equation}
We set $K=K^*$,  $ N_k= N/K$ for  $1 \leq k \leq K$. We set
 \[
  \delta_1 = c_1\left(\frac{sK \log d }{N}\right)^{1/(2\beta+1)}, \ \lambda_1 = c_2 \sqrt{\frac{N K\log d }{n^2\delta_1}},
  \]
for $2\leq k \leq K-1$,   
   \[
   \delta_k =  c_1 \left(\frac{K s \log d b_{k-1}}{N}\right)^{1/(2\beta+1)}, ~\lambda_k = c_{2} \sqrt{\frac{N K\log d}{n^2 b_{k-1}\delta_k}},
   \]
    \[
    b_{k-1}= c_3 
  M_n^{\frac{2\beta+1}{\beta+1} - \frac{\beta^{k-1}}{(\beta+1)(2\beta+1)^{k-2}}}    s^{\frac{2\beta+1}{2(\beta+1)} -\frac{\beta^{k-1}}{2(\beta+1)(2\beta+1)^{k-2}}}\left(\frac{K s \log d  }{N}\right)^{\frac{\beta}{\beta+1} 
        -\frac{\beta^{k}}{(\beta+1)(2\beta+1)^{k-1}}}, 
   \]
    for some constants $c_1, c_2, c_3> 0$, and 
\[
 \delta_K = c_{1}^\prime\left(\frac{C_N K s \log d}{N}\right)^{1/(2\beta)},
 \ \lambda_K = c_{2}^\prime \sqrt{\frac{N K\log d}{n^2 b_{K-1}\delta_K}},
 \ b_{K-1} = c_{3}^\prime\left(\frac{C_{N}^{2\beta+1}K s \log d}{N}\right)^{1/(2\beta)},    
\]
for some constants $c_{1}^\prime, c_{2}^\prime, c_{3}^\prime>0$ and $c_{3}^\prime \geq c_{1}^\prime$. If
    \begin{equation}\label{firststage}
                \left(\frac{K  \log d}{N}\right)^{\frac{1}{2\beta}-\frac{\beta}{\beta+1}+ \frac{\beta^k}{(\beta+1)(2\beta+1)^{k-1}}} = O\left(\frac{
    M_n^{\frac{2\beta+1}{\beta+1} - \frac{\beta^{k-1}}{(\beta+1)(2\beta+1)^{k-2}}}    s^{\frac{2\beta+1}{2(\beta+1)} -\frac{\beta^{k-1}}{2(\beta+1)(2\beta+1)^{k-2}}}
   }{C_N^{(2\beta+1)/(2\beta)}} \right)
    \end{equation}
    for $2 \leq k \leq K-1$, 
and
\begin{equation}\label{probfirststage}
      N \leq C_1 M_n^{\frac{\beta+1}{2\beta+1}} s^{\frac{\beta+1}{2(2\beta+1)}}(Ks \log d)^{\frac{\beta}{2\beta+1}},
\end{equation}
\begin{equation} \label{probsecondstage}
     N\leq C_2 n^{2\beta/(2\beta+1)}(Ks\log d)^{1/(2\beta+1)}C_N
\end{equation}
hold for some constants $C_1, C_2$, then with probability greater than $1-2 Kd^{-1}$,
\begin{equation*}
   \|\hat{\btheta}_K-\boldsymbol{\theta}^*\|_2 \lesssim \left(\frac{C_N K s \log d }{N}\right)^{1/2},\  \|\hat{\btheta}_K-\boldsymbol{\theta}^*\|_1 \lesssim \sqrt{s}\left(\frac{C_N K s \log d }{N}\right)^{1/2}.
\end{equation*}
\end{theorem}

\begin{theorem}
\label{cor_beta<1}
Assume that Assumptions \ref{asp3}, \ref{asp4} (i) and (ii), \ref{asp5}-\ref{asp6} hold.  We set  $ N_k= N/K$ for  $1 \leq k \leq K$, and  
  \[
  \delta_1 = c_1\left(\frac{K s \log d }{N}\right)^{1/(2\beta+1)}, \ \lambda_1 = c_2 \sqrt{\frac{N K\log d }{n^2\delta_1}},
  \]
   \[
   \ 2\leq k \leq K: \ 
   \delta_k =  c_1 \left(\frac{K s \log d b_{k-1}}{N}\right)^{1/(2\beta+1)}, ~\lambda_k = c_{2} \sqrt{\frac{N K\log d}{n^2 b_{k-1}\delta_k}},
   \]
    \[
    b_{k-1}= c_3 
  M_n^{\frac{2\beta+1}{\beta+1} - \frac{\beta^{k-1}}{(\beta+1)(2\beta+1)^{k-2}}}    s^{\frac{2\beta+1}{2(\beta+1)} -\frac{\beta^{k-1}}{2(\beta+1)(2\beta+1)^{k-2}}}
  \left(\frac{K s \log d  }{N}\right)^{\frac{\beta}{\beta+1} 
        -\frac{\beta^{k}}{(\beta+1)(2\beta+1)^{k-1}}}, 
   \]
    for some constants $c_1, c_2, c_3> 0$. If
 \begin{equation}\label{probeta2.}
   \left(\frac{K  \log d}{N}\right)^{\frac{1}{2\beta}-\frac{\beta}{\beta+1}+ \frac{\beta^k}{(\beta+1)(2\beta+1)^{k-1}}} = O\left(\frac{
    M_n^{\frac{2\beta+1}{\beta+1} - \frac{\beta^{k-1}}{(\beta+1)(2\beta+1)^{k-2}}}    s^{\frac{2\beta+1}{2(\beta+1)} -\frac{\beta^{k-1}}{2(\beta+1)(2\beta+1)^{k-2}}}
   }{C_N^{(2\beta+1)/(2\beta)}} \right)
 \end{equation}
 for all $2 \leq k \leq K$,
 and
  \begin{equation} \label{probbeta1.}
    N \leq C M_n^{\frac{\beta+1}{2\beta+1}} s^{\frac{\beta+1}{2(2\beta+1)}}(Ks \log d)^{\frac{\beta}{2\beta+1}},
\end{equation}
hold for some constant $C$, then with probability greater than $1-2 Kd^{-1}$,
 \[
 \|\hat{\btheta}_K - \btheta^*\|_2 \lesssim M_n^{\frac{\beta}{\beta+1} - \frac{\beta^K}{(\beta+1)(2\beta+1)^{K-1}}}    s^{\frac{\beta}{2(\beta+1)} -\frac{\beta^K}{2(\beta+1)(2\beta+1)^{K-1}}}\left(\frac{K s \log d  }{N}\right)^{\frac{\beta}{\beta+1} 
        -\frac{\beta^{K+1}}{(\beta+1)(2\beta+1)^K}},
\]
 \[
 \|\hat{\btheta}_K - \btheta^*\|_1 \lesssim 
M_n^{\frac{\beta}{\beta+1} - \frac{\beta^K}{(\beta+1)(2\beta+1)^{K-1}}}    s^{\frac{2\beta+1}{2(\beta+1)} -\frac{\beta^K}{2(\beta+1)(2\beta+1)^{K-1}}}\left(\frac{K s \log d  }{N}\right)^{\frac{\beta}{\beta+1} 
        -\frac{\beta^{K+1}}{(\beta+1)(2\beta+1)^K}}.
\]
\end{theorem}

\subsubsection{Supplementary Results for Path-following Algorithm}\label{sec_path_theory}
The analysis of the path-following algorithm follows the same line as \cite{feng2022nonregular} and the references therein. We only provide a sketch of the proof and refer the details to the original paper. In fact, the key difference between our proof and \cite{feng2022nonregular} is established in Proposition~\ref{prop2} and Proposition~\ref{E_1}. 

\begin{lemma}\label{algolemma1}
    Assume the conditions of Proposition~\ref{prop2} and Assumption~\ref{asp6} hold. For $\lambda \geq \lambda_{k,tgt}$, if $\boldsymbol{\theta} \in \Omega,\left\|\boldsymbol{\theta}_{S^{* c}}\right\|_0 \leq \widetilde{s}, \omega_\lambda(\boldsymbol{\theta}) \leq \frac{1}{2} \lambda$, and $\left\|\nabla R_{\delta_k}^{D_k}\left(\boldsymbol{\theta}^*\right)-
    \nabla  R_{\delta_k, \hat{\btheta}_{k-1}}(\btheta^*)
    \right\|_{\infty} \leq \lambda / 8$, we have
   \[
   \left\|\boldsymbol{\theta}-\boldsymbol{\theta}^*\right\|_2 \leq \frac{\bar{C}_1}{\rho_{-}}\left( c_{n,k}\delta_k^\beta \vee \sqrt{s} \lambda\right),
   \]
   \[
   \left\|\boldsymbol{\theta}-\boldsymbol{\theta}^*\right\|_1 \leq \frac{\bar{C}_2}{\rho_{-}}\left(\frac{c_{n,k}^2\delta_k^{2 \beta}}{\lambda} \vee \sqrt{s} c_{n,k}\delta_k^\beta \vee s \lambda\right),
   \]
  \[
  f_\lambda(\boldsymbol{\theta})-f_\lambda\left(\boldsymbol{\theta}^*\right) \leq \frac{\bar{C}_2}{2 \rho_{-}}\left( c_{n,k}^2\delta_k^{2 \beta} \vee \sqrt{s} c_{n,k}\delta_k^\beta \lambda \vee s \lambda^2\right),
  \]
  where $ f_\lambda(\boldsymbol{\theta})$ denotes the objective function $R_{\delta_k}^{D_k}(\boldsymbol{\theta})+\lambda\|\boldsymbol{\theta}\|_1$ and $\bar{C}_1, \bar{C}_2 > 0$ are constants that depend on $C_2$ in Proposition~\ref{prop2}.
\end{lemma}

\begin{proof}
Combining (\ref{lowerrho}) in Assumption~\ref{asp6} with the definition of $w_{\lambda}(\btheta)$, we can derive 
\begin{equation}\label{lemma1s1}
    \frac{3}{2} \lambda\left\|\left(\boldsymbol{\theta}-\boldsymbol{\theta}^*\right)_{S^*}\right\|_1-\left(\boldsymbol{\theta}-\boldsymbol{\theta}^*\right)^T \nabla R_{\delta_k}^{D_k}\left(\boldsymbol{\theta}^*\right) \geq \frac{1}{2} \lambda\left\|\left(\boldsymbol{\theta}-\boldsymbol{\theta}^*\right)_{S^{* c}}\right\|_1+\rho_{-}\left\|\boldsymbol{\theta}-\boldsymbol{\theta}^*\right\|_2^2 .
\end{equation}
By Proposition~\ref{prop2}, Proposition~\ref{E_1}, and notice the sparsity of $\btheta, \btheta^*$, we have
\begin{equation}\label{diff}
\begin{aligned}
\left|\left(\boldsymbol{\theta}-\boldsymbol{\theta}^*\right)^T \nabla R_{\delta_k}^{D_k}\left(\boldsymbol{\theta}^*\right)\right| & =\mid\left(\boldsymbol{\theta}-\boldsymbol{\theta}^*\right)^T (\underbrace{\nabla  R^{D_k}_{\delta_k}(\btheta^*)  - \nabla  R_{\delta_k, \hat{\btheta}_{k-1}}(\btheta^*)}_{E_1}+\underbrace{\nabla  R_{\delta_k, \hat{\btheta}_{k-1}}(\btheta^*) - \nabla R(\btheta^*)}_{E_2} ) \mid) \\
& \leq\left\|\boldsymbol{\theta}-\boldsymbol{\theta}^*\right\|_1\left\|E_1\right\|_{\infty}+
 C_2c_{n,k}\delta_k^\beta 
\left\|\boldsymbol{\theta}-\boldsymbol{\theta}^*\right\|_2,
\end{aligned}
\end{equation}
where $C_2$ is the constant defined in  Proposition~\ref{prop2}. Combining (\ref{diff}) with (\ref{lemma1s1}) we have 
\begin{equation}\label{inequalgeneral}
\begin{aligned}
\rho_{-}\left\|\boldsymbol{\theta}-\boldsymbol{\theta}^*\right\|_2^2 & \leq  C_2c_{n,k}\delta_k^\beta \left\|\boldsymbol{\theta}-\boldsymbol{\theta}^*\right\|_2+\left(\frac{3}{2} \lambda+\left\|E_1\right\|_{\infty}\right)\left\|\left(\boldsymbol{\theta}-\boldsymbol{\theta}^*\right)_{S^*}\right\|_1 \\
& -\left(\frac{1}{2} \lambda-\left\|E_1\right\|_{\infty}\right)\left\|\left(\boldsymbol{\theta}-\boldsymbol{\theta}^*\right)_{S^{* c}}\right\|_1.
\end{aligned}
\end{equation}
Now we discuss two cases. If $\rho_{-}\left\|\boldsymbol{\theta}-\boldsymbol{\theta}^*\right\|_2^2  \leq  3 C_2c_{n,k}\delta_k^\beta \left\|\boldsymbol{\theta}-\boldsymbol{\theta}^*\right\|_2$, then 
\begin{equation}
    \left\|\boldsymbol{\theta}-\boldsymbol{\theta}^*\right\|_2 \leq \frac{3}{\rho_-}
      C_2c_{n,k}\delta_k^\beta.
\end{equation}
If $\rho_{-}\left\|\boldsymbol{\theta}-\boldsymbol{\theta}^*\right\|_2^2  >  3 C_2c_{n,k}\delta_k^\beta \left\|\boldsymbol{\theta}-\boldsymbol{\theta}^*\right\|_2$, then we have
\begin{equation}
\begin{aligned}
\frac{2 \rho_{-}}{3}\left\|\boldsymbol{\theta}-\boldsymbol{\theta}^*\right\|_2^2 & \leq\left(\frac{3}{2} \lambda+\left\|E_1\right\|_{\infty}\right)\left\|\left(\boldsymbol{\theta}-\boldsymbol{\theta}^*\right)_{S^*}\right\|_1 \\
& -\left(\frac{1}{2} \lambda-\left\|E_1\right\|_{\infty}\right)\left\|\left(\boldsymbol{\theta}-\boldsymbol{\theta}^*\right)_{S^{* c}}\right\|_1 .
\end{aligned}
\end{equation}
Note that the condition of $\lambda$ ensures that $\frac{1}{2} \lambda-\left\|E_1\right\|_{\infty} \geq 0$, hence we have
\begin{equation}
\begin{aligned}
\left\|\boldsymbol{\theta}-\boldsymbol{\theta}^*\right\|_2 & \leq \frac{3 \sqrt{s}\left(\frac{3}{2} \lambda+\left\|E_1\right\|_{\infty}\right)}{2 \rho_{-}} \leq \frac{3}{\rho_{-}} \sqrt{s} \lambda.
\end{aligned}
\end{equation}
Combining the above two cases, we conclude that
\begin{equation}\label{algolemma12norm}
\left\|\boldsymbol{\theta}-\boldsymbol{\theta}^*\right\|_2 \leq \frac{3}{\rho_{-}}\left(  C_2c_{n,k}\delta_k^\beta \vee \sqrt{s} \lambda\right).
\end{equation}

For $\left\|\boldsymbol{\theta}-\boldsymbol{\theta}^*\right\|_1$, define $\gamma=\frac{\frac{3}{2} \lambda+\left\|E_1\right\|_{\infty}}{\frac{1}{2} \lambda-\left\|E_1\right\|_{\infty}} \leq \frac{13}{3}$, and we consider two cases below.

If $\left\|\left(\boldsymbol{\theta}-\boldsymbol{\theta}^*\right)_{S^{* c}}\right\|_1>2 \gamma\left\|\left(\boldsymbol{\theta}-\boldsymbol{\theta}^*\right)_{S^*}\right\|_1$, by (\ref{algolemma12norm}) we obtain 
\begin{equation}
\begin{aligned}
\left\|\boldsymbol{\theta}-\boldsymbol{\theta}^*\right\|_1 & \leq \sqrt{s}(1+2 \gamma)\left\|\boldsymbol{\theta}-\boldsymbol{\theta}^*\right\|_2 \\
& \leq \frac{29}{3} \sqrt{s}\left\|\boldsymbol{\theta}-\boldsymbol{\theta}^*\right\|_2 \\
& \leq \frac{29}{\rho_{-}}\left(C_2 \sqrt{s} c_{n,k} \delta_k^\beta \vee s \lambda\right) .
\end{aligned}
\end{equation}

If $\left\|\left(\boldsymbol{\theta}-\boldsymbol{\theta}^*\right)_{S^{* c}}\right\|_1>2 \gamma\left\|\left(\boldsymbol{\theta}-\boldsymbol{\theta}^*\right)_{S^*}\right\|_1$, we have 
\begin{equation}\label{norm1case2}
    \left\|\boldsymbol{\theta}-\boldsymbol{\theta}^*\right\|_1 \leq\left(1+\frac{1}{2 \gamma}\right)\left\|\left(\boldsymbol{\theta}-\boldsymbol{\theta}^*\right)_{S^{* c}}\right\|_1.
\end{equation}
By (\ref{inequalgeneral}) we have
\begin{align*}
    \left(\frac{1}{2} \lambda-\left\|E_1\right\|_{\infty}\right)\left\|\left(\boldsymbol{\theta}-\boldsymbol{\theta}^*\right)_{S^{* c}}\right\|_1  &\leq  C_2c_{n,k}\delta_k^\beta \left\|\boldsymbol{\theta}-\boldsymbol{\theta}^*\right\|_2
-\rho_{-}\left\|\boldsymbol{\theta}-\boldsymbol{\theta}^*\right\|_2^2
+\left(\frac{3}{2} \lambda+\left\|E_1\right\|_{\infty}\right)\left\|\left(\boldsymbol{\theta}-\boldsymbol{\theta}^*\right)_{S^*}\right\|_1 \\
 &\leq  C_2c_{n,k}\delta_k^\beta \left\|\boldsymbol{\theta}-\boldsymbol{\theta}^*\right\|_2
-\rho_{-}\left\|\boldsymbol{\theta}-\boldsymbol{\theta}^*\right\|_2^2
+\left(\frac{3}{2} \lambda+\left\|E_1\right\|_{\infty}\right)  
\frac{1}{2\gamma}
\left\|\left(\boldsymbol{\theta}-\boldsymbol{\theta}^*\right)_{S^*}\right\|_1.
\end{align*}
By the definition of $\gamma$ we have
\[
\left(\frac{1}{2} \lambda-\left\|E_1\right\|_{\infty}\right)\left\|\left(\boldsymbol{\theta}-\boldsymbol{\theta}^*\right)_{S^{* c}}\right\|_1 \leq 2\left(C_2 c_{n,k}\delta_k^\beta\left\|\boldsymbol{\theta}-\boldsymbol{\theta}^*\right\|_2-\rho_{-}\left\|\boldsymbol{\theta}-\boldsymbol{\theta}^*\right\|_2^2\right) .
\]
Combine this with (\ref{norm1case2}) and note that $\|E_1\|_\infty \leq \lambda/8$ we have
\begin{equation}
\begin{aligned}
\left\|\boldsymbol{\theta}-\boldsymbol{\theta}^*\right\|_1 & \leq \frac{\left(2+\frac{1}{\gamma}\right)}{\frac{1}{2} \lambda-\left\|E_1\right\|_{\infty}} C_2 c_{n,k}\delta_k^\beta\left\|\boldsymbol{\theta}-\boldsymbol{\theta}^*\right\|_2 \\
& \leq \frac{56 C_2 c_{n,k}\delta_k^\beta}{3 \rho_{-}}\left(\frac{C_2 c_{n,k}\delta_k^\beta}{\lambda} \vee \sqrt{s}\right) .
\end{aligned}
\end{equation}
Conclude the two cases we have
\begin{equation}
\left\|\boldsymbol{\theta}-\boldsymbol{\theta}^*\right\|_1 \leq \frac{1}{\rho_{-}}\left(29 \sqrt{s} C_2 c_{n,k}\delta_k^\beta \vee 29 s \lambda \vee \frac{56 C_2^2 c_{n,k}\delta_k^{2 \beta}}{3 \lambda}\right).
\end{equation}

Finally, note that $\boldsymbol{\xi} \in \partial\|\boldsymbol{\theta}\|_1$ implies $\bxi^T(\btheta - \btheta^*) \geq \|\btheta\|_1 - \|\btheta^*\|_1$, and 
Assumption~\ref{asp6} gives 
$
\nabla R_{\delta_k}^{D_k}(\boldsymbol{\theta})^T\left(\boldsymbol{\theta}-\boldsymbol{\theta}^*\right) \geq 
R_{\delta_k}^{D_k}\left(\boldsymbol{\theta}\right) -R_{\delta_k}^{D_k}(\boldsymbol{\theta}^*)$, therefore 
\begin{align*}
f_\lambda(\boldsymbol{\theta})-f_\lambda\left(\boldsymbol{\theta}^*\right) 
&= R_{\delta_k}^{D_k}(\boldsymbol{\theta}) - R_{\delta_k}^{D_k}(\boldsymbol{\theta}^*)+\lambda\|\boldsymbol{\theta}\|_1 - \lambda\|\boldsymbol{\theta}^*\|_1\\
& \leq\left(\boldsymbol{\theta}-\boldsymbol{\theta}^*\right)^T\left(\nabla R_{\delta_k}^{D_k}(\boldsymbol{\theta})+\lambda \boldsymbol{\xi}\right) \\
& \leq \frac{1}{2} \lambda\left\|\boldsymbol{\theta}-\boldsymbol{\theta}^*\right\|_1 \\
& \leq \frac{1}{2 \rho_{-}}\left(29 \sqrt{s} C_2 c_{n,k}\delta_k^\beta \lambda \vee 29 s \lambda^2 \vee \frac{56 C_2^2 c_{n,k}\delta_k^{2 \beta}}{3}\right).
\end{align*}
Rewriting the constants we finish the proof.
\end{proof}

The next two lemmas characterize the properties of the iterates $\btheta_{k,t}^1, \cdots$ at stage $t$.

\begin{lemma}[Lemma 2 in \cite{feng2022nonregular}]\label{algolemma2}
    Assume the conditions of Proposition~\ref{prop2} and Assumption~\ref{asp6} hold. For $\lambda \geq \lambda_{k,tgt}$, if $\boldsymbol{\theta} \in \Omega,\left\|\boldsymbol{\theta}_{S^{* c}}\right\|_0 \leq \widetilde{s}$, and $\left\|\nabla R_{\delta_k}^{D_k}\left(\boldsymbol{\theta}^*\right)-
    \nabla  R_{\delta_k, \hat{\btheta}_{k-1}}(\btheta^*)
    \right\|_{\infty} \leq \lambda / 8$, and $ f_\lambda(\boldsymbol{\theta})-f_\lambda\left(\boldsymbol{\theta}^*\right) \leq \frac{\bar{C}_2}{2 \rho_{-}}\left( c_{n,k}^2\delta_k^{2 \beta} \vee \sqrt{s} c_{n,k}\delta_k^\beta \lambda \vee s \lambda^2\right)$,
    then we have 
    \[
   \left\|\boldsymbol{\theta}-\boldsymbol{\theta}^*\right\|_2 \leq \frac{\bar{C}_1^\prime}{\rho_{-}}\left( c_{n,k}\delta_k^\beta \vee
    s^{1/4}\sqrt{ c_{n,k}\delta_k \lambda}
    \vee
   \sqrt{s} \lambda\right),
   \]
   \[
   \left\|\boldsymbol{\theta}-\boldsymbol{\theta}^*\right\|_1 \leq \frac{2\bar{C}_2}{\rho_{-}}\left(\frac{c_{n,k}^2\delta_k^{2 \beta}}{\lambda} \vee \sqrt{s} c_{n,k}\delta_k^\beta \vee s \lambda\right),
   \]
   where $\bar{C}_1^\prime$ depends on $\bar{C}_1, \bar{C}_2$, which are constants defined in Lemma~\ref{algolemma1}.
\end{lemma}

\begin{lemma}[Lemma 3 in \cite{feng2022nonregular}]\label{lemma3}
    Under the same conditions of Lemma~\ref{algolemma2}, if we choose $\lambda_{k,tgt} = C\sqrt{\frac{c_{n,k}K\log d}{n\delta_k}}$ for some large enough constant $C$, and $\delta_k = c\left( \frac{s K \log d}{n c_{n,k}}\right)^{1/(2\beta+1)} $ for some constant $c>0$, then
    \begin{equation}
\left\|\mathcal{S}_{\lambda_{k,t} \eta}\left(\boldsymbol{\theta}, \mathbb{R}^d\right)_{S^{* c}}\right\|_0 \leq \widetilde{s},
\end{equation}
where $\widetilde{s}=8\left(\frac{\bar{C}_2}{\eta \rho_{-}}+\frac{2 \bar{C}_1^{\prime 2} \rho_{+}^2}{\rho_{-}^2}+2 C_2^2\right) \cdot s$, and $\bar{C_1}^\prime, \bar{C}_2, C_2$ are constants defined in Lemma~\ref{algolemma2}, Lemma~\ref{algolemma1} and Proposition~\ref{prop2}, respectively.
\end{lemma}

Lemma~\ref{algolemma2} and Lemma~\ref{lemma3} together imply that if the initialization at stage $t$ is sparse and satisfies $\omega_{\lambda_{k,t}}\left(\boldsymbol{\theta}_{k,t}^0\right) \leq \frac{1}{2} \lambda_{k,t}$, then the subsequent iterate should also retain sparsity and exhibit favorable properties. Furthermore, under Assumption~\ref{asp6},  Lemma~\ref{lemma7} ensures that $f_\lambda(\btheta_{k,t}^0), f_\lambda(\btheta_{k,t}^1), \cdots$ are decreasing.
Consequently, the conditions delineated in Lemma~\ref{algolemma2} and Lemma~\ref{lemma3} persist throughout the entire path $\btheta_{k,t}^1, \cdots$, guaranteeing both sparsity and convergence towards a local solution. This proposition is formally articulated as follows.

\begin{proposition}[Proposition 3 in \cite{feng2022nonregular}]\label{proposition3}
    Under Assumptions~\ref{asp3}-\ref{asp6}, suppose $\delta_k = c\left( \frac{s K \log d}{n c_{n,k}}\right)^{1/(2\beta+1)} $ for some constant $c>0$. If $\left\|\nabla R_{\delta_k}^{D_k}\left(\boldsymbol{\theta}^*\right)-
    \nabla  R_{\delta_k, \hat{\btheta}_{k-1}}(\btheta^*)
    \right\|_{\infty} \leq \lambda_{k,tgt} / 8 \leq \lambda_{k,t} / 8$ and at stage $t$, the proximal gradient method is initialized with $\boldsymbol{\theta}_{k,t}^0 \in \Omega$ satisfying 
\begin{equation}
\left\|\left(\boldsymbol{\theta}_{k,t}^0\right)_{S^{* c}}\right\|_0 \leq \widetilde{s} \text { and } \omega_{\lambda_{k,t}}\left(\boldsymbol{\theta}_{k,t}^0\right) \leq \frac{1}{2} \lambda_{k,t},
\end{equation}
then for $j = 1,2, \cdots$, we have
\begin{itemize}
    \item $\left\|\left(\boldsymbol{\theta}_{k,t}^j\right)_{S^{* c}}\right\|_0 \leq \widetilde{s}$
    \item  The sequence $\left\{\boldsymbol{\theta}_{k,t}^j\right\}_{j=0}^{\infty}$ converges towards a unique local solution $\hat{\btheta}_{k,t}$ satisfying the first-order optimality $\omega_{\lambda_{k,t}}\left(\hat{\boldsymbol{\theta}}_{k,t}\right) \leq 0$ with 
     $\left\|\left(\boldsymbol{\theta}_{k,t}^j\right)_{S^{* c}}\right\|_0 \leq \widetilde{s}$
     \item $f_{\lambda_{k,t}}\left(\boldsymbol{\theta}_{k,t}^j\right)-f_{\lambda_{k,t}}\left(\widehat{\boldsymbol{\theta}}_{k,t}\right) \leq\left(1-\frac{\eta \rho_{-}}{4}\right)^j\left(f_{\lambda_{k,t}}\left(\boldsymbol{\theta}_{k,t}^0\right)-f_{\lambda_{k,t}}\left(\widehat{\boldsymbol{\theta}}_{k,t}\right)\right)$.
\end{itemize}
\end{proposition}

\begin{lemma}(\cite{nesterov2013gradient}, Theorem 1)\label{lemma7}
    Under the same conditions of Lemma~\ref{lemma3}, we have
   \[
   f_\lambda\left(\mathcal{S}_{\lambda \eta}(\boldsymbol{\theta}, \Omega)\right) \leq f_\lambda(\boldsymbol{\theta})-\frac{1}{2 \eta}\left\|\mathcal{S}_{\lambda \eta}(\boldsymbol{\theta}, \Omega)-\boldsymbol{\theta}\right\|_2^2.
   \]
\end{lemma}

\begin{lemma}(\cite{nesterov2013gradient}, Corollary 1)\label{lemma8}
    Under the same conditions of Lemma~\ref{lemma3}, we have
   \[
   \omega_{\lambda_t}\left(\mathcal{S}_{\lambda \eta}(\boldsymbol{\theta}, \Omega)\right) \leq\left(\frac{1}{\eta}+\rho_{+}\right)\left\|\mathcal{S}_{\lambda \eta}(\boldsymbol{\theta}, \Omega)-\boldsymbol{\theta}\right\|_2.
   \]
\end{lemma}


\begin{lemma}[Lemma 10 in \cite{feng2022nonregular}] \label{lemma10}
    Suppose Assumption~\ref{asp6} holds. If $\lambda \geq \lambda_{k,t g t}$, $\|\nabla  R^{D_k}_{\delta_k}(\btheta^*)  - \nabla  R_{\delta_k, \hat{\btheta}_{k-1}}(\btheta^*) \|_\infty \leq 
\lambda_{k,t g t}/8$, $\omega_{\lambda}\left(\btheta\right) \leq \frac{1}{2} \lambda$, $\left\|\boldsymbol{\theta}_{S^*}\right\|_0 \leq \widetilde{s}$ and $\widehat{\boldsymbol{\theta}}_\lambda \in \Omega$ is a minimizer of $f_{\lambda}$ satisfying $\|(\widehat{\boldsymbol{\theta}}_\lambda)_{S^{* c}}\|_0 \leq \tilde{s}$, then we have
\[
f_\lambda(\boldsymbol{\theta})-f_\lambda\left(\widehat{\boldsymbol{\theta}}_\lambda\right) \leq \frac{\bar{C}_2}{\rho_{-}}\left(\delta^{2 \beta} \vee \sqrt{s} \delta^\beta \lambda \vee s \lambda^2\right).
\]
\end{lemma}

\begin{theorem}\label{paththm1}
    Assume the conditions of Proposition~\ref{prop2} and Assumption~\ref{asp6} hold. By choosing $\nu=0.25, \phi=0.9, \eta \leq \frac{1}{\rho_{+}}$,  $ \lambda_{k,tgt} = 8 C_1\sqrt{\frac{c_{n,k}K\log d}{n\delta_k}}$, where $C_1$ is defined in Proposition~\ref{E_1} and $\delta_k = c\left( \frac{s K \log d}{n c_{n,k}}\right)^{1/(2\beta+1)} $ for some constant $c>0$, with probability greater than $1-2d^{-1}$, the final approximate local solution $\tilde{\btheta}_{k,tgt}$ from the path-following algorithm satisfies
     \begin{equation*}
\begin{aligned}
&\left\|\tilde{\boldsymbol{\theta}}_{k,t g t}-\boldsymbol{\theta}^*\right\|_2 \lesssim \left(\frac{K s \log d}{nc_{n,k}}\right)^{\frac{\beta \vee 1}{2\beta+1}}, \\
&\left\|\tilde{\boldsymbol{\theta}}_{k, t g t}-\boldsymbol{\theta}^*\right\|_1 \lesssim \sqrt{s} \left(\frac{K s \log d}{nc_{n,k}}\right)^{\frac{\beta \vee 1}{2\beta+1}}.
\end{aligned}
\end{equation*}
\end{theorem}
\begin{proof}
    By Proposition~\ref{E_1} we have
    \[
\|\nabla  R^{D_k}_{\delta_k}(\btheta^*)  - \nabla  R_{\delta_k, \hat{\btheta}_{k-1}}(\btheta^*) \|_\infty \leq 
C_1\sqrt{\frac{c_{n,k}K\log d }{n\delta_k}},
\]
hence with probability greater than $1-2d^{-1}$, $\|\nabla  R^{D_k}_{\delta_k}(\btheta^*)  - \nabla  R_{\delta_k, \hat{\btheta}_{k-1}}(\btheta^*) \|_\infty \leq \lambda_{k,tgt}/8$ holds. We prove this theorem by induction. Note that the initialization in Algorithm~\ref{pathfollow} guarantees that
\begin{equation*}
\left\|\left(\boldsymbol{\theta}_{k,0}^0\right)_{S^{* c}}\right\|_0 \leq \widetilde{s} \text { and } \omega_{\lambda_{k,0}}\left(\boldsymbol{\theta}_{k,0}^0\right) \leq \frac{1}{2} \lambda_{k,0},
\end{equation*}
where $ \widetilde{s}$ is defined in Proposition~\ref{proposition3}. Suppose at stage $t=1, \cdots, T-1$, we have 
\begin{equation*}
\left\|\left(\boldsymbol{\theta}_{k,t}^0\right)_{S^{* c}}\right\|_0 \leq \widetilde{s} \text { and } \omega_{\lambda_{k,t}}\left(\boldsymbol{\theta}_{k,t}^0\right) \leq \frac{1}{2} \lambda_{k,t}.
\end{equation*}
By Proposition~\ref{proposition3}, we know $\left\|\left(\boldsymbol{\theta}_{k,t}^j\right)_{S^{* c}}\right\|_0 \leq \widetilde{s}$ for $j = 1,\cdots,$ which implies that $\left\|\left(\tilde{\btheta}_{k,t}\right)_{S^{* c}}\right\|_0 \leq \widetilde{s}$ if exists. Recall that at stage $t$, the stopping criteria requires $\omega_{\lambda_{k,t}}\left(\btheta\right) \leq \frac{1}{4} \lambda_{k,t}$, therefore it suffices to find $j$ such that $\omega_{\lambda_{k,t}}\left(\btheta_{k,t}^j\right) \leq \frac{1}{4} \lambda_{k,t}$ to finish stage $t$. By Lemma~\ref{lemma8}, we have
   \[
   \omega_{\lambda_{k,t}}\left( \btheta_{k,t}^j\right) \leq\left(\frac{1}{\eta}+\rho_{+}\right)\left\|\btheta_{k,t}^j-\btheta_{k,t}^{j-1}\right\|_2.
   \]
   Recall that $\hat{\boldsymbol{\theta}}_{k,t}$ is defined as (\ref{thetahat}),
from Lemma~\ref{lemma7} we obtain
\begin{equation}
\begin{aligned}
\frac{1}{2 \eta}\left\|\boldsymbol{\theta}_{k,t}^j-\boldsymbol{\theta}_{k,t}^{j-1}\right\|_2^2 & \leq f_{\lambda_{k,t}}\left(\boldsymbol{\theta}_{k,t}^{j-1}\right)-f_{\lambda_{k,t}}\left(\boldsymbol{\theta}_t^j\right) \\
& \leq f_{\lambda_{k,t}}\left(\boldsymbol{\theta}_{k,t}^{j-1}\right)-f_{\lambda_{k,t}}\left(\hat{\boldsymbol{\theta}}_{k,t}\right) \\
& \leq\left(1-\frac{\eta \rho_{-}}{4}\right)^{j-1}\left(f_{\lambda_{k,t}}\left(\boldsymbol{\theta}_{k,t}^0\right)-f_{\lambda_{k,t}}\left(\widehat{\boldsymbol{\theta}}_{k,t}\right)\right) \\
& \leq\left(1-\frac{\eta \rho_{-}}{4}\right)^{j-1} \frac{\bar{C}_2}{\rho_{-}}\left(\delta_k^{2 \beta} \vee \sqrt{s} \delta_k^\beta \lambda_{k,t} \vee s \lambda_{k,t}^2\right),
\end{aligned}
\end{equation}
where the second last inequality is from Proposition~\ref{proposition3} and the last inequality follows from Lemma~\ref{lemma10}. Now it suffices to guarantee that 
\[
\left(\frac{1}{\eta}+\rho_{+}\right) \sqrt{2 \eta\left(1-\frac{\eta \rho_{-}}{4}\right)^{j-1} \frac{\bar{C}_2}{\rho_{-}}\left(\delta^{2 \beta} \vee \sqrt{s} \delta^\beta \lambda_{k,t} \vee s \lambda_{k,t}^2\right)} \leq \frac{1}{4} \lambda_{k,t}.
\]
Recall that we choose $\delta_k = c\left( \frac{s K \log d}{n c_{n,k}}\right)^{1/(2\beta+1)} $, $\lambda_{k,t} > \lambda_{k,tgt}$ and $ \lambda_{k,tgt} = 8 C_1\sqrt{\frac{c_{n,k}K\log d}{n\delta_k}}$. With some algebra we can show that it suffices to guarantee that
\[
j \geq \log \left(\frac{32\left(\frac{1}{\eta}+\rho_{+}\right)^2 \eta \bar{C}_2 s}{\rho_{-}}\right) / \log \left(\frac{4}{4-\eta \rho_{-}}\right)+1,
\]
where the RHS is independent of $\lambda$. 
\end{proof}

\subsubsection{A Variation of Theorem \ref{Kiteration}}\label{iterationexample}
\begin{theorem}
Under Assumptions~\ref{asp3}-\ref{asp6}, assume that the number of iterations $K \geq 2$ and $\beta \geq 2$.
We set $N_K=N/2$,   $N_k = \left(t N_{k+1}\right)^{1/\beta}$ for $2 \leq k \leq K-1$, $N_1 = \left(\frac{tN_2}{(s\log d)^{\beta/(2\beta+1)}  C_N^{\beta}}\right)^{(2\beta+1)/(2\beta^2)}$, and
\[
\delta_1 = c_{1}\left(\frac{s \log d}{N_1}\right)^{1/(2\beta+1)}, \ \lambda_1 =c_{2} \sqrt{\frac{N_1 K^2\log d }{n^2\delta_1}},
\]
\[
 \delta_k = c_{1}\left(\frac{C_N s \log d}{N_k}\right)^{1/(2\beta)},
 \ \lambda_k = c_{2} \sqrt{\frac{N_k K^2\log d}{n^2 b_{k-1}\delta_k}},
 \ b_{k-1} = c_{3}\left(\frac{C_{N}^{2\beta+1}s \log d}{N_k}\right)^{1/(2\beta)},     2\leq k \leq K,
\]
for some constants $c_{1}, c_{2}, c_{3}$ and $c_{3} \geq c_{1}$.
If
\begin{equation}\label{prasp2}
        N \geq t^{\frac{1}{\beta-1}} 2^{K} \vee
\left(\frac{t^{\frac{\beta-1/\beta^{(K-2)}}{\beta-1}} }{(s\log d)^{\beta/(2\beta+1)}  C_N^{\beta}}\right)^{\frac{\beta^{K-2}(2\beta+1)}{2\beta^K-(2\beta+1)}}  2^{\frac{2(K-1)\beta^K - (2\beta+1)}{2\beta^K - (2\beta+1)}} ,
     \end{equation}
where $t = s^{2\beta-1}M_n^{2\beta}(\log d)^{\beta-1}/C_N^{\beta+1}$  and
\begin{equation}\label{Kprob<1}
  N \leq C\left(n/K \right)^{2\beta/(2\beta+1)}(s \log d)^{1/(2\beta+1)}C_N
\end{equation}
hold for some constant $C$, then  $\sum_{k=1}^K N_k \leq  N$ and with probability greater than $1-2K/d$, 
\[
   \left\|\hat{\btheta}_K-\boldsymbol{\theta}^*\right\|_2 \lesssim \left(\frac{C_N s \log d }{N}\right)^{1/2},\  \left\|\hat{\btheta}_K-\boldsymbol{\theta}^*\right\|_1 \lesssim \sqrt{s}\left(\frac{C_N s \log d }{N}\right)^{1/2}.
\]

\end{theorem}
\begin{proof}
First, let's consider the case when $k=1$. For each $(X_i, \bZ_i) \in D_1$, we have
 \[
 \PP(R_i=1) = c_{n,1}, \ \PP(R_i=0) = 1-c_{n,1}.
 \]
and $N_1= nc_{n,1}/K$. 
According to Theorem~\ref{rate_k}, by selecting $\delta_1 = c_{1}\left(\frac{s \log d}{N_1}\right)^{1/(2\beta+1)}$ and $\lambda_1 = c_{2} \sqrt{\frac{N_1 K^2\log d }{n^2\delta_1}}$,  with probability greater than $1-2d^{-1}$, we obtain
\begin{equation}\label{kiterations2}
    \|\hat{\btheta}_1 -\btheta^*\|_2 \lesssim  \left(\frac{s \log d}{N_1}\right)^{\beta/(2\beta+1)},
\end{equation}
\begin{equation}\label{kiterations1}
    \|\hat{\btheta}_1 -\btheta^*\|_1 \lesssim \sqrt{s} \left(\frac{s \log d}{N_1}\right)^{\beta/(2\beta+1)}.
\end{equation}

Next, we will establish the bound for $\|\hat{\btheta}_K -\btheta^*\|_1$. The result for  $\|\hat{\btheta}_K -\btheta^*\|_2$ follows similarly. 
Following the same proof as in Theorem~\ref{Kiteration}, we can show that $\PP\left((X,\bZ)\in S_k \right) \asymp b_{k-1} $.
Note that
$N_k =  n\EE(R_i)/K = n c_{n,k} \PP\left((X,\bZ)\in S_k \right) /K$. Applying Theorem~\ref{rate_k}, we select
\begin{equation}\label{deltak}
    \delta_k \asymp \left(\frac{K s \log d}{nc_{n,k}}\right)^{1/(2\beta+1)} \asymp \left(\frac{s \log d \PP\left((X,\bZ)\in S_k \right)}{N_k}\right)^{1/(2\beta+1)} \asymp \left(\frac{b_{k-1} s \log d }{N_k}\right)^{1/(2\beta+1)},
\end{equation}
and $\lambda_k = c_{k,2} \sqrt{\frac{N_k K^2\log d}{n^2 b_{k-1}\delta_k}}$,
to ensure, with probability greater than $1-2d^{-1}$
 \begin{equation}\label{k1norm}
     \left\|\hat{\btheta}_k-\boldsymbol{\theta}^*\right\|_1 \lesssim 
         \sqrt{s}\left(\frac{\PP\left((X,\bZ)\in S_k \right)s \log d }{N_k}\right)^{\beta/(2\beta+1)}
         \lesssim \sqrt{s}\left(\frac{b_{k-1} s \log d }{N_k}\right)^{\beta/(2\beta+1)}.
 \end{equation} 
 It's important to note that for Proposition~\ref{prop2} to hold which enables us to apply Theorem~\ref{rate_k}, we need to choose $b_{k-1}$ such that
\begin{equation} \label{twoconditions}
    b_{k-1} \geq C_{N} \delta_k  \text{\ and\ }
b_{k-1} \geq 2 \|\hat{\btheta}_{k-1}-\btheta^*\|_1 M_n.
\end{equation}
To meet the first condition and by (\ref{deltak}), we have $ \left(\frac{C_{N}^{2\beta+1} s\log d}{N_k}\right)^{1/(2\beta)} = O(b_{k-1})$. Therefore, we choose $b_{k-1} = c_{3 }\left(\frac{C_{N}^{2\beta+1}s\log d}{N_k}\right)^{1/(2\beta)}$ for some constant $c_{3}$. Substituting this into (\ref{k1norm}), we get
\begin{equation}\label{knormk}
     \left\|\hat{\btheta}_k-\boldsymbol{\theta}^*\right\|_1 \lesssim \sqrt{s}\left(\frac{C_{N} s \log d }{N_k}\right)^{1/2},
\end{equation}
which implies  $ \left\|\hat{\btheta}_K-\boldsymbol{\theta}^*\right\|_1 \lesssim \sqrt{s}\left(\frac{C_{N} s \log d }{N}\right)^{1/2}$, and $ \left\|\hat{\btheta}_K-\boldsymbol{\theta}^*\right\|_2 \lesssim \left(\frac{C_{N} s \log d }{N}\right)^{1/2}$ follows similarly.

Now, let's explore the assumption (\ref{Kprob<1}).
 By definition, we have $N_k = \sum_{(X_i, \bZ_i) \in D_k}\EE\left( R_i\right)= n\EE(R_i)/K = n c_{n,k} \PP\left((X,\bZ)\in S_k \right) /K$, where $0 < c_{n,k} < 1$ is defined as
\[
\PP(R_i=1 \mid Y_i, X_i, \bZ_i, \hat{\btheta}_{k-1})=
\PP(R_i=1 \mid  X_i, \bZ_i, \hat{\btheta}_{k-1})=c_{n,k}\cdot \ind\{(X_i,\bZ_i) \in S_k\}.
\]
To ensure that $0 < c_{n,k} \leq 1$, we require
\[
N_k K \leq n  \PP\left((X,\bZ)\in S_k \right), \ 2 \leq k \leq K,
\]
and it suffices to ensure that $N_k K \leq  C b_{k-1} n$, where $b_{k-1} = c_{3}\left(\frac{C_{N}^{2\beta+1}s \log d}{N_k}\right)^{1/(2\beta)}$. 
Subsequently, some calculation yields $ N_k \leq C\left(n/K \right)^{2\beta/(2\beta+1)}(s \log d)^{1/(2\beta+1)}C_N$ for some constant $C$, which is provided in  (\ref{Kprob<1}).

Next we check that the second condition in (\ref{twoconditions}), i.e., $b_{k-1} \geq 2 \|\hat{\btheta}_{k-1}-\btheta^*\|_1 M_n$ holds for the chosen $b_{k-1}$ and $N_1, \cdots, N_K$. Recall that we choose $N_K=N/2$,   $N_k = \left(t N_{k+1}\right)^{1/\beta}$ for $2 \leq k \leq K-1$, $N_1 = \left(\frac{tN_2}{(s\log d)^{\beta/(2\beta+1)}  C_N^{\beta}}\right)^{(2\beta+1)/2\beta^2}$, where $t = s^{2\beta-1}M_n^{2\beta}(\log d)^{\beta-1}/C_N^{\beta+1}$. When $k=2$, by (\ref{kiterations1}) it suffices to show that
\begin{equation}\label{nchoice1}
    N_2 = O\left( \frac{ C_{N}^{2\beta+1} \log d}{s^{\beta-1}M_n^{2\beta}}  \left(\frac{N_1}{s \log d}\right) ^{2\beta^2/(2\beta+1)} \right),
\end{equation}
and when $2 \leq k \leq K$, by (\ref{knormk}) we need to ensure that
\begin{equation}\label{nchoice2}
    N_k = O\left( \frac{C_{N}^{\beta+1}\log d}{ s^{\beta-1}M_n^{2\beta}}  \left(\frac{N_{k-1}}{s \log d}\right) ^{\beta} \right).
\end{equation}
It is easy to check that  the choice of $N_k$ above satisfies both (\ref{nchoice1}) and (\ref{nchoice2}).

Lastly, we demonstrate that the selected values $N_1, \cdots, N_K$ satisfy the condition $\sum_{k=1}^K N_k \leq N$. Our objective is to prove that for all $0 \leq j \leq K-2$,
\begin{equation}\label{ N_{K-j}}
    N_{K-j} \leq \frac{N}{2^{j+1}},
\end{equation}
and $N_1 \leq \frac{N}{2^{K-1}}$. These inequalities together imply that $\sum_{j=1}^K N_j \leq \frac{N}{2^{K-1}}+\sum_{j=1}^{K-1} \frac{N}{2^j} = N$.
Given  $N_K=N/2$,   $N_k = \left(t N_{k+1}\right)^{1/\beta}$ for $2 \leq k \leq K-1$, we can derive 
\[
N_{K-j} = t^{\sum_{i=1}^j 1/\beta^j}\left(\frac{N}{2}\right)^{1/\beta^j} = t^{\frac{1-1/\beta^j}{\beta-1}}\left(\frac{N}{2}\right)^{1/\beta^j}.
\]
Thus to show (\ref{ N_{K-j}}), it suffices to prove $t^{\frac{1-1/\beta^j}{\beta-1}}\left(\frac{N}{2}\right)^{1/\beta^j} \leq \frac{N}{2^{j+1}}$ for all $0 \leq j \leq K-2$. Some calculations yield
\[
N \geq t^{\frac{1}{\beta -1 }} 2^{\frac{(j+1)\beta^j-1}{\beta^j-1}}.
\]
Noting that $\frac{(j+1)\beta^j-1}{\beta^j-1} = j+1 + \frac{j}{\beta^j-1} \leq K$, we have $N \geq t^{\frac{1}{\beta -1 }} 2^K$ in (\ref{prasp2}), ensuring the result.
To show that $N_1 \leq \frac{N}{2^{K-1}}$, first note that $N_1 = \left(\frac{tN_2}{(s\log d)^{\beta/(2\beta+1)}  C_N^{\beta}}\right)^{(2\beta+1)/2\beta^2}$ and 
$N_2 = t^{\frac{1-1/\beta^{(K-2)}}{\beta-1}}\left(\frac{N}{2}\right)^{1/\beta^{(K-2)}}$.  We need to verify that
\[
\left(\frac{t^{\frac{\beta-1/\beta^{(K-2)}}{\beta-1}} \left(\frac{N}{2}\right)^{1/\beta^{(K-2)}}}{(s\log d)^{\beta/(2\beta+1)}  C_N^{\beta}}\right)^{\frac{2\beta+1}{2\beta^2}} = 
\left(\frac{t^{\frac{\beta-1/\beta^{(K-2)}}{\beta-1}} }{(s\log d)^{\beta/(2\beta+1)}  C_N^{\beta}}\right)^{\frac{2\beta+1}{2\beta^2}} \left(\frac{N}{2}\right)^{\frac{2\beta+1}{2\beta^K}}
\leq \frac{N}{2^{K-1}},
\]
which implies $N  \geq 2^{\frac{2(K-1)\beta^K - (2\beta+1)}{2\beta^K - (2\beta+1)}}
\left(\frac{t^{\frac{\beta-1/\beta^{(K-2)}}{\beta-1}} }{(s\log d)^{\beta/(2\beta+1)}  C_N^{\beta}}\right)^{\frac{\beta^{K-2}(2\beta+1)}{2\beta^K-(2\beta+1)}} $, and this is provided by the second part of (\ref{prasp2}).


\end{proof}

\subsection{Justification of Assumption~\ref{asp6}}\label{rscexample}
In this section, we verify that the restricted strong convexity and restricted smoothness condition hold w.h.p for the class of conditional mean model. For simplicity, we
assume $Y \sim \operatorname{Uniform}(\{-1,1\})$ and $\gamma(y)=1 / \mathbb{P}(Y=y)=2$ is known. The model is defined as 
\[
X=\boldsymbol{\theta}^{* T} \boldsymbol{Z}+\mu Y+u,
\]
where we assume $\bZ \in \RR^d$ is a zero-mean sub-Gaussian vector with parameter $\sigma^2$, $\boldsymbol{Z} \perp u$, and $\mu>0$. We first verify the RSC/RSM condition for the case when $0 < \beta \leq 1$, and the results for the case $\beta > 1$ are similar.

For $k=1$ when data are uniformly sampled from $D_1$, the proof in  \cite{feng2022nonregular} can be easily adapted to verify the RSC/RSM conditions for $0 < \beta \leq 1$. In the following, we focus on the proof of Assumption~\ref{asp6} on the set $\Omega = \{\btheta : \|\btheta - \btheta^*\|_2  \lesssim  \delta_k \}$ at the $k$th iteration for any $2 \leq k \leq K$. We apply a proper kernel function $K$ satisfying (i) $K$ has bounded support on $[-1,1]$, (ii) $\|K\|_{\infty},\left\|K^{\prime}\right\|_{\infty}, \left\|K^{\prime\prime}\right\|_{\infty} \text { and }\tilde{K}_\beta=-\int K^{\prime}(t) t^\beta d t>0$ are bounded above by universal constants, and (iii) $K'(t)>0$ for $t<0$ and $K(\cdot)$ is symmetric. To verify that Assumption~\ref{asp6} holds w.h.p, it suffices to show that the following sparse eigenvalue condition holds w.h.p:
\begin{align}\label{rsccondi: beta<1}
    & \rho_{\max }=\sup \left\{\boldsymbol{v}^T \nabla^2 R_{\delta_k}^{D_k}(\boldsymbol{\theta}) \boldsymbol{v}:\|\boldsymbol{v}\|_2=1,\|\boldsymbol{v}\|_0 \leq C s, \boldsymbol{\theta} \in \Omega,\|\boldsymbol{\theta}\|_0 \leq C s\right\}< C_1 \delta_k^{\beta-1} c_{n,k}, \\
& \rho_{\min }=\inf \left\{\boldsymbol{v}^T \nabla^2 R_{\delta_k}^{D_k}(\boldsymbol{\theta}) \boldsymbol{v}:\|\boldsymbol{v}\|_2=1,\|\boldsymbol{v}\|_0 \leq C s, \boldsymbol{\theta} \in \Omega,\|\boldsymbol{\theta}\|_0 \leq C s\right\}> C_2 \delta_k^{\beta-1} c_{n,k},
\end{align}
for some constants $C_1, C_2 >0$, where $0 < \beta \leq 1$. We let $\bSigma_{\bZ} = \mathbf{Cov}(\bZ)$ and $g(\cdot)$ be the p.d.f of the error $u$.
Denote $\nabla^2 R_{\delta_k}(\boldsymbol{\theta}) = \EE\left[\nabla^2 R_{\delta_k}^{D_k}(\boldsymbol{\theta})|\hat\btheta_{k-1}\right]$ as the population Hessian.
\begin{proposition}\label{RE1: case 2}
    Under the setup above, suppose $g(\cdot)$ is an even function and satisfies the {H\"older condition $|g(y) - g(x)| \leq L_2 |x-y|^\beta$ for any $-\frac{3\mu}{2} \leq x \leq y \leq -\frac{\mu}{2}$, and $\sign(t)(g(-\mu+t) - g(-\mu)) \geq L_1 |t|^\beta$}. If we assume $\log d \gtrsim \log n$ and
    \begin{equation}\label{condition: empr cond for RE}
    \sqrt{\frac{c_{n,k}s \log\left(\frac{(ds\log{d)^2}}{c_{n,k}\delta_k}\right)}{n_k \delta_k^3}} + \frac{s\log(d \vee n_k)}{n_k\delta_k^2} \log\left(\frac{(ds\log{d)^2}}{c_{n,k}\delta_k}\right) \lesssim \delta_k^{\beta-1} c_{n,k},
    \end{equation}
    then with probability greater than $1-c(\log n)^{-1}$, it holds that\[
   \rho_{\min } \geq L_1 \tilde{K}_\beta \delta_k^{\beta-1} c_{n,k} \lambda_{\min }\left(\boldsymbol{\Sigma}_{\boldsymbol{Z}}\right)/2 ,~~ \ 
    \rho_{\max } \leq 2L_2 \tilde{K}_\beta \delta_k^{\beta-1} c_{n,k} \lambda_{\max }\left(\boldsymbol{\Sigma}_{\boldsymbol{Z}}\right).
    \]
\end{proposition}
\begin{proof}
Write    
\begin{equation}\label{eq: decompose}
     \boldsymbol{v}^T \nabla^2 R^{D_k}_{\delta_k}(\boldsymbol{\theta}) \boldsymbol{v} = 
    \bv^T \nabla^2 R_{\delta_k}(\boldsymbol{\theta}^*)\bv+
    \bv^T(\nabla^2 R^{D_k}_{\delta_k}(\boldsymbol{\theta})-\nabla^2 R_{\delta_k}(\boldsymbol{\theta}))\bv+
     \bv^T(\nabla^2 R_{\delta_k}(\boldsymbol{\theta})-\nabla^2 R_{\delta_k}(\boldsymbol{\theta}^*))\bv.
\end{equation}
By Lemma \ref{truetheta_hessian: case 2} we have 
\begin{align}\label{inequa1: beta<1}
        \boldsymbol{v}^T \nabla^2 R_{\delta_k}(\boldsymbol{\theta^*}) \boldsymbol{v} \leq L_2 \tilde{K}_\beta \delta_k^{\beta-1} c_{n,k} \lambda_{\max }\left(\boldsymbol{\Sigma}_{\boldsymbol{Z}}\right), \ \boldsymbol{v}^T \nabla^2 R_{\delta_k}(\boldsymbol{\theta^*}) \boldsymbol{v} \geq L_1 \tilde{K}_\beta \delta_k^{\beta-1} c_{n,k} \lambda_{\min }\left(\boldsymbol{\Sigma}_{\boldsymbol{Z}}\right).
\end{align}
By Lemma \ref{lem: stochastic term} we have with probability greater than $1-c(\log n)^{-1}$,
\[
\sup_{\norm{\bv}_2 = 1, \norm{\bv}_0 \leq Cs}\sup_{\btheta\in\Omega, \norm{\btheta}_0 \leq Cs} \bigl| \bv^T(\nabla^2 R_{\delta_k}^n(\btheta) - \nabla^2 R_{\delta_k}(\btheta)\bv \bigr| \lesssim \delta_k^{\beta-1} c_{n,k}.
\]
Then it remains to show that $|\bv^T(\nabla^2 R_{\delta_k}(\boldsymbol{\theta})-\nabla^2 R_{\delta_k}(\boldsymbol{\theta}^*))\bv| \lesssim \delta_k^{\beta-1} c_{n,k}$ on the set $\Omega = \{\btheta : \|\btheta - \btheta^*\|_2  \lesssim  \delta_k \}$. Given $\hat \btheta_{k-1}$, the population Hessian difference is
\begin{equation}
\begin{aligned}
    &\bv^T(\nabla^2R_{\delta_k}(\btheta) - \nabla^2R_{\delta_k}(\btheta^*))\bv  \\
    &= \EE\left[\gamma(Y) \frac{(\bv^T\bZ)^2}{\delta_k^2}\left(K'\left(\frac{Y(X - \btheta^T \bZ)}{\delta_k} \right) - K'\left(\frac{Y(X - \btheta^{*T} \bZ)}{\delta_k} \right) \right)R \mid \hat \btheta_{k-1} \right]\\
    &= \EE\left[\frac{(\bv^T\bZ)^2}{\delta_k^2}\left(K'\left(\frac{X - \btheta^T \bZ}{\delta_k} \right) - K'\left(\frac{X - \btheta^{*T} \bZ}{\delta_k} \right) \right)R \mid \hat \btheta_{k-1}, Y = 1 \right]\\
    &~~~~+\EE\left[\frac{(\bv^T\bZ)^2}{\delta_k^2}\left(K'\left(\frac{-X + \btheta^T \bZ}{\delta_k} \right) - K'\left(\frac{-X + \btheta^{*T} \bZ}{\delta_k} \right) \right)R \mid \hat \btheta_{k-1}, Y = -1 \right].
\end{aligned}
\end{equation}
We bound the first term and similar for the second term. Denote $\frac{-b_{k-1}\|\hat{\bw}_{k-1}\|_2 +(\hat{\btheta}_{k-1}-\btheta^*)^T\bz}{\delta_k}:=S_{u}^{*-}$ and $\frac{b_{k-1}\|\hat{\bw}_{k-1}\|_2 +(\hat{\btheta}_{k-1}-\btheta^*)^T\bz}{\delta_k}:=S_u^{*+}$.
Recall that in Proposition~\ref{prop2}, we showed that $S_u^{*+} \geq \frac{b_{k-1}\|\hat{\bw}_{k-1}\|_2 -\|\hat{\btheta}_{k-1}-\btheta^*\|_1 M_n}{\delta_k}\geq C_{N}/2$ and $S_u^{*-} \leq -C_{N}/2.$ for some large constant $C_N$. Then $\supp\{K\} = [-1,1] \subseteq [S_u^{*-},S_u^{*+}]$. And in the conditional mean model, we have $f(x \mid \bz, Y=y) = g(x-\btheta^{*T}\bz-\mu y)$. 
\begin{equation}\label{eq: hessian diff}
\begin{aligned}
    \text{First term} &= \EE\left[\EE\left[\frac{(\bv^T\bZ)^2}{\delta_k^2} \left(K'\left(\frac{X - \btheta^T \bZ}{\delta_k} \right) - K'\left(\frac{X - \btheta^{*T} \bZ}{\delta_k} \right)\right)R \mid \hat \btheta_{k-1}, Y=1, X, \bZ \right] \mid \hat \btheta_{k-1}, Y=1 \right]\\
    &= \EE\left[\frac{(\bv^T\bZ)^2}{\delta_k^2} \left(K'\left(\frac{X - \btheta^T \bZ}{\delta_k} \right) - K'\left(\frac{X - \btheta^{*T} \bZ}{\delta_k} \right)\right) c_{n,k} \ind\{(X, \bZ) \in S_k\} \mid \hat \btheta_{k-1}, Y=1\right] \\
    & = \frac{c_{n,k}}{\delta_k^2} \int (\bv^T \bZ)^2 \int_{-b_{k-1}\norm{w_{k-1}}_2 + \hat \btheta_{k-1}^T \bz}^{b_{k-1}\norm{w_{k-1}}_2 + \hat \btheta_{k-1}^T \bz} \left[K'\left(\frac{X - \btheta^T \bZ}{\delta_k} \right) - K'\left(\frac{X - \btheta^{*T} \bZ}{\delta_k} \right) \right] \\
    &\hspace{6cm} g(x - \btheta^{*T}\bz - \mu) \, dx f(\bZ) \,dz\\
    &= \frac{c_{n,k}}{\delta_k^2} \int (\bv^T \bZ)^2 \int_{S_u^{*-}}^{S_u^{*+}} \left[K'\left(\frac{(\btheta - \btheta^*)^T\bZ}{\delta_k} + u \right) - K'\left(u \right) \right] g(\delta_k u - \mu) \, du f(\bZ) \,dz\\
    &= \frac{c_{n,k}}{\delta_k^2} \int (\bv^T \bZ)^2 \int_\R \left[K'\left(\frac{(\btheta - \btheta^*)^T\bZ}{\delta_k} + u \right) - K'\left(u \right) \right] (g(\delta_k u - \mu) - g(-\mu)) \, du f(\bZ) \,dz.
\end{aligned}
\end{equation}
Note that
\[
\left|K'\left(\frac{(\btheta - \btheta^*)^T\bZ}{\delta_k} + u \right) - K'\left(u \right) \right| = \left|\int_u^{u+\frac{(\btheta^*-\btheta)^T\bZ}{\delta_k}} K''(t) \, dt\right| \leq \norm{K''}_\infty \frac{(\btheta^*-\btheta)^T\bZ}{\delta_k}
\]
and
\[
|g(\delta_k u - \mu) - g(-\mu)| \leq L_2 |\delta_k u|^{\beta} \leq C |\delta_k|^\beta.
\]
We have
\begin{equation}
\begin{aligned}
    (\ref{eq: hessian diff}) &\lesssim c_{n,k} \delta_k^{\beta-2} \int (\bv^T \bZ)^2 (\btheta - \btheta^*)^T \bZ f(\bZ) \, d\bZ\\
    &\leq c_{n,k} \delta_k^{\beta-2} \sqrt{\EE[(\bv^T \bZ)^4]} \sqrt{\EE[((\btheta - \btheta^*)^T \bZ)^2]}\\
    &\leq c_{n,k} \delta_k^{\beta-2} \sqrt{\EE[(\bv^T \bZ)^4]} \sqrt{M_1} \norm{\btheta - \btheta^*}_2\\
    &\lesssim c_{n,k} \delta_k^{\beta-2} \delta_k = c_{n,k} \delta_k^{\beta-1}.
\end{aligned}
\end{equation}
The last inequality is because $\bZ$ is a sub-Gaussian vector with a bounded sub-Gaussian norm and thus $\bv^T \bZ$ is also a sub-Gaussian random variable with bounded moment. Combining the bound for the three terms in (\ref{eq: decompose}), we finish the proof.
\end{proof}
\begin{remark}
    If we plug in $c_{n,k} = \frac{Nk K}{n \PP((X,Z) \in S_k)} \asymp \frac{N_k K}{n b_{k-1}}$ and $\delta_k \asymp (\frac{s \log d}{N_k} b_{k-1})^{\frac{1}{2\beta+1}}$, condition (\ref{condition: empr cond for RE}) is equivalent to $\frac{b_{k-1}s(\log d)^{3+\beta^{-1}}}{N_k}=O(1)$.
\end{remark}

\begin{lemma}\label{truetheta_hessian: case 2}
   Suppose $\tilde{K}_\beta=-\int K^{\prime}(t) t^\beta d t>0$, $g(\cdot)$ is an even function and satisfies the {H\"older condition $|g(y) - g(x)| \leq L_2 |x-y|^\beta$ for any $-\frac{3\mu}{2} \leq x \leq y \leq -\frac{\mu}{2}$, and $\sign(t)(g(-\mu+t) - g(-\mu)) \geq L_1 |t|^\beta$}, then for all unit vector $\bv \in \RR^d$ we have
    \[
    \boldsymbol{v}^T \nabla^2 R_{\delta_k}(\boldsymbol{\theta^*}) \boldsymbol{v} \leq L_2 \tilde{K}_\beta \delta_k^{\beta-1} c_{n,k} \lambda_{\max }\left(\boldsymbol{\Sigma}_{\boldsymbol{Z}}\right), \ \boldsymbol{v}^T \nabla^2 R_{\delta_k}(\boldsymbol{\theta^*}) \boldsymbol{v} \geq L_1 \tilde{K}_\beta \delta_k^{\beta-1} c_{n,k} \lambda_{\min }\left(\boldsymbol{\Sigma}_{\boldsymbol{Z}}\right),
    \]
    where $\|\boldsymbol{v}\|_2=1$ and $\|\boldsymbol{v}\|_0 \leq C s$.
\end{lemma}
\begin{proof}
    Denote $\frac{-b_{k-1}\|\hat{\bw}_{k-1}\|_2 +(\hat{\btheta}_{k-1}-\btheta^*)^T\bz}{\delta_k}:=S_{u}^{*-}$ and $\frac{b_{k-1}\|\hat{\bw}_{k-1}\|_2 +(\hat{\btheta}_{k-1}-\btheta^*)^T\bz}{\delta_k}:=S_u^{*+}$.
Recall that in Proposition~\ref{prop2} we showed that
\[
S_u^{*+} \geq \frac{b_{k-1}\|\hat{\bw}_{k-1}\|_2 -\|\hat{\btheta}_{k-1}-\btheta^*\|_1 M_n}{\delta_k}\geq C_{N}/2,
\]
and similarly,
\[
S_u^{*-} \leq -C_{N}/2.
\]
By definition 
\[
    \nabla^2  R^{D_k}_{\delta_k}(\btheta) = - \frac{K}{n}\sum_{(x_i, \bz_i) \in D_k} \gamma(y_i) \frac{y_i^2\bz_i\bz_i^T}{\delta_k^2}K^\prime (\frac{y_i(x_i-\btheta^{T}\bz_i)}{\delta_k})R_i.
    \]
Then,
\begin{align*}
    &\boldsymbol{v}^T \nabla^2 R_{\delta_k}(\boldsymbol{\theta}) \boldsymbol{v}\\
    =&-\frac{c_{n,k}}{\delta_k^2} \int_{\bz} \int_{-b_{k-1}\|\hat{\bw}_{k-1}\|_2 +\hat{\btheta}_{k-1}^T\bz}^{b_{k-1}\|\hat{\bw}_{k-1}\|_2+ \hat{\btheta}_{k-1}^T\bz }  \left(\boldsymbol{v}^T \boldsymbol{z}\right)^2 \\
    & \cdot \left(K^{\prime}\left(\frac{x-\boldsymbol{\theta}^T \boldsymbol{z}}{\delta_k}\right)g\left(x-\boldsymbol{\theta}^{* T} \boldsymbol{z}-\mu\right)+K^{\prime}\left(\frac{-(x-\boldsymbol{\theta}^T \boldsymbol{z})}{\delta_k}\right)g\left(x-\boldsymbol{\theta}^{* T} \boldsymbol{z}+\mu\right)\right) f(\boldsymbol{z}) d x d \boldsymbol{z}\\
    =& - \frac{c_{n,k}}{\delta_k} \int_{\bz} (\bv^T\bz)^2 \int_{S_u^-}^{S_u^+} 
    \left(K^\prime(t)g\left(\delta_k t+\boldsymbol{\theta}^T \boldsymbol{z}-\boldsymbol{\theta}^{* T} \boldsymbol{z}-\mu\right)+ K^\prime(-t)g\left(\delta_k t+ \boldsymbol{\theta}^T \boldsymbol{z}-\boldsymbol{\theta}^{* T} \boldsymbol{z}+\mu\right)\right) f(\boldsymbol{z}) dt d\bz.
\end{align*}
We choose $K(t)$ such that $K^\prime(t)$ is an odd function, then we have 
\begin{align}\label{asp3.5jus: case 2}
    &\boldsymbol{v}^T \nabla^2 R_{\delta_k}(\boldsymbol{\theta^*}) \boldsymbol{v}\\
    =& - \frac{c_{n,k}}{\delta_k}\int_{\bz} (\bv^T\bz)^2 \int_{S_u^{*-}}^{S_u^{*+}} K^\prime(t) \left(
    g \left( \delta_k t -\mu\right) -
    g \left(-\delta_k t -\mu\right)
    \right) dt f(\bz)d\bz.
\end{align}
By the choice of the kernel function we can ensure that
\[
\delta_k t -\mu,  -\delta_k t -\mu \in [\frac{-3\mu}{2}, \frac{-\mu}{2}].
\]
Note that $\tilde{K}_\beta=-\int K^{\prime}(t) t^\beta d t>0$ and $g(\cdot)$ satisfies the H\"older condition, then 
$$L_1 \tilde{K}_\beta \delta_k^{\beta-1} c_{n,k} \lambda_{\min }\left(\boldsymbol{\Sigma}_{\boldsymbol{Z}}\right) \leq \boldsymbol{v}^T \nabla^2 R_{\delta_k}(\boldsymbol{\theta^*}) \boldsymbol{v} \leq L_2 \tilde{K}_\beta \delta_k^{\beta-1} c_{n,k} \lambda_{\max }\left(\boldsymbol{\Sigma}_{\boldsymbol{Z}}\right).
$$
\end{proof}

\begin{lemma}\label{lem: stochastic term}
    If we assume $\log d \gtrsim \log n$ and
    \begin{equation}
    \sqrt{\frac{c_{n,k}s \log\left(\frac{(dslog{d)^2}}{c_{n,k}\delta_k}\right)}{n_k \delta_k^3}} + \frac{s\log(d \vee n_k)}{n_k\delta_k^2} \log\left(\frac{(dslog{d)^2}}{c_{n,k}\delta_k}\right) \lesssim \delta_k^{\beta-1} c_{n,k},
    \end{equation}
    then with probability greater than $1 - c(\log n_k)^{-1}$, it holds that
    \begin{equation}
        \sup_{\norm{\bv}_2 = 1, \norm{\bv}_0 \leq Cs}\sup_{\btheta\in\Omega, \norm{\btheta}_0 \leq Cs} \bigl| \bv^T(\nabla^2 R_{\delta_k}^n(\btheta) - \nabla^2 R_{\delta_k}(\btheta)\bv \bigr| \lesssim \delta_k^{\beta-1} c_{n,k},
    \end{equation}
    where $n_k = \frac{n}{K}$, and $c$ is some positive constant.
\end{lemma}
\begin{proof}
    Consider the class of functions
\begin{equation}\label{eq:F-class}
\mathcal{F}
=
\Bigl\{
f_{\bv,\btheta}(x,y,\bz, R)
=
-\,y\,(\bv^T \bz)^2\,
K'\!\left(\frac{x-\btheta^T \bz}{\delta_k}\right)R
:
\bv,\btheta\in\R^d, \norm{\bv}_2=1, \btheta \in \Omega, \norm{\bv}_0, \norm{\btheta}_0 \le C s
\Bigr\}.
\end{equation}
Our goal reduces to controlling the empirical process
\begin{align}\label{eq:Gn}
\|G_n\|_{\mathcal{F}}
&=
\sup_{f\in\mathcal{F}}
\Biggl|
n_k^{-1/2}
\sum_{i=1}^{n_k}
\Bigl(
f(x_i,y_i,\bz_i, R_i)-\mathbb{E}\,f(X,Y,\bZ, R)
\Bigr)
\Biggr|
\\
&=
n_k^{1/2}\delta_k^2
\sup_{\norm{\bv}_2 = 1, \norm{\bv}_0 \leq Cs}\sup_{\btheta\in\Omega, \norm{\btheta}_0 \leq Cs} \bigl| \bv^T(\nabla^2 R_{\delta_k}^n(\btheta) - \nabla^2 R_{\delta_k}(\btheta)\bv \bigr|.
\end{align}
To achieve this, we will apply the maximal inequality \ref{lem: Maximal ineq} (due to \cite{Kengo2014MaxIneq}).
In order to control the metric entropy, we introduce
\begin{equation}
\begin{aligned}
&\mathcal F_1(q)=\{(x,y,\bz,R)\mapsto (\bv^T \bz)^q: \norm{\bv}_2=1, \norm{\bv}_0 \leq Cs\}, \\
&\mathcal F_2=\{(x,y,\bz,R)\mapsto K'((x-\btheta^T \bz)/\delta_k):\ \btheta\in\Omega,\norm{\btheta}_0 \leq Cs\},\\
&\mathcal F_3=\{(x,y,\bz,R)\mapsto yR\}.
\end{aligned}
\end{equation}
For $\mathcal{F}_3$, it holds trivially that
\[
\log\sup_Q \mathcal N(\epsilon,\mathcal F_3,\|\cdot\|_{Q,2})\lesssim s\log(e/\epsilon).
\]
For $\mathcal F_1(1)$, we know it is contained in the union of at most $\binom{d}{Cs}$ VC-subgraph classes of functions with VC indices bounded by $C's$, and thus for the envelope $F_1=\sup_{\norm{\bv}_2=1,\norm{\bv}_0\le Cs}\ \norm{\bv^T \bz}_2$, we have
\[
\log\sup_Q \mathcal N\bigl(\epsilon\|F_1\|_{Q,2},\ \mathcal F_1(1),\|\cdot\|_{Q,2}\bigr)
\ \lesssim\ s\log d+s\log(e/\epsilon).
\]
For $\mathcal F_2$, notice that $\mathcal F_2' := \bigl\{(x,y,\bz,R) \mapsto x-\btheta^T \bz\bigr\}$ is contained in the union of at most $\binom{d}{Cs}$ VC-subgraph classes of functions with VC indices bounded by $C's$. By assumption $K'$ is of bounded variation, i.e., it can be written as the difference of two bounded nondecreasing functions. By Lemma 2.6.18 of \cite{vaart1996weak} the VC-subgraph property is preserved,
and we obtain
\[
\log\sup_Q \mathcal N\bigl(\epsilon\|K'\|_\infty,\ \mathcal F_2,\ \|\cdot\|_{Q,2}\bigr)
\ \lesssim\ s\log d + s\log(e/\epsilon).
\]

\noindent Since $\mathcal F$ is contained in the class of pointwise product $\mathcal F_1(2)\cdot\mathcal F_2\cdot\mathcal F_3$,
applying Lemma A.6 of \cite{Kengo2014MaxIneq}, we obtain that for the choice of envelope
\[
F=F_1^2\,\|K'\|_\infty
=\sup_{\norm{\bv}_2=1, \norm{\bv}_0\le Cs} |\bv^T \bz|^2\,\|K'\|_\infty,
\]
it holds that
\[
\log\sup_Q \mathcal N\bigl(\epsilon\|F\|_{Q,2},\ \mathcal F,\ \|\cdot\|_{Q,2}\bigr)
\ \lesssim\ s\log d + s\log(e/\epsilon).
\]
Thus, we could choose $v = s, a = d^{C1}$ for some $C_1$ large enough in Lemma \ref{lem: Maximal ineq}. Next, we calculate bound for $\rho^2$, $\norm{F}_{P,2}$ and $\norm{M}_{P,2}$.
\begin{equation}
\begin{aligned}
    \E(f^2 \mid \hat \btheta_{k-1}) 
    &= \E\left[(\bv^T\bz)^4 K'^2\left(\frac{x-\btheta^T\bz}{\delta_k}\right)R \mid \hat \btheta_{k-1} \right]\\
    &= \E\left[ \E\left((\bv^T\bz)^4 K'^2\left(\frac{x-\btheta^T\bz}{\delta_k}\right)R \mid \hat \btheta_{k-1}, x, \bz \right) \mid \hat \btheta_{k-1} \right]\\
    &= \E\left[(\bv^T\bz)^4 K'^2\left(\frac{x-\btheta^T\bz}{\delta_k}\right) c_{n,k} \ind\{(x,\bz) \in S_k\} \mid \hat \theta_{k-1} \right]\\
    & \leq c_{n,k} \int (\bv^T\bz)^4 \int K'^2\left(\frac{x - \btheta^T\bz}{\delta_k}\right)f(x\mid \bz,y)\, dx f(\bz\mid y)\, dz\\
    &= c_{n,k} \delta_k \int (\bv^T\bz)^4 \int K'^2\left(u\right)f(\btheta^T\bz + \delta_k u\mid \bz,y)\, du f(\bz\mid y)\, dz\\
    &\leq c_{n,k} \delta_k \norm{K'}_\infty^2 p_{max} \E(\bv^T\bz)^4 \leq c_{n,k} \delta_k.
\end{aligned}
\end{equation}
Hence we could define $\rho^2 = C_2c_{n,k} \delta_k$ for some positive constant $C_2$.
Now for any sparse index set $S\subset\{1,\dots,d\}$, $|S|\le Cs$, consider the sphere $\mathbb S_S=\bigl\{v:\ \|v\|_2=1,\ \mathrm{supp}(v)\subset S\bigr\}$.
Let $\mathcal N^{(\mathbb S_S)}_{1/6}\subset\mathbb S_S$ denote a $1/6$-net of $\mathbb S_S$ such that $\forall \bv\in\mathbb S_S$,
$\exists \bu\in\mathcal N^{(\mathbb S_S)}_{1/6}$ with $\norm{\bu-\bv} \le 1/6$. We know the cardinality of $\mathcal N^{(\mathbb S_S)}_{1/6}$ is no
larger than $18^{Cs}$ and
\[
\sup_{\bv\in\mathbb S_S} |\bv^T \bZ|^4\ \le\ c \max_{\bv\in\mathcal N^{(S)}_{1/6}} |\bv^T \bZ|^4
\]
for some absolute constant $c$. Since there are no more than $\binom{d}{Cs}$ such sparse index sets, we
know the cardinality of
\[
\bar{\mathcal S}
=\bigcup_{S\subset\{1,\dots,d\},\ |S|\le Cs}\ \mathcal N^{(\mathbb S_S)}_{1/6}
\]
is no larger than $(18d)^{Cs}$. Therefore we obtain
\[
\|F\|_{P,2}^2
=\ \|K'\|_\infty^2\,\mathbb E_P\!\left[\sup_{\|\bv\|_2=1,\ \|\bv\|_0\le Cs} |\bv^T \bZ|^4\right]
\ \lesssim\ \|K'\|_\infty^2\,\mathbb E_P\!\left[\max_{\bv\in\bar{\mathcal S}} |\bv^T \bZ|^4\right]
\ \lesssim\ \|K'\|_\infty^2\,\sigma^4\,(s\log d)^2,
\]
where the last inequality follows from the proof of Lemma 14.12 in \cite{buhlmann2011statistics}. With a similar derivation, we can
also show that
\[
\|M\|_{P,2}^2
=\ \bigl\|\max_{i\le n} F(x_i,y_i,\bz_i,R_i)\bigr\|_{P,2}^2
\ \lesssim\ \|K'\|_\infty^2\,\sigma^4\,(s\log(d\vee n))^2.
\]
Plugging the above results into Lemma \ref{lem: Maximal ineq}, we obtain when $a = d^{C_1} \geq n$ or equivalently $\log d \gtrsim \log n$, with probability greater than $1- c(\log n)^{-1}$,
\begin{equation}
\begin{aligned}
    &\sup_{\norm{\bv}_2 = 1, \norm{\bv}_0 \leq Cs}\sup_{\btheta\in\Omega, \norm{\btheta}_0 \leq Cs} \bigl| \bv^T(\nabla^2 R_{\delta_k}^n(\btheta) - \nabla^2 R_{\delta_k}(\btheta)\bv \bigr|\\
    &\lesssim \sqrt{\frac{c_{n,k}s \log\left(\frac{(dslog{d)^2}}{c_{n,k}\delta_k}\right)}{n_k \delta_k^3}} + \frac{s\log(d \vee n_k)}{n_k\delta_k^2} \log\left(\frac{(dslog{d)^2}}{c_{n,k}\delta_k}\right)\\
    &\lesssim \delta_k^{\beta-1} c_{n,k}.
\end{aligned}
\end{equation}
\end{proof}

\begin{lemma}\label{lem: Maximal ineq}
Under the setup above, suppose $F \ge \sup_{f\in\mathcal{F}} |f|$ is a measurable envelope with $\|F\|_{P,2}<\infty$.
Let $M=\max_{i\le n_k} F(x_i,y_i,\bz_i, R_i)$ and let $\rho^2$ be any positive number such that
$\sup_{f\in\mathcal{F}}\|f\|_{P,2}^2 \le \rho^2 \le \|F\|_{P,2}^2$.
Suppose there exist $a\ge e$ and $\nu\ge 1$ such that
\[
\log \sup_Q N\!\bigl(\epsilon\|F\|_{Q,2},\,\mathcal{F},\,\|\cdot\|_{Q,2}\bigr)
\;\le\;
\nu \log(a/\epsilon),\qquad \epsilon\in(0,1].
\]
Then
\[
\mathbb{E}_P\bigl[\|G_n\|_{\mathcal{F}}\bigr]
\;\le\;
K\Bigl(
\rho\,\sqrt{\nu\log(a\|F\|_{P,2}/\rho)}
\;+\;
\frac{\nu\|M\|_{P,2}}{\sqrt{n_k}}\,
\log(a\|F\|_{P,2}/\rho)
\Bigr).
\]
Moreover, when $a\ge n$, with probability greater than $1-c(\log n)^{-1}$, 
\[
\|G_n\|_{\mathcal F}\ \le\
K(q,c)\Bigl(
\rho\,\sqrt{\nu\log(a\|F\|_{P,2}/\rho)}
\;+\;
\frac{\nu\|M\|_{P,2}}{\sqrt{n_k}}\,
\log(a\|F\|_{P,2}/\rho)
\Bigr),
\]
where $K, K(q,c)$ are absolute constants, $Q$ is any finitely discrete probability measure, $N(\epsilon, \mathcal{F}, \norm{\cdot})$ is the covering number w.r.t the class $\mathcal{F}$ and radius $\epsilon$, and $\norm{f}_{P,p} = (\int|f|^p \,dP)^{1/p}$.
\end{lemma}

When $\beta > 1$, we have similar results. To be specific, we also verify the Assumption \ref{asp6} on the set $\Omega = \{\btheta : \|\btheta - \btheta^*\|_2  \lesssim  \delta_k \}$ at the $k$th iteration for any $2 \leq k \leq K$. We apply a proper kernel function $K$ such that (i) $K$ has bounded support on $[-1,1]$, (ii) $\|K\|_{\infty},\left\|K^{\prime}\right\|_{\infty}, \left\|K^{\prime\prime}\right\|_{\infty} \text { and }\tilde{K} =-\int K^{\prime}(t) t d t>0$ are bounded above by universal constants. Then, it suffices to show that the following sparse eigenvalue condition holds w.h.p:
\begin{align}\label{rsccondi: beta>1}
    & \rho_{\max }=\sup \left\{\boldsymbol{v}^T \nabla^2 R_{\delta_k}^{D_k}(\boldsymbol{\theta}) \boldsymbol{v}:\|\boldsymbol{v}\|_2=1,\|\boldsymbol{v}\|_0 \leq C s, \boldsymbol{\theta} \in \Omega,\|\boldsymbol{\theta}\|_0 \leq C s\right\}< C_1 c_{n,k}, \\
    & \rho_{\min }=\inf \left\{\boldsymbol{v}^T \nabla^2 R_{\delta_k}^{D_k}(\boldsymbol{\theta}) \boldsymbol{v}:\|\boldsymbol{v}\|_2=1,\|\boldsymbol{v}\|_0 \leq C s, \boldsymbol{\theta} \in \Omega,\|\boldsymbol{\theta}\|_0 \leq C s\right\}> C_2 c_{n,k},
\end{align}
for some constants $C_1, C_2 >0$. We let $\bSigma_{\bZ} = \mathbf{Cov}(\bZ)$ and $g(\cdot)$ be the p.d.f of the error $u$.
Denote $\nabla^2 R_{\delta_k}(\boldsymbol{\theta}) = \EE\left[\nabla^2 R_{\delta_k}^{D_k}(\boldsymbol{\theta})|\hat\btheta_{k-1}\right]$ as the population Hessian. We have the following lemmas and propositions:
\begin{proposition}\label{RE1: case beta>1}
    Under the setup above, suppose $g(\cdot)$ is an even function and satisfies the Lipschitz condition $|g(y) - g(x)| \leq L_2 |x-y|$ for any $-\frac{3\mu}{2} \leq x \leq y \leq -\frac{\mu}{2}$, and $\sign(t)(g(-\mu+t) - g(-\mu)) \geq L_1 |t|$. If we assume $\log d \gtrsim \log n$ and
    \begin{equation}
    \sqrt{\frac{c_{n,k}s \log\left(\frac{(dslog{d)^2}}{c_{n,k}\delta_k}\right)}{n_k \delta_k^3}} + \frac{s\log(d \vee n_k)}{n_k\delta_k^2} \log\left(\frac{(dslog{d)^2}}{c_{n,k}\delta_k}\right) \lesssim c_{n,k},
    \end{equation}
    then with probability greater than $1-c(\log n)^{-1}$, it holds that\[
    \rho_{\min } \geq L_1 \tilde{K} c_{n,k} \lambda_{\min }\left(\boldsymbol{\Sigma}_{\boldsymbol{Z}}\right)/2 ,~~ \ 
    \rho_{\max } \leq 2L_2 \tilde{K} c_{n,k} \lambda_{\max }\left(\boldsymbol{\Sigma}_{\boldsymbol{Z}}\right).
    \]
\end{proposition}

\begin{lemma}\label{truetheta_hessian: case beta>1}
   Suppose $\tilde{K}=-\int K^{\prime}(t) t d t>0$, $g(\cdot)$ is an even function and satisfies the Lipschitz condition $|g(y) - g(x)| \leq L_2 |x-y|$ for any $-\frac{3\mu}{2} \leq x \leq y \leq -\frac{\mu}{2}$, and $\sign(t)(g(-\mu+t) - g(-\mu)) \geq L_1 |t|$, then for all unit vector $\bv \in \RR^d$ we have
    \[
    \boldsymbol{v}^T \nabla^2 R_{\delta_k}(\boldsymbol{\theta^*}) \boldsymbol{v} \leq L_2 \tilde{K} c_{n,k} \lambda_{\max }\left(\boldsymbol{\Sigma}_{\boldsymbol{Z}}\right), \ \boldsymbol{v}^T \nabla^2 R_{\delta_k}(\boldsymbol{\theta^*}) \boldsymbol{v} \geq L_1 \tilde{K} c_{n,k} \lambda_{\min }\left(\boldsymbol{\Sigma}_{\boldsymbol{Z}}\right).
    \]
\end{lemma}

\begin{lemma}\label{lem: stochastic term beta >1}
    If we assume $\log d \gtrsim \log n$ and
    \begin{equation}
    \sqrt{\frac{c_{n,k}s \log\left(\frac{(dslog{d)^2}}{c_{n,k}\delta_k}\right)}{n_k \delta_k^3}} + \frac{s\log(d \vee n_k)}{n_k\delta_k^2} \log\left(\frac{(dslog{d)^2}}{c_{n,k}\delta_k}\right) \lesssim c_{n,k},
    \end{equation}
    then with probability greater than $1 - c(\log n_k)^{-1}$, it holds that
    \begin{equation}
        \sup_{\norm{\bv}_2 = 1, \norm{\bv}_0 \leq Cs}\sup_{\btheta\in\Omega, \norm{\btheta}_0 \leq Cs} \bigl| \bv^T(\nabla^2 R_{\delta_k}^n(\btheta) - \nabla^2 R_{\delta_k}(\btheta)\bv \bigr| \lesssim c_{n,k},
    \end{equation}
    where $n_k = \frac{n}{K}$, and $c$ is some positive constant.
\end{lemma}

We also verify the RSC/RSM condition on the population level, i.e. (\ref{eq_population_curvature_beta}) for $\beta > 1$ and (\ref{eq_population_curvature_beta_2}) for $0 < \beta \leq 1$ under mild conditions. For $\beta > 1$, the result is presented in the following proposition.

\begin{proposition}
    If $c_1f(\bz) \leq f'(\btheta^* \bz \mid \bz, y=1) f(\bz \mid y=1) - f'(\btheta^* \bz \mid \bz, y=-1) f(\bz \mid y=-1) \leq c_2f(\bz)$ for some positive constant $c_1$ and $c_2$, then $c_1 \lambda_{\min}(\Sigma_{\bz}) \norm{\bv}_2^2 \leq \bv^T \nabla^2 R(\btheta^*) \bv \leq c_2 \lambda_{\max}(\Sigma_{\bz})\norm{\bv}_2^2$.
\end{proposition}
\begin{proof}
    By definition (\ref{objective})
    \begin{equation}
    \begin{aligned}
        R(\btheta) &= \EE\left[\gamma(Y) L_{01}\{Y(X-\btheta^T\bZ)\}\right]\\
        &= \EE\left[L_{01}\{X-\btheta^T\bZ\} \mid Y=1 \right] + \EE\left[L_{01}\{-X+\btheta^T\bZ\} \mid Y=-1 \right]\\
        &= \PP(X < \btheta^T\bZ \mid Y=1) + \PP(X > \btheta^T\bZ \mid Y=-1)\\
        &= \int \int_{-\infty}^{\btheta^T\bz} f(x\mid \bz, y=1)\ dx f(\bz|y=1)\ d\bz + \int \int_{\btheta^T\bz}^{\infty} f(x\mid \bz, y=-1)\ dx f(\bz|y=-1)\ d\bz.
    \end{aligned}
    \end{equation}
    Hence, \[
    \nabla R(\btheta) = \int \bz f(\btheta^T\bz \mid \bz, y=1) f(\bz \mid y=1) \ d\bz - \int \bz f(\btheta^T\bz \mid \bz, y=-1) f(\bz \mid y=-1) \ d\bz,
    \]
    \[
        \nabla^2 R(\btheta) = \int \bz\bz^T f'(\btheta^T\bz \mid \bz, y=1) f(\bz \mid y=1) \ d\bz - \int \bz\bz^T f'(\btheta^T\bz \mid \bz, y=-1) f(\bz \mid y=-1) \ d\bz,
    \]
    and we can get the quadratic form $\bv^T \nabla^2 R(\btheta^*) \bv =  \int (\bv^T\bz)^2 h(\bz) f(\bz) \ d\bz$, where\[
        c_1 \leq h(\bz) = \left(f'(\btheta^T\bz \mid \bz, y=1) f(\bz \mid y=1) - f'(\btheta^T\bz \mid \bz, y=-1) f(\bz \mid y=-1)\right)/f(\bz) \leq c_2.
    \]
    Thus, $c_1 \lambda_{\min}(\Sigma_{\bz}) \norm{\bv}_2^2 \leq \bv^T \nabla^2 R(\btheta^*) \bv \leq c_2 \lambda_{\max}(\Sigma_{\bz})\norm{\bv}_2^2$.
\end{proof}
For the case $0 < \beta \leq 1$, the result follows directly from Lemma \ref{lem: beta-smooth class} and the fact that $\nabla R(\btheta^*) = 0$.

\subsection{Adaptation Theory}\label{sec_adaptation}
As shown in Section \ref{sec_rate}, the optimality of the proposed estimator $\hat\btheta_K$ critically depends on the choice of tuning parameters $ \{\lambda_k\}_{k=1}^K, \{b_k\}_{k=1}^{K-1}, \{\delta_k\}_{k=1}^K$ and the number of iterations $K$, which requires the knowledge of the smoothness parameter $\beta$ and the sparsity parameter $s$. In this section, we propose Lepski’s methods \citep{lepskii1991problem} to achieve adaptation to unknown smoothness $\beta$ and sparsity $s$ respectively.

\subsubsection{Adaptation to unknown sparsity $s$}\label{sec_adaptation_1}
In this section, we assume that the smoothness parameter $\beta$ is known and construct adaptive estimators $\hat s$ and $\hat \btheta_{\hat s}$, which achieves the optimal rate. Consider a discrete set $\mathcal{D} = \{s_j\}_{j=0}^m$ such that $s_j = 2^j$ and $2^m \leq d \leq 2^{m+1}$. If $\beta > 1$, define the optimal rate $\psi(s) = (\frac{s \log d}{N})^\frac{1}{2}$ and parameters $(b_{k-1}(s), \delta_k(s), \lambda_k(s))$ to be those as in Theorem \ref{beta>beta_*} and Theorem \ref{beta^{**}< beta < beta^*} in each iteration $k =1, \dots, K$, that is:
if $\beta > \frac{1+\sqrt{3}}{2}$, let
\[
    \delta_1(s) = c_1 \left(\frac{s \log d}{N}\right)^{1/(2\beta+1)}, \ \lambda_1(s) = c_2 \sqrt{\frac{N \log d}{n^2 \delta_1(s)}},
\]
and for $2 \leq k \leq K$,
\[
    \delta_{k}(s) = c_1 \left(\frac{s\log d}{N}\right)^{1/(2\beta)}, \ \lambda_k(s) = c_{2} \sqrt{\frac{N \log d}{n^2 b_{k-1}(s)\delta_k(s)}}, b_{k-1}(s) = c_3 \left(\frac{s\log d}{N}\right)^{1/(2\beta)}.
\]
If $1 < \beta \leq \frac{1+\sqrt{3}}{2}$, let
\[
    \delta_1(s) = c_1 \left(\frac{s \log d}{N}\right)^{1/(2\beta+1)}, \ \lambda_1(s) = c_2 \sqrt{\frac{N \log d}{n^2 \delta_1(s)}},
\]
for $2 \leq k \leq K-1$,
\begin{align*}
    &\ b_{k-1}(s) = c_3 
        \left(\log(\frac{N}{s \log d}) \right)^{\frac{(2\beta+1)(1-(\frac{\beta}{2\beta+1})^{k-1})}{2(\beta+1)}} \left(\frac{s \log d}{N}\right)^{\frac{\beta}{\beta+1}(1-(\frac{\beta}{2\beta+1})^{k-1})},\\
    &\delta_k(s)  = c_1 \left(\frac{b_{k-1}(s) s \log d}{N}\right)^{1/(2\beta+1)},\ \lambda_k(s) = c_{2} \sqrt{\frac{N \log d}{n^2 b_{k-1}(s)\delta_k(s)}}
\end{align*}
and
\[
    b_{K-1}(s) = c_3 \left(\frac{s\log d}{N}\right)^{1/(2\beta)}, \delta_{K}(s) = c_1 \left(\frac{s\log d}{N}\right)^{1/(2\beta)}, \ \lambda_K(s) = c_{2} \sqrt{\frac{N \log d}{n^2 b_{K-1}(s)\delta_K(s)}}.
\]
If $0 < \beta \leq 1$, define the optimal rate $\psi(s) = \left(\log(\frac{N}{K s \log d})\right)^{\frac{1}{4\beta}} \left(\frac{K s \log d}{N}\right)^{\frac{1}{2\beta}}$ and parameters $(b_{k-1}(s), \delta_k(s), \lambda_k(s))$ to be those as in Theorem \ref{beta < beta^**} in each iteration $k =1, \dots, K$, that is:
\[
    \delta_1(s) = c_1 \left(\frac{s \log d}{N}\right)^{1/(2\beta+1)}, \ \lambda_1(s) = c_2 \sqrt{\frac{N \log d}{n^2 \delta_1(s)}},
\]
and for $2 \leq k \leq K$,
\begin{align*}
    b_{k-1}(s) =& c_3 \left(\log (\frac{N}{K s \log d})\right)^{\frac{2\beta+1}{4\beta}[1 -  (\frac{1}{2\beta+1})^{k-1}]} \left(\frac{K s \log d}{N}\right)^{\frac{1}{2\beta}[1 - (\frac{1}{2\beta+1})^{k-1}]}, \\
    \delta_k(s)  =& c_1 \left(\frac{b_{k-1}(s) K s \log d}{N}\right)^{1/(2\beta+1)},\ \lambda_k(s) = c_{2} \sqrt{\frac{N K \log d}{n^2 b_{k-1}(s)\delta_k(s)}}. 
\end{align*}
Let $\hat \btheta_s$ be the final estimator after $K$ iterations of Algorithm \ref{algoK} with parameter $(b_{k-1}(s), \delta_k(s), \lambda_k(s))$, where we note that the value of $K$ only depends on $\beta$ which is known in this case.  Lepski's estimator is defined as $\hat \btheta_{\hat s}$, where
\begin{equation}  \label{def: adapt s}
\hat s = \min_{s \in D'}\{\norm{\hat \btheta_s - \hat \btheta_{s'}}_2 \leq \bar c \psi(s'), \forall s' \geq s, s' \in \cD\},
\end{equation}
for some constant $\bar c > 0$.
If the set to be minimized over is empty, we choose $\hat s = 2^m$. Let $s^*$ be the true sparsity, then there must exists $j \in\{0,\dots,m\}$ such that $s_{j-1} < s^* \leq s_j$.

\begin{lemma} \label{lem: hat s leq s_j}
Assume that Assumptions~\ref{asp3}-\ref{asp6} hold, then with probability greater than $1 - 2 K \log_2(d)/d$, we have $\hat s \leq s_j$.
\end{lemma}
\begin{proof}
    Define the sparsity class $C(s) = \{\norm{\btheta^*}_0 \leq s, \norm{\btheta^*}_2 \leq C\}$ for some constant $C>0$. Then, for any $s \geq s^*$, $C(s^*) \subseteq C(s)$, which means Theorem \ref{rate_k} holds for any $s \geq s^*$. That is, $\norm{\hat \btheta_{s} - \btheta^*}_2 \leq \psi(s)$ holds with probability greater than $1 - 2K/d$ for a single $s \geq s^*$. From the union bound, we know that with probability greater than $1 - 2K \log_2(d)/d$, $\norm{\hat \btheta_{s} - \btheta^*}_2 \leq \psi(s)$ holds for each $s \geq s^*, s \in \cD$. We prove $\hat s \leq s_j$ by verifying that $s_j$ is also in the set $\{\norm{\hat \btheta_s - \hat \btheta_{s'}}_2 \leq \psi(s'), \forall s' \geq s, s' \in \cD\}$. This is because by Theorem \ref{rate_k} for any $s' \geq s_j \geq s^*, s' \in \cD$,
    \begin{align*}
    \norm{\hat \btheta_{s_j} - \hat \btheta_{s'}}_2
     & \leq \norm{\hat \btheta_{s_j} - \btheta^*}_2 + \norm{\hat \btheta_{s'} - \btheta^*}_2 \\
     & \leq C\psi(s_j) + C\psi(s') \leq \bar c \psi(s'),
    \end{align*}
    with probability greater than $1 - 2K \log_2(d)/d$. The last inequality is because the rate $\psi(s)$ is a monotonically increasing function with respect to $s$ and by choosing constant $\bar c$ large enough.
    By definition (\ref{def: adapt s}), we complete the proof.
\end{proof}

The following theorem shows that the proposed adaptive estimator $\hat \btheta_{\hat s}$ attains the optimal rate. 

\begin{theorem}
    Assume that Assumptions~\ref{asp3}-\ref{asp6} hold, and $\log \frac{N}{K d \log d} \geq \frac{\log 2}{1-c}$ for some constant $c \in (0,1)$. With probability greater than $1 - 2 K \log_2(d)/d$, we have
    \begin{equation}
        \norm{\hat \btheta_{\hat s} - \btheta^*}_2 \lesssim \psi(s^*).
    \end{equation}
\end{theorem}

\begin{proof}
    From Lemma \ref{lem: hat s leq s_j}, we know that with probability greater than $1 - 2 K \log_2(d)/d$, $\hat s \leq s_j$. Thus, by definition of $\hat s$ and Theorem \ref{rate_k}, we have 
    \begin{align*}
        \norm{\hat \btheta_{\hat s} - \btheta^*}_2 &\leq \norm{\hat \btheta_{\hat s} - \hat \btheta_{s_j}}_2 + \norm{\hat \btheta_{s_j} - \btheta^*}_2\\
        &\leq \bar c \psi(s_j) + C\psi(s_j) \asymp \psi(s_j).
    \end{align*}
    Thus, it suffices to show that $\psi(s_j) \asymp \psi(s^*)$. By the definition of $\psi(s)$ and that $s_j = 2 s_{j-1}$, we obtain for $\beta > 1$
    \[
    \frac{\psi(s_j)}{\psi(s_{j-1})} = \frac{\left(s_j \log d/N \right)^{\frac{1}{2}}}{\left(s_{j-1} \log d/N \right)^{\frac{1}{2}}} = \sqrt{2},
    \] and for $0 < \beta \leq 1$,
    \begin{align*}
    \frac{\psi(s_j)}{\psi(s_{j-1})} &= \frac{\left(\frac{K s_j \log d}{N}\right)^{\frac{1}{2\beta}} \left(\log \frac{N}{K s_j \log d}\right)^{\frac{1}{4\beta}}}{\left(\frac{K s_{j-1} \log d}{N}\right)^{\frac{1}{2\beta}} \left(\log \frac{N}{K s_{j-1} \log d}\right)^{\frac{1}{4\beta}}} \\
    &= \left(\frac{s_j}{s_{j-1}}\right)^{\frac{1}{2\beta}} \left(\frac{\log \frac{N}{2K s_{j-1} \log d}}{\log \frac{N}{K s_{j-1} \log d}}\right)^\frac{1}{4\beta} \\
    &= 2^{\frac{1}{2\beta}} \left(1 - \frac{\log 2}{\log \frac{N}{K s_{j-1} \log d}}\right)^{\frac{1}{4\beta}} \in [2^{\frac{1}{2\beta}}c, 2^{\frac{1}{2\beta}}],
    \end{align*}
    so $\psi(s_{j-1}) \asymp \psi(s_j)$. Since the rate $\psi(s)$ is a monotonically increasing function with respect to $s$, we know that $\psi(s_{j-1}) \leq \psi(s^*) \leq \psi(s_j)$, which implies $\psi(s_{j-1}) \asymp \psi(s^*) \asymp \psi(s_j)$. This completes the proof.
\end{proof}

\subsubsection{Adaptation to unknown smoothness $\beta$}\label{sec_adaptation_2}

In this section, we assume that the sparsity $s$ is known or can be specified based on domain knowledge. The goal is to construct an estimator $\hat\btheta$, which does not require the knowledge of the smoothness parameter $\beta$. Recall that our algorithm involves a set of tuning parameters $ \{\lambda_k\}_{k=1}^K, \{b_k\}_{k=1}^{K-1}, \{\delta_k\}_{k=1}^K$ and the number of iterations $K$, and it produces a sequence of estimators $ \{\hat\btheta_k\}_{k=1}^K$. To emphasize the dependence of $\hat\btheta_k$ on $\delta_k$, we denote the estimator by $\hat\btheta_{\delta_k}$. 

For now, let us simply set $K =\lceil \log\log N \rceil$ and $N_k=N/K$. For simplicity, we assume $n\gg N$, that is the total budget $N$ is much smaller than the sample size $n$. Let $\Delta=\frac{s\log d}{N_k}$. Given any two points $\btheta$ and $\btheta'$ in $\Omega_k$ (defined in Assumption \ref{asp6}), define 
$$
\rho_k(\delta_k)=\frac{R_{\delta_k}^{D_k}(\btheta')-R_{\delta_k}^{D_k}(\btheta)-\nabla R_{\delta_k}^{D_k}(\btheta)^T(\btheta'-\btheta)}{\|\btheta'-\btheta\|_2^2}.
$$
For $k\geq 2$, we set the size of the active set to be 
$$
\hat b_{k-1}(\delta_k)=C\Big(\delta_k\vee \frac{1}{\rho_{k-1}(\hat\delta_{k-1})} \sqrt{\frac{N Ks \log d}{n^2\hat \delta_{k-1}\hat b_{k-2}(\hat\delta_{k-1})} \log \frac{1}{\Delta}}\Big),
$$
where $\hat\delta_{k-1}$ is to be defined later and $C$ is a large constant. Note that for $k=2$, we just set $\hat b_{k-2}(\hat\delta_{k-1})=1$ in the above definition. Throughout the proof, the constant $C$ may differ from line to line. By the proof of Theorem \ref{beta>beta_*}, we set the sampling probability in (\ref{sampling}) to be $\hat c_k(\delta_k)=\frac{C N}{n\hat b_{k-1}(\delta_k)}$ for $k\geq 2$, and $\hat c_1=c_{n,1}=\frac{N}{n}$ which is independent of $\delta_1$. Moreover, set the regularization
parameter to be
$$
\hat \lambda_k(\delta_k)=C\sqrt{\frac{\hat c_k(\delta_k)K\log d}{n\delta_k}},
$$
for $k\geq 2$, and $\hat \lambda_1=C\sqrt{\frac{NK\log d}{n^2\delta_1}}$. 
Now we are ready to introduce the Lepski’s method. Consider the discrete set $\cD = \{1,\frac{1}{2},\dotso, \frac{1}{2^m} \}$ such that $\frac{1}{2^m} \leq \frac1n < \frac{1}{2^{m-1}}.$ The key idea is to choose the largest bandwidth in $\cD$ such that the bias is still dominated by the variance. Formally, the Lepski's method selects
 \begin{equation}\label{leiski_beta_1}
 \hat{\delta}_1 = \max_{\delta_1 \in \cD } \bigg \{
 \norm{\hat{\btheta}_{\delta_1} - \hat{\btheta}_{\delta_1'}}_2 \leq \frac{c}{\rho_1(\delta_1')} \sqrt{\frac{\hat c_1 Ks \log d}{n\delta_1'}}~~\forall \delta_1' \leq \delta_1, \delta_1' \in \cD
 \bigg  \},     
 \end{equation}
 and for $2\leq k\leq K$
 \begin{equation}\label{leiski_beta_k}
 \hat{\delta}_k = \max_{\delta_k \in \cD } \bigg \{
 \norm{\hat{\btheta}_{\delta_k} - \hat{\btheta}_{\delta_k'}}_2 \leq \frac{c}{\rho_k(\delta_k')} \sqrt{\frac{\hat c_k(\delta_k') Ks \log d}{n\delta_k'}}~~\forall \delta_k' \leq \delta_k, \delta_k' \in \cD
 \bigg  \},         
 \end{equation}
 for some constant $c > 0$. If the set to be maximized over is empty, we  choose $\hat{\delta}_k = 1/n$ by default. We further define the oracles as
$$
\delta^*_1 = \max \bigg \{
	\delta_1 \in \cD, \delta_1^\beta \leq c'\sqrt{\Delta \frac{1}{\delta_1}}
	\bigg \},
$$
and for $k\geq 2$
	\[
	\delta^*_k = \max \bigg \{
	\delta_k \in \cD, \delta_k^\beta \leq c'\sqrt{\Delta \frac{\hat b_{k-1}(\delta_k)}{\delta_k}}
	\bigg \},
	\]
	where $c'>0$ is a constant. Note that $\delta^*_k$ is a random quantity as it depends $\hat\delta_{k-1}$ via  $\hat b_{k-1}(\delta_k)$. 

Define the event $\cE=\cap_{1\leq k\leq K} \cE_k$, where for $k\geq 2$
$$
\cE_k=\Big\{\|\nabla  R^{D_k}_{\delta_k}(\btheta^*)  - \nabla  R_{\delta_k, \hat{\btheta}_{\hat\delta_{k-1}}}(\btheta^*) \|_\infty \leq 
C_1\sqrt{\frac{\hat c_{k}(\delta_k)K\log d }{n\delta_k}}, \forall \delta_k\in\cD\Big\},
$$
and 
$$
\cE_1=\Big\{\|\nabla  R^{D_1}_{\delta_1}(\btheta^*)  - \nabla  R_{\delta_1}(\btheta^*) \|_\infty \leq 
C_1\sqrt{\frac{\hat c_{1} K\log d }{n\delta_1}}, \forall \delta_1\in\cD\Big\}.
$$
By the proof of Proposition \ref{E_1}, we can easily prove the following lemma. 
\begin{lemma}\label{lem_adaptive_grid}
 Assume that Assumptions~\ref{asp3}-\ref{asp6} hold. The event $\cE$ holds with probability greater than $1 - 2K\log_{2}(n)/d$. 
\end{lemma}

Recall that
$$
 \nabla \bar R_{\delta}(\btheta) = \sum_{y = \pm 1} \int_{\bZ} \bz y \Big[\int_X\frac 1\delta   K\bigg(\frac{x - \btheta^T\bz}{\delta}\bigg) f(x|\bz,y)dx \Big]f(\bz|y) d\bz. 
$$
 The following theorem shows the convergence rate of the estimator $\hat{\btheta}_{\hat{\delta}_K}$.
 
 \begin{theorem}\label{thm_adapt_beta}
 Assume that Assumptions~\ref{asp3}-\ref{asp6} hold, 
 $\|\nabla_S \bar R_{\delta}(\btheta^*)\|_2\geq C\delta^\beta$, and 
 \begin{equation}\label{eq_thm_adapt_beta_signal}
\min\{|\btheta^*_j|:\btheta^*_j\neq 0\}\geq 
\begin{cases}
C\delta^{*\beta}_1, & \text{if } \beta > 1, \\[6pt]
C\delta_1^* & \text{if } \beta \leq 1,
\end{cases}
\end{equation}
for some sufficiently large constant $C$. On the event $\cE$, we have
\begin{equation}\label{eq_thm_adapt_beta_1}
\norm{\hat{\btheta}_{\hat\delta_1} - \btheta^*}_2\lesssim \frac{1}{\rho_1(\delta_1^*)} \sqrt{\frac{\hat c_1 Ks \log d}{n\delta_1^*}},
\end{equation}
and for $2\leq k\leq K$
\begin{equation}\label{eq_thm_adapt_beta_2}
\norm{\hat{\btheta}_{\hat\delta_k} - \btheta^*}_2\lesssim \frac{1}{\rho_k(\delta_k^*)} \sqrt{\frac{\hat c_k(\delta_k^*) Ks \log d}{n\delta_k^*}}.
\end{equation}
In particular,  with probability greater than $1 - 2K\log_{2}(n)/d$
\begin{equation}\label{eq_thm_adapt_beta_3}
 \|\hat{\btheta}_{\hat{\delta}_K} - \btheta^*\|_2 \lesssim 
\begin{cases}
\Delta^{1/2}, & \text{if } \beta > 1, \\[6pt]
\log^{\frac{1}{4\beta}}(\frac{1}{\Delta})\Delta^{\frac{1}{2\beta}} & \text{if } 1/2\leq \beta \leq 1.
\end{cases}
\end{equation}
 \end{theorem}

To achieve adaptive estimation, we require two more conditions, $\|\nabla_S \bar R_{\delta}(\btheta^*)\|_2\geq C\delta^\beta$, and (\ref{eq_thm_adapt_beta_signal}). The first condition ensures that the approximation bias induced by the smoothed surrogate loss scales exactly as  $\delta^\beta$. The condition is further discussed and verified in Section \ref{sec_bias_lower_bound}. The second condition, known as the signal strength condition, appears to be necessary for deriving a sharp lower bound for Lasso-type estimators \citep{lounici2011oracle}. This lower bound plays a key role in establishing the adaptivity property and can be also of interest in its own right.

This theorem shows that we can achieve adaptive estimation of $\btheta^*$ without knowing the value of the smoothness parameter $\beta$. When $1/2\leq \beta\leq 1$, the rate of the estimator $\hat{\btheta}_{\hat{\delta}_K}$ in (\ref{eq_thm_adapt_beta_3}) exactly matches the upper bound in Theorem \ref{beta < beta^**}. However, when $\beta>1$, compared with Theorems \ref{beta>beta_*} and \ref{beta^{**}< beta < beta^*}, the rate $\Delta^{1/2}=(\frac{Ks\log d}{N})^{1/2}$ incurs an additional $K=\log\log N$ factor. Indeed, the extra factor can be removed with a mild modification of our algorithm. Specifically, we allocate half of the label budget to the last batch, that is $N_K=N/2$, and $|D_K|=n/2$. For the first $K-1$ batches, the label budget is allocated equally $N_k=N/(2K)$. The Lepski's method remains the same except that in the last batch we define
$$
 \hat{\delta}_K = \max_{\delta_K \in \cD } \bigg \{
 \norm{\hat{\btheta}_{\delta_K} - \hat{\btheta}_{\delta_K'}}_2 \leq \frac{c}{\rho_K(\delta_K')} \sqrt{\frac{\hat c_K(\delta_K') s \log d}{n\delta_K'}}~~\forall \delta_K' \leq \delta_K, \delta_K' \in \cD
 \bigg  \}. 
$$

 \begin{corollary}\label{cor_adapt_beta}
Under the same conditions as in Theorem \ref{thm_adapt_beta}, with probability greater than $1 - 2K\log_{2}(n)/d$,
\begin{equation}\label{eq_thm_adapt_beta_4}
 \|\hat{\btheta}_{\hat{\delta}_K} - \btheta^*\|_2 \lesssim 
\begin{cases}
\Big(\frac{s\log d}{N}\Big)^{1/2}, & \text{if } \beta > 1, \\[6pt]
\log^{\frac{1}{4\beta}}(\frac{1}{\Delta})\Delta^{\frac{1}{2\beta}} & \text{if } 1/2\leq \beta \leq 1.
\end{cases}
\end{equation}
 \end{corollary}

The proof of this corollary is nearly the same as the proof of Theorem \ref{thm_adapt_beta}. The only difference is that the oracle $\delta^*_K$ is defined as	
\[
	\delta^*_K = \max \bigg \{
	\delta_K \in \cD, \delta_K^\beta \leq c'\sqrt{\frac{s\log d}{N} \cdot\frac{\hat b_{K-1}(\delta_K)}{\delta_K}}
	\bigg \}.
	\]
We omit the proof. 

\begin{proof}[Proof of Theorem \ref{thm_adapt_beta}]    
By the definition of oracles, we know that 
$$
\delta_1^{*\beta} \asymp \sqrt{\Delta \frac{1}{\delta^*_1}}, ~~\textrm{and}~~ \delta_k^{*\beta} \asymp \sqrt{\Delta \frac{\hat b_{k-1}(\delta^*_k)}{\delta^*_k}}
$$ 
for $k\geq 2$. In addition, Assumption \ref{asp6} implies that $\rho_k(\delta_k)\asymp \hat c_k(\delta_k)$ for $\beta\geq 1$ and $\rho_k(\delta_k)\asymp \hat c_k(\delta_k)\delta_k^{\beta-1}$ for $\beta< 1$.  Similar results hold for $k=1$ if $\hat c_k(\delta_k)$ is replaced with $\hat c_1$. For notational simplicity, let 
    $$
    \psi_1(\delta)=\frac{1}{\rho_1(\delta)} \sqrt{\frac{\hat c_1 Ks \log d}{n\delta}}, ~~~\psi_k(\delta)=\frac{1}{\rho_k(\delta)} \sqrt{\frac{\hat c_k(\delta) Ks \log d}{n\delta}},
    $$
    for $k\geq 2$. 
    
    We now prove the theorem by induction. First, for $k=1$, by Lemma \ref{lem_delta_1_bound}, we have on the event $\cE_1$, $\delta_1^*\leq \hat\delta_1$. Then
    \begin{align*}
	\norm{\hat{\btheta}_{\hat{\delta}_1} - \btheta^*}_2 &\leq \norm{\hat{\btheta}_{\hat{\delta}_1} - \hat{\btheta}_{\delta^*_1}}_2 + \norm{\hat{\btheta}_{\delta^*_1} - \btheta^*}_2\\
	&\lesssim \psi_1(\delta_1^*) + \psi_1(\delta_1^*),        
    \end{align*}
where we use the definition (\ref{leiski_beta_1}) and the proof of Theorem \ref{rate_k} shows that, on the event $\cE_1$, the estimator $\hat{\btheta}_{\delta^*_1}$ with tuning parameters $\delta_1^*$, $\hat\lambda_1$ satisfies 
$$
\norm{\hat{\btheta}_{\delta^*_1} - \btheta^*}_2\lesssim \frac{1}{\hat c_1}(\hat c_1\delta_1^{*\beta}+\sqrt{\frac{\hat c_1 Ks\log d}{n\delta_1^*}})\asymp \delta_1^{*\beta}+\sqrt{\Delta \frac{1}{\delta^*_1}}\asymp \sqrt{\Delta \frac{1}{\delta^*_1}} \asymp\psi_1(\delta_1^*)
$$
for $\beta\geq 1$ and 
$$
\norm{\hat{\btheta}_{\delta^*_1} - \btheta^*}_2\lesssim \frac{(\delta^*_1)^{1-\beta}}{\hat c_1}(\hat c_1\delta_1^{*\beta}+\sqrt{\frac{\hat c_1 Ks\log d}{n\delta_1^*}})\asymp \delta_1^{*}+\sqrt{\Delta \delta_1^{*(1-2\beta)}}\asymp \sqrt{\Delta \delta_1^{*(1-2\beta)}} \asymp\psi_1(\delta_1^*)
$$
for $\beta< 1$. This completes the proof of (\ref{eq_thm_adapt_beta_1}). 
     
Assuming that (\ref{eq_thm_adapt_beta_2}) holds for any $j\leq k-1$, our goal is to verify (\ref{eq_thm_adapt_beta_2}) holds for $k$. By Lemma \ref{lem_delta_k_bound}, we know on the event $\cE$, $\hat\delta_{k-1}\leq C\delta_{k-1}^*$ and $\delta_{k}^*\leq \hat\delta_{k}$. Then
    \begin{align}
	\norm{\hat{\btheta}_{\hat{\delta}_k} - \btheta^*}_2 &\leq \norm{\hat{\btheta}_{\hat{\delta}_k} - \hat{\btheta}_{\delta^*_k}}_2 + \norm{\hat{\btheta}_{\delta^*_k} - \btheta^*}_2\nonumber\\
	&\lesssim \psi_k(\delta_k^*) +\norm{\hat{\btheta}_{\delta^*_k} - \btheta^*}_2, \label{eq_them_adapt_beta_k_bound}
    \end{align}
where we use the definition (\ref{leiski_beta_k}). To upper bound $\norm{\hat{\btheta}_{\delta^*_k} - \btheta^*}_2$, we need to verify the proof of Theorem \ref{rate_k} is still valid with tuning parameters $\delta^*_k$, $\hat\lambda_k(\delta_k^*)$ and $\hat b_{k-1}(\delta_k^*)$. The key is to verify the size of the active set satisfies (\ref{eq_rate_k_3}). Clearly, $\hat b_{k-1}(\delta_k^*)\geq C\delta^*_k$ and we can further show that
$$
\hat b_{k-1}(\delta_k^*)\geq C\frac{1}{\rho_{k-1}(\hat\delta_{k-1})} \sqrt{\frac{N Ks \log d}{n^2\hat \delta_{k-1}\hat b_{k-2}(\hat\delta_{k-1})} \log \frac{1}{\Delta}}\geq C\|\hat\btheta_{\hat\delta_{k-1}}-\btheta^*\|_2\sqrt{\log \frac{1}{\Delta}}.
$$
To see why the last step holds, we first use the assumption that (\ref{eq_thm_adapt_beta_2}) holds for $k-1$,
$$
\norm{\hat{\btheta}_{\hat\delta_{k-1}} - \btheta^*}_2\lesssim \frac{1}{\rho_{k-1}(\delta_{k-1}^*)} \sqrt{\frac{\hat c_{k-1}(\delta_{k-1}^*) Ks \log d}{n\delta_{k-1}^*}}\lesssim \frac{1}{\rho_{k-1}(\delta_{k-1}^*)} \sqrt{\frac{NKs \log d}{n^2\delta_{k-1}^*\hat b_{k-2}(\delta^*_{k-1})}}.
$$
For $\beta\geq 1$, it reduces to
$$
\frac{1}{\rho_{k-1}(\delta_{k-1}^*)} \sqrt{\frac{NKs \log d}{n^2\delta_{k-1}^*\hat b_{k-2}(\delta^*_{k-1})}}\asymp \sqrt{\frac{\Delta \hat b_{k-2}(\delta^*_{k-1})}{\delta^*_{k-1}}}\lesssim \sqrt{\frac{\Delta \hat b_{k-2}(\hat\delta_{k-1})}{\hat\delta_{k-1}}},
$$
as $\hat\delta_{k-1}\leq C\delta_{k-1}^*$ by Lemma \ref{lem_delta_k_bound}. Similarly, for $1/2\leq\beta<1$, 
$$
\frac{1}{\rho_{k-1}(\delta_{k-1}^*)} \sqrt{\frac{NKs \log d}{n^2\delta_{k-1}^*\hat b_{k-2}(\delta^*_{k-1})}}\asymp \sqrt{\Delta \hat b_{k-2}(\delta^*_{k-1})\delta^{*(1-2\beta)}_{k-1}}\lesssim \sqrt{\Delta \hat b_{k-2}(\hat\delta_{k-1})\hat\delta^{(1-2\beta)}_{k-1}},
$$
where in the last step we use both sides of the inequality $\delta_{k-1}^*\leq \hat\delta_{k-1}\leq C\delta_{k-1}^*$. Therefore, the condition (\ref{eq_rate_k_3}) holds, and Theorem \ref{rate_k} implies 
$$
\norm{\hat{\btheta}_{\delta^*_k} - \btheta^*}_2\lesssim \frac{1}{\hat c_k(\delta^*_k)}(\hat c_k(\delta^*_k)\delta_k^{*\beta}+\sqrt{\frac{\hat c_k(\delta^*_k) Ks\log d}{n\delta_k^*}})\asymp \delta_k^{*\beta}+ \sqrt{\Delta \frac{\hat b_{k-1}(\delta_k^*)}{\delta^*_k}}\asymp\psi_k(\delta_k^*)
$$
for $\beta\geq 1$ and 
$$
\norm{\hat{\btheta}_{\delta^*_k} - \btheta^*}_2\lesssim \frac{(\delta^*_k)^{1-\beta}}{\hat c_k(\delta^*_k)}(\hat c_k(\delta^*_k)\delta_k^{*\beta}+\sqrt{\frac{\hat c_k(\delta^*_k) Ks\log d}{n\delta_k^*}})\asymp \delta_k^{*}+\sqrt{\Delta\hat b_{k-1}(\delta_k^*) \delta_k^{*(1-2\beta)}} \asymp\psi_k(\delta_k^*)
$$
for $\beta< 1$. Together with (\ref{eq_them_adapt_beta_k_bound}), we obtain that $\norm{\hat{\btheta}_{\hat\delta_k} - \btheta^*}_2\lesssim \psi_k(\delta_k^*)$. This completes the proof of (\ref{eq_thm_adapt_beta_2}). 

Since the bound (\ref{eq_thm_adapt_beta_2}) also applies to $K$, what remains to do is to simplify the rate. However, the bound for $\hat{\btheta}_{\hat\delta_k}$ depends on $\hat b_{k-1}(\delta^*_k)$ and $\hat\delta_{k-1}$ recursively. Thus, we need to understand the order of $\hat\delta_{k-1}$, that is the order of $\delta_{k-1}^*$ by Lemma \ref{lem_delta_k_bound}, from the first iteration. It is easily seen that $\delta_1^*\asymp \Delta^{\frac{1}{2\beta+1}}$. Therefore, 
$$
\norm{\hat{\btheta}_{\hat\delta_1} - \btheta^*}_2\lesssim \Delta^{\frac{\beta \vee 1}{2\beta+1}}.
$$
In the following, let us consider the three cases $\beta>(1+\sqrt{3})/2$, $1<\beta\leq (1+\sqrt{3})/2$ and $\beta\leq 1$ separately. For $\beta>(1+\sqrt{3})/2$, (\ref{eq_thm_adapt_beta_2}) simplifies to 
\begin{equation}\label{eq_thm_adapt_beta_temp_1}
\norm{\hat{\btheta}_{\hat\delta_k} - \btheta^*}_2\lesssim \sqrt{\Delta \frac{\hat b_{k-1}(\delta_k^*)}{\delta^*_k}}.
\end{equation}
We start from $k=2$. Since $\delta_1^*\leq \hat\delta_1\leq C\delta_1^*$ by Lemma \ref{lem_delta_1_bound}, $\hat b_{1}(\delta_2^*)\asymp \delta_2^*\vee \sqrt{\Delta^{\frac{2\beta}{2\beta+1}}\log \frac{1}{\Delta}}$. In addition, the definition of $\delta_2^*$ implies $\delta_2^{*(2\beta+1)}\asymp \Delta (\delta_2^*\vee \sqrt{\Delta^{\frac{2\beta}{2\beta+1}}\log \frac{1}{\Delta}})$. Given $\beta>(1+\sqrt{3})/2$, we can easily verify that $\delta_2^*$ is the dominating term and it reduces $\delta_2^{*(2\beta+1)}\asymp \Delta \delta_2^*$, that is $\delta_2^*\asymp \Delta^{\frac{1}{2\beta}}$. By (\ref{eq_thm_adapt_beta_temp_1}), we obtain
$$
\norm{\hat{\btheta}_{\hat\delta_2} - \btheta^*}_2\lesssim \Delta^{\frac{1}{2}}.
$$
Assuming that $\norm{\hat{\btheta}_{\hat\delta_{k-1}} - \btheta^*}_2\lesssim \Delta^{\frac{1}{2}}$ and $\delta_{k-1}^*\asymp \Delta^{\frac{1}{2\beta}}$, we can similarly show that $\norm{\hat{\btheta}_{\hat\delta_{k}} - \btheta^*}_2\lesssim \Delta^{\frac{1}{2}}$ and $\delta_{k}^*\asymp \Delta^{\frac{1}{2\beta}}$. The key is to verify $\hat b_{k-1}(\delta^*_k)\asymp \delta^*_k$, and therefore we can exactly replicate the analysis for $k=2$. The analysis for $1<\beta\leq (1+\sqrt{3})/2$ is similar and is omitted. This completes the proof of the first part of (\ref{eq_thm_adapt_beta_3}). 

Finally, for $\beta<1$, (\ref{eq_thm_adapt_beta_2}) simplifies to $\norm{\hat{\btheta}_{\hat\delta_k} - \btheta^*}_2\lesssim \delta_k^*$. For $k=2$, note that $\hat b_1(\delta_2^*)\asymp \delta^*_2\vee \Delta^{\frac{1}{2\beta+1}}\log^{1/2}(\frac{1}{\Delta})$. Similarly, we can verify that $\Delta^{\frac{1}{2\beta+1}}\log^{1/2}(\frac{1}{\Delta})$ is the dominating term, such that $\hat b_1(\delta_2^*)\asymp \Delta^{\frac{1}{2\beta+1}}\log^{1/2}(\frac{1}{\Delta})$. Then the definition of $\delta_2^*$ implies $\delta_2^{*(2\beta+1)}\asymp\Delta \hat b_1(\delta_2^*)$, which reduces to $\delta_2^*\asymp \Delta^{(1+\frac{1}{2\beta+1})\frac{1}{2\beta+1}}\log^{\frac{1}{2(2\beta+1)}} (\frac{1}{\Delta})$. That is 
$$
\norm{\hat{\btheta}_{\hat\delta_2} - \btheta^*}_2\lesssim \Delta^{(1+\frac{1}{2\beta+1})\frac{1}{2\beta+1}}\log^{\frac{1}{2(2\beta+1)}} (\frac{1}{\Delta}).
$$
Finally, we can use induction method to show that
$$
\norm{\hat{\btheta}_{\hat\delta_k} - \btheta^*}_2\lesssim \Delta^{(1+\frac{1}{2\beta+1}+...+\frac{1}{(2\beta+1)^{k-1}})\frac{1}{2\beta+1}}\log^{\frac{1}{2(2\beta+1)}(1+...+\frac{1}{(2\beta+1)^{k-2}})} (\frac{1}{\Delta}),
$$
which agrees with (\ref{eq_beta < beta^**_rate_sub}). Following the same argument as in the proof of Theorem \ref{beta < beta^**}, we complete the proof of the second part of (\ref{eq_thm_adapt_beta_3}). 
\end{proof}

\begin{lemma}\label{lem_delta_1_bound}
Under the same conditions as in Theorem \ref{thm_adapt_beta}, on the event $\cE_1$, it holds that $\delta_1^*\leq \hat\delta_1\leq C\delta_1^*$ for some sufficiently large constant $C$. 
\end{lemma}

\begin{proof}[Proof of Lemma \ref{lem_delta_1_bound}]   
Let us first prove $\delta_1^*\leq \hat\delta_1$. For any $\delta_1\in\cD$ and $\delta_1\leq \delta_1^*$, on the event $\cE_1$, we have
\begin{align}
\|\hat\btheta_{\delta_1}-\hat\btheta_{\delta_1^*}\|_2&\leq \|\hat\btheta_{\delta_1}-\btheta^*\|_2+\|\hat\btheta_{\delta_1^*}-\btheta^*\|_2\nonumber\\
&\lesssim \frac{1}{\rho_1(\delta_1)} \Big(\hat c_1\delta_1^\beta+\sqrt{\frac{\hat c_1 Ks \log d}{n\delta_1}}\Big)+\frac{1}{\rho_1(\delta_1^*)} \Big(\hat c_1\delta_1^{*\beta}+\sqrt{\frac{\hat c_1 Ks \log d}{n\delta_1^*}}\Big)\nonumber\\
&\lesssim \psi_1(\delta_1)+\psi_1(\delta_1^*),\label{eq_lem_lem_delta_1_bound_1}
\end{align}
where the last step holds by noticing the definition of $\delta_1^*$ and $\delta_1\leq \delta_1^*$. Note that the notations $\psi_1(\delta)$ and $\psi_k(\delta)$ are defined in the proof of Theorem \ref{thm_adapt_beta}. When $\beta\geq 1$, 
$$
\psi_1(\delta_1)\asymp \sqrt{\Delta \frac{1}{\delta_1}} \geq  \sqrt{\Delta \frac{1}{\delta_1^*}}\asymp\psi_1(\delta_1^*),
$$
and then (\ref{eq_lem_lem_delta_1_bound_1}) implies $\|\hat\btheta_{\delta_1}-\hat\btheta_{\delta_1^*}\|_2\lesssim \psi_1(\delta_1)$. When $1/2\leq \beta<1$, 
$$
\psi_1(\delta_1)\asymp\sqrt{\Delta \delta_1^{(1-2\beta)}} \geq \sqrt{\Delta \delta_1^{*(1-2\beta)}} \asymp\psi_1(\delta_1^*),
$$
and then (\ref{eq_lem_lem_delta_1_bound_1}) implies $\|\hat\btheta_{\delta_1}-\hat\btheta_{\delta_1^*}\|_2\lesssim \psi_1(\delta_1)$. By the definition (\ref{leiski_beta_1}), we obtain that $\delta_1^*\leq \hat\delta_1$. 

Next, we prove $\hat\delta_1\leq C\delta_1^*$ by contradiction. Assume that for any constant $C_a$, we always have $\hat\delta_1\geq C_a\delta_1^*$. Then there exists $\bar\delta_1\geq C_a\delta_1^*$ and $\bar\delta_1\in\cD$ such that 
$$
\|\hat\btheta_{\bar\delta_1}-\hat\btheta_{\delta_1^*}\|_2\leq \frac{c}{\rho_1(\delta_1^*)} \sqrt{\frac{\hat c_1 Ks \log d}{n\delta_1^*}}. 
$$
The triangle inequality yields,
\begin{equation}\label{eq_lem_lem_delta_1_bound_2}
\|\hat\btheta_{\bar\delta_1}-\btheta^*\|_2\leq \frac{c}{\rho_1(\delta_1^*)} \sqrt{\frac{\hat c_1 Ks \log d}{n\delta_1^*}}+\|\hat\btheta_{\delta_1^*}-\btheta^*\|_2\leq C'\psi_1(\delta_1^*).
\end{equation}
However, Lemma \ref{lem_lasso_lower_bound} shows that $\|\hat\btheta_{\bar\delta_1}-\btheta^*\|_2\geq \bar C{\hat c_1\bar\delta_1^\beta}/{\rho_1(\bar\delta_1)}\geq C_a \bar C\psi_1(\delta_1^*)$. Since $C_a$ can be any constant, as long as we take $C_a>C'/\bar C$, where $C'$ is the constant in (\ref{eq_lem_lem_delta_1_bound_2}), we will have $\|\hat\btheta_{\bar\delta_1}-\btheta^*\|_2> C'\psi_1(\delta_1^*)$, which leads to the contradiction with  (\ref{eq_lem_lem_delta_1_bound_2}). This completes the proof.
\end{proof}

\begin{lemma}\label{lem_lasso_lower_bound}
Under the same conditions as in Theorem \ref{thm_adapt_beta}, for $\delta\geq C_a\delta_1^*$ with $C_a$ being a sufficiently large constant, we have $\|\hat\btheta_1-\btheta^*\|_2\geq \bar C{\hat c_1\delta^\beta}/{\rho_+}$.
\end{lemma}

\begin{proof}[Proof of Lemma \ref{lem_lasso_lower_bound}]
By the KKT condition, we know that $\nabla_j R_{\delta}^{D_1}(\hat\btheta_1)=\lambda_1 \sign(\hat\btheta_j)$ for any $\hat\btheta_j\neq 0$, which yields $|\nabla_j R_{\delta}^{D_1}(\hat\btheta_1)|=\lambda_1$. Recall that we take $\lambda_1=C\sqrt{\frac{\hat c_1 K\log d}{n\delta_1}}$. With a sufficiently large constant $C$, we can claim that the event $\cE_1$ implies $\cE'_1=\{\|\nabla  R^{D_1}_{\delta_1}(\btheta^*)  - \nabla  R_{\delta_1}(\btheta^*) \|_\infty \leq \lambda_1/2\}$. Write 
$$
\nabla_j R_{\delta}^{D_1}(\hat\btheta_1)=\nabla_j R_{\delta}^{D_1}(\hat\btheta_1)-\nabla_j R_{\delta}^{D_1}(\btheta^*)+I_j+II_j,
$$
where $I_j=\nabla_j R_{\delta}^{D_1}(\btheta_1^*)-\nabla_j R_{\delta}(\btheta^*)$ and $II_j=\nabla_j R_{\delta}(\btheta^*)-\nabla_j R(\btheta^*)$. On the event $\cE'_1$, 
$$
\lambda_1/2\leq |\nabla_j R_{\delta}^{D_1}(\hat\btheta_1)-\nabla_j R_{\delta}^{D_1}(\btheta^*)+II_j|\leq 3\lambda_1/2. 
$$
Then $\|\nabla R_{\delta}^{D_1}(\hat\btheta_1)-\nabla R_{\delta}^{D_1}(\btheta^*)+II\|_{2,\hat S}\leq 3\sqrt{\hat s}\lambda_1/2$, where $\|v\|_{2,\hat S}=\sqrt{\sum_{j:(\hat\btheta_1)_j\neq 0} v_j^2}$, and $\hat S$ is the support set of $\hat\btheta_1$ with $\hat s=|\hat S|$. We further have 
\begin{equation}\label{eq_lem_lasso_lower_bound_temp_1}
\|\nabla R_{\delta}^{D_1}(\hat\btheta_1)-\nabla R_{\delta}^{D_1}(\btheta^*)\|_{2,\hat S}\geq \|II\|_{2,\hat S}-3\sqrt{\hat s}\lambda_1/2.
\end{equation}
Note that
\begin{equation}\label{eq_lem_lasso_lower_bound_temp_2}
\rho_+\|\hat\btheta_1-\btheta^*\|_2\geq \|\nabla_j R_{\delta}^{D_1}(\hat\btheta_1)-\nabla_j R_{\delta}^{D_1}(\btheta^*)\|_2\geq \|\nabla R_{\delta}^{D_1}(\hat\btheta_1)-\nabla R_{\delta}^{D_1}(\btheta^*)\|_{2,\hat S}.
\end{equation}
Combining (\ref{eq_lem_lasso_lower_bound_temp_1}) and (\ref{eq_lem_lasso_lower_bound_temp_2}), we obtain 
\begin{equation}\label{eq_lem_lasso_lower_bound_temp_3}
\|\hat\btheta_1-\btheta^*\|_2\geq \frac{1}{\rho_+}(\|II\|_{2,\hat S}-3\sqrt{\hat s}\lambda_1/2)\geq \frac{1}{\rho_+}(\|II\|_{2,\hat S}-3C \sqrt{s}\lambda_1/2),
\end{equation}
where the last step follows from $\hat s\leq C^2s$ by Lemma \ref{lemma3}. The next step to lower bound  $\|II\|_{2,\hat S}$. In the following we first prove that $S\subseteq\hat S$. We prove by contradiction. Assume that there exists $j\in S\backslash \hat S$. Then 
$$
|\btheta^*_j|=|(\hat\btheta_1-\btheta^*)_j|\leq \|\hat\btheta_1-\btheta^*\|_2\lesssim 
\begin{cases}
\delta^\beta+\sqrt{s}\lambda_1\lesssim \delta^\beta, & \text{if } \beta > 1, \\[6pt]
\delta+\sqrt{s}\lambda_1 \delta^{1-\beta} \lesssim \delta& \text{if } \beta \leq 1,
\end{cases}
$$
where the last step holds because $\delta\geq C_a\delta_1^*$ with $C_a$ being a sufficiently large constant. In other words, when the bandwidth is sufficiently large, the bias can dominate the variance. This contradicts with the minimum signal strength condition (\ref{eq_thm_adapt_beta_signal}). 

Back to (\ref{eq_lem_lasso_lower_bound_temp_3}), we further have $\|II\|_{2,\hat S}\geq \|II\|_{2,S}$, which yields
$$
\|\hat\btheta_1-\btheta^*\|_2\geq \frac{1}{\rho_+}(\|II\|_{2,S}-3\sqrt{\hat s}\lambda_1/2)\geq \frac{1}{\rho_+}(C'\hat c_1 \delta^\beta-3C \sqrt{s}\lambda_1/2)\geq  \frac{\bar C}{\rho_+}\hat c_1 \delta^\beta.
$$
\end{proof}

\begin{lemma}\label{lem_delta_k_bound}
Under the same conditions as in Theorem \ref{thm_adapt_beta}, assume that (\ref{eq_thm_adapt_beta_2}) holds for $j\leq k-1$. On the event $\cap_{1\leq j\leq j}\cE_j$, it holds that $\delta_{k}^*\leq \hat\delta_{k}\leq C\delta_k^*$. 
\end{lemma}

\begin{proof}[Proof of Lemma \ref{lem_delta_k_bound}]   
Let us first prove $\delta_k^*\leq \hat\delta_k$. For any $\delta_k\in\cD$ and $\delta_k\leq \delta_k^*$, on the event $\cE_k$, we have
\begin{align}
\|\hat\btheta_{\delta_k}-\hat\btheta_{\delta_k^*}\|_2&\leq \|\hat\btheta_{\delta_k}-\btheta^*\|_2+\|\hat\btheta_{\delta_k^*}-\btheta^*\|_2\nonumber\\
&\lesssim \frac{1}{\rho_k(\delta_k)} \Big(\hat c_k(\delta_k)\delta_k^\beta+\sqrt{\frac{\hat c_k(\delta_k) Ks \log d}{n\delta_k}}\Big)+\frac{1}{\rho_k(\delta_k^*)} \Big(\hat c_k(\delta_k^*)\delta_k^{*\beta}+\sqrt{\frac{\hat c_k(\delta^*) Ks \log d}{n\delta_k^*}}\Big)\nonumber\\
&\lesssim \psi_k(\delta_k)+\psi_k(\delta_k^*).\label{eq_lem_delta_k_bound_1}
\end{align}
To see why the last step holds, we need to simplify the rate. Note that for $\beta\geq 1$,
$$
\frac{1}{\hat c_k(\delta_k)}(\hat c_k(\delta_k)\delta_k^{\beta}+\sqrt{\frac{\hat c_k(\delta_k) Ks\log d}{n\delta_k}})\asymp \delta_k^{\beta}+ \sqrt{\Delta \frac{\hat b_{k-1}(\delta_k)}{\delta_k}}\lesssim \psi_k(\delta_k)
$$
and for $\beta<1$
$$
\frac{(\delta_k)^{1-\beta}}{\hat c_k(\delta_k)}(\hat c_k(\delta_k)\delta_k^{\beta}+\sqrt{\frac{\hat c_k(\delta_k) Ks\log d}{n\delta_k}})\asymp \delta_k+\sqrt{\Delta\hat b_{k-1}(\delta_k) \delta_k^{(1-2\beta)}} \lesssim \psi_k(\delta_k).
$$
We can similarly verify that $\psi_k(\delta_k)\geq C\psi_k(\delta_k^*)$.  Then (\ref{eq_lem_delta_k_bound_1}) implies $\|\hat\btheta_{\delta_k}-\hat\btheta_{\delta_k^*}\|_2\lesssim \psi_k(\delta_k)$ and thus $\delta_k^*\leq \hat\delta_k$. We can also prove $\hat\delta_1\leq C\delta_1^*$ by contradiction using the same argument. Indeed, we can similarly prove 
$$
\|\hat\btheta_k-\btheta^*\|_2\geq 
\begin{cases}
\bar C\delta_k^\beta, & \text{if } \beta > 1, \\[6pt]
\bar C \delta_k& \text{if } \beta \leq 1,
\end{cases}
$$
for $\delta_k\geq C_a \delta_k^*$. We omit the details. 
\end{proof}

\subsubsection{Remark on $\|\nabla_S \bar R_{\delta}(\theta^*)\|_2\geq C\delta^\beta$}\label{sec_bias_lower_bound}
Recall that
$$
 \nabla \bar R_{\delta}(\btheta) = \sum_{y = \pm 1} \int_{\bZ} \bz y \Big[\int_X\frac 1\delta   K\bigg(\frac{x - \btheta^T\bz}{\delta}\bigg) f(x|\bz,y)dx \Big]f(\bz|y) d\bz. 
$$
For any $\|\bv\|_2=1$ with $\|\bv\|_0\leq s$, we have $\| \nabla_S \bar R_{\delta}(\btheta)\|_2\geq |\bv^T\nabla \bar R_{\delta}(\btheta)|$. By the calculation in Proposition \ref{prop2}, 
\begin{align*}
|\bv^T\nabla \bar R_{\delta}(\btheta)|&=\Big|\int \bv^T\bz\int K(u) \frac{(u\delta)^l}{l!}\big(A_1(\btheta^{*T}\bz+\tau u\delta) - A_1(\btheta^{*T}\bz)  \big)du f(z|y=1)dz\\
&-\int \bv^T\bz\int K(u) \frac{(u\delta)^l}{l!}\big(A_{-1}(\btheta^{*T}\bz+\tau u\delta) - A_{-1}(\btheta^{*T}\bz)  \big)du f(z|y=-1)dz\Big|,
\end{align*}
where $A_1(x)=f^{(l)}(x|\bz,Y =1)$ and $A_{-1}(x)=f^{(l)}(x|\bz,Y =-1)$. For simplicity, assume that $\|\btheta^*\|_2=1$, so we can take $\bv=\btheta^*$.  Thus, the condition $\|\nabla_S \bar R_{\delta}(\theta^*)\|_2\geq C\delta^\beta$ holds, if 
\begin{equation}\label{eq_bias_lower_bound_1}
    \int \EE\Big[(A_1(\btheta^{*T}\bz+\tau u\delta)-A_1(\btheta^{*T}\bz))\btheta^{*T}\bz|Y=1\Big] K(u)u^l du\geq C\delta^{\beta-l},
\end{equation}
and 
$$
-\int \EE\Big[(A_{-1}(\btheta^{*T}\bz+\tau u\delta)-A_{-1}(\btheta^{*T}\bz))\btheta^{*T}\bz|Y=1\Big] K(u)u^l du\geq C\delta^{\beta-l}.
$$
It remains to verify the above conditions case by case. Let us consider the following example to verify (\ref{eq_bias_lower_bound_1}). Assume that $A_1(x)$ is piecewise linear, 
$$
A_1(x)=C_1x, ~~\textrm{for}~~ x\geq 0, ~~\textrm{and}~~A_1(x)=C_2x, ~~\textrm{for}~~ x< 0,
$$
where $C_1>C_2\geq 0$. Then we have $\beta=1$ and $l=0$.  Assume that the pdf of $N:=\btheta^*\bz$ given $Y=1$,  is continuous, upper bounded by a constant, and symmetric. The kernel function $K(\cdot)$ is non-negative. To show (\ref{eq_bias_lower_bound_1}), we have 
\begin{align*}
& \int \EE\Big[(A_1(\btheta^{*T}\bz+ u\delta)-A_1(\btheta^{*T}\bz))\btheta^{*T}\bz|Y=1\Big] K(u) du\\
&=\int K(u)[I(u\geq 0)+I(u<0)] \EE[(A_1(N+u\delta)-A_1(N))N|Y=1]du\\
&=\int K(u)I(u\geq 0)\EE[(A_1(N+u\delta)-A_1(N-u\delta))N|Y=1]du
\end{align*}
where the last step holds by using the change  of variable for $u<0$. Furthermore, given $u>0$, 
\begin{align*}
& \EE[(A_1(N+u\delta)-A_1(N-u\delta))N|Y=1]\\
&=\EE\Big\{I(N<-u\delta)N C_2u\delta+I(-u\delta<N<u\delta)N (C_1(u\delta+N)-C_2(N-u\delta))+I(N>u\delta)N C_1u\delta |Y=1\Big\}\\
&=\delta u\Big\{C_1\EE(I(N>u\delta)N|Y=1)+C_2\EE(I(N<-u\delta)N|Y=1)\Big\}+O(\delta^2)\\
&=\delta u\Big\{C_1\EE(I(N>0)N|Y=1)+C_2\EE(I(N<0)N|Y=1)\Big\}+O(\delta^2).
\end{align*}
Let $C:=\EE(I(N>0)N|Y=1)$. Then
\begin{align*}
&\int \EE\Big[(A_1(\btheta^{*T}\bz+ u\delta)-A_1(\btheta^{*T}\bz))\btheta^{*T}\bz|Y=1\Big] K(u)du\\
&=C(C_1-C_2)\delta\int K(u)uI(u\geq 0)du+O(\delta^2)\geq C'\delta. 
\end{align*}
This verifies (\ref{eq_bias_lower_bound_1}). Intuitively, condition  (\ref{eq_bias_lower_bound_1}) requires the non-smoothness of $A_1(x)$ to be asymmetric, so that the bias from the positive and negative sides does not cancel each other out.

\subsection{Binary Response Model with Diverging 
 \texorpdfstring{$\|\theta^*\|_2$}{theta}
}\label{sec_example}

For the binary response model, we can construct a counterexample that shows that the eigenvalues of the Hessian matrix of $R(\btheta^*)$ at the true value $\btheta^*$ are of order $1/\|\btheta^*\|_2$. When  $\|\btheta^*\|_2$ is diverging, the eigenvalues of $\nabla^2 R(\btheta^*)$ shrink to 0 and thus the RSC condition would fail even if we only consider  $\btheta=\btheta^*$. Thus, the  statistical rate presented in Section \ref{sec_theory} is generally not applicable to the case with diverging $\|\btheta^*\|_2$. In the following, we provide the counterexample.

 Recall that the binary response model takes the form $Y = \sign (X - \btheta^{*T}\bZ + u)$. For simplicity we assume $X\sim N(0,1)$, $\bZ\sim N(0,\Ib)$ independent of $X$ and $\gamma(y) \equiv 1$. In addition, we assume the error $u$ is independent of $X$ (but possibly depends on $\bZ$), and the conditional density given $\bZ = \bz$, denoted by $f_{u|\bz}(\cdot)$, is bounded by a constant for all $\bz$. The median of $u$ given $\bZ$ is 0.

By the definition of $R(\btheta)$, it is easily shown that for any $\|\bv\|_2=1$
$$
\bv^T\nabla^2 R(\btheta^*)\bv=2\int (\bz^T\bv)^2f_{u|\bz}(0)\phi(\btheta^{*T}\bz)\phi(z_1)...\phi(z_n)d\bz,
$$
where $\phi$ is the pdf of $N(0,1)$. Since $f_{u|\bz}(\cdot)$ is upper bounded by a constant $C$, we have
\begin{align*}
\bv^T\nabla^2 R(\btheta^*)\bv&\leq 2C\int  (\bz^T\bv)^2 \frac{1}{(2\pi)^{1/2}} \frac{1}{(2\pi)^{d/2}}\exp(-\frac{\bz^T(\Ib+\btheta^*\btheta^{*T})\bz}{2})d\bz\\
&=2C\frac{\bv^T(\Ib+\btheta^*\btheta^{*T})^{-1}\bv}{(2\pi)^{1/2}|\Ib+\btheta^*\btheta^{*T}|^{1/2}},
\end{align*}
where in the last step we use the integral of a normal distribution with variance $(\Ib+\btheta^*\btheta^{*T})^{-1}$. Since the matrix $\Ib+\btheta^*\btheta^{*T}$ has $d-1$ eigenvalues 1 and 1 eigenvalue $1+\|\btheta^*\|^2_2$, we obtain that
$$
\bv^T\nabla^2 R(\btheta^*)\bv\leq \frac{2C}{(2\pi)^{1/2}(1+\|\btheta^*\|^2_2)^{1/2}}.
$$
Thus, when $\|\btheta^*\|_2$ is diverging, $\bv^T\nabla^2 R(\btheta^*)\bv$ is of order $1/\|\btheta^*\|_2$. Note that we can make the approximation error $\bv^T\nabla^2(R_{\delta_k} (\btheta^*) - c_{n,k}R(\btheta^*))\bv$ ignorable with a small bandwidth $\delta_k$ and the stochastic error $\bv^T(\nabla^2R_{\delta_k}^{D_k}(\btheta^*) - \nabla^2 R_{\delta_k}(\btheta^*))\bv$ ignorable via concentration inequality. Thus, we can show that $\bv^T\nabla^2R_{\delta_k}^{D_k}(\btheta) \bv$ is also of order $c_{n,k}/\|\btheta^*\|_2$, violating our RSC condition (even at the true value $\btheta^*$) when $\|\btheta^*\|_2$ is diverging.

\subsection{Analysis of Region-based sampling}\label{sec_sampling}

Recall that 
    \[
    Q_i = g_{3i}(\bar{H}_{i-1}) \ind\{f_i(\bZ_i, \bar{H}_{i-1}) - g_{1i}(\bar{H}_{i-1})\leq X_i \leq f_i(\bZ_i, \bar{H}_{i-1}) + g_{2i}(\bar{H}_{i-1})\}.
    \]
Since $\EE_P(\sum_{i=1}^n Q_i)\leq N$ holds and $\inf_{x \in \RR, \boldsymbol{z} \in \mathbb{R}^d} f(x \mid \boldsymbol{z}) \geq p_{\min}$ for $P\in \mathcal{P}(\beta,s)$, we have $$
\sum_{i=1}^n \EE_P(g_{3i}(\bar{H}_{i-1})(g_{1i}(\bar{H}_{i-1})+ g_{2i}(\bar{H}_{i-1})) )= O(N).$$ 
Finally, we obtain that
\begin{align*}
    &\sum_{i=1}^n\EE\left[\bZ_i \bZ_i^T \cdot Q_i\right]\\
    =& \sum_{i=1}^n\EE\left[\bZ_i \bZ_i^T \cdot 
    g_{3i}(\bar{H}_{i-1}) \ind\{f_i(\bZ_i, \bar{H}_{i-1}) - g_{1i}(\bar{H}_{i-1})\leq X_i \leq f_i(\bZ_i, \bar{H}_{i-1}) + g_{2i}(\bar{H}_{i-1})\}
    \right] \\
    \leq& p_{\max} \sum_{i=1}^n \EE\left[\bZ_i \bZ_i^T \cdot 
    g_{3i}(\bar{H}_{i-1}) (g_{1i}(\bar{H}_{i-1})+g_{2i}(\bar{H}_{i-1}))
    \right]\\
    =& p_{\max}\EE(\bZ\bZ^T) \sum_{i=1}^n \EE\left[g_{3i}(\bar{H}_{i-1}) (g_{1i}(\bar{H}_{i-1})+g_{2i}(\bar{H}_{i-1}))\right]\\
    =& O(N) \cdot \EE(\bZ\bZ^T),
\end{align*}
where we used the fact that $(X_i, \bZ_i) \perp \bar{H}_{i-1} $. Since the sparse eigenvalue assumption (\ref{sparseeigen2}) holds, the last condition in (\ref{sampling1}) holds as well.

\subsection{Additional simulation results}\label{sec_additional_sim}

In this section, we consider the numerical results for our algorithm with $N_1=N/5$ and $N_2=4N/5$ and the comparison of the four methods concerning  $\|\hat{\btheta} - \btheta^*\|_{\infty}$ and prediction errors, which exhibit similar patterns as $\|\hat{\btheta} - \btheta^*\|_1$ and $\|\hat{\btheta} - \btheta^*\|_2$. 

\begin{figure}
\centering
\subfigcapskip -5pt
\subfigure[]{
\includegraphics[width=0.48\textwidth]{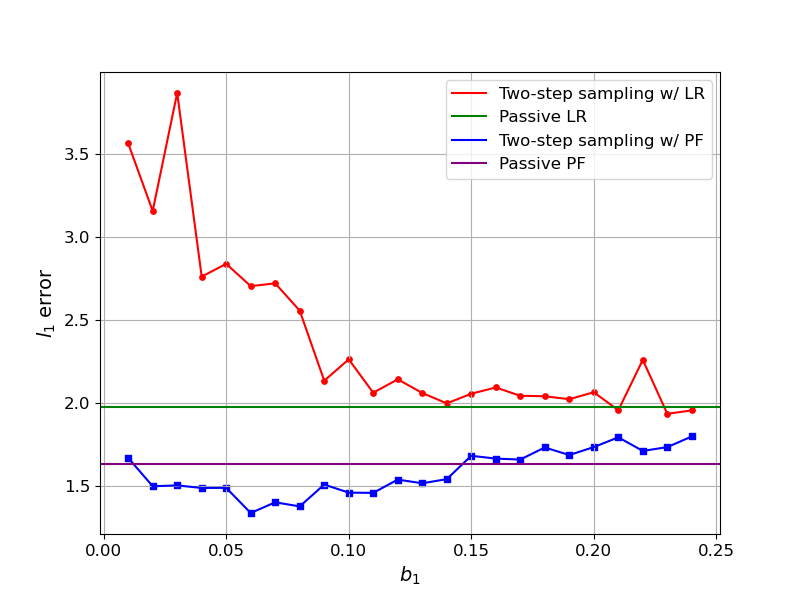}}
\subfigure[]{
\includegraphics[width=0.48\textwidth]{
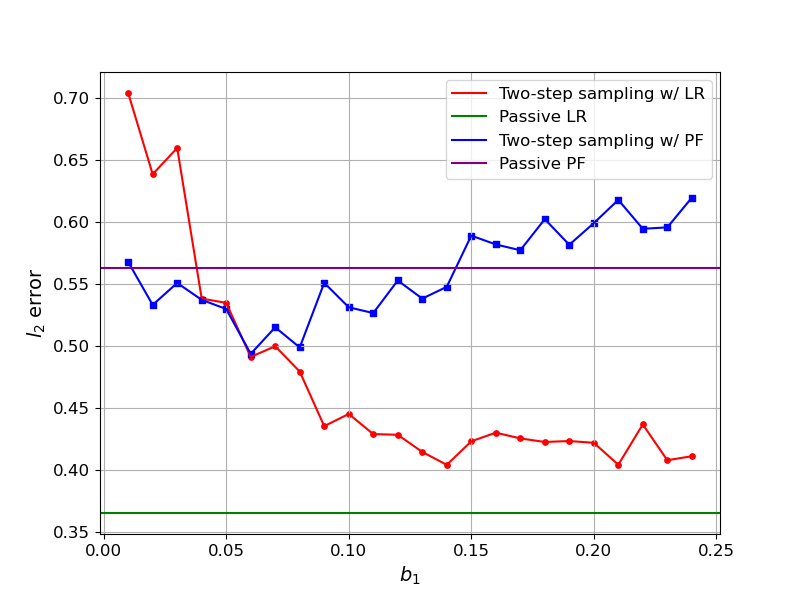}} \vskip -5pt 
\caption{$\|\hat{\btheta} - \btheta^*\|$ in $\ell_1$ and $\ell_2$ for logistic regression. LR: $\ell_1$ penalized logistic regression; PF: path-following algorithm. 1/5 of the label budget is used in the first step for both two-step sampling methods.}
\label{log_5}
\end{figure}

\begin{figure}
\centering
\subfigcapskip -5pt
\subfigure[]{
\includegraphics[width=0.48\textwidth]{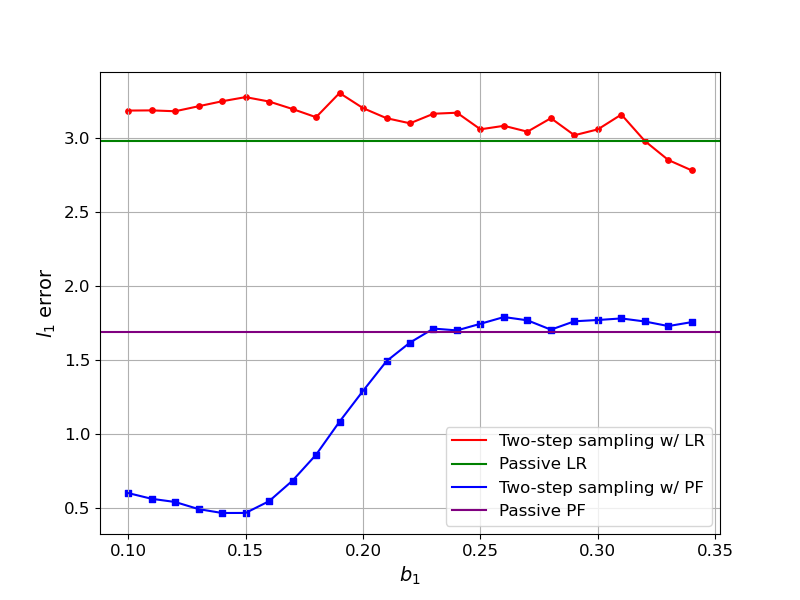}}
\subfigure[]{
\includegraphics[width=0.48\textwidth]{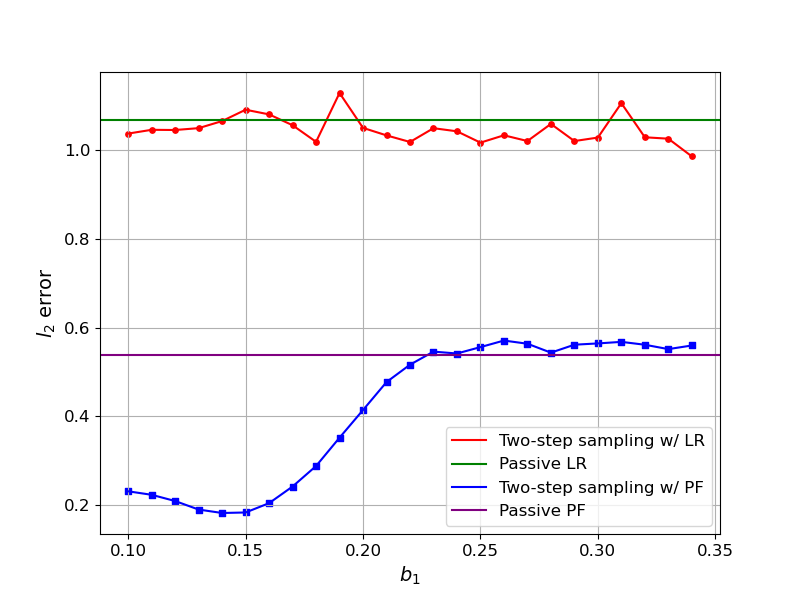}} \vskip -5pt 
\caption{$\|\hat{\btheta} - \btheta^*\|$ in $\ell_1$ and $\ell_2$ for conditional mean model. LR: $\ell_1$ penalized logistic regression; PF: path-following algorithm. 1/5 of the label budget is used in the first step for both two-step sampling methods.}
\label{condi_5}
\end{figure}

\begin{figure}
\centering
\subfigcapskip -5pt
\subfigure[]{
\includegraphics[width=0.48\textwidth]{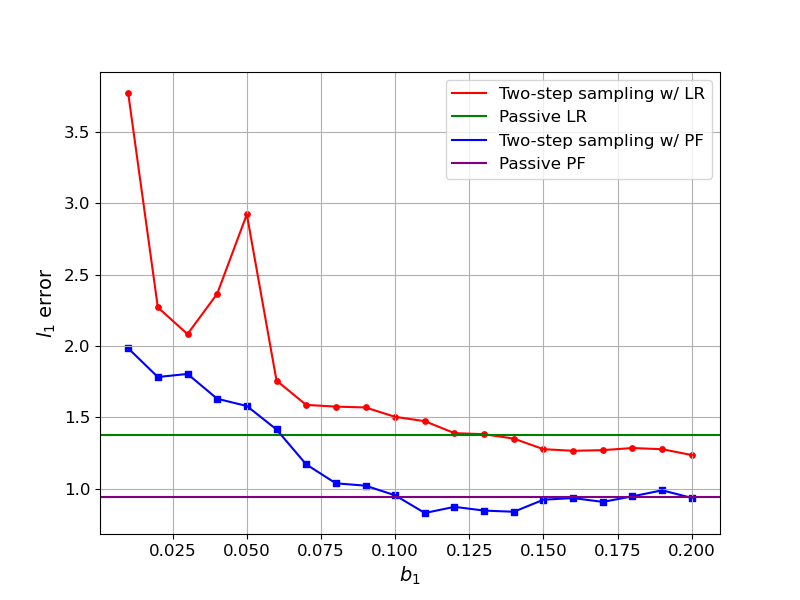}}
\subfigure[]{
\includegraphics[width=0.48\textwidth]{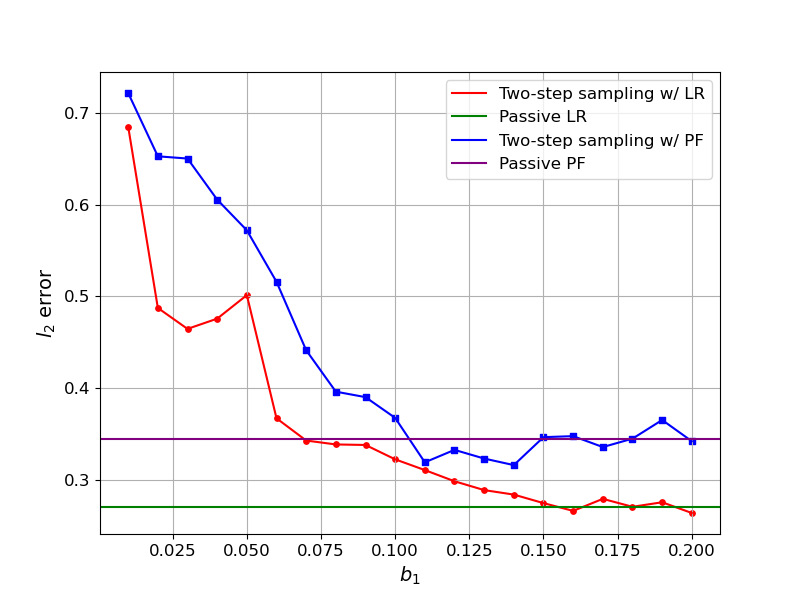}} \vskip -5pt 
\caption{$\|\hat{\btheta} - \btheta^*\|$ in $\ell_1$ and $\ell_2$ for binary response model. LR: $\ell_1$ penalized logistic regression; PF: path-following algorithm. 1/5 of the label budget is used in the first step for both two-step sampling methods.}
\label{manski_5}
\end{figure}

\begin{figure}
\centering
\subfigcapskip -5pt
\subfigure[]{
\includegraphics[width=0.48\textwidth]{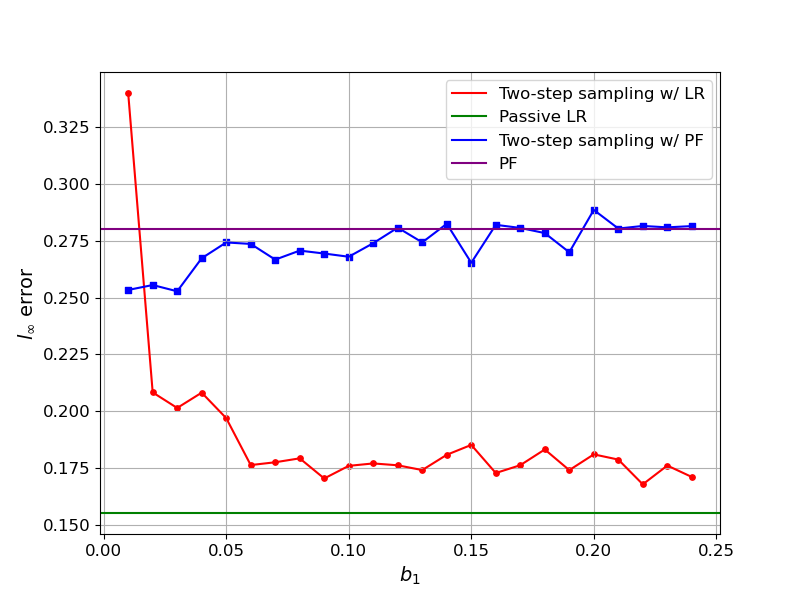}}
\subfigure[]{
\includegraphics[width=0.48\textwidth]{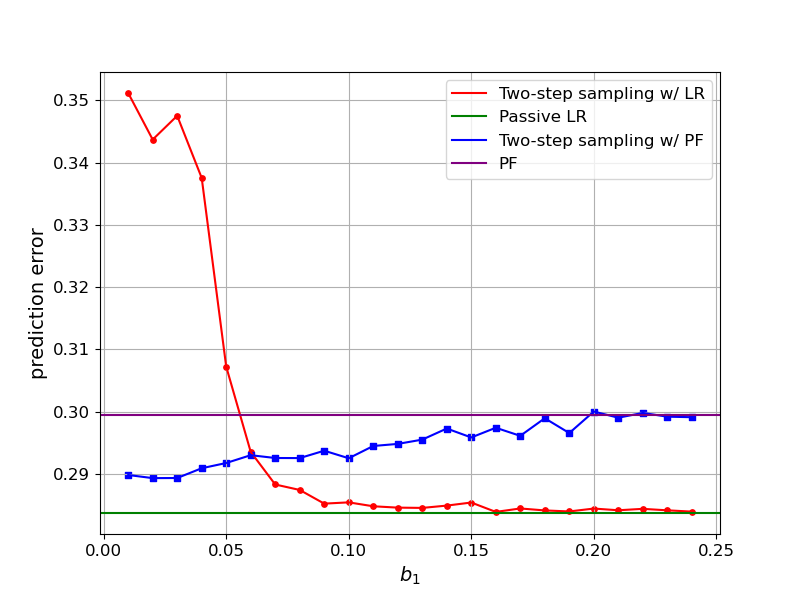}} \vskip -5pt 
\caption{$\|\hat{\btheta} - \btheta^*\|_{\infty}$ and prediction error for logistic regression. LR: $\ell_1$ penalized logistic regression; PF: path-following algorithm. 1/8 of the total budget of labeled data is used in the first step for both two-step sampling methods.}
\label{log_infpred_8}
\end{figure}

\begin{figure}
\centering
\subfigcapskip -5pt
\subfigure[]{
\includegraphics[width=0.48\textwidth]{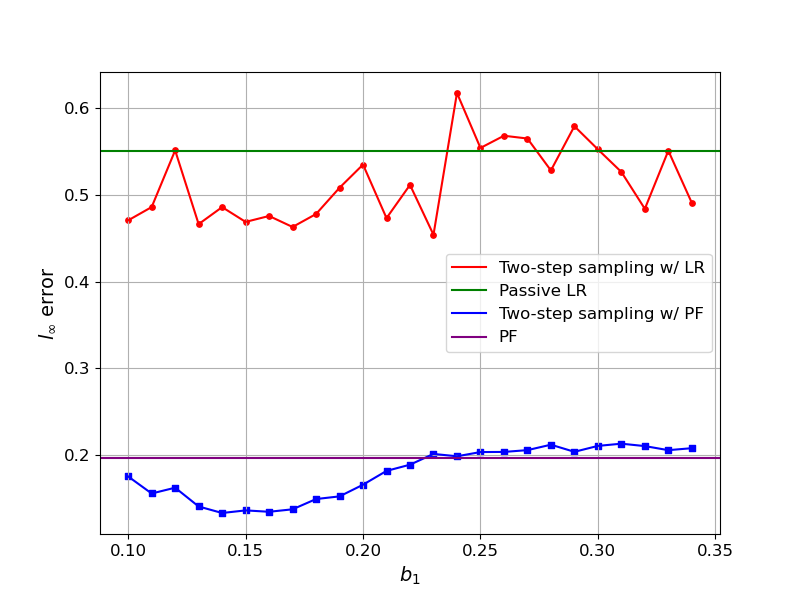}}
\subfigure[]{
\includegraphics[width=0.48\textwidth]{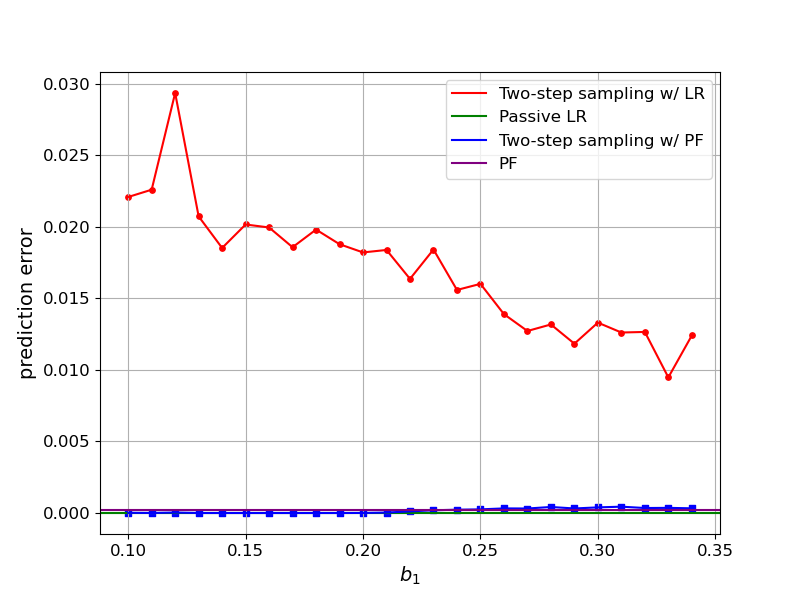}} \vskip -5pt 
\caption{$\|\hat{\btheta} - \btheta^*\|_{\infty}$ and prediction error for conditional mean model. LR: $\ell_1$ penalized logistic regression; PF: path-following algorithm. 1/8 of the total budget of labeled data is used in the first step for both two-step sampling methods.}
\label{condi_infpred_8}
\end{figure}

\begin{figure}
\centering
\subfigcapskip -5pt
\subfigure[]{
\includegraphics[width=0.48\textwidth]{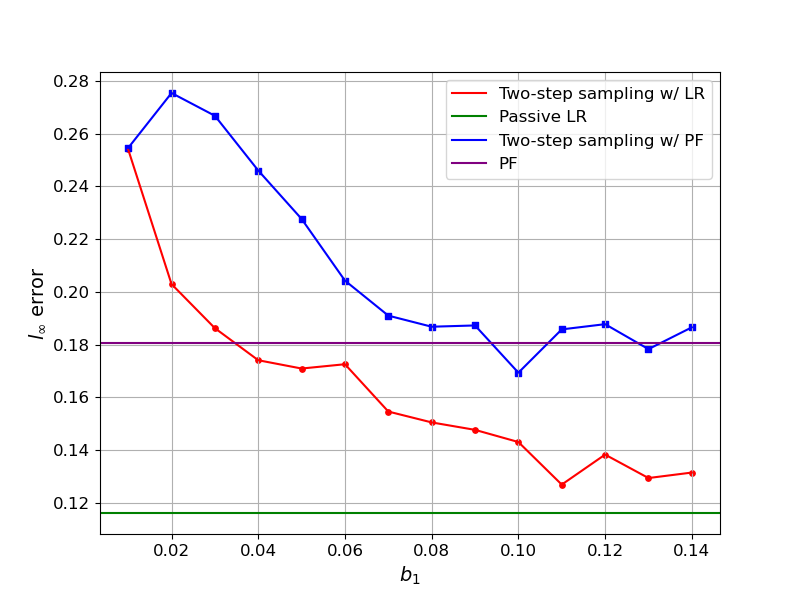}}
\subfigure[]{
\includegraphics[width=0.48\textwidth]{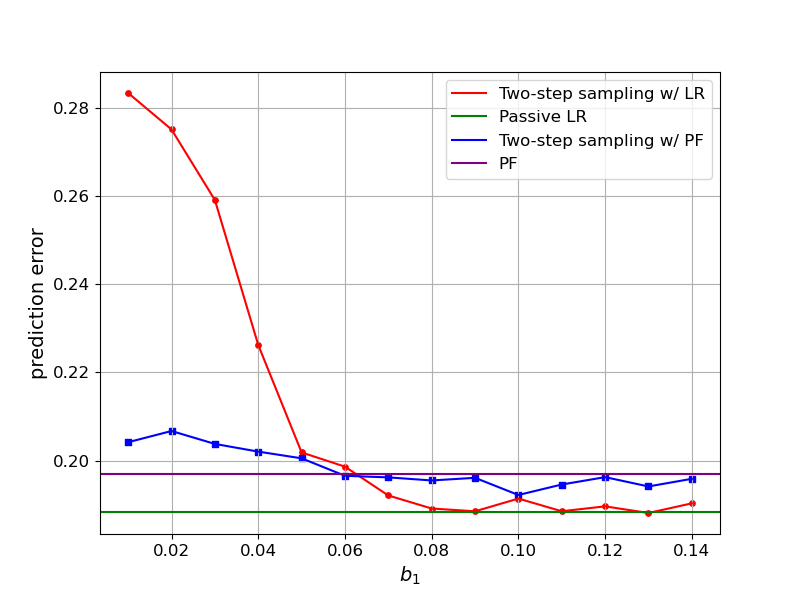}} \vskip -5pt 
\caption{$\|\hat{\btheta} - \btheta^*\|_{\infty}$ and prediction error for binary response model. LR: $\ell_1$ penalized logistic regression; PF: path-following algorithm. 1/8 of the total budget of labeled data is used in the first step for both two-step sampling methods.}
\label{manski_infpred_8}
\end{figure}

\begin{figure}
\centering
\subfigcapskip -5pt
\subfigure[]{
\includegraphics[width=0.48\textwidth]{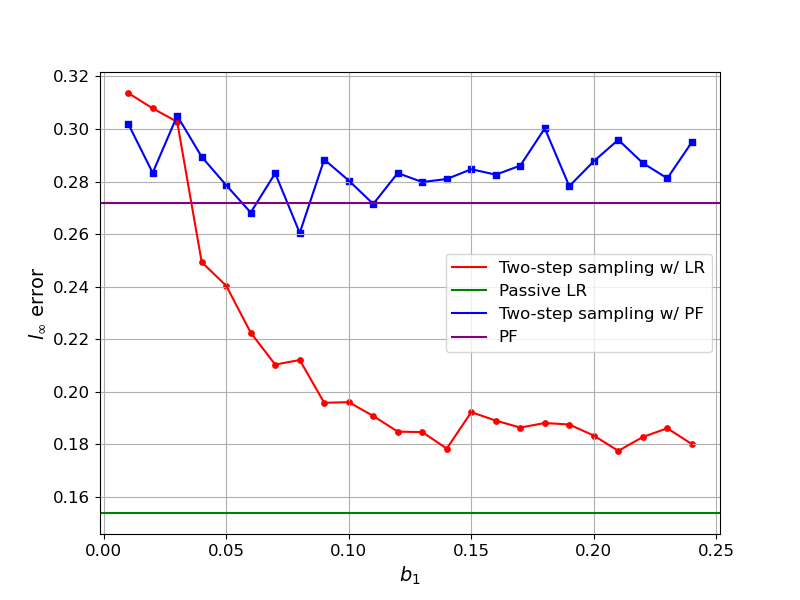}}
\subfigure[]{
\includegraphics[width=0.48\textwidth]{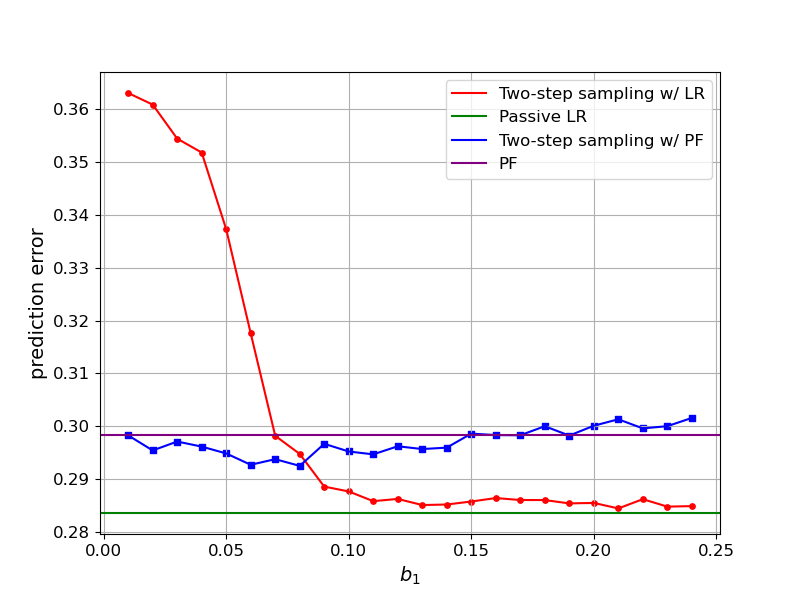}} \vskip -5pt 
\caption{$\|\hat{\btheta} - \btheta^*\|_{\infty}$ and prediction error for logistic regression. LR: $\ell_1$ penalized logistic regression; PF: path-following algorithm. 1/5 of the total budget of labeled data is used in the first step for both two-step sampling methods.}
\label{log_infpred_5}
\end{figure}

\begin{figure}
\centering
\subfigcapskip -5pt
\subfigure[]{
\includegraphics[width=0.48\textwidth]{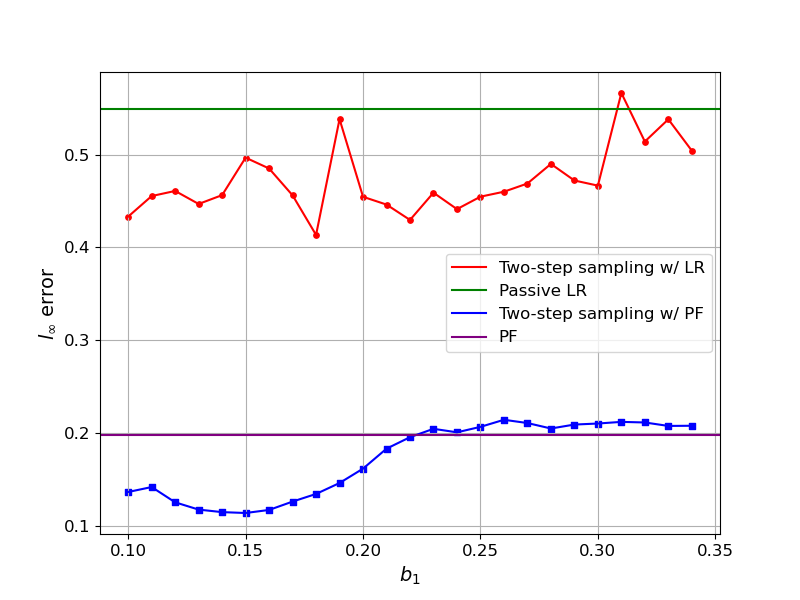}}
\subfigure[]{
\includegraphics[width=0.48\textwidth]{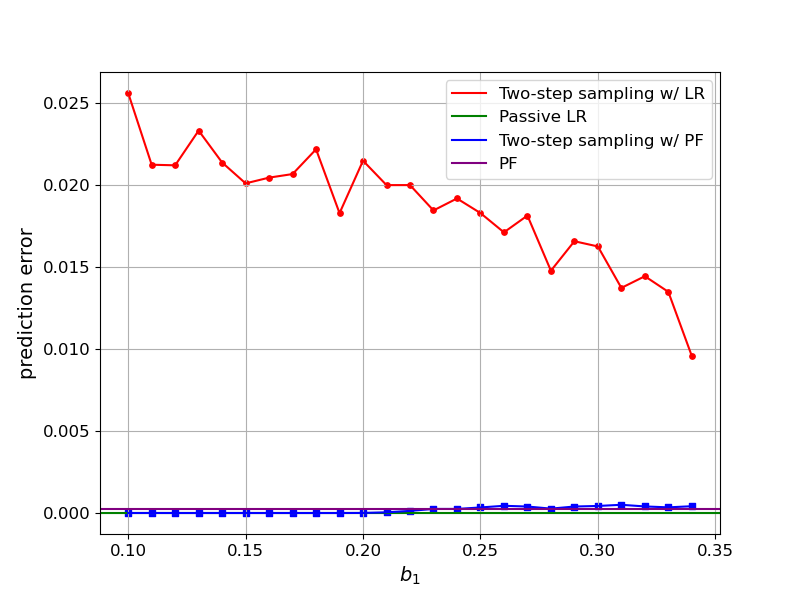}} \vskip -5pt 
\caption{$\|\hat{\btheta} - \btheta^*\|_{\infty}$ and prediction error for conditional mean model. LR: $\ell_1$ penalized logistic regression; PF: path-following algorithm. 1/5 of the total budget of labeled data is used in the first step for both two-step sampling methods.}
\label{condi_infpred_5}
\end{figure}

\begin{figure}
\centering
\subfigcapskip -5pt
\subfigure[]{
\includegraphics[width=0.48\textwidth]{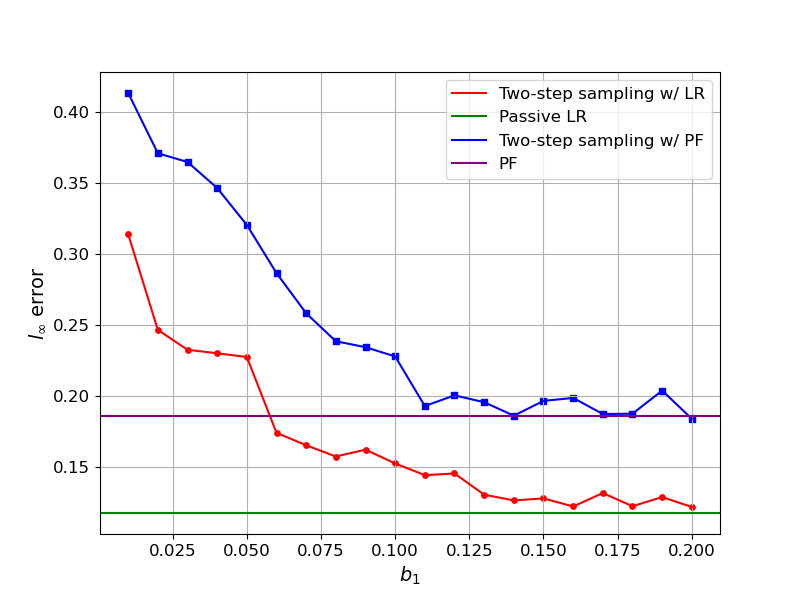}}
\subfigure[]{
\includegraphics[width=0.48\textwidth]{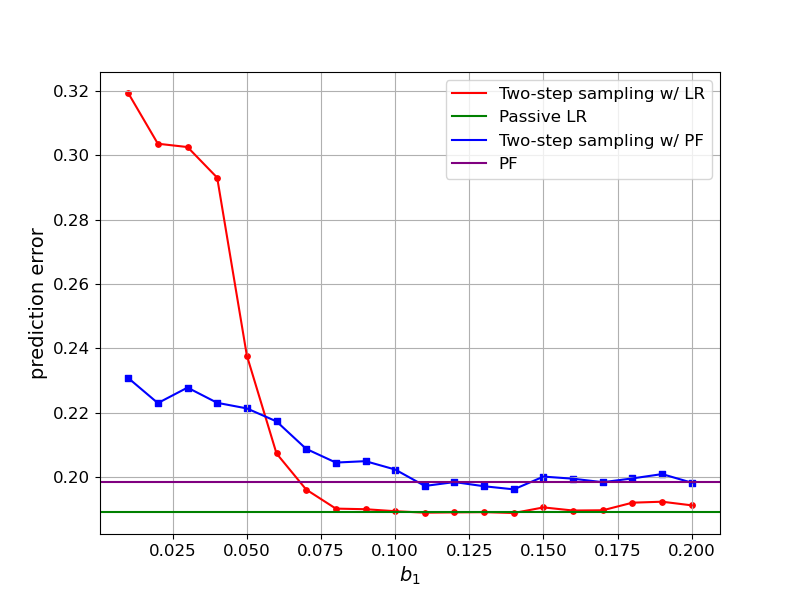}} \vskip -5pt 
\caption{$\|\hat{\btheta} - \btheta^*\|_{\infty}$ and prediction error for binary response model. LR: $\ell_1$ penalized logistic regression; PF: path-following algorithm. 1/5 of the total budget of labeled data is used in the first step for both two-step sampling methods.}
\label{manski_infpred_5}
\end{figure}

\subsection{Additional data analysis}\label{sec_additional_data}
In this section, we conduct an additional real data analysis using the MNIST handwritten digit dataset, a widely used benchmark in machine learning \citep{lecun1998mnist}. The dataset contains 70,000 grayscale images of handwritten digits (0-9), each represented as a $28 \times 28$ pixel array.
To evaluate the effectiveness of our active sampling strategy, we consider a binary classification task using the digits "3" and "5". We choose this pair because they are visually similar and relatively difficult to distinguish, making the classification task more challenging. To further increase the difficulty and to better reflect real-world scenarios in which one class may be underrepresented, we introduce class imbalance.
Specifically, we retain all images of the digit "3" as the majority class ($n_- = 7,141$) and randomly sample images of the digit "5" to create a minority class. We consider two imbalance settings: (i) sampling probability $= 10\%$, resulting in $n_+ = 793$ images of the digit "5"; and (ii) sampling probability $= 20\%$, resulting in $n_+ = 1,785$ images of the digit "5". The former setting represents a more difficult task, where we expect active sampling to provide greater improvement over passive sampling. The response variable is defined as $Y_i = -1$ for digit "3" and $Y_i = 1$ for digit "5". The sample size of the full dataset is $n=n_-+n_+$. 

To incorporate the classification task into our framework, we choose a single pixel (pixel 350) as the primary measurement $X_i$, with the remaining 783 pixels plus an intercept term serving as covariates $\bZ_i \in \mathbb{R}^{784}$. Similar to the setup in Section \ref{sec_data}, our goal is to estimate the optimal individualized threshold $\btheta^T \bZ_i$ such that the digit class can be predicted based on whether the pixel intensity exceeds this threshold ($X_i \geq \btheta^T \bZ_i$) or falls below it ($X_i < \btheta^T \bZ_i$).

We randomly split the dataset into 80\% training set and 20\% test set, with stratified sampling to preserve the class proportions. All pixel values are standardized to have zero mean and unit variance using statistics computed from the training set. 
We compare the performance of our proposed two-step active subsampling method (Active PF) with the passive path-following method (Passive PF), where the labeled data are uniformly sampled from the training set. We consider label budgets of $N = 250, 500, 1000$, and $2000$. For each budget, the experiment is repeated 10 times with different random seeds for sampling, while keeping the training/test split fixed.
Since the dataset is imbalanced, we use the F1 score as the primary evaluation metric, which accounts for both precision and recall and is more informative than accuracy in imbalanced settings. For completeness, we also report the test error.

The results are summarized in Figures~\ref{fig: data imbalanced 10}--\ref{fig: data imbalanced 20} and Tables~\ref{tab: data imbalanced 10}--\ref{tab: data imbalanced 20}. 
Figures~\ref{fig: data imbalanced 10} and \ref{fig: data imbalanced 20} display the averaged F1 scores and test errors over 10 trials under the two imbalance settings. Recall that an F1 score of 1 corresponds to perfect classification performance. In both cases, our proposed active subsampling method consistently outperforms the passive path-following method across all label budgets. The improvement is more substantial in the more challenging setting with sampling probability $= 10\%$, where the F1 score improves by approximately 0.03--0.05 compared to passive sampling. When the sampling probability is $20\%$, the improvement is smaller but still consistent, ranging from 0.01 to 0.02.
Similar patterns are observed for test errors. The active subsampling method achieves lower test errors across all budgets in both settings, with a larger performance gap in the more severely imbalanced case.

Moreover, Tables~\ref{tab: data imbalanced 10} and \ref{tab: data imbalanced 20} show that the active subsampling method not only achieves higher averaged F1 scores but also exhibits smaller standard deviations in most cases, indicating more stable performance. For example, when the sampling probability is $10\%$ and the label budget is $N = 250$, the standard deviation of F1 score is 0.025 for Active PF compared to 0.045 for Passive PF. 
In conclusion, our proposed two-step active subsampling method consistently outperforms passive sampling across different budget levels and imbalance settings, with larger improvements in more challenging scenarios.

\begin{figure}
\centering
\subfigcapskip -5pt
\subfigure[]{
\includegraphics[width=0.48\textwidth]{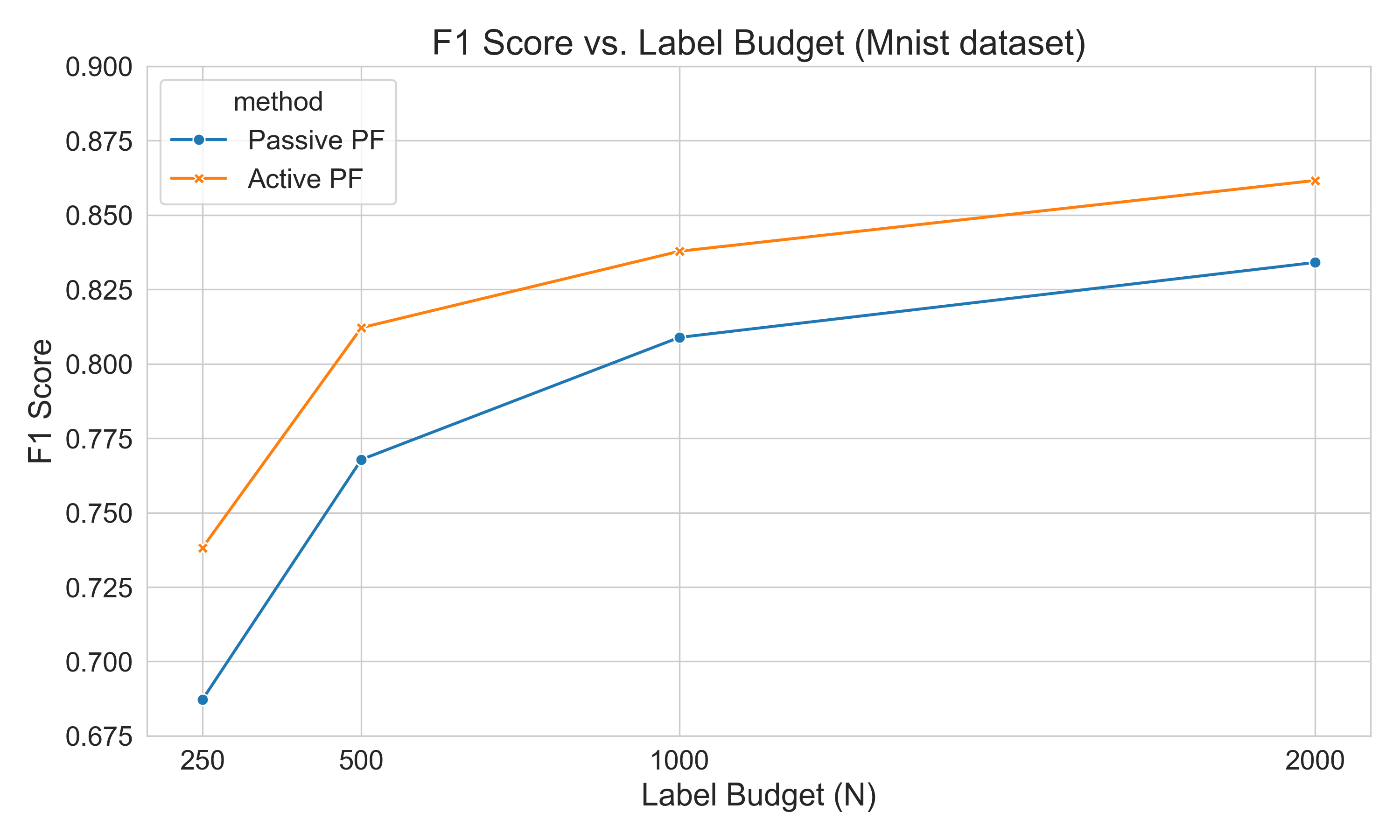}}
\subfigure[]{
\includegraphics[width=0.48\textwidth]{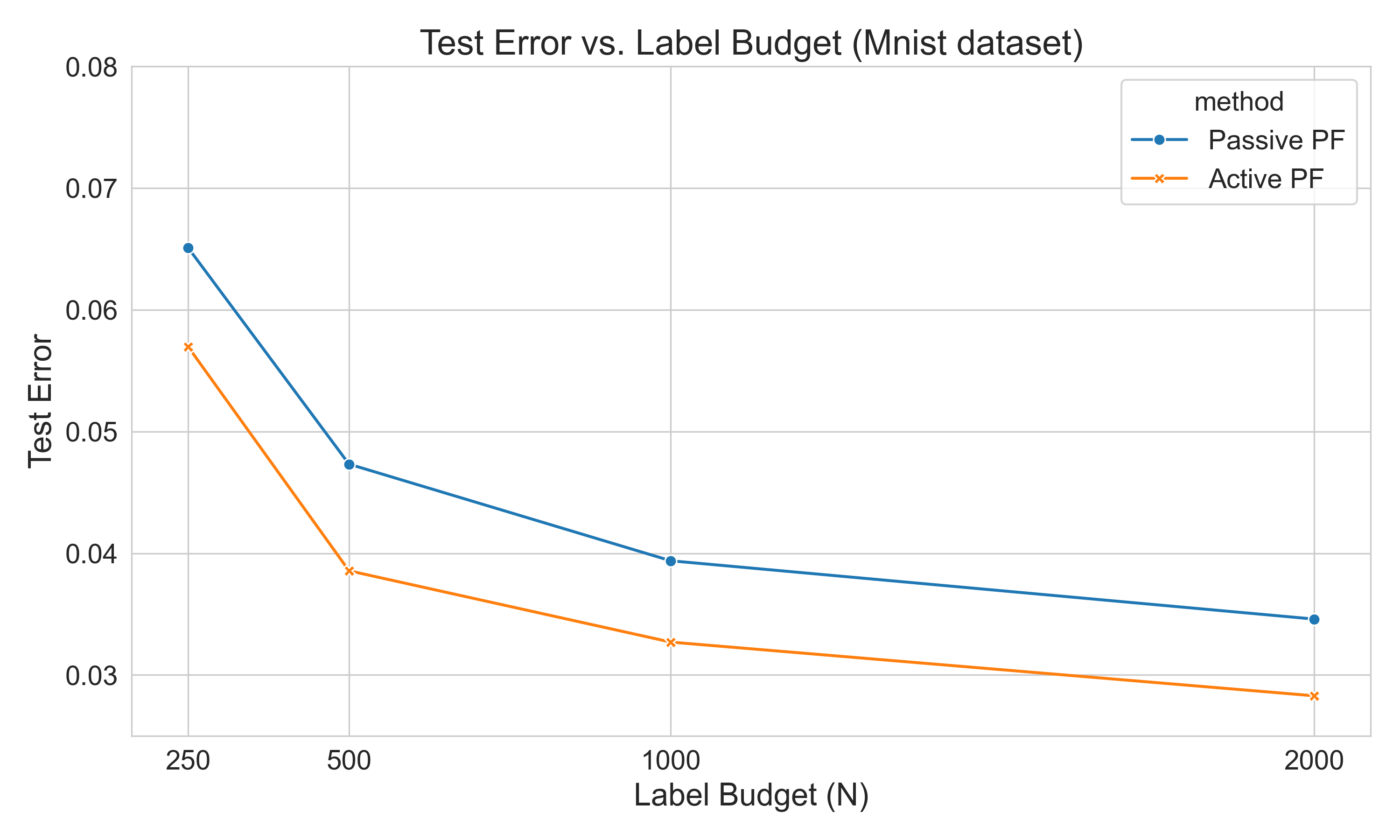}} \vskip -5pt 
\caption{F1 Score and test error vs. label budget with sampling probability $= 10\%$}
\label{fig: data imbalanced 10}
\end{figure}

\begin{table}[htbp]
\centering
\caption{Comparison of Active PF and Passive PF with sampling probability $= 10\%$. Numbers in parentheses are standard deviations. Diff = Active PF $-$ Passive PF for F1 Score and Test Error. Win Rate denotes the proportion of trials where Active PF outperforms Passive PF.}
\label{tab: data imbalanced 10}
\begin{tabular}{c|cccc|cccc}
\hline
 & \multicolumn{4}{c|}{F1 Score} & \multicolumn{4}{c}{Test Error (\%)} \\
$N$ & Passive PF & Active PF & Diff & Win Rate & Passive PF & Active PF & Diff & Win Rate \\
\hline
250  & 0.687(0.045) & 0.738(0.025) & 0.051 & 80\% & 6.51(0.93) & 5.70(0.78) & -0.81 & 80\% \\
500  & 0.768(0.021) & 0.812(0.016) & 0.044 & 90\% & 4.73(0.51) & 3.86(0.37) & -0.88 & 90\% \\
1000 & 0.809(0.009) & 0.838(0.013) & 0.029 & 100\% & 3.94(0.25) & 3.27(0.29) & -0.67 & 100\% \\
2000 & 0.834(0.016) & 0.862(0.012) & 0.028 & 100\% & 3.46(0.38) & 2.83(0.29) & -0.63 & 100\% \\
\hline
\end{tabular}
\end{table}

\begin{figure}
\centering
\subfigcapskip -5pt
\subfigure[]{
\includegraphics[width=0.48\textwidth]{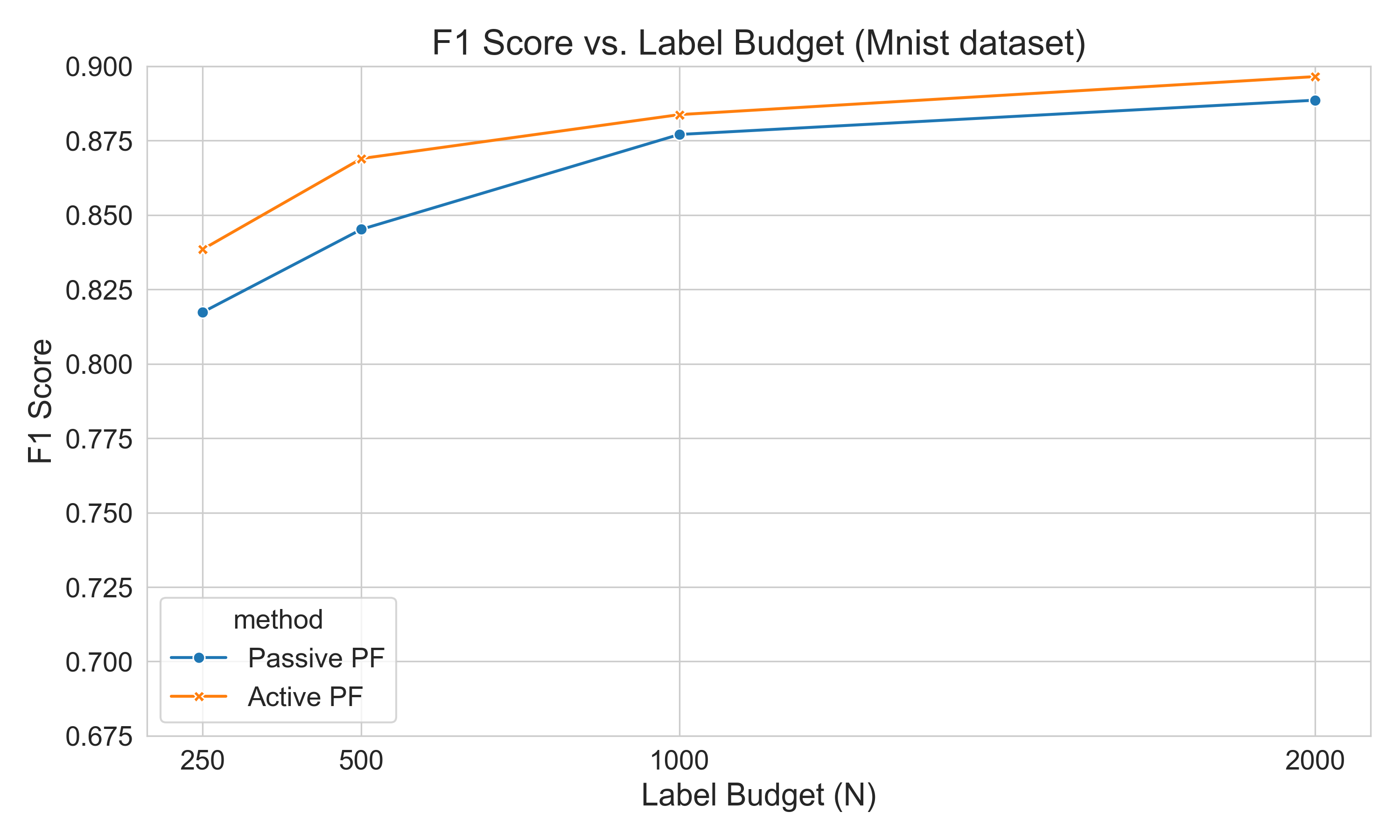}}
\subfigure[]{
\includegraphics[width=0.48\textwidth]{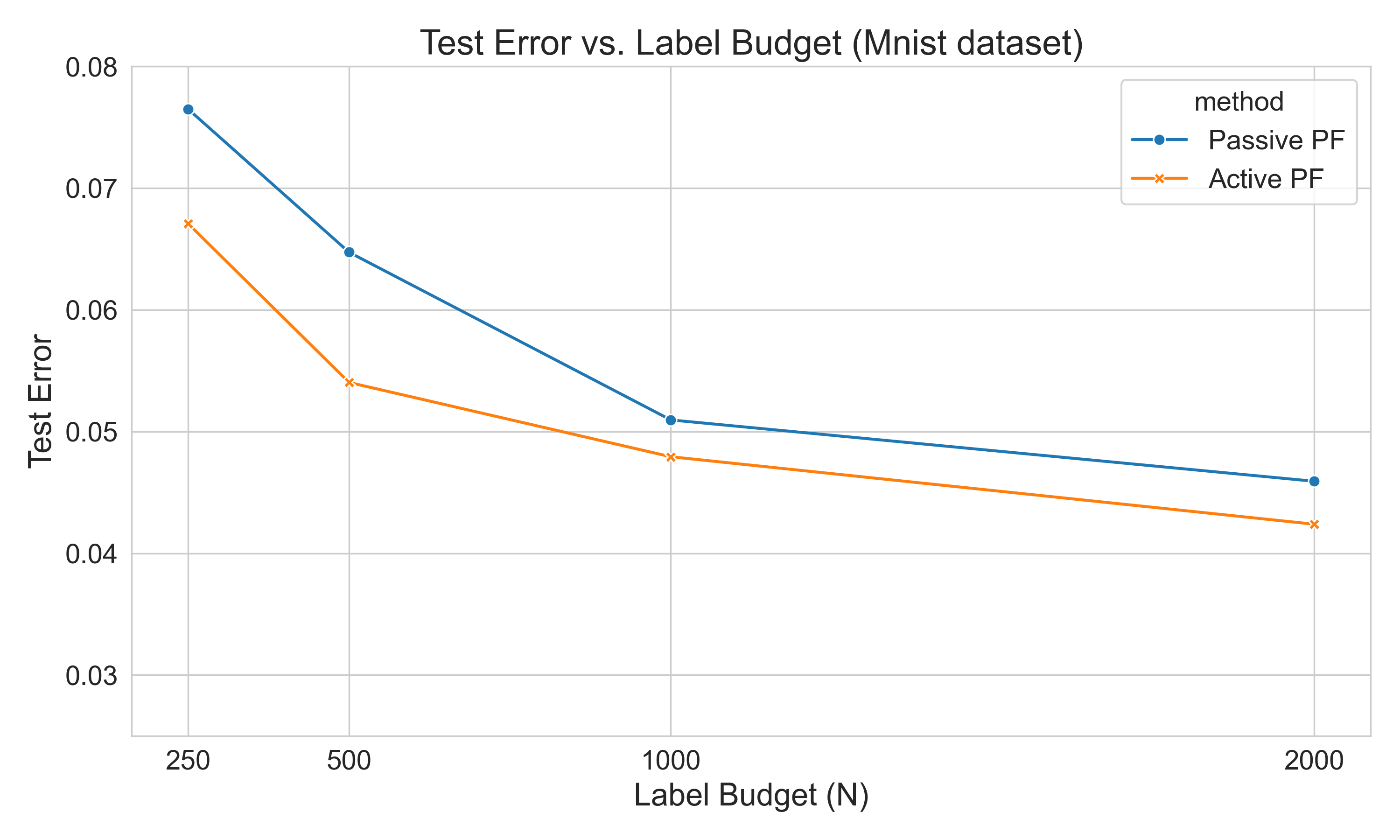}} \vskip -5pt 
\caption{F1 Score and test error vs. label budget with sampling probability $= 20\%$}
\label{fig: data imbalanced 20}
\end{figure}

\begin{table}[htbp]
\centering
\caption{Comparison of Active PF and Passive PF with sampling probability $= 20\%$. Numbers in parentheses are standard deviations. Diff = Active PF $-$ Passive PF for F1 Score and Test Error. Win Rate denotes the proportion of trials where Active PF outperforms Passive PF.}
\label{tab: data imbalanced 20}
\begin{tabular}{c|cccc|cccc}
\hline
 & \multicolumn{4}{c|}{F1 Score} & \multicolumn{4}{c}{Test Error (\%)} \\
$N$ & Passive PF & Active PF & Diff & Win Rate & Passive PF & Active PF & Diff & Win Rate \\
\hline
250  & 0.817(0.022) & 0.838(0.009) & 0.021 & 80\% & 7.65(0.97) & 6.71(0.60) & -0.94 & 80\% \\
500  & 0.845(0.019) & 0.869(0.013) & 0.024 & 90\% & 6.47(0.80) & 5.40(0.63) & -1.07 & 80\% \\
1000 & 0.877(0.009) & 0.884(0.011) & 0.007 & 60\% & 5.10(0.43) & 4.79(0.50) & -0.30 & 60\% \\
2000 & 0.889(0.010) & 0.897(0.010) & 0.008 & 50\% & 4.59(0.42) & 4.24(0.39) & -0.35 & 50\% \\
\hline
\end{tabular}
\end{table}

\end{document}